\documentclass[twoside, usletter]{article}
\usepackage[textsize=scriptsize]{todonotes}

\usepackage{etoolbox}
\usepackage[margin=1in,marginparwidth=0.8in]{geometry}
\usepackage{tikz-cd}
\usepackage{amsmath}
\usepackage{amssymb}
\usepackage{amsthm}
\usepackage{amscd}
\usepackage{enumerate}
\usepackage[pdfusetitle,unicode,hidelinks]{hyperref}
\usepackage{bbm}
\usepackage{etoolbox}

\usepackage[utf8]{inputenc}
\usepackage[T1]{fontenc}

\newtheorem{proposition}[equation]{Proposition}
\newtheorem{corollary}[equation]{Corollary}
\newtheorem{lemma}[equation]{Lemma}
\newtheorem{theorem}[equation]{Theorem}

\newtheorem*{conjecture*}{Conjecture}
\newtheorem*{theorem*}{Theorem}
\newtheorem*{corollary*}{Corollary}
\newtheorem*{proposition*}{Proposition}
\newtheorem*{lemma*}{Lemma}
\theoremstyle{definition}
\newtheorem{definition}[equation]{Definition}
\newtheorem{construction}[equation]{Construction}

\newtheorem*{definition*}{Definition}
\newtheorem*{construction*}{Construction}
\newtheorem{remark}[equation]{Remark}

\newtheorem*{remark*}{Remark}
\newtheorem*{variant*}{Variant}

\newtheorem{example}[equation]{Example}

\newtheorem*{example*}{Example}

\newcommand{\id}{\operatorname{id}}
\newcommand{\Z}{\mathbb{Z}}

\newcommand{\N}{\mathbb{N}}
\newcommand{\Q}{\mathbb{Q}}
\newcommand{\F}{\mathbb{F}}

\let\scr=\mathcal
\let\bb=\mathbb
\newcommand{\Gm}{{\mathbb{G}_m}}

\def\A{\bb A}
\def\P{\bb P}

\newcommand{\eff}{{\text{eff}}}
\newcommand{\veff}{{\text{veff}}}
\newcommand{\DM}{\mathrm{DM}}
\newcommand{\SH}{\mathrm{SH}}
\newcommand{\Spt}{\mathrm{Spt}}
\newcommand{\D}{\mathrm{D}}
\newcommand{\Perf}{\mathrm{Perf}}
\newcommand{\Spc}{\mathrm{Spc}}

\DeclareMathOperator*{\colim}{colim}

\let\lim=\relax
\DeclareMathOperator*{\lim}{lim}
\def\Map{\mathrm{Map}}

\def\map{\mathrm{map}}
\def\CAlg{\mathrm{CAlg}}

\def\Cat{\mathcal{C}\mathrm{at}{}}

\def\Fun{\mathrm{Fun}}

\newcommand{\Spec}{\mathrm{Spec}}

\newcommand{\wequi}{\simeq}
\newcommand{\Mod}{\mathrm{Mod}}

\DeclareRobustCommand{\ul}{\underline}

\newcommand{\Hom}{\operatorname{Hom}}

\def\op{\mathrm{op}}

\let\cat=\mathrm
\def\Sm{{\cat{S}\mathrm{m}}}
\def\Sch{\cat{S}\mathrm{ch}{}}

\def\Nis{\mathrm{Nis}}
\def\Zar{\mathrm{Zar}}
\def\cdh{\mathrm{cdh}}

\newcommand{\et}{{\acute{e}t}}

\newcommand{\lra}[1]{\langle #1 \rangle}
\def\ph{\mathord-}

\newtoggle{final}
\toggletrue{final}

\input{xy}
\xyoption{all}

\DeclareFontFamily{T1}{cbgreek}{}
\DeclareFontShape{T1}{cbgreek}{m}{n}{<-6>  grmn0500 <6-7> grmn0600 <7-8> grmn0700 <8-9> grmn0800 <9-10> grmn0900 <10-12> grmn1000 <12-17> grmn1200 <17-> grmn1728}{}
\DeclareSymbolFont{quadratics}{T1}{cbgreek}{m}{n}
\DeclareMathSymbol{\qoppa}{\mathord}{quadratics}{19}
\DeclareMathSymbol{\Qoppa}{\mathord}{quadratics}{21}

\newcommand{\sm}{\mathrm{sm}}
\newcommand{\BGL}{\mathrm{BGL}}

\newcommand{\GL}{\mathrm{GL}}
\newcommand{\NB}[1]{\todo[color=gray!40]{#1}}

\newcommand{\tom}[1]{\todo[color=green]{#1}}
\newcommand{\elden}[1]{\todo[color=purple]{#1}}
\newcommand{\matthew}[1]{\todo[color=yellow]{#1}}
\newcommand{\KU}{\mathrm{KU}}
\newcommand{\lse}{\mathrm{lse}}
\renewcommand{\et}{\text{\'et}}

\newcommand{\PShv}{\mathrm{PShv}}

\newcommand{\Fil}{\operatorname{Fil}}
\newcommand{\gr}{\mathrm{gr}}
\newcommand{\Gr}{\operatorname{Gr}}
\newcommand{\Alg}{\mathrm{Alg}}

\newcommand{\qcqs}{\mathrm{qcqs}}
\newcommand{\KGL}{\mathrm{KGL}}
\newcommand{\KO}{\mathrm{KO}}
\newcommand{\KW}{\mathrm{KW}}

\newcommand{\HW}{\mathrm{HW}}

\newcommand{\kgl}{\mathrm{kgl}}
\newcommand{\K}{\mathrm{K}}
\renewcommand{\H}{\mathrm{H}}
\newcommand{\KH}{\mathrm{KH}}

\newcommand{\Vect}{\mathrm{Vect}}

\newcommand{\comp}{{{\kern -.5pt}\wedge}}
\newcommand{\syn}{\mathrm{syn}}
\newcommand{\cn}{\mathrm{cn}}

\renewcommand{\L}{\mathrm{L}}
\def\map{\mathrm{map}}

\newcommand{\h}{\mathrm{h}}

\newcommand{\Sel}{\mathrm{Sel}}

\newcommand{\TC}{\mathrm{TC}}

\newcommand{\fib}{\mathrm{fib}}
\newcommand{\cof}{\mathrm{cof}}
\newcommand{\THH}{\mathrm{THH}}
\newcommand{\Pic}{\mathcal{P}\mathrm{ic}}
\usepackage{relsize}
\usepackage[bbgreekl]{mathbbol}
\usepackage{amsfonts}
\DeclareSymbolFontAlphabet{\mathbb}{AMSb} 
\DeclareSymbolFontAlphabet{\mathbbl}{bbold}
\newcommand{\Prism}{{\mathlarger{\mathbbl{\Delta}}}}

\numberwithin{equation}{section}

\theoremstyle{proposition}

\usepackage{framed}
\renewcommand{\bb}[1]{\mathbb{#1}}
\newcommand{\sub}[1]{{\mbox{\rm \scriptsize #1}}}
\newcommand{\roi}{\mathcal{O}}
\newcommand{\isoto}{\stackrel{\simeq}{\to}}
\newcommand{\quis}{\stackrel{\sim}{\to}}
\newcommand{\opp}{\operatorname}
\renewcommand{\hat}{\widehat}

\newcommand{\To}{\longrightarrow}
\renewcommand{\cal}{\mathcal}
\newcommand{\comment}[1]{}
\newcommand{\Hyp}{\mathrm{Val}}
\renewcommand{\Spt}{\mathrm{Sp}}
\newcommand{\xto}{\xrightarrow}
\renewcommand{\tilde}{\widetilde}
\newcommand{\ep}{\epsilon}

\newcommand{\al}{\alpha}
\newcommand{\into}{\hookrightarrow}
\newcommand{\onto}{\twoheadrightarrow}

\newcommand{\slice}{\mathrm{slice}}

\renewcommand{\PShv}{\opp{PSh}}

\usepackage{sectsty}
\sectionfont{\Large\sc\centering}
\chapterfont{\large\sc\centering}
\chaptertitlefont{\LARGE\centering}
\partfont{\centering}

\iftoggle{final} {
\renewcommand{\todo}[1]{}
\renewcommand{\NB}[1]{}
\renewcommand{\tom}[1]{}
\renewcommand{\matthew}[1]{}
\renewcommand{\elden}[1]{}
}

\usepackage{fancyhdr}

\title{$\bb A^1$-invariant motivic cohomology of schemes}
\date{\today}

\author{Tom Bachmann\footnote{University of Mainz}, Elden Elmanto\footnote{University of Toronto}, Matthew Morrow\footnote{CNRS et Laboratoire de Math\'ematiques d'Orsay}}

\begin{document}

\maketitle

\begin{abstract}
Voevodsky outlined a conjectural programme that his slice filtration in motivic homotopy theory should give rise to a good theory of $\bb A^1$-invariant motivic cohomology. This paper achieves his vision in the generality of arbitrary quasicompact, quasiseparated schemes, by introducing a theory of $\bb A^1$-invariant motivic cohomology which is related to Weibel's homotopy $K$-theory via an Atiyah--Hirzebruch spectral sequence, and which we compare to \'etale and syntomic cohomology in the style of the original conjectures of Beilinson and Lichtenbaum. In addition, it is represented by an absolute motivic spectrum and therefore satisfies cdh descent, and modules over it offer a candidate for the derived category of $\bb A^1$-invariant motives. We establish some of Voevodsky's open conjectures on slices, in particular relating the zeroth slice of the motivic sphere to homotopy $K$-theory. In the final section we prove analogous results for the Hermitian $K$-theory of qcqs schemes on which $2$ is invertible.

As an auxiliary tool we introduce cdh-motivic cohomology, defined as the cdh sheafification of the left Kan extension of the motivic cohomology of smooth $\Z$-schemes. We offer a new approach to control the latter, independent of previous work on $\bb A^1$-invariant motivic cohomology of smooth schemes over mixed characteristic Dedekind domains: our approach is based on recent developments in $p$-adic cohomology, in particular syntomic and prismatic cohomology. The cdh-motivic cohomology is also a necessary ingredient in the last two authors' and Bouis' construction of non-$\bb A^1$-invariant motivic cohomology of qcqs schemes. 
\end{abstract}

\tableofcontents

\section{Introduction}
Beilinson and Lichtenbaum \cite{Beilinson1987a, bms-zero, Lichtenbaum1984} envisaged that the algebraic $K$-theory of any reasonable scheme would be equipped with a filtration whose graded pieces would correspond to a conjectural theory of motivic cohomology. The associated motivic cohomology groups would then form the $E_2$-page of a spectral sequence abutting to the algebraic $K$-groups, thereby providing an analogue for schemes of the Atiyah--Hirzebruch spectral sequence of a topological space, whose $E_2$-page is singular cohomology and which abuts to topological $K$-theory. Such a theory of motivic cohomology for arbitrary quasicompact quasiseparated schemes, which we stress is not $\bb A^1$-invariant on arbitrary schemes, is constructed in recent work by the last two authors \cite{ElmantoMorrow2023} (for schemes of equal characteristic) and by Bouis \cite{Bouis2025} (in the general case).

This paper is concerned with the $\mathbb{A}^1$-invariant approximation of motivic cohomology. That is, instead of studying a motivic filtration on algebraic $K$-theory itself, here we focus on Weibel's homotopy $K$-theory $\KH$ \cite{Weibel1989a}, which is the universal $\mathbb{A}^1$-invariant approximation to algebraic $K$-theory. While this agrees with algebraic $K$-theory on regular schemes (where $K$-theory was shown to be $\mathbb{A}^1$-invariant by Quillen), the theories differ in general. For example, it is know that homotopy $K$-theory is insensitive to nilpotents, whereas algebraic $K$-theory itself is not, as one can see already in unit groups and so in $\K_1$. Even better, thanks to a theorem of Cisinski \cite{Cisinski2013}, generalizing an earlier result of Haesemeyer in characteristic zero \cite{Haesemeyer2004}, homotopy $K$-theory converts abstract blowup squares of schemes to cartesian squares of spectra; in other words, it is a cdh sheaf. This is a reflection of the fact that homotopy $K$-theory may be analysed through the powerful lens of Morel--Voevodsky's motivic homotopy theory \cite{MorelVoevodsky1999}.

The results of this paper serve two distinct purposes, reflecting the above discussion. The first is as a foundation for the construction of the motivic filtration on actual algebraic $K$-theory \cite{ElmantoMorrow2023, Bouis2025}, which involves glueing a suitable motivic filtration on topological cyclic homology to a filtration on homotopy $K$-theory; this paper constructs and studies the latter filtration. (Conversely, once the actual motivic cohomology has been constructed, the $\bb A^1$-invariant theory of this paper may be obtained by imposing $\bb A^1$-invariance on non-$\bb A^1$-invariant theory; see Remark \ref{rem:nonA1_intro}.) Secondly, our results realise Voevodsky's vision of a motivic filtration on $\KH$ via his slice filtration in motivic homotopy theory, as laid out in \cite{Voevodsky2002a,Voevodsky2002}. Indeed, within the framework of motivic homotopy theory, one of the principal outcomes of this paper is a description of the zeroth slice of the motivic sphere spectrum of any qcqs scheme: we show that it represents $\bb A^1$-invariant motivic cohomology, analogously to the fact in topology that the zeroth stable homotopy group of spheres is isomorphic to $\mathbb{Z}$.

\subsection{$\bb A^1$-invariant  motivic cohomology and its main properties}
Ignoring for a moment our applications to motivic homotopy theory, the following theorem summarises our theory of $\mathbb{A}^1$-invariant motivic  cohomology of arbitrary qcqs schemes:

\begin{theorem}\label{thm:main-intro}
There exists a multiplicative family of presheaves of complexes 
\[
\Z(j)^{\bb A}: \Sch^\sub{qcqs,op}\To \D(\Z) \qquad j \geq 0
\]
called \emph{$\bb A^1$-motivic cohomology} which enjoys the following properties:
\begin{enumerate}
\item Atiyah--Hirzebruch filtration: for any qcqs scheme $X$ there exists a natural, multiplicative, $\bb N$-indexed  filtration $\mathrm{Fil}_{\bb A}^{\star}\KH(X)$ on the homotopy $K$-theory $\KH(X)$, such that the graded pieces are naturally and multiplicatively given by
\[
\mathrm{gr}_{\bb A}^j\KH(X)\simeq \Z(j)^{\bb A}(X)[2j],\qquad j\ge0.
\]
 In particular, writing $H^i_{\bb A}(X,\bb Z(j)):=H^i(\bb Z(j)^{\bb A}(X))$ for the corresponding {\em $\bb A^1$-motivic cohomology groups}, there exists an Atiyah--Hirzebruch spectral sequence
\[
E_2^{i,j}=\H_{\bb A}^{i-j}(X, \Z(-j)) \implies \KH_{-i-j}(X).
\]
If $X$ has finite valuative dimension $\le d$ then this filtration is complete: more precisely, $\Fil_{\bb A}^j\KH(X)$ is supported in cohomological degrees $\le d-j$.
\item\label{item-A1-invar} $\bb A^1$-invariance: for any qcqs scheme $X$, the map
\[
\Z(j)^{\bb A}(X) \To \Z(j)^{\bb A}(\A_X^1),
\]
induced by the projection $\A_X^1 \rightarrow X$, is an equivalence for each $j\ge0$.
\item Finitariness and sheafiness: $\bb Z(j)^{\bb A}$ is a finitary cdh sheaf on the category of qcqs schemes.

\item Relation to \'etale cohomology: for any integer $\ell>0$, there are natural, multiplicative equivalences
\[
\Z(j)^{\bb A}/\ell \simeq L_{\cdh}\tau^{\leq j}R\Gamma_{\et}(-,\mu_{\ell}^{\otimes j})\qquad j\ge0
\]
on the category of qcqs $\bb Z[\tfrac1\ell]$-schemes. In particular, the Beilinson--Lichtenbaum isomorphisms hold: for any qcqs $\bb Z[\tfrac1\ell]$-scheme $X$ there are natural isomorphisms \[H^i_\bb A(X,\bb Z/\ell(j))\cong H^i_\sub{\'et}(X,\mu_\ell^{\otimes j})\] in the range $i\le j$.

\item Relation to syntomic cohomology in characteristic $p$: for any prime $p$ and $r\ge1$ there are 
natural, multiplicative equivalences
\[
\Z(j)^{\bb A}(X)/p^r \simeq R\Gamma_{\cdh}(X,W_r\Omega^j_{\log})[-j],\qquad j\ge0,
\]
on the category of qcqs $\bb F_p$-schemes. For the case of schemes in mixed characteristic, see Theorem \ref{thm:BL_intro}.

\item Low weights: for any qcqs scheme $X$ there are natural equivalences
\[
R\Gamma_{\cdh}(X, \bb Z) \xrightarrow{\simeq} \bb Z(0)^{\bb A}(X)
\]
and
\[c_1^{\bb A}: R\Gamma_{\cdh}(X,\bb G_m)[-1] \xrightarrow{\simeq} \bb Z(1)^{\bb A}(X). \]
\item \label{item-PBF} Projective bundle formula: for any qcqs scheme $X$ the powers of the first Chern class of the tautological bundle $c_1^\bb A(\scr O(1)) \in H^{2}_{\bb A}(\P^d_X, \Z(1))$
induce a natural equivalence
\begin{equation}\bigoplus_{i=0}^d\bb Z(j-i)^{\bb A}(X)[-2i]\quis \bb Z(j)^{\bb A}(\bb P^d_X).\end{equation}

\item Comparison to classical motivic cohomology: for any smooth algebraic variety $X$, there are equivalences
\[
\Z(j)^{\bb A}(X) \simeq z^j(X,\bullet)[-2j],\qquad j\ge0,
\]
where $ z^j(X,\bullet)$ is Bloch's cycle complex \cite{Bloch1986b}. (For the case of smooth schemes over mixed characteristic Dedekind domains, see Remark \ref{rem:Spitzweck}.)

\item Determined by smooth schemes: letting $\bb Z(j)^{\cdh}:\Sch^{\qcqs,\op}\to\D(\bb Z)$ be the cdh sheafification of the left Kan extension of $\bb Z(j)^{\bb A}|_{\Sm_{\bb Z}}$, the canonical maps
\[
L_{\bb A^1}\bb Z(j)^{\cdh} \To \bb Z(j)^{\bb A},\qquad j\ge0,
\]
(arising from properties (2) and (3)) are equivalences.

\end{enumerate}
\end{theorem}

\begin{remark}[cdh topology]
The cdh topology and cohomology of Suslin--Voevodsky \cite{SuslinVoevodsky2000,ElmantoHoyoisIwasaKelly2021} appears repeatedly in Theorem \ref{thm:main-intro} and plays an important role in both the constructions and proofs of this paper. Indeed, as already mentioned above, this is due to the surprising fact that $\KH$ is a cdh sheaf. Meanwhile, concerning our proofs, the cdh topology allows us to reduce many statements to the case of henselian valuation rings (i.e., points for the cdh topology), or even using excision to henselian valuation rings $V$ of rank $\le 1$, where we then leverage the fact that small schemes over such a ring (e.g., $\bb P_V^1$, $\bb A_V^1$) still have small dimension for the cdh topology. This is in contrast to the Zariski or Nisnevich topology, where even the affine line over a local ring can have large cohomological dimension.
\end{remark}

 \subsection{Voevodsky's slice conjectures}\label{subsec:vv-approach}
We broadly follow and realize Voevodsky's conjectural vision for the construction of a motivic filtration on $\KH$, as explained in his survey articles from more than two decades ago \cite{Voevodsky2002a,Voevodsky2002}. We now briefly sketch his ideas. In algebraic topology, we can functorially equip any spectrum $E \in \Spt$ with its Postnikov filtration\footnote{Sometimes also called its \emph{Whitehead tower}.} $\Fil_\sub{Pos}^{\star}E$; this is a decreasing, multiplicative, exhaustive filtration on $E$ whose graded pieces are given by shifts of the Eilenberg--MacLane spectra of its homotopy groups:
 \[
 \gr^j_\sub{Pos}E:=\mathrm{cofib}(\Fil_\sub{Pos}^{j+1}E \to \Fil_\sub{Pos}^jE) \simeq (H\pi_j(E))[j].
 \]
The Postnikov filtration originates from the filtration on the whole $\infty$-category of spectra given by
\[
\cdots \subset \Spt_{\geq j} \subset \Spt_{\geq j-1} \subset \cdots \subset\Spt_{\geq 0}=\Spt_\sub{conn} \subset \cdots \subset\Spt,
\]
where $\Spt_{\geq j}$ is the subcategory of $j$-connective spectra, or in other words the full subcategory of $\Spt$ generated under colimits and extensions by $\mathbb{S}[j]$. From this point of view, the spectrum $\Fil^j_\sub{Pos}E$ is then the final spectrum over $E$ which lives in $\Spt_{\geq j}$. The Atiyah--Hirzebruch spectral sequence in topology now arises from the fact that the graded pieces of the Postinkov filtration on $\KU$ are given, as a graded $\mathbb{E}_{\infty}$-ring, by
\[
\mathrm{gr}_\sub{Pos}^{\star}\KU \simeq H\mathbb{Z}[\beta^{\pm 1}],\qquad |\beta| = 2.
\]
In particular, the unit map $\mathbb{S} \rightarrow \KU$ induces an isomorphism on the zeroth graded piece of the Postnikov filtration (i.e., on $H\pi_0$), and their common value is $H\mathbb{Z}$.

We now turn to algebraic geometry. Voevodsky's theory of motivic spectra assigns to each qcqs scheme $X$ a well-behaved stable $\infty$-category $\SH(X)$ which should be thought of as an algebro-geometric version of ``spectra over $X$''. It is equipped with a functor $M_X:\Sm_X\to\SH(X)$ sending any smooth $X$-scheme to its ``motive''; for example, $1_X=M_X(X)$ is the unit of $\SH(X)$ and is the motivic analogue over $X$ of the sphere spectrum. Voevodsky defined a filtration on $\SH(X)$ in an analogous fashion to the topological context:
\[
\cdots \subset \SH(X)^{\eff}(j) \subset \SH(X)^{\eff}(j-1) \subset  \cdots \subset\SH(X)^{\eff}(0) = \SH(X)^{\eff} \subset \cdots \subset \SH(X) 
\]
where $\SH(X)^{\eff}(j)$ is the full stable subcategory of $\SH(X)$ generated under colimits and desuspensions by $M_X(Y) \otimes \mathbb{T}_X^{\otimes j}$, where $Y$ varies over smooth $X$-schemes and $\mathbb{T}_X$ denotes the Tate motive over $X$, i.e., the motive of the reduced projective line. This definition is inspired by thinking of the Tate motive as a kind of 2-sphere. These categories give rise, for any motivic spectrum $E \in \SH(X)$, to its so-called {\em slice filtration} $\Fil_\sub{slice}^{\star}E$, which is again a natural, decreasing, multiplicative, exhaustive filtration on $E$. Its graded pieces $s^\star E$ are known as the slices of $E$. In particular, letting $\KGL_X\in \SH(X)$ be the object representing homotopy $K$-theory on smooth $X$-schemes, the analogy with algebraic topology suggests that the motivic analogue of the Eilenberg--MacLane spectrum of $\bb Z$ should be the motivic ring spectrum
\begin{equation}\label{eq:def-main}
H\mathbb{Z}_X^{\bb A}:=s^0\KGL_X = \mathrm{cofib}(\Fil_\sub{slice}^{1}\KGL_X \rightarrow \Fil_\sub{slice}^{0}\KGL_X), 
\end{equation}
and that this should represent an analogue of singular cohomology. We adopt this as our definition of $\bb A^1$-motivic cohomology:

\begin{definition}\label{def:a1-inv}
For each qcqs scheme $X$ and $j \geq 0$, the \emph{weight-$j$ $\bb A^1$-motivic cohomology} of $X$ is defined as
\[
\bb Z(j)^{\bb A}(X) := \map_{\SH(X)}(1_X, H\mathbb{Z}_X^{\bb A} \otimes \mathbb{T}_X^{\otimes j})\in\Spt.
\]
Here $\map_{\SH(X)}$ refers to the spectrum of maps in the stable $\infty$-category $\SH(X)$.
\end{definition}

Several facts formally follow from this definition. For example, the presheaves $\bb Z(j)^{\bb A}:\Sch^\sub{qcqs,op}\to\Spt$ are $\bb A^1$-invariant Nisnevich sheaves and, since Bott periodicity still holds in the context of motivic homotopy theory, they assemble to form an $\bb E_\infty$-algebra in graded presheaves. On the other hand, it is not at all clear that $\bb Z(j)^{\bb A}$ has the structure of a presheaf of complexes (as opposed to spectra), nor that it is a cdh sheaf, nor that it should be related to \'etale/syntomic cohomology, nor that $\bb Z(1)^{\bb A}$ should be related to line bundles, and so on. Such non-obvious structure is the subject of Theorem \ref{thm:main-intro}.

Meanwhile, internally to motivic homotopy theory, we establish the following structural results about $H\bb Z_X^\bb A$; the first part is a motivic analogue of the fact that $\pi_0(\bb S)=\bb Z$:

\begin{theorem}\label{thm:vv-main}
Let $X$ be a qcqs scheme with structure map $f: X \rightarrow \Spec(\bb Z)$. Then the following hold in $\SH(X)$:
\begin{enumerate}
\item the unit map $1_X \rightarrow \KGL$ induces an equivalence
\[
s^0(1_X) \xrightarrow{\simeq} H\mathbb{Z}_X^{\bb A}
\]
(see Corollary~\ref{cor:conjectures});

\item the canonical comparison map
\[
f^*(H\mathbb{Z}_{\Spec(\bb Z)}^{\bb A}) \rightarrow H\mathbb{Z}_X^{\bb A}
\]
is an equivalence (see Theorem~\ref{thm:a1-a1cdh}(2)).

\end{enumerate}

\end{theorem}

\begin{remark}[Relationship with Voevodsky's slice conjectures]
In \cite{Voevodsky2002a}, Voevodsky made several conjectures about slices in motivic homotopy theory. These conjectures refer to a motivic spectrum $H\bb Z$ representing an $\bb A^1$-invariant theory of motivic cohomology; the intended meaning of this object was clear over a field, but less so over an arbitrary base. Taking $H\bb Z_X^{\bb A}$ as the definition of $H\bb Z$, then \cite[Conjectures 1, 7]{Voevodsky2002a} are vacuous, while Theorem~\ref{thm:vv-main} establishes \cite[Conjectures 10, 17]{Voevodsky2002a}; moreover, the weight zero equivalence of Theorem~\ref{thm:main-intro}(6) establishes \cite[Conjecture 12]{Voevodsky2002a}.
\end{remark}

A consequence of Theorem \ref{thm:vv-main}(2) is that the $\infty$-category of $H\bb Z_X^\bb A$-modules in $\SH(X)$, i.e., the presentably symmetric monoidal stable $\infty$-category \[\opp{Mod}_{H\bb Z_X^\bb A}(\SH(X)),\] is a good candidate for the derived $\infty$-category of $\bb A^1$-invariant motives over $X$. In fact, when $X$ is the spectrum of a characteristic zero field $k$ this is known to agree with Voevodsky's category $\DM(k)$ \cite{rondigs-ostvaer}, and it enjoys a six-functor formalism:

\begin{corollary}[See Corollary \ref{corol:6functor}]
The assignment on qcqs schemes $X\mapsto \opp{Mod}_{H\bb Z_X^\bb A}(\SH(X))$ inherits a six-functor formalism from that of $\SH(-)$.
\end{corollary}

\subsection{Cdh-motivic cohomology and the key hypothesis}
Theorem \ref{thm:main-intro}(9) states, for any fixed weight $j$, that the weight-$j$ $\bb A^1$-motivic cohomology of arbitrary qcqs schemes is obtained from the case of smooth $\bb Z$-schemes via a simple recipe. Namely, letting $\bb Z(j)^{\cdh}:\Sch^{\qcqs,\op}\to\D(\bb Z)$ be the cdh sheafification of the left Kan extension of $\bb Z(j)^{\bb A}|_{\Sm_{\bb Z}}$, there are canonical comparison maps \begin{equation}\bb Z(j)^\sub{cdh}\To L_{\bb A^1}\bb Z(j)^{\cdh} \To \bb Z(j)^{\bb A},\label{eqn:cdhA1intro}\end{equation} and the second map is an equivalence. In light of Definition \ref{def:a1-inv}, this equivalence relates $H\bb Z_{\Spec(\bb Z)}^{\bb A}$ to $H\bb Z_{X}^{\bb A}$ and is thus intimately related to Theorem \ref{thm:vv-main}(2). We refer to $\bb Z(j)^{\cdh}$ as cdh-motivic cohomology; it is a version, for arbitrary qcqs schemes, of the motivic cohomology of singular varieties introduced by Friedlander--Voevodsky (see Remark \ref{rem_FV_intro} for further discussion).

As mentioned above, the second comparison map in \eqref{eqn:cdhA1intro} is an equivalence. It turns out that more is true: for many classes of schemes we can show that the the first comparison map in \eqref{eqn:cdhA1intro} is also an equivalence, i.e., the cdh-motivic cohomology $\bb Z(j)^\sub{cdh}$ is often $\bb A^1$-invariant. Conjecturally this should hold for all schemes. We investigate this phenomenon in \S\ref{ss:key}. For any qcqs scheme $X$, prime $p$, and integer $j\ge0$, we introduce the ``key hypothesis'' $\Hyp(X, p, j)$: it asks, whenever $V$ is a henselian valuation ring of mixed characteristic $(0,p)$ equipped with a map $\Spec(V)\to X$, that the mod-$p$ syntomic cohomology $\F_p(j)^\syn(V)$ is mostly given by the \'etale cohomology of $V[\tfrac1p]$. This is a regularity condition inspired by the Beilinson--Lichtebaum equivalence in motivic cohomology of smooth schemes over mixed characteristic Dedekind domains, and it would follow if such valuation rings were known to be $F$-smooth in the sense of Bhatt--Mathew \cite{BhattMathew2023}; the latter is a non-Noetherian regularity condition defined in terms of prismatic cohomology. It is expected that the hypothesis $\Hyp(X, p, j)$ is {\em always} true, but for the moment this remains conjectural; see Example~\ref{example:hypothesis-holds} for known cases.

We establish the following conditional $\bb A^1$-invariance theorem about the cdh-motivic cohomology $\bb Z(j)^\sub{cdh}$:

\begin{theorem}[See Theorem \ref{thm:A1inv}]\label{thm:A1_for_cdh}
Let $X$ be a qcqs scheme and $j\ge0$. Then the canonical comparison map \[\bb Z(j)^\sub{cdh}(X)\To \bb Z(j)^\sub{cdh}(\bb A_X^1)\] is an equivalence rationally. For a fixed prime number $p$, if $X$ has finite valuative dimension and satisfies the hypothesis $\Hyp(X, p, n)$ for all $n\le j$, then the map is also an equivalence modulo $p$ and so $\bb Z_{(p)}(j)^\sub{cdh}(X)\quis L_{\bb A^1}\bb Z_{(p)}(j)^\sub{cdh}(X)$.
\end{theorem}

For example, whenever $X$ is a scheme of equal characteristic the hypothesis $\Hyp(X, p, j)$ holds for all $p,j$, and so we obtain the following consequence by combining Theorems \ref{thm:main-intro}(9) and \ref{thm:A1_for_cdh}:

\begin{corollary}\label{corol:FV}
For any scheme $X$ of equal characteristic, the comparison map $\bb Z(j)^\sub{cdh}(X)\to \bb Z(j)^{\bb A}(X)$ from \eqref{eqn:cdhA1intro} is an equivalence.
\end{corollary}

\begin{remark}[Friedlander--Voevodsky motivic cohomology]\label{rem_FV_intro}
Let $k$ be a field which is assumed to satisfy strong resolution of singularities. Friedlander--Voevodsky introduced an $\bb A^1$-invariant motivic cohomology theory of finite type $k$-schemes, by taking the universal cdh sheaf extending motivic cohomology from smooth $k$-schemes \cite[Definition 9.2]{FriedlanderVoevodsky2000}. A consequence of the previous corollary is that the $\bb A^1$-motivic cohomology $\bb Z(j)^{\bb A}$ agrees with the Friedlander--Voevodsky theory on such schemes. See Remark \ref{rmk:FV} for further details.

On the other hand, for algebraic varieties over fields of finite characteristic (where resolution of singularities remains unknown), the $\bb A^1$-motivic cohomology provides an unconditional replacement of the Friedlander--Voevodsky theory. One may thus sometimes eliminate the resolution of singularities hypothesis from results in the literature, for example in Geisser's arithmetic cohomology theory \cite{Geisser2006} or in the study of zero-cycles on singular varieties, e.g., \cite{BindaKrishna2022}.
\end{remark}

Similarly to Theorem \ref{thm:A1_for_cdh}, we also prove that the cdh-motivic cohomology satisfies the projective bundle formula rationally, and modulo $p$ conditionally upon the hypothesis; see \S\ref{sec:pbf-fields}. In particular it follows that cdh-motivic cohomology of equicharacteristic schemes satisfies the projective bundle formula, which is a crucial input into the key technical result of this paper, Theorem~\ref{thm:comparison}, stating that the $\bb A^1$-localisation of cdh motivic cohomology unconditionally satisfies the projective bundle formula. In turn, the latter result is a necessary input into showing that the second comparison map in \eqref{eqn:cdhA1intro} is an equivalence.

Thus the key hypothesis primarily appears as a technical input for bridging the gap between the various motivic cohomology theories of the article, allowing us to ultimately prove properties of all of them. However, we also need this condition to prove our optimal Beilinson--Lichtenbaum equivalence in mixed characteristic; note that the next result holds unconditionally over fields, and away from residue characteristics, by Theorem~\ref{thm:main-intro}(4)\&(5):

\begin{theorem}[See Remark \ref{rem:proof_of_BL_from_intro}]\label{thm:BL_intro}
Let $S$ be a qcqs scheme of finite valuative dimension, $p$ a prime number, and $N\ge0$; assume that $\Hyp(S,p,j)$ holds for all $j\le N$. Then, for all $j\le N$ and all qcqs $S$-schemes $X$, the cofiber of the canonical comparison map \[\bb Z(j)^\bb A(X)/p\To R\Gamma_\sub{\'et}(X[\tfrac1p],\mu_p^{\otimes j})\] is supported in cohomological degrees $\ge j-1$.
\end{theorem}

Without knowing the validity of the key hypothesis, we still have the following explicit description of $\bb A^1$-motivic cohomology with finite coefficients: \[\bb Z(j)^\bb A(X)/p\simeq L_{\bb A^1}L_\sub{cdh}\tau^{\le j}\bb F_p(j)^\sub{syn},\] i.e., the $\bb A^1$-localisation of the cdh sheafification of a truncation of mod-$p$ syntomic cohomology. Here we use syntomic cohomology as introduced for $p$-complete rings in \cite{BhattMorrowScholze2, AntieauMathewMorrowNikolaus2022}, and then extended to arbitrary schemes by Bhatt--Lurie \cite{BhattLurie2022} by gluing contributions from $p$-adic \'etale cohomology. In \S\ref{sec:syntomic} we provide a different exposition of how to extend syntomic cohomology beyond $p$-complete rings, using $K$-theory and topological cyclic homology instead of the theory of absolute prismatic cohomology. Syntomic cohomology is one of the main new tools which makes the results of this article possible.

\begin{remark}[Role of cdh-motivic cohomology in non-$\bb A^1$-invariant motivic cohomology]\label{rem:nonA1_intro}
The cdh-motivic cohomology, even without knowing unconditionally whether it coincides with the true $\bb A^1$-motivic cohomology, plays an important role in the construction of the motivic cohomology of qcqs schemes by the last two authors \cite{ElmantoMorrow2023} and Bouis \cite{Bouis2025}. The point of departure of their construction is the pullback square of spectra, natural for each qcqs scheme $X$,
\begin{equation}\label{eq:mainsq}
\begin{tikzcd}
\K(X) \ar{r}{\text{tr}} \ar{d} & \TC(X) \ar{d} \\
\KH(X) \ar[swap]{r}{L_\sub{cdh}\text{tr}} & L_{\cdh}\TC(X),
\end{tikzcd}
\end{equation}
established by Kerz--Strunk--Tamme \cite{KerzStrunkTamme2018} and Land--Tamme \cite{LandTamme2019} (see also \cite[Theorem 3.8]{ElmantoMorrow2023}). Here, $\TC$ is the topological cyclic homology of B\"okstedt--Hsiang--Madsen and Nikolaus--Scholze \cite{BokstedtHsiangMadsen1993, NikolausScholze2018}, the top horizontal map is the cyclotomic trace map, and the bottom horizontal map is its cdh sheafification. The past few years have seen rapid developments of ``motivic'' filtrations on $\TC$ \cite{BhattMorrowScholze2, HahnRaksitWilson2022}; cdh sheafifying then produces a compatible filtration on $L_\sub{cdh}\TC(X)$.

To produce a motivic filtration on $\K(X)$ itself (whose graded pieces will correspond to the desired theory of motivic cohomology) the strategy adopted in \cite{ElmantoMorrow2023, Bouis2025} is to ``glue'' along $L_\sub{cdh}\mathrm{tr}$ the existing filtrations on $\TC(X)$ and $L_\sub{cdh}\TC(X)$ to a suitable motivic filtration on $\KH(X)$. The relevant filtration on $\KH(X)$ to carry out this argument turns out to be that corresponding to cdh-motivic cohomology since, in contexts where the key hypothesis is not known, it remains an open question whether $L_\sub{cdh}\mathrm{tr}$ carries the $\bb A^1$-motivic filtration to the existing motivic filtration on $L_\sub{cdh}\TC(X)$ (the difficulty being that the latter is far from being $\bb A^1$-invariant \cite{Elmanto2021}).

Upon then imposing $\bb A^1$-invariance everywhere, the terms $\TC$, $L_\sub{cdh}\TC$, and their graded pieces mostly vanish, and one deduces that the $\bb A^1$-localisation of the motivic cohomology agrees with $L_{\bb A^1}\bb Z(j)^\sub{cdh}$, which in turn is $\bb Z(j)^\bb A$ by  Theorem \ref{thm:main-intro}(9). See \cite[Remark 6.2]{ElmantoMorrow2023} \cite[Theorem 10.5]{Bouis2025} for precise statements.
\end{remark}

\subsection{Relation to, and dependence on, existing work in motivic cohomology}\label{ss:dep_on_sh}
Many of our main results (e.g., Theorem \ref{thm:main-intro}) make no reference to motivic homotopy theory or the slice filtration, even though they are important tools for some of our constructions and proofs. We have therefore made an effort to make this paper accessible to a broad range of readers interested in algebraic geometry and cohomology theories, by making the use of motivic homotopy theory as self-contained as possible. In \S\ref{sec:prelim2-motsp}, we give a detailed and uniform review of the necessary aspects of the theory. This is possible because we only require a relatively small amount of the diverse corpus of motivic homotopy theory. In particular, regarding previous work on motivic cohomology, the existing results which we use only concern the case of smooth algebraic varieties over fields.

In the rest of this subsection we explain in more detail the relation of this paper to past work.

\subsubsection{Classical motivic cohomology over fields}
Let $k$ be a field. Bloch's cycle complex \cite{Bloch1986b}, built via algebraic cycles, was traditionally adopted as a candidate for the motivic cohomology of smooth algebraic varieties over $k$ (perhaps under a hypothesis of quasi-projectivity, which we ignore). However, it is very difficult to relate the cycle complex to algebraic $K$-theory; in particular, the original construction of the motivic filtration on the $K$-theory of smooth varieties by Bloch and Lichtenbaum \cite{BlochLichtenbaum} contained gaps which were later fixed by Levine \cite{Levine2008}. The cycle complex is also not contravariantly functorial for arbitrary morphisms between smooth varieties, and is not multiplicative in the modern sense: more precisely the cycle complexes do not assemble to form an $\bb E_\infty$-algebra in graded presheaves of complexes on smooth varieties.

Instead of directly relating the cycle complex to algebraic $K$-theory, Voevodsky's methods in motivic homotopy theory (outlined above in \S\ref{subsec:vv-approach}) provide a robust construction of a motivic filtration on algebraic $K$-theory of smooth varieties, which was then a posteriori related to the cycle complex by Voevodsky in characteristic zero \cite{voevodsky-zero} and Levine \cite{Levine2008} in general. Replacing the cdh topology by the coarser Nisnevich topology, parts (4) and (5) of Theorem \ref{thm:main-intro} are deep results of Rost--Voevodksy \cite{voevodsky-z2,voevodsky-zl}  (see \cite{HaesemeyerWeibel2019} for a textbook treatment) and Geisser--Levine \cite{GeisserLevine2000}.

We freely quote all these classical results about motivic cohomology of smooth algebraic varieties. However, we are careful not to attribute more to them than they say: they are typically formulated at the level of homotopy categories, or at best in model categories, and do not always treat multiplicative structure. As an example, we are not aware of any source in the literature which establishes a Beilinson--Lichtenbaum equivalence relating motivic cohomology and truncated \'etale cohomology as $\bb E_\infty$-algebras in graded presheaves of complexes on smooth varieties. A goal of \S\ref{ss_coherent_BL} is to resolve such issues: we take the main classical results concerning motivic cohomology and motivic homotopy theory of smooth algebraic varieties, and upgrade them to highly coherent $\infty$-categorical statements.

\subsubsection{Classical motivic cohomology over Dedekind domains}\label{sss:classical_Ded}
Now let $B$ be a mixed characteristic Dedekind domain. Bloch's cycle complex was extended to smooth $B$-schemes by Levine \cite{Levine2001}, again to provide a candidate for motivic cohomology (and, more generally, for Borel--Moore motivic homology of finite type $B$-schemes). Its key properties, including Nisnvich versions of Theorem \ref{thm:main-intro}(4)\&(5), were established by Geisser \cite{Geisser2004}. As in the case of smooth varieties, the Bloch--Levine cycle complex is not functorial or multiplicative in an $\infty$-categorical sense.

Spitzweck \cite{Spitzweck2018} showed that the Bloch--Levine cycle complex on smooth $B$-schemes is represented by a motivic spectrum in $\SH(B)$, which work of the first author partly joint with Hoyois \cite{Bachmann2022,BachmannHoyois2021} then identified with $s^0(1_B)$.

However, we completely circumvent the Bloch--Levine cycle complex in mixed characteristic, as well as  the above papers of Levine, Geisser, and Spitzweck, and propose a different development of the foundations of motivic cohomology of smooth $B$-schemes as follows. Riou's theory of Adams operations on $\KGL$ \cite{Riou2010} reduces questions about rational motivic cohomology to $\KGL$, so in practice it is sufficient to analyse motivic cohomology of smooth $B$-schemes $p$-adically for each prime number $p$. Our strategy is to first construct a theory of so-called $p$-adic Beilinson--Lichtenbaum cohomology in \S\ref{ss_BL}, by truncating syntomic cohomology in degrees $\le$ weight. We prove in \S\ref{ss_BL} that this cohomology theory is $\bb A^1$-invariant and satisfies the projective bundle formula, first on smooth algebraic varieties using Voevodsky's theory of presheaves with transfer, then on smooth $B$-schemes by establishing a localisation sequence.

The $p$-adic Beilinson--Lichtenbaum cohomology therefore corresponds to a motivic ring spectrum $H\bb Z_{p,B}^\sub{BL}\in\SH(B)$, and the theory of syntomic cohomology shows that it behaves well under base change (see \S\ref{ss_BL2}; this step is similar to one in Spitzweck's construction \cite[\S8]{Spitzweck2018}, but the modern approach to syntomic cohomology makes it much cleaner). Combined with the above cited work of the first author and Hoyois, this allows us in \S\ref{ss:slice-dedekind} to show that $s^0(1_B)_p^\comp\simeq H\bb Z_{p,B}^\sub{BL}\simeq s^0(\KGL_B)_p^\comp$ by reducing to the classical case of fields. Running the machinary of \S\ref{ss_coherent_BL} again then produces a highly structured Beilinson--Lichtenbaum equivalence for  the motivic cohomology of smooth $B$-schemes.

\subsubsection{The six functor formalism and cdh descent}
We use, and have nothing new to say about, Ayoub's six functor formalism in motivic homotopy theory and his proper base change theorem \cite{Ayoub2007,Ayoub2007II}. These underly the relation of motivic homotopy theory to cdh descent \cite{Cisinski2013} and are crucial to an argument we make repeatedly, namely reduction to the case of $\SH$ of fields.

\subsubsection{Framed correspondences and motivic infinite loop spaces} In \cite{ElmantoHoyoisKhanSosniloYakerson2021}, Hoyois, Khan, Sosnilo, Yakerson and the second author developed a motivic analog of infinite loop space theory. One of the important consequences of this theory is an explicit description of various effective spectra in terms of the ``framed suspension spectra'' of various nice moduli stacks. These geometric descriptions are used to prove that certain spectra are stable under pullbacks. In this paper we use the motivic spectrum $\cal V$ \cite{HoyoisJelisiejewNardinTotaroYakerson2021}, which we recall in \S\ref{subsub:V}. Its usefulness lies in the facts that, as well as being stable under pullback, it provides a candidate model for the effective cover of $\KGL$ (in fact, we will eventually see that they agree) and its zeroth slice agrees with that of the motivic sphere (Proposition~\ref{prop_1_vs_V}).

\subsection{Summary}
We highlight the main content of each section for the reader's convenience.
\begin{enumerate}
\item[\S\ref{s:prelimI}] There are two preliminaries sections, the first of which covers three topics. We discuss our conventions on filtered and graded objects, before moving on in \S\ref{subsec:topologies} to recalling the Grothendieck topologies we need, and the related topic of Milnor squares and excision. Lastly, \S\ref{subsec:lke} is a collection of results on left Kan extensions, some of which we could not find in the literature. Most notably, Proposition~\ref{prop:lke-restrict} is used throughout the paper to show that the value of the left Kan extension of a functor on a particular type of local ring (e.g., henselian) can be calculated by left Kan extending from the same type of local rings. 

\item[\S\ref{sec:prelim2-motsp}] The second preliminary section is mostly aimed towards non-specialists in motivic homotopy theory, although we also present some machinary which seems not to be available in the literature.  In particular, one of the key technical themes of this paper is that we exploit the formal, universal aspects of motivic spectra while simultaneously explicitly manipulating the cohomology theories they represent. To bridge this gap between motivic spectra and cohomology theories, we use the symmetric monoidal $\infty$-category of graded $\bb A^1$-invariant Nisnevich sheaves ``equipped with a first Chern class'', which we formally define as modules over a certain ring object in Construction~\ref{cons:lin-cohom-theories}. Objects in this $\infty$-category are relatively explicit and of a classical flavour, but nevertheless retain enough coherence to induce an associated ``Eilenberg--MacLane'' motivic spectrum; see Definition~\ref{definition:graded-p1-mult}. 


\item[\S\ref{sec:a1-mot-schemes}] Here we introduce the main cohomology theory of the paper, namely $\bb A^1$-motivic cohomology, denoted by $\bb Z(\star)^{\bb A}$. This is constructed by applying the slice filtration to the motivic spectrum $\KGL$ representing homotopy $K$-theory, which we review in \S\ref{subsec:KGL}.  A variant, called $\bb A^1$-cdh-motivic cohomology, maps to $\bb A^1$-motivic cohomology and will be used to control the latter. The rest of the section is dedicated to studying these cohomology theories in low weights (using  only considerations in motivic homotopy theory) in \S\ref{subsec:low-weights}, explaining their basic properties which follow from motivic homotopy theory, and finally analysing their rational structure using Adams operators.
 
\item[\S\ref{sec:syntomic}] We present the theory of syntomic cohomology of qcqs schemes, defined in Definition~\ref{def:bms-decomplete} by gluing the syntomic cohomology of the $p$-adic completion \cite{BhattMorrowScholze2} to the $p$-adic \'etale cohomology of the generic fibre of the scheme. Our exposition of the theory freely uses $K$-theory and topological cyclic homology, in contrast to Bhatt-Lurie's \cite{BhattLurie2022} prismatic approach. The main properties of syntomic cohomology are summarized in Theorem~\ref{theorem_syntomic_properties}. By Nisnevich locally truncating syntomic cohomology in degrees $\le$ weight we define Beilinson--Lichtenbaum cohomology in \S\ref{ss_BL}; it is a candidate for the $p$-adic motivic cohomology of sufficient regular schemes. As justification, we show on smooth schemes over fields and mixed characteristic Dedekind domains that it satisfies $\bb A^1$-invariance, the $\bb P^1$-bundle formula, and a localisation sequence.

\item[\S\ref{section_motivic_DD}] The goal of this section is to relate $\bb A^1$-motivic cohomology to the Beilinson--Lichtenbaum cohomology of \S\ref{ss_BL}. We do this by axiomatically defining, for any prime number $p$, a category $\frak{BL}^p$ of ``motivically regular'' schemes which satisfy versions of the $p$-adic Beilinson--Lichtenbaum conjectures. Although the precise axioms are mostly formulated at the level of homotopy categories without multiplicative structure, we show that the $p$-adic $\bb A^1$-motivic cohomology of schemes in $\frak{BL}^p$ identifies with Beilinson--Lichtenbaum cohomology in a highly coherent and multiplicative fashion. We check that smooth algebraic varieties belong to $\frak{BL}^p$ by quoting the main classical theorems about their motivic cohomology, and then use the resulting highly coherent Beilinson--Lichtenbaum comparison to carry out a base change argument; the conclusion is that smooth schemes over mixed characteristic Dedekind domains also belong to $\frak{BL}^p$, and so their motivic cohomology also enjoys a highly coherent Beilinson--Lichtenbaum comparison. As already discussed in \S\ref{sss:classical_Ded}, this argument is independent of the previous work of Levine, Geisser, and Spitzweck \cite{Geisser2004, Levine2001, Spitzweck2018} on motivic cohomology in mixed characteristic.

\item[\S\ref{sec:cdh-mot}] Here we introduce the most elementary of our motivic filtrations for $\KH$ and the associated so-called cdh-motivic cohomology; the latter is given by cdh-sheafifiying the left Kan extension of the $\bb A^1$-motivic cohomology of smooth $\bb Z$-schemes. A priori it is an auxiliary invariant used to analyse $\bb A^1$-motivic cohomology. In \S\ref{ss:Milnor_excision} we show that it satisfies Milnor excision, which provides a crucial technique for later reducing questions about cdh-motivic cohomology of general schemes to the case of henselian valuation rings of rank $\leq 1$. In Theorem~\ref{thm:singular-bl} we present an analogue for arbitrary schemes of the Beilinson--Lichtenbaum equivalence.

\item[\S\ref{sec:pbf-a1}] In this technical section, we establish $\bb A^1$-invariance (Theorem~\ref{thm:A1inv}) and $\bb P^1$-bundle formula (Theorem~\ref{thm:pbf-fields}) for cdh-motivic cohomology, conditionally on the validity of the key hypothesis $\Hyp(S, p, j)$. This hypothesis is formulated entirely in terms of syntomic and \'etale cohomology, is discussed in \S\ref{ss:key}, and is known to hold over fields and in low weights.

\item[\S\ref{sec:a1-comparison}]
Here we prove our main theorems: we relate slices of $\KGL$ and the motivic sphere spectrum (Corollary~\ref{eq:1-to-k}), compare $\bb Z(\star)^{\bb A}$ and $\bb Z(\star)^{\bb A, \cdh}$ (Theorem~\ref{thm:a1-a1cdh}), and describe $\bb A^1$-motivic cohomology in low weights (Corollary~\ref{corol:low-wts-integral}). These theorems all depend on Theorem~\ref{thm:comparison}, identifying $\bb A^1$-cdh-motivic cohomology with the $\bb A^1$-localisation of cdh-motivic cohomology; in turn, that theorem reduces to Theorem \ref{thm:pbf-a1}, asserting that the $\bb A^1$-localisation of cdh-motivic cohomology unconditionally satisfies the $\bb P^1$-bundle formula. The proof of the latter theorem is somewhat indirect (the reader is encouraged to read the summary at the beginning of \S\ref{ss:main_proof}), but it involves several results and constructions which are of independent interest. Of note is the introduction of \emph{Kato $K$-theory} $\K^\sub{Kato}$ in \S\ref{subsub:Kinfsel} (resp.~its mod-$p$ motivic cohomology analog $\bb F_p(\star)^\sub{Kato}$, defined in \S\ref{subsub:infinitesimal-mot}) which measures the difference between $K$-theory and Clausen's Selmer $K$-theory (resp.~the difference between mod-$p$ cdh-motivic cohomology and cdh-sheafified mod-$p$ syntomic cohomology). A crucial input is that these invariants enjoy a \emph{localisation property} as formulated in \S\ref{subsub:localisation} which allows us to ``separate'' their values on schemes of different characteristics.

\item[\S 10] The last section speed-runs the paper through a quadratic, more precisely Grothendieck-Witt, variant of the main results. Since our arguments are mostly similar to the $K$-theory versions, we omit some details. The main result is Theorem~\ref{thm:ko} which calculates certain generalised slices of the motivic spectrum $\KO$ representing homotopy Hermitian $K$-theory on schemes over $\bb Z[\tfrac12]$, which one can think of as an algebro-geometric refinement of Bott's famous calculation of the homotopy groups of the  $8$-periodic real $K$-theory spectrum. In place of $\bb A^1$-motivic cohomology, we have $\bb A^1$-Milnor--Witt motivic cohomology which is defined in Definition~\ref{definition_A^1_mot-_coh_mw}. In Theorem~\ref{thm:hz-tilde-main} we establish a pullback square describing this motivic spectrum, which in the case of fields was first suggested by Morel and established by the first author\cite{bachmann-very-effective}. This motivic spectrum, in turn, appears as the zero-th graded piece of a generalised slice filtration on $\KO$. \end{enumerate}

\subsection{Notation} We denote the $\infty$-category of spaces/anima/$\infty$-groupoids by $\Spc$. The stable $\infty$-category of spectra is denoted by $\Spt$ and the derived $\infty$-category of complexes of $\Lambda$-modules, where $\Lambda$ is a commutative ring, is denoted by $\D(\Lambda)$. For any $\infty$-category $\scr C$, and objects $X,Y\in\cal C$, we write $\Map_{\cal C}(X, Y) \in \opp{Spc}$ for the mapping space between the objects. If $\scr C$ is furthermore stable, we can also form the spectrum of maps $\opp{map}_{\cal C}(X, Y)$. The rest of notation will be introduced as needed throughout the paper.

For ease of reference, the following is a list of the main cohomology theories we study in this paper and the first place where they appear. 

\begin{center}
\begin{tabular}{lll}
$\Z(j)^{\A}$ & $\bb A^1$-motivic cohomology & Definition~\ref{definition_A^1_mot-_coh}\\
$\Z(j)^{\A,\sub{cdh}}$ & $\bb A^1$-cdh-motivic cohomology & Definition~\ref{definition_A^1_cdh_mot_coh}\\
$\Z(j)^{\sub{cdh}}$ & cdh-motivic cohomology & Definition~\ref{def:cdh}\\
$\Z_p(j)^\sub{syn}$ & $p$-adic syntomic cohomology & Definition~\ref{def:bms-decomplete}\\
$\Z_p(j)^\sub{BL}$ & $p$-adic Beilinson--Lichtenbaum cohomology & Definition~\ref{def:bl-cohomology}\\
\end{tabular}
\end{center}

\subsection{Acknowledgments}
This paper has undergone a long fermentation since 2021, when we first understood that the results were within reach. Over that time we have benefited from discussions with numerous people, including Federico Binda, Tess Bouis, Lars Hesselholt, Marc Hoyois, Ryomei Iwasa, Markus Land, Jacob Lurie, Marc Levine, and Georg Tamme; we express our gratitude to them and apologise to anyone whose name we have forgotten to include.

EE acknowledges support from the NSERC Discovery grant RGPIN-2025-07114, ``Motivic cohomology: theory and applications'' and an Erik Ellentuck fellowship during his stay at the Institute for Advanced Study. TB acknowledges support from Deutsche Forschungsgemeinschaft (DFG, German Research
Foundation) through the Collaborative Research Centre TRR 326 Geometry and Arithmetic of Uniformized Structures, project number 444845124. MM acknowledges support of the James D.~Wolfensohn Fund and the Membership in Mathematics fund at the Institute for Advanced Study, for his stay there in Spring 2024. This project has received funding from the European Research Council (ERC) under the European Union's Horizon 2020 research and innovation programme (grant agreement No.~101001474). We thank the Mathematisches Forschungsinstitut Oberwolfach for excellent working conditions during the Algebraic $K$-theory workshop in August 2025, during which this paper was completed.

\section{Preliminaries I}\label{s:prelimI}

\subsection{Filtrations and gradings} \label{ss_filtrations}
Our conventions regarding filtrations follow \cite{ElmantoMorrow2023}, which we briefly review in this section. For a stable $\infty$-category $\cal C$ we have the associated stable $\infty$-categories of \emph{filtered objects} and {\em graded objects}:
\[
\Fil\cal C := \Fun((\Z, \geq)^{\op}, \cal C) \qquad \text{and}\qquad \Gr\cal C := \Fun(\Z^{\delta}, \cal C).
\]
Here $(\Z, \geq)$ denotes the totally ordered set of the integers and $\Z^{\delta}$ is the discrete category of the integers. Our filtrations are thus, by convention, $\bb Z$-indexed and decreasing. The functor of taking associated graded is denoted by $\mathrm{gr}^\star:\Fil\cal C \rightarrow  \Gr\cal C$.

Filtered objects of $\cal C$ are usually written as $\Fil^\star M$, for some $M \in \cal C$. This means implicitly that there is a map in $\cal C$,
\[
\Fil^{-\infty}M:=\colim_{j\to -\infty}\Fil^jM\to M,
\]
and we say that the {\em filtration on $M$ is exhaustive} if this map is an equivalence,  {\em $\bb N$-indexed} when $\Fil^jM\to M$ is an equivalence for all $j\le 0$ and {\em complete} if $\lim_{j\to\infty}\Fil^jM=0$.

Assume now that $\cal C$ is presentably symmetric monoidal.
Then $\cal C^{\Z^{\op}}$ and $ \cal C^{\Z^{\delta}}$ admit canonical symmetric monoidal structures given by Day convolution, such that the functor of taking associated graded is strong symmetric monoidal\NB{Add references?}. For the sake of simplicity, we sometimes say that a filtered (respectively graded) object is {\em multiplicative} if it is equipped with the structure of $\mathbb{E}_{\infty}$-algebra, i.e., it is an object of $\CAlg(\Fil \cal C)$ (respectively of $\CAlg(\Gr \cal C)$); similarly, we will sometimes say that maps between such structured objects are \emph{multiplicative}.

\begin{remark}[Shearing]\label{rmk:shearing1}
Continue to assume that $\cal C$ is a presentably symmetric monoidal stable $\infty$-category. We may ask if there is a symmetric monoidal equivalence \begin{equation}\Gr \scr C \quis \Gr \scr C,\qquad E^{\star}\mapsto E^{\star}[2\star].\label{eqn:shear1}\end{equation} Below we sketch the positive answer to this question when $\cal C$ is $\bb Z$-linear; see \cite[Proposition 3.3.4]{Raksit2020} for further details.

Firstly, by \cite[proof of Lemma 6.24]{bachmann-linearity}, such an equivalence is the same as a map of symmetric monoidal $\infty$-groupoids \begin{equation} \Z \to \Pic(\Gr \scr C) \wequi \Pic(\scr C) \times \Z \label{eqn:shear2}\end{equation} sending $1$ to $1\lra{1}_{\cal C}[2]$ ($:=$ the unit placed in grading degree $1$ and homological degree $2$), or in other words a map of commutative monoids $\Z \to \Pic(\scr C)$ sending $1$ to $1_\scr{C}[2]$. For any commutative ring $A$ the $\infty$-groupoid $\Pic(\D(A))$ is $1$-truncated and equivalent, as a symmetric monoidal $1$-groupoid (i.e., as a Picard groupoid) to $\Pic^\bb Z(A)$ ($:=$ the symmetric monoidal $1$-groupoid of $\bb Z$-graded line bundles). We may therefore define a map of symmetric monoidal $1$-groupoids $\bb Z\to \Pic(\D(A))$ by sending each $m\in\bb Z$ to $A[2m]$, and declaring the lax monoidal structure $A[2m]\otimes_AA[2m']\to A[2(m+m')]$ to be the identity map. In the case of $A=\bb Z$ (where all units are $2$-torsion), this is even the unique symmetric monoidal functor $\bb Z\to\Pic(\D(\bb Z))$ sending $1$ to $\bb Z[2]$.

More generally, for any $\bb Z$-linear presentably symmetric monoidal stable $\infty$-category $\cal C$, precomposing with the unit map $\D(\bb Z)\to \cal C$ thus defines a map as in \eqref{eqn:shear2} and so induces a symmetric monoidal equivalence as in \eqref{eqn:shear1}. We will refer to this process as {\em shearing}.

On the other hand, the $\infty$-category $\cal C=\Spt$ of spectra does not admit any map as in \eqref{eqn:shear2} and does not support any equivalence as in \eqref{eqn:shear1}.
\end{remark}

\begin{remark}[Beilinson--Lichtenbaum truncation]\label{rem:BL_truncation}
Given a graded complex $C\in\opp{Gr}\D(\bb Z)$, we will often consider the graded complex $\tau^{\le\star}C$ obtained by truncating $C$ in degrees less than or equal to the weight, i.e., $(\tau^{\le\star}C)^j:=\tau^{\le j}C^j$ for $j\in\bb Z$. The endofunctor $\tau^{\le\star}:\opp{Gr}\D(\bb Z)\to \opp{Gr}\D(\bb Z)$ is the connective cover for a $t$-structure on $\opp{Gr}\D(\bb Z)$ compatible with the symmetric monoidal structure \cite[\S3.3]{Raksit2020}; in particular, the functor is lax symmetric monoidal and so preserves $\bb E_\infty$-algebras.
\end{remark}

\subsection{Grothendieck Topologies}\label{subsec:topologies}
Here we recall the main Grothendieck topologies which will appear: the Zariski, Nisnevich, \'etale, and fppf topologies, and their cousins the rh, cdh, \'eh, and h topologies. They are topologies on $\opp{Sch}^\sub{qcqs}$, the category of all quasi-compact, quasi-separated schemes, which we will also restrict to $\mathrm{Sch}_S^\sub{qcqs}$, the category of qcqs schemes over a base qcqs scheme $S$ (often affine). We also restrict the Nisnevich and \'etale topologies to $\Sm_S$, which denotes the category of qcqs smooth $S$-schemes (equivalently, assuming $S$ qcqs, finitely presented smooth $S$-schemes).

We define our topologies using the language of covers and pretopologies in the sense of \cite[Exp. II, Definition 1.3]{SGA_IV_I} (in particular, when we define the covers of a pretopology, we do not require that a collection of morphisms refinable by a cover is itself a cover); any pretopology (defined in terms of covers) gives rise to a topology (defined in terms of sieves) \cite[Exp. II, 1.1.6]{SGA_IV_I}, and the notion of a sheaf depends only on the latter.

\begin{definition}\label{definition_topologies}
The \emph{Zariski}, {\em Nisnevich}, {\em \'etale}, and {\em fppf} pretopologies are defined on $\opp{Sch}^\sub{qcqs}$ by declaring covers to be \emph{finite} collections of morphisms $\{f_i:X_i\to X\}_i$  such that
\begin{enumerate}
\item Zariski: each $f_i$ is an open immersion and $\bigsqcup_iX_i\to X$ is surjective;
\item Nisnevich: each $f_i$ is \'etale and $\bigsqcup_iX_i\to X$ is surjective on $k$-points for every field $k$;
\item \'etale: each $f_i$ is \'etale and $\bigsqcup_iX_i\to X$ is surjective;
\item fppf: each $f_i$ is flat of finite presentation and $\bigsqcup_iX_i\to X$ is surjective.
\end{enumerate}

We next discuss the h cousins of the topologies of the previous definition. An \emph{abstract blowup square} is a cartesian square of qcqs schemes
\begin{equation}
\begin{tikzcd}
Y' \ar{r} \ar{d} & X' \ar{d}{p} \\
Y \ar[swap]{r}{i} & X,
\end{tikzcd}
\end{equation}
where $i$ is a finitely presented closed immersion and $p$ is a finitely presented, proper morphism inducing an isomorphism $X'\setminus Y'\isoto X\setminus Y$.

The {\em cdh} pretopology on $\opp{Sch}^\sub{qcqs}$ is defined to be the pretopology generated by the Nisnevich pretopology and by $\{Y\to X,\,X'\to X\}$ as one runs over all abstract blow-up squares of qcqs schemes. Similarly, the {\em rh}, {\em \'eh}, respectively {\em h}, pretopology on $\opp{Sch}^\sub{qcqs}$ is defined to be the pretopology generated by the Zariski, \'etale, respectively fppf, topology and by $\{Y\to X,\,X'\to X\}$ as one runs over all abstract blow-up squares of qcqs schemes.
\end{definition}

Diagrammatically, these topologies may be summarised as follows, where an arrow $\sigma\leftarrow \tau$ indicates that $\tau$ is finer than $\sigma$:

\[\xymatrix{
\text{rh} \ar[d] & \text{cdh} \ar[d] \ar[l] & \text{\'eh}\ar[l]\ar[d] & \text{h}\ar[d]\ar[l] \\
 \text{Zar} & \text{Nis} \ar[l] &\text{\'et}\ar[l] & \text{fppf}\ar[l]
}\]
The rh topology will not appear again, and the h topology will only be used in \S\ref{sec:singular-bl}.

\begin{remark}[Restriction to finitely presented schemes]\label{remark_fp}
For any of our pretopologies $\tau=$ Zar, Nis, \'et, fppf, cdh, \'eh, and h, every $\tau$-cover $\{X_i\to X\}_i$ in $\opp{Sch}^\sub{qcqs}$ is a finite set consisting of finitely presented morphisms. Indeed, this follows formally from the facts that these properties hold for the generators of the pretopology and that the generators are closed under base change.\NB{ref?}

The $\tau$-topology therefore restricts to $\opp{Sch}^\sub{fp}_X$ for any qcqs scheme $X$, and the restriction functor commutes with sheafification.
Furthermore, suppose that $X=\lim_\lambda Y_\lambda$ is a cofiltered limit of qcqs schemes along affine transition maps. Then any $\tau$-cover $\{X_i\to X\}_i$ comes via base change from a $\tau$-cover of $Y_\lambda$ for some index $\lambda\gg0$ (e.g. use \cite[Tag 01ZM]{Stacks}).
\end{remark}

\subsubsection{Relationship with other definitions in the literature}
The topologies of Definition \ref{definition_topologies} agree with other definitions found in the literature. Firstly, as far as we are aware there is no ambiguity in the literature about what is meant by the Zariski, \'etale, or fppf topologies on qcqs schemes, and in any case our definition agrees with what is found in all standard sources.

In the case of the Nisnevich topology, there are several possible definitions which may a priori seem distinct, but in fact they all coincide. See \cite[Proposition~A.2]{BachmannHoyois2021}.

Next, we follow the convention of \cite[Appendix~A]{LandTamme2019} and \cite[\S 4]{AntieauMathewMorrow} that the morphisms in our abstract blowup squares are required be finitely presented, which is indispensable to ensure cdh sheafiness of certain non-finitary presheaves (e.g., the fibre of the trace map $\K\to \TC$).
Such squares are called finitely presented abstract blowup squares in \cite{ElmantoHoyoisIwasaKelly2021}.
In any case, all of the aforementioned sources define the cdh topology in terms of finitely presented maps, and so agree. Its restriction to finitely presented schemes over a field agrees with the original definition of Suslin--Voevodsky \cite{SuslinVoevodsky2000}.

For finite type schemes over a Noetherian ring, the \'eh topology was introduced by Geisser \cite[Definition~2.1]{Geisser2006} in the context of his arithmetic cohomology, and its points are studied in \cite{GabberKelly2015}; in that context it agrees with our definition.

Finally, the results of \cite[\S 2]{BhattScholze2017} and \cite[Theorem~3.12]{Rydh2010} show that our h-topology coincides with the following equivalent topologies on $\opp{Sch}^\sub{qcqs}$: the topology generated by Zariski covers and by covers $\{X'\to X\}$, where $X'\to X$ is any finitely presented proper surjection of qcqs schemes, or where $X' \to X$ is any finitely presented v-cover of qcqs schemes in the sense of \cite[Definition~2.1]{BhattScholze2017}, or the h-topology on $\opp{Sch}^\sub{qcqs}$ as found in \cite[Tag 0ETQ]{Stacks}.
In particular, the restriction of our h-topology to finitely presented schemes over a qcqs base coincides with that of \cite[Definition~2.7]{BhattScholze2017}, and coincides with the h-topology of Voevodsky's thesis in the case of finite type schemes over a field \cite{Voevodsky1996}.

\subsubsection{Detecting sheaves}
Next recall how to detect sheafiness with respect to the topologies of Definition \ref{definition_topologies}. For the rest of this subsection, we let $\cal D$ be a complete $\infty$-category (that is, it admits all limits) and let $S$ denote a qcqs scheme.

\begin{proposition}\label{proposition_checking_sheafiness}
Let $F:\opp{Sch}^\sub{qcqs,op}_S\to\cal D$ be a presheaf.
\begin{enumerate}
\item $F$ is a Nisnevich sheaf if and only if it sends any Nisnevich distinguished square (as in \cite[Example 2.1.2(2)]{AsokHoyoisWendt2017}) to a cartesian square and $F(\emptyset) = *$.
\item Assume $F$ is a Nisnevich (respectively~\'etale, respectively~fppf) sheaf; then it is a cdh (respectively~\'eh, respectively~h) sheaf if and only if it sends any abstract blow-up square to a cartesian square.
\end{enumerate}
\end{proposition}
\begin{proof}
We may assume that $\scr D = \Spc$.\footnote{Observe that a presheaf $F$ with values in $\scr D$ is a sheaf if and only if, for every $X \in \scr D$, the presheaf $\Map(X, F(\ph))$ is.}
Part (1), and (2) for the cdh topology, are covered by a result of Voevodsky on cd-structures~\cite[Corollary~5.10]{Voevodsky2010a}; see \cite[Theorem~3.2.5]{AsokHoyoisWendt2017} for a modern proof.

By formal closure properties of the collections $\{X_i\to X\}$ for which a presheaf satisfies descent (see, e.g., \cite[Lemma 3.3.2]{LiuZheng2014}), we see that a presheaf is an \'eh (respectively~h) sheaf if and only if it is simultaneously an \'etale (respectively~fppf) sheaf and a cdh sheaf. Therefore the remaining cases of part (2) follow from the cdh case.
\end{proof}

\begin{corollary} \label{cor:cdh-cpt-gen}
Assume furthermore that $\scr D$ is cocomplete and stable (e.g., it is a presentable stable $\infty$-category). Then $\scr D$-valued Nisnevich and cdh sheaves on $\Sch_S^\qcqs$ are closed under colimits in presheaves.\footnote{Using the terminology of \cite[Definition~3.4]{ClausenMathew2021}, the site of $\Sch^\qcqs$ equipped with the Nisnevich or cdh topology is excisive.}
\end{corollary}
\begin{proof}
Indeed by Proposition \ref{proposition_checking_sheafiness}, being a sheaf means sending certain squares to cocartesian squares ($\scr D$ being stable), and cocartesian squares are stable under colimits.
\end{proof}

\begin{example}[Rationalization as a filtered colimit] \label{ex:rationalization-cdh}
Given an object $E$ in a presentable stable $\infty$-category $\scr C$, one has the \emph{rationalization} which can be calculated as: \[ E \otimes \Q \simeq \colim\left( E \xrightarrow{2} E \xrightarrow{3} E \xrightarrow{4} \dots \right). \]
This construction has the useful property that $E=0$ if and only if $E \otimes \Q = 0$ and $E/p = 0$ for all primes $p$ (indeed $E/p = 0$ means that multiplication by $p$ is an autoequivalence of $E$, so the directed system in the definition of $E \otimes \Q$ consists of equivalences and has colimit $E$ itself).
Taking $\scr C$ to be Nisnevich or cdh sheaves valued in a presentable stable $\infty$-category and $X$ a qcqs scheme, we learn from Corollary \ref{cor:cdh-cpt-gen} that \[ (E \otimes \Q)(X) \wequi E(X) \otimes \Q. \]
\end{example}

We make two further remarks on detecting sheafiness for the h and \'eh topologies.

\begin{remark}[Descent for the $h$-topology]\label{remark:h_descent}
In the case of the $h$-topology one can improve Proposition \ref{proposition_checking_sheafiness}. Namely, suppose that $F:\opp{Sch}^\sub{qcqs,op}_S\to\cal D$ is a Zariski sheaf which sends any abstract blow-up square to a cartesian square and satisfies finite flat descent, i.e., for any finite flat morphism $\Spec(B)\to\Spec(A)$ of affine $S$-schemes, $F$ satisfies \v{C}ech descent in the sense that $F(A)\quis\lim_{n\in \Delta}F(B^{\otimes_An+1})$; then we claim that $F$ is an h sheaf. Indeed, as soon as $F$ satisfies Zariski descent and finite flat descent, then it satisfies fppf descent by \cite[Tag 05WM]{Stacks}.
\end{remark}

\begin{remark}[\'{e}h-sheafification]\label{rem:eh-sheaf} The \'{e}h topology is an intermediary topology between the $\cdh$ topology and the $h$ topology. An important technical result about this topology is that the $\cdh$-sheafification of an \'etale sheaf computes the \'eh-sheafification in some generality. More precisely, assuming $\scr D$ is a presentable stable $\infty$-category compactly generated by cotruncated objects (see \cite[Definition~3.1.4]{ElmantoHoyoisIwasaKelly2021}), then \cite[Theorems~A.3,A.4]{ElmantoMorrow2023} shows that the canonical map of endofunctors
\begin{equation}\label{eq:cdh-eh}
L_{\cdh}L_\sub{\'et} \To L_\sub{\'{e}h}
\end{equation}
of $\cal D$-valued presheaves on $\opp{Sch}_S^\sub{qcqs}$  is an equivalence.
\end{remark}

\subsubsection{Finitary (pre)sheaves}
We now discuss finitary presheaves in the sense of the next definition; it is the condition that lets us access stalks of sheaves by simply evaluating them at points of the topology. 

\begin{definition}\label{def:finitary}
Assume further that the complete $\infty$-category $\scr D$ is cocomplete. Then a presheaf $F:\opp{Sch}^\sub{qcqs}_S\to \cal D^{\sub{op}}$ is said to be {\em finitary} if it commutes with cofiltered limits along affine transition maps.
\end{definition}

The following result shows that, in our cases of interest, left Kan extension produces finitary presheaves, and that it is preserved by sheafification:

\begin{proposition}\label{prop:finitary_conditions} Assume that $\scr D$ is a presentable $\infty$-category. 
\begin{enumerate}
\item Let $\cal C$ be a full subcategory of $\opp{Sch}^\sub{fp}_S$ and $F:\cal C^\sub{op} \to\scr D$  a presheaf. Then the left Kan extension $L_{\opp{Sch}^\sub{qcqs,op}_S/\cal C^\sub{op}}F:\opp{Sch}^\sub{qcqs,op}_S\to\scr D$ of $F$ to $\opp{Sch}^\sub{qcqs,op}_S$ is finitary.
\item Let $\tau$ be any of the topologies of Definition \ref{definition_topologies}.  Assume that either $\scr D$ is compactly generated by cotruncated objects,\footnote{This holds for example if $\scr D = \scr C_{\le 0}$ where $\scr C$ is compactly generated \cite[Example 3.1.5]{ElmantoHoyoisIwasaKelly2021}} or $\scr D$ is stable and $\tau\in\{\text{Nis},\text{cdh}\}$.
  Then, for any finitary presheaf $F:\opp{Sch}^\sub{qcqs}_S\to\scr D$, its sheafification $L_\tau F:\opp{Sch}^\sub{qcqs}_S\to\scr D$ is also finitary.
\end{enumerate}
\end{proposition}
\begin{proof}
(1): By the pointwise formula for left Kan extensions \cite[Definition 4.3.2.2]{Lurie2009}, for $X \in \Sch_S$ we have \[ L_{\opp{Sch}^\sub{qcqs,op}_S/\cal C^\sub{op}}F(X) \wequi \colim_{\cal C_{X/}} F. \]
Now let $X = \lim_i X_i$ be written as a cofiltered limit along affine transition maps. Then it is enough to check that $\cal C_{X/} \quis \colim_i \cal C_{X_i/}$. This follows from the fact that any $Y\in\scr C$ is a finitely presented scheme, whence $\colim_i \Map(X_i, Y)\to\Map(X, Y)$ is an equivalence \cite[Proposition 8.14.2]{EGA_IV_III}

(2): Our proof is inspired by \cite[Proposition~7.11]{ClausenMathew2021}. For any presheaf $G:\opp{Sch}^\sub{qcqs}_S\to\opp{Sp}$ we will write $G|_\sub{fp}$ for its restriction to $\opp{Sch}^\sub{fp}_S$, and $L_\sub{qcqs/fp}G|_\sub{fp}:\opp{Sch}^\sub{qcqs}_S\to\opp{Sp}$ for the left Kan extension of $G|_\sub{fp}$ along the inclusion $\opp{Sch}^\sub{fp}_S\subset \opp{Sch}^\sub{qcqs}_S$. Note that there is a natural map $L_\sub{qcqs/fp}G|_\sub{fp}\to G$ of presheaves on $\opp{Sch}^\sub{qcqs}_S$, and that $L_\sub{qcqs/fp}G|_\sub{fp}$ is finitary by part (1).

Set $\tilde F = L_\sub{qcqs/fp}(L_\tau F)|_\sub{fp}$.
Since $L_\sub{qcqs/fp}$ preserves $\tau$-equivalences and $F \wequi L_\sub{qcqs/fp}(F|_\sub{fp})$ ($F$ being finitary), we see have a $\tau$-equivalence $F \to \tilde F$.
Since $\tilde F$ is finitary, it suffices  thus to show that $\tilde F$ is a $\tau$-sheaf.

First assume that $\scr D$ is compactly generated by cotruncated objects.
Considering $\Map(X, \ph)$ for any $X \in \scr D$ compact and cotruncated, we may assume $\scr D = \Spc_{\le n}$ for some $n\ge0$.
Let $\{X_i\to X\}$ be a cover in $\opp{Sch}^\sub{qcqs}_S$ for the $\tau$-pretopology, and set $X':=\bigsqcup_iX_i$; as in Remark \ref{remark_fp}, we may write $X=\lim_\lambda Y_\lambda$ as a cofiltered limit along affine transition maps of finitely presented $S$-schemes, and we may assume that the indexing set has a minimal element $\lambda_0$ such that the cover $X'\to X$ descends to a cover $Y_{\lambda_0}'\to Y_{\lambda_0}$. For any $\lambda_0\to \lambda$, the $\tau$-sheaf $L_\tau F$ satisfies descent with respect to the $\tau$-cover $Y_\lambda':=Y_{\lambda_0}'\times_{Y_{\lambda_0}}Y_\lambda\to Y_\lambda$, i.e., $L_\tau F(Y_\lambda)\quis\lim_{\Delta}(F(Y_\lambda')\rightrightarrows\cdots)$. Taking the filtered colimit over $\lambda$ we obtain \[\tilde F(X) = \colim_\lambda\lim_{\Delta}(F(Y_\lambda')\rightrightarrows\cdots) \] Since $F$ takes values in $n$-truncated spaces, the filtered colimit may be exchanged with the $\lim_{\Delta}$\NB{ref} and we deduce that $\tilde F$ satisfies descent with respect to the cover $X'\to X$ as desired. Therefore $\tilde F$ is indeed a $\tau$-sheaf, as required.

Now suppose instead that $\scr D$ is stable and $\tau$ is the Nisnevich or cdh topology.
We appeal to the criteria of Proposition \ref{proposition_checking_sheafiness} stating that Nisnevich and cdh descent may be checked via excision for Nisnevich and abstract blow-up squares; these again may be approximated and the relevant cartesian squares are then preserved by the colimit, as $\cal D$ is stable.
\end{proof}

\begin{remark}\label{lem:filtr-coconn}
The argument in the second to last paragraph also proves the following: if $F_i$ is a filtered diagram of $\tau$-sheaves for any of the topologies in Definition \ref{definition_topologies}, such that each $F_i$ is coconnective, then the filtered colimit in presheaves is, in fact, a $\tau$-sheaf. This boils down to commuting a totalization and a filtered colimit, which one is allowed to do in the $\infty$-category of coconnective objects.
\end{remark}

\subsubsection{Milnor excision}
The final topic of this preliminary subsection on topologies is Milnor excision and its applications to cdh sheaves. 

\begin{definition}
Recall that a {\em Milnor square}, or {\em excision square}, is a commutative square of rings
\[\xymatrix{
A\ar[r]\ar[d] & A'\ar[d]\\
A/I\ar[r] & A'/I'
}\]
which is both cartesian and co-cartesian as a square of $A$-modules; equivalently, the map (of non-unital $A$-algebras) $I\to I'$ is an isomorphism. A functor $F$, defined on a class of rings including the square and valued in an $\infty$-category $\cal D$, is said to satisfy {\em Milnor excision} for the square if it carries it to a cartesian square in $\cal D$.
\end{definition}

\begin{remark}
Milnor squares do not underlie a Grothendieck topology, unless combined with nilinvariance; see \cite[Remark~3.2.10]{ElmantoHoyoisIwasaKelly2021}. 
\end{remark}

An important class of Milnor square arise as follows: given a valuation ring $V$ and prime ideal $\frak p\subset V$, the square
\begin{equation}
\xymatrix{
V\ar[r]\ar[d] & V_{\frak p}\ar[d]\\
V/\frak p\ar[r] & k(\frak p)
}
\label{eqn_hv_square}\end{equation}
is a Milnor square \cite[Proposition 2.8]{BhattMathew2021}, which we typically call the Milnor square associated with the pair $(V,\frak p)$. For any (respectively henselian) valuation ring $V$, a functor $F$, defined on a class of rings including the square~\eqref{eqn_hv_square} and valued in a stable $\infty$-category $\cal D$, is said to satisfy {\em $v$-excision} (respectively {\em henselian $v$-excision}) for the square if it carries it to a cartesian square in $\cal D$. 

For cdh sheaves, henselian $v$-excision suffices to guarantee general excision:

\begin{theorem}[{\cite[Theorem~3.3.4 \& Remark~3.3.3]{ElmantoHoyoisIwasaKelly2021}.}]\label{thm:ehik}
Let $B$ be a base ring of finite valuative dimension and $F:\Sch_B^\sub{qcqs,op} \rightarrow \opp{Sp}$ a finitary $\cdh$ sheaf. Assume that, for all finite rank henselian valuation rings $V$ under $B$ and all prime ideals $\frak p\subset V$, the functor $F$ satisfies Milnor excision for the square (\ref{eqn_hv_square}). Then $F$ satisfies Milnor excision for all Milnor squares of $B$-algebras. 
\end{theorem}

\begin{remark}\label{rem:rigid-exc}
In the case of a functor which is sufficiently rigid (i.e., invariant under henselian ideals), excision for henselian valuation rings is automatic. For example, suppose that $V$ is a henselian valuation ring and $F:\opp{CAlg}_V\to\cal D$ is a functor which is invariant under henselian ideals. Then, for any prime ideal $\frak p\subset V$, the square
\[
\begin{tikzcd}
F(V) \ar{r} \ar{d} & F(V_{\mathfrak{p}}) \ar{d}\\
F(V/\mathfrak{p}) \ar{r} & F(\kappa(\mathfrak{p})).
\end{tikzcd}
\]
is cartesian and cocartesian. Indeed much more is true: the vertical maps in the square are equivalences, since both $(V,\mathfrak{p})$ and $(V_{\mathfrak{p}}, \mathfrak{p})$ are henselian pairs.
\end{remark}

Finally we recall how to check equivalences of sheaves on points:
\begin{proposition}\label{proposition_checking_on_points}
Let $\cal D = \opp{Sp}$ or $\D(\bb Z)$, let $B$ be a base ring of finite valuative dimension, and let $F:\opp{Sch}^\sub{qcqs,op}_B\to \cal D$ a finitary presheaf.
\begin{enumerate}
\item  Assuming $F$ is a cdh sheaf, then $F=0$ if and only $F(\Spec(V))=0$ for every henselian valuation ring of finite rank $V$ under $B$.
\item Assuming $F$ is a \'eh sheaf and takes coconnective values, then $F=0$ if and only if $F(\Spec(V))=0$ for every strictly henselian valuation ring of finite rank $V$ under $B$.
\item Assuming $F$ is a h sheaf and takes coconnective values, then $F=0$ if and only if $F(\Spec(V))=0$ for every absolutely integrally closed valuation ring of finite rank $V$ under $B$.
\end{enumerate}
In each case, if $F$ moreover satisfies excision, then it even suffices to restrict to such valuation rings of rank $\le 1$.
\end{proposition}

\begin{proof}
In all the cases, it will suffice to prove that $F|_{\Sch_B^\sub{fp}} = 0$. We first treat (1); the restricted cdh topos is coherent and locally coherent \cite[Proposition 2.1.4]{ElmantoHoyoisIwasaKelly2021}. Furthermore, it is hypercomplete in case (1) by \cite[Corollary 2.4.16]{ElmantoHoyoisIwasaKelly2021}. Therefore, by Deligne's completeness theorem \cite[Theorem A.4.0.5]{LurieSAG}, it thus suffices to check that $F$ vanishes on points.
These are necessarily given by evaluation on a pro-scheme (see e.g. \cite[Remark A.9.1.4]{LurieSAG}). In order for a pro-scheme to define a point, any covering of it must split.
From this one deduces that the pro-scheme must correspond to henselian valuation rings, perhaps of infinite rank.
See, e.g.~\cite{GabberKelly2015}, for the description of points for the cdh topology; note that this is proved under additional finiteness hypotheses which are not necessary. However, since $F$ is assumed to be finitary, it is enough to check vanishing on finite rank valuation rings. 

To prove (2) and (3), we note that hypercompleteness of $F$ is assured since it is bounded above \cite[Lemma 2.7]{ClausenMathew2021}. In these cases, we refer the reader again to \cite{GabberKelly2015} for a description of the points of both the \'eh and h topologies and reduction to finite rank valuation rings again follows by from the finitary assumption. 

Lastly, we can reduce to rank $\le 1$ valuation rings using excision for squares of type \eqref{eqn_hv_square}.
\end{proof}

\subsection{Rigidity and left Kan extending from smooth schemes/algebras}\label{subsec:lke}
We will frequently left Kan extend functors from smooth algebras or smooth schemes, and sometimes then sheafify the result. We will need to know, for example, that the stalks of the resulting sheaves can be computed by left Kan extending the original functor from essentially smooth local rings. It will also be notationally convenient to know that if we left Kan extend from smooth schemes and then evaluate on an affine, then the result is the left Kan extension from smooth algebras. We collect such results in the following proposition which allow us, in the remainder of the paper, to unambiguously denote by $L_B^\sub{sm}$ any of the denoted left Kan extensions from smooth $B$-schemes/algebras/etc.

Let $B$ be a ring. The categories appearing in the following proposition are as follows:
\begin{enumerate}
\item $\Sch_B^\sub{qcqs}$ is the category of qcqs $B$-schemes, and $\Sm_B$ its full subcategory of qcqs smooth $B$-schemes;
\item $\opp{CAlg}_B\supset \opp{IndSmAlg}_B\supset \opp{SmAlg}_B$ denote respectively the categories of $B$-algebras,\footnote{All rings and algebras in this paper are commutative if not stated otherwise.} of ind-smooth $B$-algebras, and of smooth $B$-algebras;
\item $\opp{CAlg}_B^\sub{loc}$ is the category of local $B$-algebras, and $\opp{EssSmAlg}_B^\sub{loc}$ its full subcategory of essentially smooth local $B$-algebras (i.e., localisations of smooth $B$-algebras);
\item $\opp{CAlg}_B^\sub{hloc}$ is the category of henselian local $B$-algebras, and $\opp{EssSmAlg}_B^\sub{hloc}$ its full subcategory consisting of henselisations of essentially smooth local $B$-algebras;
\item $\opp{CAlg}_B^\sub{shloc}$ is the category of strictly henselian local $B$-algebras, and $\opp{EssSmAlg}_B^\sub{shloc}$ its full subcategory consisting of strict henselisations of essentially smooth local $B$-algebras.
\end{enumerate}

\begin{proposition}\label{prop:lke-restrict}
Let $B$ be a ring and $\cal C$ a presentable $\infty$-category. Then the following diagram commutes

\[\xymatrix@C=3mm{
\Fun(\Sch_B^\sub{qcqs,op},\cal C)\ar[r] & \Fun(\opp{CAlg}_B,\cal C)\ar[r]^{=} & \Fun(\opp{CAlg}_B,\cal C)\ar[r] & \Fun(\opp{CAlg}_B^\sub{?loc},\cal C)
\\
\Fun(\Sm^\sub{op}_B,\cal C)\ar[r]\ar[u]_{L_B^\sub{sm}}& \Fun(\opp{SmAlg}_B,\cal C)\ar[u]_{L_B^\sub{sm}}\ar[r]^-{\simeq} & \Fun^\sub{fin}(\opp{IndSmAlg}_B,\cal C)\ar[u]_{L_B^\sub{sm}}\ar[r] &\Fun(\opp{EssSmAlg}_B^\sub{?loc},\cal C)\ar[u]_{L_B^\sub{sm}}
}\]

where
\begin{itemize}
\item the vertical arrows all denote left Kan extensions along the indicated fully faithful inclusions of categories;
\item the horizontal arrow labelled $\simeq$ denotes left Kan extension;\footnote{In practice we will start with a finitary functor $F$ defined at least on all $B$-algebras, and study the left Kan extension of its restriction to smooth algebras; in that case the functor harmlessly sends $F|_{\opp{SmAlg}_B}$ to $F|_{\opp{IndSmAlg}_B}$.}
\item the other horizontal arrows denote the obvious restriction maps;
\item {\rm ?loc} stands for {\rm loc}, {\rm hloc}, or {\rm shloc}.
\end{itemize}
\end{proposition}
\begin{proof}
Commutativity of the middle square is clear, since left Kan extensions compose. We now treat the right hand square. Write $\opp{IndSmAlg}_B^\sub{?loc}=\opp{CAlg}_B^\sub{?loc}\cap \opp{IndSmAlg}_B$ for the category of ind-smooth $B$ algebras which also satisfy ?loc. We expand the square as
\begin{equation*}
\begin{CD}
\Fun(\opp{CAlg}_B,\cal C) @>>> \Fun(\opp{CAlg}_B^\sub{?loc},\cal C) @= \Fun(\opp{CAlg}_B^\sub{?loc},\cal C) \\
@AAA  @AAA @AAA \\
\Fun^\sub{fin}(\opp{IndSmAlg}_B,\cal C) @>>> \Fun(\opp{EssSmAlg}_B^\sub{?loc},\cal C) @>>> \Fun(\opp{IndSmAlg}_B^\sub{?loc},\cal C), \\
\end{CD}
\end{equation*}
where all maps except for the horizontal ones in the left hand square are left Kan extensions.
In particular, the right hand square commutes, and so it suffices to show that the outer square commutes.
The bottom composite is just given by restriction (since the source consists of finitary functors).
Now let $C \in \opp{CAlg}_B^\sub{?loc}$.
Using the pointwise formula for left Kan extensions, it suffices to show that the inclusion of index categories \[ (\opp{IndSmAlg}^\sub{?loc}_B)_{/C} \subset (\opp{IndSmAlg}_B)_{/C} \] is cofinal.
By Quillen's theorem A \cite[Theorem 4.1.3.1]{Lurie2009}, this reduces to showing, for any $D \in (\opp{IndSmAlg}_B)_{/C}$, that the category \[ ((\opp{IndSmAlg}^\sub{?loc}_B)_{/C})_{D/} \] is weakly contractible. This is true because it has an initial object, namely the localisation/henselisation/strict henselisation of $D$ at the inverse image of the maximal ideal of $C$ along the map $D \to C$.

Next we prove commutativity of the left most square. The argument is similar to the previous square, but more involved. Again using the pointwise formula for left Kan extensions and Quillen's theorem A, it suffices to show, for $\Spec(A) \in \Sch_B$, $Y \in \Sm_B$, and any morphism $f: \Spec(A) \to Y$, that the category \[ \scr T = ((\opp{SmAff}_B)_{\Spec(A)/})_{/Y} \] is weakly contractible.
Concretely, this category is described as follows: it consists of commutative diagrams of $B$-schemes 
\[\xymatrix{
\Spec (A)\ar[r]\ar@/_5mm/[rr]_f & \Spec (R)\ar[r] & Y
}\]
where $R$ is a smooth $B$-algebra; morphisms are defined in the obvious way. To prove that $\scr T$ is weakly contractible we will even show that it is filtered, i.e., that it is non-empty, any two objects admit morphisms to a common target, and any two parallel morphisms can be equalised.

For non-emptiness, let $A'$ be an ind-smooth $B$-algebra equipped with a surjection $A'\onto A$ whose kernel is a Henselian ideal. Then we claim that there is a dashed arrow making the following diagram commute:
\[\xymatrix{
\Spec (A)\ar[d]\ar[r]^f & Y\ar[d]\\
\Spec (A)'\ar[r]\ar@{-->}[ur] & \Spec(B)
}\]
If $Y$ were affine then this would be the usual lifting property for Henselian surjections along smooth maps of rings, but in fact it holds also when $Y$ is not affine: see either the proof of \cite[Proposition~A.0.4]{ElmantoHoyoisKhanSosniloYakerson2020}, or else \cite[Proposition~2.1.4]{BouthierCesnavicius2022}. Finally, writing $A'=\colim_i A_i'$ as a filtered colimit of smooth $B$-algebras, the new map $\Spec (A')\to Y$ factors through a map $\Spec (A'_{i_o})\to Y$ for some $i_o\gg0$ since $Y$ is finitely presented over $B$; the diagram
\[\xymatrix{
\Spec (A)\ar[r]\ar@/_5mm/[rr]_f & \Spec (A'_{i_o}) \ar[r] & Y
}\]
then shows that $\scr T$ is non-empty.

Next we check that any two objects 
\[\xymatrix{
\Spec (A)\ar[r]\ar@/_5mm/[rr]_f & \Spec (R_i)\ar[r] & Y
}\qquad i=1,2\]
of $\scr T$ admit a morphism to a common target. Form the pushout $\Spec (R_1)\sqcup_{\Spec (A)}\Spec (R_2)=\Spec (R_1\times_AR_2)$ and let $R$ be an ind-smooth $B$-algebra equipped with a surjection $R\onto R_1\times_AR_2$ whose kernel is a Henselian ideal. Just as in the proof of non-emptiness, the induced map $\Spec (R_1\times_AR_2)\to Y$ lifts first to a map $\Spec (R)\to Y$ and then to a map $\Spec (R_{i_o})\to Y$ for some smooth $B$-algebra $R_{i_o}$ mapping to $R$. By construction the object \[\xymatrix{
\Spec (A)\ar[r]\ar@/_5mm/[rr]_f & \Spec (R_{i_o})\ar[r] & Y
}\] of $\scr T$ receives a map from the original two objects, as required.

Equalisation of maps is proved similarly.\NB{?}
\end{proof}

We also record here the following known sufficient condition for being left Kan extended from smooth algebras.
\begin{definition} \label{def:rigid}
Let $F: \Sch^\sub{qcqs,op} \to \scr C$ be a functor.
We call $F$ \emph{rigid} if, for any ring $A$ and henselian ideal $I \subset A$, the canonical map $F(\Spec(A)) \to F(\Spec(A/I))$ is an equivalence.
We use similar terminology when $F$ is defined only on affines, or on a smaller categories of schemes (e.g., over a fixed base), etc.
\end{definition}

\begin{lemma}\label{lemma_rigid_implies_lke}
Let $B$ be a ring, $\cal C$ a compactly generated, presentable $\infty$-category, and $F:\opp{CAlg}_B\to\cal C$ a finitary functor. If $F$ is rigid, then it is left Kan extended from $\opp{SmAlg}_B$.
\end{lemma}
\begin{proof}
Considering maps out of compact generators, we reduce to the case $\scr C = \Spc$.
The claim now is a special case of \cite[Proposition A.0.1]{ElmantoHoyoisKhanSosniloYakerson2020}.
\end{proof}

We also have the following lemma which implies base independence of certain left Kan extensions which we will consider later; we thank Jacob Lurie for pointing this out to us.

\begin{lemma}\label{lem:lke-b-b'}
Let $\cal C$ be a presentable $\infty$-category and $F': \opp{SmAlg}_{\bb Z} \rightarrow \cal C$ a functor; let $F:= L^\sub{sm}_{\bb Z}F': \CAlg_{\bb Z} \rightarrow \cal C$ be its left Kan extension to all rings. Then, for any ring $B$, the canonical counit map
\[
L^\sub{sm}_{B}( F|_{\opp{SmAlg}_B}) \To F|_{\CAlg_B}
\]
is an equivalence, i.e, the restriction of $F$ to $B$-algebras is left Kan extended from smooth $B$-algebras.
\end{lemma}

\begin{proof}
Apply left Kan extension to the commutative square of categories
\begin{equation*}
\begin{CD}
\opp{SmAlg}_\Z @>>> \opp{SmAlg_B} \\
@VVV                 @VVV \\
\CAlg_\Z @>>> \CAlg_B,
\end{CD}
\end{equation*}
where the vertical arrows are the inclusions and the horizontal arrows are given by base change. From the fact that left Kan extensions compose, we get that
\[
L^\sub{sm}_B(L_{\opp{SmAlg_B}/\opp{SmAlg}_\Z}F')  \simeq L_{\CAlg_B/\CAlg_{\Z}}F
\]
The result then follows from the observations that $L_{\opp{SmAlg_B}/\opp{SmAlg}_\Z}F' \simeq F|_{\opp{SmAlg}_B}$ (which follows from the pointwise formula for left Kan extension) and that Kan extension along the bottom arrow is restriction (since base change is left adjoint to the forgetful functor).
\end{proof}

\section{Preliminaries II: Motivic spectra}\label{sec:prelim2-motsp}
In this section we present some aspects of the theory of ($\bb A^1$-invariant) motivic spectra. We set up some machinery in \S\ref{ss:omega} which is convenient for passing between a motivic spectrum and the cohomology theory it represents in a highly coherent fashion. The slice filtration is reviewed in \S\ref{ss:slice} and aspects of its convergence are discussed via its very effective version. We end the section with a short review of the theory of orientations in \S\ref{subsec:orientations}.

\subsection{The $\infty$-category $\SH(X)$}\label{s:SH}
In this subsection we review Morel--Voevodsky's motivic stable homotopy theory. Useful modern references include \cite{Robalo2015} and \cite[\S4.1]{BachmannHoyois2021}. Given a qcqs scheme $X$, its {\em motivic stable homotopy category} $\SH(X)$ is a stable, presentably symmetric monoidal $\infty$-category equipped with a symmetric monoidal functor $M_X:\Sm_X\to\SH(X)$ having the following properties:
\begin{enumerate}
\item $1_{X}:=M_X(X)$ is the unit of the symmetric monoidal $\infty$-category $\SH(X)$, and for any $Y_1,Y_2\in\Sm_X$ there is a natural equivalence $M_X(Y_1\times_X Y_2)\simeq M_X(Y_1)\otimes M_X(Y_2)$ (indeed, these properties both follow from $M_X$ being monoidal). (Note: when $X=\Spec(A)$ is affine, then we write $1_A$ in place of $1_{\Spec(A)}$, and similarly $M_A$ in place of $M_{\Spec(A)}$.)
\item The map $M_X(\bb A_X^1)\to M_X(X)$ induced by the canonical map $\bb A_X^1\to X$ is an equivalence.
\item The {\em Tate motive} $\bb T_X:=\fib(M_X(\bb P_X^1)\to M_X(X))$ is invertible in $\SH(X)$, where the map is induced by the canonical map $\bb P_X^1\to X.$\footnote{We remark that the canonical map $\bb P^1_X \rightarrow X$ admits a section given by the closed immersion $\infty: X \hookrightarrow \bb P^1_X$. This then induces a functorial decomposition: $M_X(\bb P^1_X) \simeq 1_{X} \oplus \bb T_X$.}
\item $M_X$ carries Nisnevich distinguished squares in $\Sm_X$ to cartesian squares in $\SH(X)$.
\end{enumerate}
Furthermore, $\SH(X)$ together with the functor $M_X$ is initial with respect to the above properties 
\cite[Corollary 2.39]{Robalo2015}. In particular, given a morphism $f:X'\to X$ of qcqs schemes, there is an induced symmetric monoidal functor $f^*:\SH(X)\to\SH(X')$ which preserves colimits and satisfies $f^*(M_X(Y))=M_{X'}(Y\times_XX')$ for all $Y\in \Sm_X$.
Via $f^*$, the assignment $\SH$ assembles into a functor $\SH: \Sch^{\qcqs,\op} \rightarrow \Cat_{\infty}$. We refer to objects of $\SH(X)$ as {\em motivic spectra} ({\em over $X$}, if clarification is necessary).

\begin{remark}[Shift and suspension notation]
Being a stable $\infty$-category, $\SH(X)$ admits inverse equivalences given by the suspension and loops functor; we will use homological notation, denoting these respectively by \[E\mapsto E[1]:=\opp{cofib}(E\to 0),\qquad E\mapsto E[-1]:=\opp{fib}(0\to E).\] Common notation in motivic homotopy theory, which we will avoid, is to introduce the $(i,j)$-bigraded spheres $S^{i,j}_X := \bb T_X^{\otimes j}[i-2j]$ and suspension functors $\Sigma^{i,j}E := E \otimes S^{i,j}$. We will instead write out homological shifts and twists by the Tate motive in full.
\end{remark}

It is useful to know that motivic spectra are determined by their pullbacks to fields, through the following proposition:

\begin{proposition}\label{prop:conservative}
Let $X$ be a qcqs scheme locally of finite Krull dimension and $E\in\SH(X)$. Then $E=0$ if and only if, for every point $x\in X$, the pullback $i_x^*E\in\SH(k(x))$ vanishes.
\end{proposition}
\begin{proof}
See \cite[Proposition B.3]{BachmannHoyois2021}.
\end{proof}

\subsubsection{Motives of (pointed) stacks}\label{sss:stacks}
By various universal properties, the functor $M_X$ uniquely extends to presheaves of spaces $\opp{PSh}(\Sm_X)$ and to its pointed analogue. More precisely, there is a natural commutative diagram
\begin{equation}\label{eq:sm-sh}
\xymatrix{
\Sm_X\ar[d]_{\sub{Yoneda}}\ar[r]^{M_X} & \SH(X)\\
\opp{PSh}(\Sm_X)\ar[d]_{\sub{Add disjoint base point}\, E\mapsto E_+} &\\
\opp{PSh}(\Sm_X,\opp{Spc}_*)\ar[d]_{F\mapsto \Sigma^\infty F}\ar@{-->}[uur] & \\
\opp{PSh}(\Sm_X,\opp{Sp}) \ar[r]_{L_{\sub{Nis},\bb A^1}} & \opp{Sh}_{\sub{Nis},\bb A^1}(\Sm_X,\opp{Sp}) \ar[uuu]_{\sigma^\infty}
}
\end{equation}

As usual we will suppress the Yoneda functor from notation. We will continue writing $M_X$ for the extension of $M_X$ to presheaves of spaces. However, in the case of pointed presheaves of spaces this abuse of notation would lead to confusion. Namely, given a pointed presheaf of spaces $E$ (where the point is usually hidden in the notation) on $\Sm_X$, the notation $M_X(E)$ could refer to the functor out of either $\opp{PSh}(\Sm_X)$ or $\opp{PSh}(\Sm_X,\opp{Spc}_*)$; in the first case we have forgotten the fact that $E$ was pointed, and in the second case we have remembered the point. The standard solution is that the extension of $M_X$ to pointed presheaves of spaces is denoted by $\tilde M_X$ (i.e., the dotted arrow in the diagram), and often referred to as the associated {\em reduced motive}. In the case of our pointed presheaf of spaces $E$, we thus have
\[\tilde M_X(E\textrm{ with its point remembered})=\opp{cofib}(M_X(X)\to M_X(E\textrm{ with its point forgotten}))\] where the arrow on the right is induced by the point of $E$. If the base point is clear, we will simply write $\tilde M_X(E)$.

For example, the (Yoneda of the) projective line $\bb P_X^1$ together with its section $\infty:X\to\bb P_X^1$ defines $(\bb P_X^1,\infty)\in \opp{PSh}(\Sm_X,\opp{Spc}_*)$, and then \[\tilde M_X((\bb P_X^1,\infty))=\tilde M(\bb P_X^1)=\opp{cofib}(M_X(X)\xto{\infty^*}M_X(\bb P_X^1))=\bb T_X.\]

Another example comes from the Picard stack. For a qcqs scheme $X$ we write $\opp{Pic}(X)$ for its usual Picard group, and $\Pic(X)$ for its Picard groupoid. That is, $\Pic(X)$ is the symmetric monoidal $1$-groupoid consisting of line bundles on $X$ and monoidal structure given by tensor product of line bundles. Viewing $\Pic(X)$ as a grouplike  $\bb E_\infty$-monoid in anima, or equivalently as a connective spectrum, we write
\[
 \Pic \in \PShv(\Sch^\sub{qcqs}, \Spt),\qquad X\mapsto\Pic(X)
\]
for this presheaf. The underlying presheaf of pointed spaces will be denoted by $\Omega^\infty\Pic \in \PShv(\Sch, \opp{Spc}_{*})$; note that the base point is classified by the map $\roi: \Spec(\bb Z) \rightarrow \Pic$ corresponding to the trivial line bundle. For any qcqs scheme $X$ we may restrict this to a pointed presheaf on $\Sm_X$ (with base point given by the trivial line bundle $\roi: X \rightarrow \Pic$) and form the reduced motive  $\tilde M_X(\Omega^\infty\Pic)\in\SH(X)$.

\begin{remark}[Two remarks on the Picard stack] \label{rem:pic} We make two further remarks on the Picard stack. First we note that the presheaf of spectra $\Pic$ admits a natural $\Z$-linear structure (in other words it is the Eilenberg--MacLane spectrum associated with a $\D(\bb Z)$-valued presheaf).  Indeed, the inclusion of symmetric monoidal groupoids $\bb G_m(X)\subset \Pic(X)$ (given by the natural identification of groups $\bb G_m(X)\isoto \opp{End}_{\Pic(X)}(\roi_X)$) induces a map of spectra $\bb G_m(X)[1]\to \Pic(X)$ (where the left hand side is really the associated Eilenberg--MacLane spectrum). Furthermore the presheaf $X\mapsto\Omega^\infty\Pic(X)$ satisfies fpqc descent (in the groupoid sense, consequently also as a presheaf of spaces), whence the same is true of $\Pic$ viewed as a presheaf valued in connective space (as sheafiness may be detected after applying $\Omega$). Zariski sheafifying the map $\bb G_m(-)[1]\to \Pic(-)^\otimes$ therefore defines a natural map of presheaves of spectra \begin{equation}(\tau^{\le 1}R\Gamma_\sub{Zar}(-,\bb G_m))[1]\To\Pic,\label{eqn_pic_as_Gm_coh}\end{equation} which is an equivalence (e.g., by checking on stalks).

Second, note that the counit of the $\Sigma^\infty$-$\Omega^\infty$ adjunction between pointed spaces and connective spectra defines a map of connective spectra \begin{equation}\Sigma^\infty\Omega^\infty\Pic(X)\To \Pic(X),\label{eqn_Pic_mess}\end{equation} i.e., ``the map from the free $\bb E_\infty$-group on the underlying space of $\Pic(X)$ back to itself''. The left hand side of \eqref{eqn_Pic_mess} does not see the fact that line bundles may be tensored together, but the right hand side and the map do.
\end{remark}

\subsection{Motivic spectra versus cohomology theories}\label{ss:omega}
The first main goal of this section is Definition \ref{definition:graded-p1-mult} below, in which we present a functor associating motivic spectra to suitable graded presheaves of spectra. Curiously, the construction is most easily motivated by approaching the problem in the other direction: how can we associate graded presheaves of spectra to motivic spectra, and how much structure does the graded presheaf carry?

To be more precise about this question, let $X$ be a qcqs scheme. Firstly, the functor $\sigma^\infty$ from \eqref{eq:sm-sh} admits a right adjoint, leading to an adjunction
\begin{equation*}
\sigma^{\infty}: \opp{Sh}_{\Nis, \bb A^1}(\Sm_X,\Spt) \rightleftarrows \SH(X): \omega^{\infty}
\end{equation*}
where the right adjoint $\omega^\infty$ is given by the mapping spectrum construction $\omega^\infty(E)(-)=\opp{map}_{\SH(X)}(M_X(-),E)$, for any $E\in \SH(X)$. In particular for each $j \in \Z$ we can extract an $\bb A^1$-invariant Nisnevich sheaf from $E$ via the formula
\begin{equation}
E(j):= \opp{map}_{\SH(X)}(M_X(-),E\otimes \bb T^{\otimes  j}_X[-2j]): \Sm^\sub{op}_X \To \Spt.
\label{eqn:EjNis}\end{equation}
It will be expedient to record all of these sheaves at the same time, and so we consider a graded variant of the above adjunction \[ \sigma^{\infty,\gr}: \Gr\opp{Sh}_{\Nis, \bb A^1}(\Sm_X,\Spt) \rightleftarrows \SH(X): \omega^{\infty,\gr} \] where \[ \sigma^{\infty,\gr}(E) = \bigoplus_{j \in \Z} \sigma^\infty(E^j) \otimes \bb T_X^{\otimes j} \] and thus \[ \omega^{\infty,\gr}(E)^j \wequi E(j)[2j]=\omega^\infty(E\otimes\bb T_X^{\otimes j}). \]

\begin{remark}[Notational warning!]
We highlight the change of indexing in the first equivalence of the previous line ($j$ on the left and $(2j, j)$ on the right). More generally, we will see various functors below between graded sheaves $F$ and motivic spectra $E$, possibly equipped with extra structure, and they will tend to satisfy $F^j=E(j)[2j]$. Since we work in the generality of sheaves of spectra, this convention cannot be fixed by shearing: see Remark \ref{rmk:shearing1}.
\end{remark}

\begin{remark}[Incompatibility with pullback]\label{rem:pullback1}
Given a morphism of qcqs schemes $f:Y\to X$, there is an induced symmetric monoidal pullback functor $f^*:\Gr\opp{Sh}_{\Nis, \bb A^1}(\Sm_X,\Spt)\to \Gr\opp{Sh}_{\Nis, \bb A^1}(\Sm_Y,\Spt)$ given by termwise left Kan extending from $\Sm_X$ to $\Sm_Y$, then Nisnevich sheafifying, then $\bb A^1$-localising, i.e., $f^*=L_{\bb A^1}L_\sub{Nis}L_X^\sub{sm}$. Via the functor $f^*:\SH(X)\to\SH(Y)$, this pullback functor does commute with $\sigma^{\infty,\gr}$ (since both are compatible with left Kan extension from $\Sm_X$ to $\Sm_Y$) but not with $\omega^{\infty,\gr}$.
\end{remark}

The functor $\omega^{\infty,\gr}$ is unfortunately also not very multiplicative: it admits a natural $\bb E_1$-monoidal structure, but in general not more. More precisely recall that, for a presentably symmetric monoidal $\infty$-category $\scr C$, we have $\Gr \scr C \wequi \PShv(\Z^\delta,\opp{Spc}) \otimes \scr C$ by \cite[Proposition 4.8.1.17]{LurieHA}, where the tensor product takes place inside $\Pr^L$  ($=$ presentable $\infty$-categories with small colimit preserving functors). Therefore $\bb E_n$-monoidal structures on $\sigma^{\infty,\gr}$ are the same as $\bb E_n$-monoidal maps $\Z \to \Pic(\SH(X))$ sending $1$ to $\bb T$. When $n=1$ then, since $\Z$ is the free $\bb E_1$-monoid on an invertible generator, there is the canonical such $\bb E_1$-monoidal map and so an induced $\bb E_1$-monoidal structure on $\sigma^{\infty,\gr}$. Its right adjoint $\omega^{\infty,\gr}$ then inherits the desired natural lax $\bb E_1$-monoidal structure. This induces, in particular, a functor \[\Alg_{\bb E_1}(\Gr\opp{Sh}_{\Nis, \bb A^1}(\Sm_X,\Spt))\longleftarrow\Alg_{\bb E_1}(\SH(X)):\omega^{\infty,\gr}_{\bb E_1}.\]

\begin{remark}[Warning: not better than $\bb E_1$]\label{rem:not_E2}
When $n=2$ the existence of an $\bb E_2$-monoidal structure on $\sigma^{\infty,\gr}$ (extending the above $\bb E_1$-monoidal structure) implies that the switch map on $\bb T_X\otimes\bb T_X$ is homotopic to the identity, which need not always be true (for example, if $X$ is the spectrum of a field $k$, then it holds if and only $-1$ is a square in $k$.\footnote{Indeed under the identification between $\pi_0\Map_{\SH(k)}(1_k, 1_k) \cong \mathrm{GW}(k)$ of \cite[Theorem 6.4.1, Remark 6.4.2]{Morel2004}, the endomorphism of $1_k$ induced by the switch map on $\bb T_k^{\otimes 2}$ is the class of $\langle -1 \rangle$ \cite[Lemma 6.1.1(2), Remark 6.3.5]{Morel2004}. On the other $\langle -1 \rangle = 1$ if and only if $-1 =1 \in k^{\times}/(k^{\times})^2$; see \cite[Chapter III, Theorem 5.2]{MilnorHusemoller}.} It follows that in general $\omega^{\infty,\gr}$ need not admit any lax $\bb E_2$-monoidal structure extending the above lax $\bb E_1$-monoidal structure, and certainly not lax $\bb E_\infty$-monoidal structure. In particular, $\omega^{\infty,\gr}_{\bb E_1}$ does not appear to upgrade to a functor $\CAlg(\SH(X))\to \CAlg(\Gr\opp{Sh}_{\Nis, \bb A^1}(\Sm_X,\Spt))$.
\end{remark}

The standard partial solution to the failure of symmetric multiplicativity of $\omega^{\infty,\gr}$ is the introduction of periodizations. A {\em periodization} of an $\bb E_\infty$-algebra $E\in\SH(X)$ is the data of a compatible $\bb E_\infty$-algebra structure on $E\otimes\bb T_X^{\otimes \star}\in \Gr\SH(X)$ (this does not exist for every $E$). Let $\CAlg(\SH(X))_\sub{per}$ denote the $\infty$-category of $\bb E_\infty$-algebras in $\SH(X)$ together with a chosen periodization, with morphisms required to respect the periodizations.

Observe that the functor $\omega^{\infty,\gr}$ does now upgrade to a functor \begin{equation} \CAlg(\Gr\opp{Sh}_{\Nis, \bb A^1}(\Sm_X,\Spt))\longleftarrow \CAlg(\SH(X))_\sub{per}: \omega^{\infty,\gr}_\sub{per}.\label{eqn:omega-otimes}\end{equation} Indeed, given $E\in \CAlg(\SH(X))_\sub{per}$, applying the symmetric monoidal functor $\omega^\infty$ to $E\otimes\bb T_X^{\otimes \star}$ naturally equips $\omega^{\infty,\gr}(E)\in \Gr\opp{Sh}_{\Nis, \bb A^1}(\Sm_X,\Spt)$ with an $\bb E_\infty$-algebra structure.

\begin{construction}[Periodized motivic ring spectra and $\cal P_X'$]\label{cons:periodized}
To show that \eqref{eqn:omega-otimes} arises from an lax symmetric monoidal version of $\omega^{\infty,\gr}$ (and so, in particular, also treat modules over periodized motivic ring spectra), we  reinterpret periodized motivic ring spectra as $\bb E_\infty$-algebras over a certain $\bb E_\infty$-ring in $\SH(X)$. To define it, first let $\cal P_X$ be the free $\bb E_\infty$-algebra in $\Gr \SH(X)$ on $\bb T\lra{1}$, i.e., $\bb T$ placed in grading degree $1$; more explicitly, $\cal P_X^j=(\bb T_X^{\otimes j})_{h\Sigma_j}$ for $j\ge0$, and $=0$ for $j<0$. Write $u:\cal P_X\otimes \bb T_X\lra{1}\to \cal P_X$ for the canonical map of $\cal P_X$-modules induced by the tautological map $\bb T_X\lra{1}\to \cal P_X$ in $\Gr \SH(X)$, and let
\[\opp{Mod}_{\cal P_X}(\Gr \SH(X))\supset L_u\opp{Mod}_{\cal P_X}(\Gr \SH(X))\]
be the full subcategory of $\cal P_X$-modules which are $u$-periodic; see \cite[\S12.1]{BachmannHoyois2021}. It admits a symmetric monoidal structure such that the inclusion is lax symmetric monoidal and its left adjoint $L_u$ is symmetric monoidal. Applying \cite[Lemma 12.1]{BachmannHoyois2021} we see that $L_u\cal P_X$ is an idempotent $\cal P_X$-algebra, and that $L_u\opp{Mod}_{\cal P_X}(\Gr \SH(X))$ is precisely the $\infty$-category of $L_u\cal P_X$-modules. Finally let $\cal P_X':=(L_{u}\cal P_X)^0\in\CAlg(\SH(X))$ be its degree zero piece.

We claim that periodized motivic ring spectra are the same as $\bb E_\infty$-$\cal P_X'$-algebras, i.e., that there is a natural equivalence \begin{equation}\CAlg(\SH(X))_\sub{per}\simeq \CAlg(\SH(X))_{\cal P_X'/}.\label{eq:per-px}\end{equation}
Indeed firstly, given a map $\cal P_X' \to E$ in $\CAlg(\SH(X))$, the equivalence $E\otimes\bb T_X^{\otimes \star}\simeq E\otimes_{\cal P_X'}L_u\cal P_X$ in $\Gr \SH(X)$ lets us transfer the $\bb E_\infty$-algebra structure from the right hand side to the left, i.e., defines a periodization on $E$. Conversely, given a periodization on some $E\in\CAlg(\SH(X))$, the unit map of $E$ induces $\bb T_X\to E\otimes \bb T_X$, which then induces a map of $\bb E_\infty$-algebras $\cal P_X\to E\otimes\bb T_X^\star$ (using the periodization), which then factors as $L_u\cal P_X\to E\otimes\bb T_X^\star$, and finally passing to degree zero pieces defines the desired map of $\bb E_\infty$-algebras $\cal P_X'\to E$. The two functors are clearly inverse to each other.

(Although it is not necessary for what follows, we note by \cite[Lemma 12.1]{BachmannHoyois2021} that $L_{u}\cal P_X$ is given by the telescope construction 
\[\colim (\cal P_X \xrightarrow{u} \cal P_X\otimes\bb T_X^{\otimes -1}\lra{-1} \xrightarrow{u} \cal P_X\otimes\bb T_X^{\otimes -2}\lra{-1}\xrightarrow{u} \cdots), \] and so \[\cal P_X'=\colim(1_X\to \bb T_X\otimes\bb T_X^{\otimes -1}\to (\bb T_X^{\otimes 2})_{h\Sigma_2}\otimes \bb T^{\otimes -2}_X\to (\bb T_X^{\otimes 3})_{h\Sigma_3}\otimes \bb T^{\otimes -3}_X\to\cdots)\] by passing to degree zero pieces.)
\end{construction}

One consequence of Proposition~\ref{prop:P1-spectrum-from-graded} below will be that $\omega^{\infty,\gr}$ upgrades to a lax symmetric monoidal functor from $\cal P_X'$-modules to graded sheaves of spectra. We may even identify the image in terms of the following construction:

\begin{construction}[$\cal Q_X$]\label{cons:lin-cohom-theories}
Similarly to Construction \ref{cons:periodized}, let $\cal Q_X\in\CAlg(\Gr\opp{Sh}_{\Nis, \bb A^1}(\Sm_X,\Spt))$ be the free $\bb E_\infty$-algebra on $(\Sigma^\infty \P_X^1)\lra{1}$, i.e., $\Sigma^\infty \P_X^1$ placed in degree $1$, and let $v:\cal Q_X\otimes \Sigma^\infty \P_X^1\lra1\to \cal Q_X$ be the canonical map of $\cal Q_X$-modules which is induced by the tautological map $\Sigma^\infty \P_X^1\lra1\to \cal Q_X$ in $\Gr\opp{Sh}_{\Nis, \bb A^1}(\Sm_X,\Spt)$. Let \[\opp{Mod}_{\cal Q_X}(\Gr\opp{Sh}_{\Nis, \bb A^1}(\Sm_X,\Spt))\supset L_\sub{pbf}\opp{Mod}_{\cal Q_X}(\Gr\opp{Sh}_{\Nis, \bb A^1}(\Sm_X,\Spt))\] be the full subcategory of $\cal Q_X$-modules spanned by the $v$-periodic objects. It admits a symmetric monoidal structure such that the inclusion is lax symmetric monoidal and its left adjoint \[L_\sub{pbf}:\opp{Mod}_{\cal Q_X}(\Gr\opp{Sh}_{\Nis, \bb A^1}(\Sm_X,\Spt))\To L_\sub{pbf}\opp{Mod}_{\cal Q_X}(\Gr\opp{Sh}_{\Nis, \bb A^1}(\Sm_X,\Spt))\] is symmetric monoidal. See \cite[\S3]{Hoyois2020} for the theory of periodization in this context where, unlike in Construction \ref{cons:periodized}, we are periodizing with respect to a map from a non-invertible module (namely $\cal Q_X\otimes \Sigma^\infty \P^1\lra1$); in particular, $L_\sub{pbf}\opp{Mod}_{\cal Q_X}\Gr\opp{Sh}_{\Nis, \bb A^1}(\Sm_X,\Spt)$ does not seem to be the same as the $\infty$-category of $L_\sub{pbf}\cal Q_X$-modules, and we will not be interested in the latter.

We henceforth adopt the following notation for these symmetric monoidal categories:
\[\Gr\opp{Sh}_{\Nis, \bb A^1}(\Sm_X,\Spt)_\sub{c}:=\opp{Mod}_{\cal Q_X}(\Gr\opp{Sh}_{\Nis, \bb A^1}(\Sm_X,\Spt))\]
and
\[\Gr\opp{Sh}_{\Nis, \bb A^1}(\Sm_X,\Spt)_\sub{c,pbf}:=L_\sub{pbf}\opp{Mod}_{\cal Q_X}(\Gr\opp{Sh}_{\Nis, \bb A^1}(\Sm_X,\Spt)).\]
Indeed, note that an object $M\in \Gr\opp{Sh}_{\Nis, \bb A^1}(\Sm_X,\Spt)_\sub{c}$ consists of a graded sheaf $M^\star \in \Gr\opp{Sh}_{\Nis, \bb A^1}(\Sm_X,\Spt)$ together with maps of sheaves $b_{M}^j: \Sigma^\infty \P_X^1 \otimes M^j \to M^{j+1}$, for $j\in\bb Z$, equipped with certain coherent symmetry conditions which we do not make explicit; but we think of these maps as multiplication by some first Chern class of the tautological line bundle. The subcategory $\Gr\opp{Sh}_{\Nis, \bb A^1}(\Sm_X,\Spt))_\sub{c,pbf}$ consists of those $M$ for which the adjoint maps \[b_{M}^{j,\dagger}: M^j \to \Omega_{\Sigma^\infty \P_X^1} M^{j+1}=\fib(M^{j+1}(\bb P^1_-)\xto{\infty^*}M^{j+1})\] are equivalences for all $j\in\bb Z$, i.e., for which the $\bb P^1$-bundle formula holds. In particular, the functor $L_\sub{pbf}$ imposes the $\bb P^1$-bundle formula, whence the notation.
\end{construction}

We may now prove the main result providing a structured relation between motivic spectra and graded sheaves of spectra; see also \cite[Construction 6.5]{AnnalaHoyoisIwasa2024} and \cite[\S3]{CarmeliFeng2025} for similar results.

\begin{proposition} \label{prop:P1-spectrum-from-graded}
There is an $\bb E_1$-monoidal equivalence $\omega^{\infty,\gr}_{\otimes}$ filling in the following commutative square of $\bb E_1$-monoidal categories and lax $\bb E_1$-monoidal functors:
\begin{equation*}
\begin{CD}
\Gr\opp{Sh}_{\Nis, \bb A^1}(\Sm_X,\Spt)_\sub{c,pbf} @<{\omega^{\infty,\gr}_{\otimes}}<{\wequi}< \opp{Mod}_{\cal P_X'}(\SH(X)) \\
@VVV                @VVV \\
\Gr\opp{Sh}_{\Nis, \bb A^1}(\Sm_X,\Spt) @<{\omega^{\infty,\gr}}<< \SH(X).
\end{CD}
\end{equation*}
(where the vertical arrows are the forgetful functors, and the bottom horizontal arrow is the lax $\bb E_1$-monoidal upgrade of $\omega^{\infty,\gr}$ from before Remark \ref{rem:not_E2}).

Moreover, $\omega^{\infty,\gr}_{\otimes}$ upgrades to a symmetric monoidal equivalence (note that its domain and codomain are symmetrical monoidal categories).
\end{proposition}
\begin{proof}
When we periodize with respect to $v$ to pass from $\Gr\opp{Sh}_{\Nis, \bb A^1}(\Sm_X,\Spt)_\sub{c}$ to $\Gr\opp{Sh}_{\Nis, \bb A^1}(\Sm_X,\Spt)_\sub{c,pbf}$, we obtain the same result if we first symmetrically monoidally invert the source $\cal Q_X\otimes\Sigma^\infty\bb P_X^1\lra 1$ of the map $v$ \cite[Proposition~3.1]{Hoyois2020}. But since the symmetric monoidal functor $\sigma^\infty$ carries $\cal Q_X$ to $\cal P_X$, this inversion yields $\opp{Mod}_{\cal P_X}(\Gr\SH(X))$. Thus $\sigma^\infty$ induces a symmetric monoidal equivalence \[\Gr\opp{Sh}_{\Nis, \bb A^1}(\Sm_X,\Spt)_\sub{c,pbf} \quis L_u\opp{Mod}_{\cal P_X}(\Gr \SH(X))=\opp{Mod}_{L_u\cal P_X}(\Gr \SH(X)).\] 

It remains to note that the functor $\opp{Mod}_{L_u\cal P_X}(\Gr \SH(X)) \to \opp{Mod}_{\cal P_X'}(\SH(X))$, given by restricting to degree zero, is an equivalence. Indeed, its inverse is given by sending a $\cal P_X'$-module $M$ to $M\otimes_{\cal P_X'}L_u\cal P_X$.
\end{proof}

\begin{definition} \label{definition:graded-p1-mult}
Define the {\em motivic Eilenberg--MacLane functor} \[H:\Gr\opp{Sh}_{\Nis, \bb A^1}(\Sm_X,\Spt)\sub{c,pbf}\To \SH(X)\] to be the lax symmetric monoidal functor defined as the inverse of $\omega^{\infty,\gr}_\otimes$ followed by forgetting the $\cal P_X'$-module structure. Observe that, by the commutative diagram in Proposition \ref{prop:P1-spectrum-from-graded}, the diagram
\[\xymatrix{
\Gr\opp{Sh}_{\Nis, \bb A^1}(\Sm_X,\Spt)_\sub{c,pbf}\ar[r]^-H\ar[d] & \SH(X)\ar[dl]^{\omega^{\infty,\gr}}\\
\Gr\opp{Sh}_{\Nis, \bb A^1}(\Sm_X,\Spt)&
}\]
commutes. That is, given $F\in\Gr\opp{Sh}_{\Nis, \bb A^1}(\Sm_X,\Spt)_\sub{c,pbf}$, one has $(HF)(j)[2j]\simeq F^j$ for all $j\in\bb Z$.

\end{definition}

\begin{remark}[Compatibility with pullback]\label{rem:pullbackQ}
Given a morphism of qcqs schemes $f:Y\to X$, we defined a lax symmetric monoidal pullback functor $f^*=L_{\Nis,\bb A^1}L^\sub{sm}:\Gr\opp{Sh}_{\Nis, \bb A^1}(\Sm_X,\Spt)\to \Gr\opp{Sh}_{\Nis, \bb A^1}(\Sm_Y,\Spt)$ in Remark \ref{rem:pullback1}. Since $f^*$ carries $(\Sigma^\infty\bb P_X^1)\lra1$ to $(\Sigma^\infty\bb P_Y^1)\lra1$, we see that it induces a symmetric monoidal pullback functor from $\cal Q_X$-modules to $\cal Q_Y$-modules, which then induces a symmetric monoidal pullback functor
\[L_\sub{pbf}L_{\Nis,\bb A^1}L^\sub{sm}:\Gr\opp{Sh}_{\Nis, \bb A^1}(\Sm_X,\Spt)_\sub{c,pbf}\To \Gr\opp{Sh}_{\Nis, \bb A^1}(\Sm_Y,\Spt)_\sub{c,pbf}.\] This is compatible with $H$ in that the diagram
\[\xymatrix{
\Gr\opp{Sh}_{\Nis, \bb A^1}(\Sm_Y,\Spt)_\sub{c,pbf}\ar[r]^-H&\SH(Y)\\
\Gr\opp{Sh}_{\Nis, \bb A^1}(\Sm_X,\Spt)_\sub{c,pbf}\ar[r]_-H\ar[u]^{L_\sub{pbf}L_{\Nis,\bb A^1}L^\sub{sm}}&\SH(X)\ar[u]_{f^*}
}\]
commutes (because $\sigma^{\infty,\gr}$ commutes with $f^*$, as noted in Remark \ref{rem:pullback1}).
\end{remark}

In the remainder of the section we specialize the above constructions to the case of algebras, where there are superficial simplifications. We have described periodized ring spectra in terms of $\cal P_X'$-algebras as in~\eqref{eq:per-px}. On the (graded) sheaf side of the story, let \begin{equation} \CAlg(\Gr\opp{Sh}_{\Nis, \bb A^1}(\Sm_X,\Spt)_\sub{c})\supset \CAlg(\Gr\opp{Sh}_{\Nis, \bb A^1}(\Sm_X,\Spt)_\sub{c,pbf})\label{eqn_Ch_pbf}\end{equation} denote the categories of $\bb E_\infty$-algebras in $\Gr\opp{Sh}_{\Nis, \bb A^1}(\Sm_X,\Spt)_\sub{c}$ and $\Gr\opp{Sh}_{\Nis, \bb A^1}(\Sm_X,\Spt)_\sub{c,pbf}$ respectively. Explicitly, an object $F$ of $\CAlg(\Gr\opp{Sh}_{\Nis, \bb A^1}(\Sm_X,\Spt)_\sub{c})$ is an $\bb E_\infty$-algebra $F^\star\in \CAlg(\Gr\opp{Sh}_{\Nis, \bb A^1}(\Sm_X,\Spt))$ together with a map $c_F:\Sigma^\infty\bb P^1_X\to F^1$ in $\opp{PSh}(\Sm_X,\opp{Sp})$. Note further that $c_F$ induces, for each $j\in\bb Z$, a morphism \begin{equation}F^j\To \fib(F^{j+1}(\bb P^1_-)\xto{\infty^*}F^{j+1})\label{eqn:pbf}\end{equation} of sheaves on $\Sm_X$, and $\CAlg(\Gr\opp{Sh}_{\Nis, \bb A^1}(\Sm_X,\Spt)_\sub{c,pbf})$ is precisely the full subcategory where the maps \eqref{eqn:pbf} are equivalences for all $j$, i.e., those $F=(F^\star,c_F)$ for which the $\bb P^1$-bundle formula holds.



\begin{remark}[Traditional cohomology theories]\label{remark:traditional}
``Traditional cohomology theories'' of algebraic geometry which are $\bb A^1$-invariant give rise to objects of $\CAlg(\Gr\opp{Sh}_{\Nis, \bb A^1}(\Sm_X,\Spt)_\sub{c,pbf})$. Indeed, such cohomology theories have the following structure: they are a family of $\bb A^1$-invariant Nisnevich $\D(\bb Z)$-valued sheaves $F(j)$, for $j\in\bb Z$, equipped with compatible multiplication maps $F(i)\otimes F(j)\to F(i+j)$; they admit a theory of Chern classes, in particular a first Chern class map $c_1:R\Gamma_\sub{Nis}(-,\bb G_m)[1]\to F(1)$; and the projective bundle formula holds.

From a modern point of view the multiplication maps assemble to form an $\bb E_\infty$-algebra structure in $\Gr\opp{Sh}_{\Nis, \bb A^1}(\Sm_X,\D(\bb Z))$; shearing and forgetting the $\bb Z$-linearity yields $F(\star)[2\star]\in \CAlg(\Gr\opp{Sh}_{\Nis, \bb A^1}(\Sm_X,\Spt))$. Meanwhile, by restricting attention to the line bundle $\roi(1)\in\opp{Pic}(\bb P^1_X)$, the first Chern class map induces a morphism $\Sigma^\infty\bb P^1_X\to F(1)[2]$, defined as the composition
\[\Sigma^\infty\bb P^1_X\xto{[\roi(1)]} \Pic\simeq (\tau^{\le 1}R\Gamma_\sub{Zar}(-,\bb G_m))[1]\To  F(1)[2].\] Here the first map classifies $\roi(1)\in\opp{Pic}(\bb P_X^1)$, and the equivalence is \eqref{eqn_pic_as_Gm_coh}. We thus indeed obtain an object of $\CAlg(\Gr\opp{Sh}_{\Nis, \bb A^1}(\Sm_X,\Spt)_\sub{c,pbf})$.

For example, for any integer $m$ invertible on $X$, \'etale cohomology gives us an object $R\Gamma_\sub{\'et}(-,\mu_m^{\otimes \star})[2\star]$ of $\CAlg(\Gr\opp{Sh}_{\Nis, \bb A^1}(\Sm_X,\Spt)_\sub{c,pbf})$, with first Chern class induced by the Kummer equivalence $R\Gamma_\sub{\'et}(-,\bb G_m)/m\simeq R\Gamma_\sub{\'et}(-,\mu_m^{\otimes j})[1]$.

From this point of view, the purpose of introducing $\cal Q_X$, its modules, and the $v$-periodization functor $L_\sub{pbf}$ is precisely to put such cohomologies into a rigorous framework and to relate them to motivic homotopy theory. 
\end{remark}

\begin{remark}[Periodized motivic ring spectrum]\label{rem:qx}
Restricting the symmetric monoidal equivalence $\omega^{\infty,\gr}_{\otimes}$ of Proposition \ref{prop:P1-spectrum-from-graded} to $\bb E_\infty$-algebras defines a symmetric monoidal equivalence
\begin{equation} \CAlg(\Gr\opp{Sh}_{\Nis, \bb A^1}(\Sm_X,\Spt)_\sub{c,pbf})\stackrel\sim\longleftarrow \CAlg(\SH(X))_\sub{per}: \omega^{\infty,\gr}_\sub{per}\label{eqn:omega-otimes3}\end{equation}
which, unravelling the proof of Proposition \ref{prop:P1-spectrum-from-graded}, is given more explicitly as follows: given $E\in \CAlg(\SH(X))_\sub{per}$, the object of $\CAlg(\Gr\opp{Sh}_{\Nis, \bb A^1}(\Sm_X,\Spt))$ underlying $\omega^{\infty,\gr}_\otimes(E)$ is the $\bb E_\infty$-algebra $\omega^{\infty,\gr}_\sub{per}(E)=\omega^\infty(E\otimes\bb T_X^\star)$ of \eqref{eqn:omega-otimes}, and the map $c_E:\Sigma^\infty\bb P^1_X\to \omega^{\infty,\gr}_\otimes(E)^1=\omega^\infty(E\otimes\bb T_X)$ is adjoint to that induced by the unit $\bb T_X\to E\otimes\bb T_X$.
\end{remark}

\begin{remark}[Motivic ring spectra vs algebras in graded presheaves]\label{rem:recover}
Restricting the Eilenberg--MacLane functor $H$ to $\bb E_\infty$-algebras defines \begin{equation}H:\CAlg(\Gr\opp{Sh}_{\Nis, \bb A^1}(\Sm_X,\Spt)_\sub{c,pbf})\To\CAlg(\SH(X)),\label{eqn:multEM}\end{equation}
which (often combined with Remark \ref{remark:traditional}) we will use to construct motivic ring spectra. Note that, again by definition of $H$, this construction has the following properties:
\begin{enumerate}
\item The diagram
\[\xymatrix{
\CAlg(\Gr\opp{Sh}_{\Nis, \bb A^1}(\Sm_X,\Spt)_\sub{c,pbf})\ar[d]\ar[r]^-{H}&\CAlg(\SH(X))\ar[d]\\\opp{Alg}_{\bb E_1}(\Gr\opp{Sh}_{\Nis, \bb A^1}(\Sm_X,\Spt)) & \opp{Alg}_{\bb E_1}(\SH(X))\ar[l]^-{\omega^{\infty,\gr}_{\bb E_1}}\\
}\]
commutes, where both vertical arrow are the obvious forgetful maps. Furthermore, given any $F\in\CAlg(\Gr\opp{Sh}_{\Nis, \bb A^1}(\Sm_X,\Spt)_\sub{c,pbf})$, then the map $c_F:\Sigma^\infty\bb P_X^1\to F^1$ can be functorially recovered from $HF$: indeed, via the $\sigma^\infty$-$\omega^\infty$ adjunction, it corresponds to the map $\bb T_X\to HF\otimes\bb T_X$ induced by the unit of $HF$. In summary, the functor \eqref{eqn:multEM} remembers the $\bb E_1$-algebra structure and ``first Chern class of $\roi(1)$'' of the input, but seems to forget some of the $\bb E_\infty$-algebra structure (despite outputting $\bb E_\infty$-algebras).
\item The following diagram commutes:
\begin{equation}\xymatrix{
&\CAlg(\SH(X))_\sub{per}\ar[d]^{\sub{forget periodization}}\ar[dl]_{\omega^{\infty,\gr}_\sub{per}}\\
\CAlg(\Gr\opp{Sh}_{\Nis, \bb A^1}(\Sm_X,\Spt)_\sub{c,pbf})\ar[r]^-H & \CAlg(\SH(X))
}\label{eqn:H_of_per}\end{equation}
I.e., given an $\bb E_\infty$-algebra $E\in\SH(X)$ equipped with a compatible $\bb E_\infty$-algebra structure on $E\otimes \bb T_X^\star\in\Gr\SH(X)$, then the $\bb E_\infty$-algebra $F:=\omega^\infty(E\otimes T^\star)\in\Gr\opp{Sh}_{\Nis,\bb A^1}(\Sm_X,\Spt)_\sub{c,pbf}$ satisfies $HF=E$.
\end{enumerate}
\end{remark}

\begin{remark}[Modules over $HF$]\label{rem:modules_over_HF}
Note that, for any fixed $\bb E_\infty$-algebra $F\in \Gr\opp{Sh}_{\Nis,\bb A^1}(\Sm_X,\Spt)_\sub{c,pbf}$, the motivic Eilenberg--MacLane functor induces a symmetric monoidal equivalence \[H:\opp{Mod}_F(\Gr\opp{Sh}_{\Nis,\bb A^1}(\Sm_X,\Spt)_\sub{c,pbf})\quis\opp{Mod}_{HF}(\SH(X)).\] Indeed, the forgetful functor $\opp{Mod}_{\cal P_X'}(\SH(X))\to \SH(X)$ becomes an equivalence once we pass to modules over $(\omega^{\infty,\gr}_\otimes)^{-1}(F)\in\CAlg(\SH(X))_\sub{per}$ on both sides.
\end{remark}

We finish this discussion of the functor $H$ for algebras with a result which will be required in the proof of Theorem \ref{theorem_BL1}(2). From the point of view of ``traditional cohomology theories'' of Remark \ref{remark:traditional}, it may seem unsatisfactory that, even in the presence of a first Chern class map $c_1:R\Gamma_\sub{Nis}(-,\bb G_m)[1]\to F^1$, the $\infty$-category $\CAlg(\Gr\opp{Sh}_{\Nis, \bb A^1}(\Sm_X,\Spt)_\sub{c,pbf})$ and the functor $H$ a priori only remember the first Chern class $c_1(\roi_{\bb P_X^1}(1))\in \tilde \pi_0(F^1(\bb P_X^1))$. But it turns out just having this first Chern class of $\roi(1)$ is sometimes sufficient to recover the behaviour of the original map $c_1$ on units:

\begin{lemma}\label{lem:units-Chern}
Let $X$ be a qcqs scheme, let $F^\star\in \CAlg(\Gr\opp{Sh}_{\Nis, \bb A^1}(\Sm_X,\Spt))$, and suppose that $c_1, c'_1: R\Gamma_{\Nis}(-,\Gm)[1] \rightarrow F^1$ are maps of presheaves of spectra on $\Sm_X$ such that the following conditions are satisfied:
\begin{enumerate}
\item $F^j=0$ for $j<0$;
\item $c'_1(\roi_{\bb P^1_X}(1)) =c_1(\roi_{\bb P^1_X}(1))$ in $\pi_0(F^1(\bb P_X^1))$;
\item The projective bundle formula holds over $X$ with respect to $c_1$, i.e., for each $d,j \ge 0$, the map 
\[
\bigoplus_{i=0}^d F^{j-i}(X)\To  F^j(\bb P_X^d) 
\]
induced by powers of the class $c_1(\roi_{\bb P_X^d}(1))\in \pi_0(F^1(\bb P^d_X))$ is an equivalence.
\end{enumerate}
Then the maps of presheaves of abelian groups on $\Sm_X$
\[
c_1, c_1': \bb G_m= \pi_1(R\Gamma_{\Nis}(-,\Gm)[1]) \To \pi_{1}F^1(-)
\]
coincide.
\end{lemma}
\begin{proof}
We first claim that the maps
\begin{equation}\Sigma^\infty\Omega^\infty\Pic\xto{\eqref{eqn_Pic_mess}}\Pic\stackrel{\sub{\eqref{eqn_pic_as_Gm_coh}}}\simeq (\tau^{\le 1}R\Gamma_\sub{Nis}(-,\bb G_m))[1]\xto{c_1\mathrm{\,or\,}c_1'}  F^1.\label{eqn_2maps}\end{equation}
of presheaves of spectra on $\Sm_X$ are homotopic. We argue similarly to Definition \ref{definition:orientation_of_K} and Lemma \ref{lemma:HZA-orient-KGL-compat} below, so we will be brief.

A result in motivic homotopy theory \cite[\S 4, Proposition 3.7]{MorelVoevodsky1999} shows that the map $\colim_m (\bb P_X^m,\infty)\to\Omega^\infty\Pic$ is an equivalence after applying $L_{\sub{Nis},\bb A^1}$. This yields a short exact sequence of abelian groups
 \[0\To \opp{lim}_m^1\tilde\pi_1F^1(\bb P_X^m)\To \pi_0\Map_{\opp{PSh}(\Sm_X,\Spt)}(\Sigma^\infty\Omega^\infty\Pic,F^1)\To \opp{lim}_m\tilde\pi_0(F^1(\bb P^m_X))\To 0,\] where $\tilde\pi_*$ refers to reduced homotopy groups (i.e., the elements which vanish via $\infty^*$ on $\pi_*(F^1(X))$). Condition (1) and (3) imply that the transition maps $F^1(\bb P_X^{m+1})\to F^1(\bb P_X^{m})$ are equivalences for all $m\ge1$, so that the $\opp{lim}^1_m$ term vanishes and the $\opp{lim}_m$ term is simply $\pi_0(F^1(\bb P^1_X))$. That is, we have shown that the map
\[\pi_0\Map_{\opp{PSh}(\Sm_X,\Spt)}(\Sigma^\infty\Omega^\infty\Pic,F^1)\To \pi_0\Map_{\opp{PSh}(\Sm_X,\Spt)}(\Sigma^\infty\bb P^1_X,F^1)=\tilde\pi_0(F^1(\bb P^1_X)),\]
induced by $[\roi(1)]:\Sigma^\infty\bb P_X^1\to\Sigma^\infty\Omega^\infty\Pic$, is an isomorphism.

Therefore, to show that the maps \eqref{eqn_2maps} are in fact the same, it is enough to check that they agree after precomposing with the map $[\roi(1)]$ of the previous paragraph. But that is precisely hypothesis (2), and so the proof of the claim is complete.

It is now enough to show that the map of presheaves of abelian groups $\pi_1(\Sigma^\infty\Omega^\infty\Pic(-))\to\pi_1(\Pic(-))$ induced by \eqref{eqn_Pic_mess} is surjective. In fact the following stronger statement is true: the map of pointed presheaves of spaces spaces $\Omega^\infty\Sigma^\infty\Omega^\infty\Pic\to\Omega^\infty\Pic$ is split. Indeed this follows formally from the theory of adjoints: a splitting is provided by the unit map $\Omega^\infty\Pic\to \Omega^\infty\Sigma^\infty\Omega^\infty\Pic$ for $\Omega^\infty\Pic$.

\end{proof}

\subsubsection{The cdh variant and a pullback formula}\label{sss:cdhvariant}
The cdh topology plays a crucial role in this paper; here we recall how it interacts with motivic homotopy theory and with the theory developed so far in this subsection. Unlike the relatively formal nature of the theory so far, the forthcoming cdh results ultimately depend on Ayoub's proper base change theorem for $\SH$ \cite{Ayoub2007, Ayoub2007II}.

Given a qcqs scheme $X$ and $E\in\SH(X)$, the Nisnevich sheaves $E(j)$ associated with $E\in\SH(X)$ in \eqref{eqn:EjNis} in fact extend to finitary $\bb A^1$-invariant cdh sheaves on all qcqs $X$-schemes, via the formula
\begin{equation}
E(j):(\Sch^{\qcqs}_X)^\sub{op} \To \Spt,\qquad (f:Y\to X) \mapsto (\omega^\infty f^*(E\otimes\bb T_X^{\otimes j}[-2j])) (Y).\label{eqn:Ejcdh}
\end{equation}
Note that, on smooth $X$-schemes $Y$, this indeed agrees with $E(j)(Y)$ as originally defined in \eqref{eqn:EjNis} because then $f^*$ has a left adjoint $f_\sharp:\SH(Y)\to\SH(X)$ sending $1_Y$ to $M_X(Y)$. The fact that \eqref{eqn:Ejcdh} is a finitary $\bb A^1$-invariant cdh sheaf is obtained by applying the following proposition to $G=E\otimes\bb T_X^{\otimes j}[-2j]$:

\begin{proposition}[Cisinski] \label{prop:auto-cdh}
For any $G \in \SH(X)$, the presheaf \[ (\Sch^{\qcqs}_X)^\sub{op} \To \Spt,\qquad (f:Y\to X) \mapsto (\omega^\infty f^*G) (Y) \] is a finitary $\bb A^1$-invariant cdh sheaf.
\end{proposition}
\begin{proof}
The cdh descent claim is due to Cisinski \cite[Proposition 3.7]{Cisinski2013}. For more references, including the finitariness, see \cite[Lemma 3.4.1]{ElmantoHoyoisIwasaKelly2021}.
\end{proof}

Our next goal is Corollary \ref{corol_pullback_of_EM}, which we will use to control the cdh sheaves \eqref{eqn:Ejcdh} in the case that $E$ comes from the motivic Eilenberg--MacLane functor. This requires introducing the following cdh variant of Construction \ref{cons:lin-cohom-theories}. Namely, let $\cal Q_{X,\sub{cdh}}\in \CAlg(\Gr\opp{Sh}_{\cdh, \bb A^1}(\Sch^\sub{qcqs}_X,\Spt))$ be the free $\bb E_\infty$-algebra on $(\Sigma^\infty \P_X^1)\lra{1}$, or equivalently the image of $\cal Q_X$ under the functor induced by \[L_{\cdh,\bb A^1}L^\sub{sm}:\opp{Sh}_{\Nis, \bb A^1}(\Sm_X,\Spt)\To \opp{Sh}_{\cdh, \bb A^1}(\Sch^\sub{qcqs}_X,\Spt).\] Again write $v:\cal Q_{X,\sub{cdh}}\otimes \Sigma^\infty \P^1\lra1\to \cal Q_{X,\sub{cdh}}$ for the canonical map and let \[\Gr\opp{Sh}_{\cdh, \bb A^1}(\Sch_X^\sub{qcqs},\Spt)_\sub{c}\supset \Gr\opp{Sh}_{\cdh, \bb A^1}(\Sch_X^\sub{qcqs},\Spt)_\sub{c,pbf},\] and denote the $\infty$-category of $\cal Q_{X,\sub{cdh}}$-modules  and its subcategory of $v$-periodic modules. Again write $L_\sub{pbf}$ for the symmetric monoidal left adjoint to this inclusion. The next lemma lets us calculate pullbacks of certain motivic spectra via cdh methods:

\begin{lemma}\label{lem:pullback}
Let $f:Y\to X$ be a morphism of qcqs schemes. Then the following diagram commutes
\begin{equation}
\begin{tikzcd}
 & \Gr\opp{Sh}_{\cdh, \bb A^1}(\Sch^\sub{qcqs}_X,\Spt)_\sub{c} \ar{r}{L_\sub{pbf}} & \Gr\opp{Sh}_{\cdh, \bb A^1}(\Sch^\sub{qcqs}_X,\Spt)_\sub{c,pbf} \ar{dd}{\sub{restrict and forget}}\\
\Gr\opp{Sh}_{\Nis, \bb A^1}(\Sm_X,\Spt)_\sub{c,pbf} \ar{ur}{L_{\cdh, \bb A^1}L^\sub{sm}} \ar{dr}[swap]{L_{\Nis, \bb A^1}L^\sub{sm}}& & \\
 & \Gr\opp{Sh}_{\Nis, \bb A^1}(\Sm_Y,\Spt)_\sub{c} \ar{r}[swap]{L_\sub{pbf}} &  \Gr\opp{Sh}_{\Nis, \bb A^1}(\Sm_Y,\Spt)_\sub{c,pbf} \\
 \end{tikzcd}
\end{equation}
\end{lemma}
\begin{proof}
For the proof (but not needed for the rest of the paper), we introduce $\infty$-categories
\[
\Gr\opp{Sh}_{\Nis, \bb A^1}(\Sch_X^\sub{qcqs},\Spt)_\sub{c}\supset \Gr\opp{Sh}_{\Nis, \bb A^1}(\Sch_X^\sub{qcqs},\Spt)_\sub{c,pbf},
\] 
which are again variants of Construction \ref{cons:lin-cohom-theories}, this time using $\Sch_X^\sub{qcqs}$  with the Nisnevich topology. That is, $\Gr\opp{Sh}_{\Nis, \bb A^1}(\Sch_X^\sub{qcqs},\Spt)_\sub{c}$ are $\cal Q_{X,\Nis}$-modules where $\cal Q_{X,\Nis}$ is the image of $\cal Q_X$ under the functor $L_{\Nis,\bb A^1}L^\sub{sm}:\opp{Sh}_{\Nis, \bb A^1}(\Sm_X,\Spt)\to \opp{Sh}_{\Nis, \bb A^1}(\Sch^\sub{qcqs}_X,\Spt)$, and $\Gr\opp{Sh}_{\Nis, \bb A^1}(\Sch_X^\sub{qcqs},\Spt)_\sub{c,pbf}$ is the full subcategory of $v$-periodic modules where $v$ is defined in the same way.

%
%
%

These $\infty$-categories fit into the following diagram, whose commutation we will proceed to justify. 
\begin{equation}\hspace{-0.7cm}
\begin{tikzcd}[column sep=small]
 & \Gr\opp{Sh}_{\cdh, \bb A^1}(\Sch_X^\sub{qcqs},\Spt)_\sub{c} \ar{rr}{L_\sub{pbf}} & & \Gr\opp{Sh}_{\cdh, \bb A^1}(\Sch_X^\sub{qcqs},\Spt)_\sub{c,pbf} \ar{dd} \ar{dl}{\sub{forget}}\\
\Gr\opp{Sh}_{\Nis, \bb A^1}(\Sm_X,\Spt)_\sub{c,pbf} \ar[bend left=15]{ur}{L_{\cdh, \bb A^1}L^\sub{sm}} \ar{r}{L_{\Nis, \bb A^1}L^\sub{sm}}  \ar[swap,bend right=15]{dr}{L_{\Nis, \bb A^1}L^\sub{sm}}& \Gr\opp{Sh}_{\Nis, \bb A^1}(\Sch^\sub{qcqs}_X,\Spt)_\sub{c} \ar{d}{\sub{restrict}} \ar{r}{L_\sub{pbf}} & \Gr\opp{Sh}_{\Nis, \bb A^1}(\Sch^\sub{qcqs}_X,\Spt)_\sub{c,pbf} \ar{dr}{\sub{restrict}} & \\
 & \Gr\opp{Sh}_{\Nis, \bb A^1}(\Sm_Y,\Spt)_\sub{c} \ar[swap]{rr}{L_\sub{pbf}} &  & \Gr\opp{Sh}_{\Nis, \bb A^1}(\Sm_Y,\Spt)_\sub{c,pbf} \\
 \end{tikzcd}
\end{equation}
The commutation of the bottom left triangle is obvious, and the far right triangle commutes by definition of the vertical map. The bottom trapezoid commutes since both composites satisfy the same universal property: namely, they send $F \in \Gr\opp{Sh}_{\Nis, \bb A^1}(\Sch^\sub{qcqs}_X,\Spt)_\sub{c}$ to the initial $v$-periodic $\cal Q_{Y}$-module under $F|_{\Sm_Y}$.

We now justify the commutation of the top parallelogram, which is non-trivial. First we claim that the composite $L_\sub{pbf}L_{\Nis, \bb A^1}L^\sub{sm}$ lands in $\Gr\opp{Sh}_{\cdh, \bb A^1}(\Sch_X^\sub{qcqs},\Spt)_\sub{c,pbf}$, in other words the procedure creates graded $\bb A^1$-invariant cdh sheaves. Indeed, by Remark~\ref{rem:pullbackQ}, for any $F \in  \Gr\opp{Sh}_{\Nis, \bb A^1}(\Sm_X,\Spt)_\sub{c,pbf}$ we have an equivalence of presheaves on $\Sch_X^\sub{qcqs}$ \[
(L_\sub{pbf}L_{\Nis, \bb A^1}L^\sub{sm}F)^j \simeq HF(j)[2j]
\]
(on the right we apply notation \eqref{eqn:Ejcdh} to $HF\in\SH(X)$) and the right hand side is an $\bb A^1$-invariant cdh sheaf by Proposition~\ref{prop:auto-cdh}. Therefore the canonical transformation of functors  $L^\sub{sm} \rightarrow  L_\sub{pbf}L_{\Nis, \bb A^1}L^\sub{sm}$ on graded presheaves on $\Sm_X$ induces a transformation $L_\sub{pbf}L_{\cdh, \bb A^1}L^\sub{sm} \quis L_\sub{pbf}L_{\Nis, \bb A^1}L^\sub{sm}$, which is an equivalence because there is an inverse induced by $L_\sub{Nis}\to L_\sub{cdh}$. This establishes the desired commutation.
\end{proof}

We now obtain our ``explicit'' calculation of \eqref{eqn:Ejcdh} in the case of motivic spectra coming from the motivic Eilenberg--MacLane construction:

\begin{corollary}\label{corol_pullback_of_EM}
Let $X$ be a qcqs scheme and $j\in\bb Z$. Then the following diagram commutes:
\[\xymatrix{
\Gr\opp{Sh}_{\cdh, \bb A^1}(\Sch_X^\sub{qcqs},\Spt)_\sub{c,pbf}\ar[r]^-{G\mapsto G^j} & \opp{Sh}_{\cdh, \bb A^1}(\Sch_X^\sub{qcqs},\Spt)\\
\Gr\opp{Sh}_{\Nis, \bb A^1}(\Sm_X,\Spt)_\sub{c,pbf}\ar[r]_-H\ar[u]^{L_\sub{pbf}L_{\cdh,\bb A^1}L^\sub{sm}} & \SH(X)\ar[u]_{E\mapsto E(j)[2j]}
}\]
\end{corollary}
\begin{proof}
Let $f:Y\to X$ be any qcqs scheme over $X$. Combining the commutative diagram of Remark \ref{rem:pullbackQ} with Lemma \ref{lem:pullback} shows that the following diagram commutes:
\begin{equation}\label{eq:h-cdh}
\begin{tikzcd}
\Gr\opp{Sh}_{\Nis, \bb A^1}(\Sm_Y,\Spt)_\sub{c,pbf} \ar{r}{H} & \SH(Y)\\
\Gr\opp{Sh}_{\cdh, \bb A^1}(\Sch_X,\Spt)_\sub{c,pbf} \ar{u}{\text{restrict and forget}}& \\
\Gr\opp{Sh}_{\Nis, \bb A^1}(\Sm_X,\Spt)_\sub{c,pbf} \ar{u}{L_\sub{pbf}L_{\cdh, \bb A^1}L^\sub{sm}} \ar{r}{H}  & \SH(X) \ar{uu}{f^*}.
\end{tikzcd}
\end{equation}
Given $F\in \Gr\opp{Sh}_{\Nis, \bb A^1}(\Sm_X,\Spt)_\sub{c,pbf}$, the final identity in Definition \ref{definition:graded-p1-mult} now establishes the desired natural equivalence $HF(j)[2j]\simeq (L_\sub{pbf}L_{\cdh,\bb A^1}L^\sub{sm}F)^j$ on smooth $Y$-schemes. Varying $Y$ completes the proof.
\end{proof}

%

\subsection{The slice filtration}\label{ss:slice}
We next discuss Voevodsky's slice filtrations on $\SH(X)$. We continue to let $X$ be a qcqs scheme and we refer to \cite{Voevodsky2002, Voevodsky2002a} for the original references. The $\infty$-category of \emph{effective motivic spectra}
\begin{equation*}
\SH(X)^{\eff} \subset \SH(X)
\end{equation*}
is the smallest stable subcategory of $\SH(X)$ which is closed under colimits and contains $M_X(Y)[i]$ for all  $Y\in\Sm_X$, $i \in \Z$. Effective motivic spectra are closed under tensor product, and so $\SH(X)^{\eff}$ is a presentably symmetric monoidal $\infty$-category. More generally, for any $j\in\bb Z$, the $\infty$ category of \emph{$j$-effective motivic spectra} is
\begin{equation*}
\SH(X)^{\eff}(j) := \bb T^{\otimes j} \otimes \SH(X)^\eff \subset \SH(X)
\end{equation*}
Observe that $\SH(X)^{\eff}(j)\supset \SH(X)^{\eff}(j+1)$.

The inclusions $\SH(S)^\eff(j) \into \SH(S)^\eff$ preserve colimits and hence admit compatible right adjoints as $n$ varies (see \cite[\S 13.1]{BachmannHoyois2021}, \cite[\S 3.0.1]{ElmantoLevineSpitzweckOstvaer2022} for $\infty$-categorical treatments, and \cite[\S 2]{Voevodsky2002a} for the original); we display these adjunctions as
\begin{equation*}
\iota^j: \SH(X)^{\eff}(j) \rightleftarrows \SH(X): r^j.
\end{equation*}
Therefore, setting $\Fil^j_\sub{slice} := \iota^jr^j$, we define a natural filtration \cite[Lemma~3.8]{ElmantoLevineSpitzweckOstvaer2022} on any motivic spectrum $E\in \SH(X)$ \begin{equation}E\leftarrow\cdots\leftarrow \Fil^{j-1}_\sub{slice}E\leftarrow \Fil^j_\sub{slice}E\leftarrow\Fil^{j+1}_\sub{slice}E\leftarrow\label{eq:slice_filt} \cdots \end{equation} such that, for any $j\in\bb Z$, the following hold:
\begin{enumerate}
\item the formation of $E \mapsto \Fil^j_\sub{slice}E$ is functorial in $E$;
\item $\Fil^j_\sub{slice}E\in \SH(X)^{\eff}(j)$;
\item the map $\Fil^j_\sub{slice}E \rightarrow E$ is universal for maps out of an $j$-effective motivic spectrum to $E$ in the sense that any other such map $F \rightarrow E$ factors uniquely through $\Fil^j_\sub{slice}E \rightarrow E$;
\item $s^jE:=\opp{cofib}(\Fil^{j+1}_\sub{slice}E\to\Fil^{j}_\sub{slice}E)$ is right orthogonal to $\SH(X)^{\eff}(j+1)$, i.e., the mapping spectrum $\opp{map}_{\SH(X)}(F,s^jE)$ vanishes for all $F\in \SH(X)^{\eff}(j+1)$.
\end{enumerate}
We will call $\Fil^0_\sub{slice}E$ the \emph{effective cover} of $E$. 

The filtration (\ref{eq:slice_filt}) is known as the {\em slice filtration},\footnote{The terse notation $f_nE=\Fil^n_\sub{slice}E$ is often used in the literature. We use upper indexing since the filtration is decreasing.} and the graded pieces $s^jE$ are known as the {\em slices}. Applying the functor \begin{equation}\SH(X)\xto{\omega^\infty}\opp{Sh}_{\Nis, \bb A^1}(\Sm_X,\Spt)\xto{\sub{global sections}}\Spt\label{eqn:omega_glob}\end{equation} equips $(\omega^\infty E)(X)$ with the filtration $\Fil^\star_\sub{slice}(\omega^\infty E)(X):=(\omega^\infty\Fil_\sub{slice}^\star E)(X)$, which by a slight abuse of terminology we will also refer to as the slice filtration on $(\omega^\infty E)(X)$; by exactness of $\omega^\infty$, the graded pieces of this filtration are $\gr_\sub{slice}^j(\omega^\infty E)(X)=\omega^\infty s^jE(X)\in\Spt$ for $j\in\bb Z$.

\begin{remark}[Exhaustiveness]\label{rem:exhaustive}
The slice filtration is always exhaustive in the sense that for any motivic spectrum $E \in \SH(X)$, the natural map
\[
\colim_{n \rightarrow - \infty} \Fil^j_\sub{slice}E \To E
\]
is an equivalence. See e.g. \cite[Lemma 3.8]{ElmantoLevineSpitzweckOstvaer2022}.
\end{remark}

\begin{remark}[Multiplicativity]\label{rem:mult}
For any scheme $X$, the formation of the slice filtration defines a lax symmetric monoidal functor \cite[\S13.4]{BachmannHoyois2021}
\begin{equation}
\Fil_\sub{slice}^{\star}:\SH(X) \To \Fil\SH(X),
\label{eqn:slice:mult}
\end{equation}
where the target is endowed with the Day convolution symmetric monoidal structure. Thus it induces a functor
\[
\Fil_\sub{slice}^{\star}:\CAlg(\SH(X)) \To \CAlg(\Fil\SH(X)).
\]
Similarly, passing to graded pieces (which is a symmetric monoidal functor) defines a functor \[\gr^\star_\sub{slice}:\CAlg(\SH(X)) \To \CAlg(\Gr\SH(X)).\] 

Furthermore, each of the functors in (\ref{eqn:omega_glob}) are lax monoidal, whence they induce lax monoidal functors on the corresponding categories of filtered objects, which then induce functors on the categories of $\bb E_\infty$-algebras in filtered objects:
\[\CAlg(\Fil\SH(X))\xto{\omega^\infty}\CAlg(\Fil\opp{Sh}_{\Nis, \bb A^1}(\Sm_X,\Spt)))\xto{\sub{global sections}}\CAlg(\Fil\Spt).\] Composing this with (\ref{eqn:slice:mult}) defines a functor $\CAlg(\SH(X))\to \CAlg(\Fil\Spt)$, showing that whenever $E$ is an $\bb E_\infty$-algebra in motivic spectra, then the slice filtration on $(\omega^\infty E)(X)$ naturally makes it into an $\bb E_\infty$-algebra in filtered spectra.
\end{remark}

\begin{remark}[Naturality]\label{rem:naturality} Let $f: Y \rightarrow X$ be a morphism of qcqs schemes. Then, for any $n \in \bb Z$, we have that $f^*(\SH(X)^{\eff}(j)) \subset \SH(Y)^{\eff}(j)$ (since $f^*$ preserves colimits and $f^*M_X(T) \simeq M_Y(T \times_X Y)$ for all smooth $X$-schemes $T$). Therefore the induced map $f^*\Fil_\sub{slice}^{\star}E \rightarrow f^*E$ of filtered spectra, where the target is given the constant filtration, factors uniquely through a map 
\begin{equation}\label{eq:natural_slice}
f^*\Fil_\sub{slice}^{\star}E \To \Fil^{\star}_\sub{slice}f^*E,
\end{equation}
and this induces a map of graded spectra 
\begin{equation}\label{eq:natural_graded}
f^*s^{\star}E \To s^{\star}f^*E.
\end{equation}
The discussion on multiplicativity in Remark~\ref{rem:mult} further shows that both the induced maps on filtered and graded spectra are multiplicative.

Now suppose that $f: X \rightarrow Y$ is a smooth morphism. Then $f^*$ admits a left adjoint $f_{\sharp}$ which preserves the inclusion $\SH(X)^{\eff}(j) \hookrightarrow \SH(X)$ since it preserves colimits and $f_{\sharp}M_X(T) \simeq M_Y(T)$ for all smooth $X$-schemes $T$. This implies that $f^*r^j \simeq r^jf^*$. Since $f^*$ and the inclusion of effective spectra always commute, we thus conclude that $f^*$ preserves the formation of slices in this case, i.e., \eqref{eq:natural_slice} and \eqref{eq:natural_graded} are equivalences. More generally, they are equivalences if $f$ is a cofiltered limit of smooth affine morphisms: this follows from the smooth case and the fact that $\SH$ is finitary \cite[Proposition C.12(4)]{Hoyois2014}.
\end{remark}

A difficulty with the slice filtration $\Fil_\sub{slice}^\star E$ is that it is not clear whether it is convergent. Towards addressing this question, we define the subcategory of \emph{very effective motivic spectra}
\[
\SH(X)^\sub{veff} \subset \SH(X)
\]
to be the full subcategory of $\SH(X)$ generated under colimits and extensions by $M_X(Y)$, for $Y\in\Sm_X$. We remark that $\SH(X)^\sub{veff}$ does not contain negative shifts of $M_X(Y)$ and is thus not a stable $\infty$-category. In fact, $\SH(X)^\sub{veff}$ is the non-negative part of a $t$-structure on $\SH(X)^\sub{eff}$ by \cite[Proposition 1.4.4.11]{LurieHA}. 

The following proposition will be used to control the connectivity of the slice filtration on $K$-theory; see Lemma~\ref{lemm:convergence-bootstrap} below. 

\begin{proposition} \label{prop:detecting-veff}
Let $X$ be Noetherian scheme of dimension $\le d$.
\begin{enumerate}
\item  Given $E\in\SH(X)^\sub{veff}$, then $\omega^\infty(E)$ is $-d$-connective as a Nisnevich sheaf of spectra on smooth $X$-schemes; more generally, $\omega^{\infty,\gr}(E)^j$ is $j-d$-connective for any $j\ge0$.
\item Suppose $d=1$ and let $E \in \SH(X)^\sub{eff}$. If $\omega^\infty(E)$ is connective then $E \in \SH(X)^\sub{veff}$.
\end{enumerate}
\end{proposition}
\begin{proof}
(1): Since connectivity is preserved under colimits and extensions, it suffices to treat the case $E = M_X(Y)$, where $Y \in \Sm_X$. Thus we reduce to the shifted stable connectivity theorem \cite{Druzhinin2022stable}. We remark however that in this paper we will only require part (1) in the case $d\le 1$, and so the reference \cite{SchmidtStrunk2018} is sufficient. The ``more generally'' claim follows by replacing $E$ by $E\otimes\bb T_X^{\otimes j}\simeq E\otimes M_X(\bb G_{m,X})^{\otimes j}[j]$ and noting that $E\otimes M_X(\bb G_{m,X})^{\otimes j}$ is still very effective.

(2) In the case of fields, this was proved in \cite[Proposition 4]{bachmann-very-effective}.
We give here a proof which works for bases of dimension $\le 1$.
Out of the adjunction between $\sigma^\infty$ and $\omega^\infty$ we extract a comonad $C = \sigma^\infty\omega^\infty$ and thus an augmented simplicial object \[ E \leftarrow C^{\circ \bullet+1} E. \]
The functor $\omega^\infty: \SH(X)^\eff \to \opp{Sh}_{\Nis, \bb A^1}(\Sm_X,\Spt)$ is conservative and preserves colimits.
Since the map $\colim_{m \in \Delta^\op} \omega^\infty C^{\circ m+1} E \to \omega^\infty E$ is an equivalence (indeed, after applying $\omega^\infty$ the simplicial object is split, c.f., \cite[Proof of Lemma 4.7.3.13]{LurieHA}), we deduce that the same is true for $\colim_{m \in \Delta^\op} C^{\circ m+1} E \to E$.

Now assume in addition that $\omega^\infty(E)$ is connective.
We seek to prove that $E \in \SH(X)^\veff$.
Write $E_m := \colim_{\Delta^{\op}_{\leq m}} E|_{\Delta^{\op}_{\leq m}}$ for the $m^\sub{th}$ partial geometric realization of our simplicial object, so that $E \wequi \colim_m E_m$ and we have cofiber sequences $E_m \to E_{m+1} \to M_{m+1}[m+1]$, where $M_{m+1}$ is a summand of $C^{\circ m+2} E$ \cite[Remark 4.3.4]{LurieDAGVIII}.
It will thus be enough to show that \[ C^{\circ m+1}E = \sigma^\infty \omega^\infty C^{\circ m} E \quad\text{lies in } \SH(X)^\veff[-m] \]
The functor $\sigma^\infty$ sends connective sheaves of spectra into $\SH(X)^\veff$.
Our assumption thus implies that $CE \in \SH(X)^\veff$.
By (1), if $F \in \SH(X)^\veff$ then $\omega^\infty(F)$ is $(-1)$-connective (here we use that $X$ has dimension $\le 1$) and so $CF \in \SH(X)^\veff[-1]$.
We thus prove by induction that $C^{\circ m+1}E \in \SH(X)^\veff[-m]$, as needed.
\end{proof}

Finally we record that, similarly to Proposition \ref{prop:conservative}, (very) effectivity may be detected by pulling back to fields:

\begin{proposition} \label{prop:detect-effective}
Let $X$ be a qcqs scheme locally of finite Krull dimension and $E \in \SH(X)$. Then $E$ is effective (resp.~very effective) if and only if, for every point $x\in X$, the pullback $i_x^*E\in\SH(k(x))$ is effective (resp.~very effective).
\end{proposition}
\begin{proof}
See \cite[Proposition B.3]{BachmannHoyois2021}.
\end{proof}

\subsection{Orientations}\label{subsec:orientations}
Another aspect of motivic homotopy theory that we will need is the theory of orientations, which is completely analogous to its classical counterpart from topology. For a more thorough treatment, we refer the reader to \cite{Deglise2019}. Given $E \in \CAlg(\h\SH(X))$,\footnote{We follow a standard convention of setting up the framework of orientations in the generality of homotopy commutative motivic ring spectra, but will in fact only be interested in the case that $E$ arises from an actual $\bb E_\infty$-algebra in $\SH(X)$.}
an \emph{orientation of $E$} is a homotopy class of maps $c_E:\tilde M_X(\Omega^\infty\Pic)\to E\otimes\bb T_X$ which fits into a commutative (up to unspecified homotopy) diagram
\[\xymatrix{
\tilde M_X(\Omega^\infty\Pic)\ar[r]^{c_E} & E\otimes\bb T_X\\
\bb T_X\ar[ur]_{\eta_E \otimes\bb T_X}\ar[u]^{[\roi(1)]}&
}\]
where $\eta_E:1_X\to E$ is the unit of $E$ and the vertical arrow is the map corresponding to $\scr O(1)$.

\begin{remark} \label{rmk:orientation-covers}
It is a standard fact from $\bb A^1$-homotopy theory  \cite[\S 4, Proposition 3.7]{MorelVoevodsky1999} that $\tilde M_X(\Omega^\infty\Pic)\in\SH(X)^\sub{eff}(1)$. In fact, there is an equivalence \begin{equation}\colim_m \tilde M_X(\bb P_X^m)\quis \tilde M_X(\Omega^\infty\Pic),\label{eqn_O1_infty}\end{equation} (where the transition maps $\bb P_X^m\to \bb P_X^{m+1}$ are the closed inclusions given by the last $m+1$-coordinates, and each map $\bb P_X^m\to\Pic$ of pointed presheaves of spaces on $\Sm_X$ corresponds to $\roi(1)\in\opp{Pic}(\bb P_X^m)$), and it is easily checked that each $\tilde M_X(\bb P_X^m)$ is $1$-effective.

Therefore, given $E \in \CAlg(\h\SH(X))$, applying $\Fil^1_\sub{slice}$ to any orientation $\tilde{M}_X(\Omega^{\infty}\Pic) \to E\otimes \bb T_X$ yields a map $\tilde{M}_X(\Omega^{\infty}\Pic) \to\Fil^1_\sub{slice}(E\otimes\bb T_X)\wequi \Fil_\sub{slice}^0E\otimes \bb T_X$, which is an orientation of $\Fil^0_\sub{slice}E$.
\end{remark}

\begin{remark} \label{rmk:orientation-morphisms}
Let $E\in \CAlg(\h\SH(X))$ and let $c_E$ be an orientation of $E$.
\begin{enumerate}
\item Given a morphism $E \to F$ in $\CAlg(\h\SH(X))$, the composition $\tilde M_X(\Omega^\infty\Pic)\xto{c_E}E\otimes \bb T_X\to F\otimes \bb T_X$ is an orientation of $F$.
\item Let $f:Y\to X$ be a morphism of qcqs schemes. Then, pulling back, the map \[\tilde M_Y(\Omega^\infty\Pic)=f^*(\tilde M_X(\Omega^\infty\Pic))\xto{f^*c_E}f^*(E\otimes \bb T_X)=f^*E\otimes \bb T_X\] is an orientation of $f^*E$.
\item Combining part (1) with Remark \ref{rmk:orientation-covers} shows that $c_E$ naturally induces an orientation of $s^0E$.
\end{enumerate}
\end{remark}


One of the key points of the theory of orientations is that it allows us to deduce the projective bundle formula for cohomology theories; this is completely analogous to the story for classical oriented spectra. Let $Y\in\Sm_X$. By Yoneda and adjunctions in the left hand side of diagram \eqref{eq:sm-sh}, any line bundle $L$ on $Y$ corresponds uniquely to a homotopy class of maps $\Sigma^\infty_+Y\to \Pic$ (in presheaves of spectra on $\Sm_X$); moving to $\SH(X)$ then naturally induces a homotopy class of maps $[L]:M_X(Y)\to\tilde M_X(\Omega^\infty\Pic)$. Consequently, given $E\in\CAlg(\h\SH(X))$ equipped with an orientation $c_E$, and $Y\in\Sm_X$ equipped with a line bundle $ L$, we obtain the composition \[ c_{1,E}(L) := c_E\circ[L] \in H^2(E(1)(Y)).\] The projective bundle formula in motivic homotopy theory is then as follows:\NB{Replace $c_1$ by $c_E$? Keep an eye on.}

\begin{theorem}[Morel]\label{thm_pbf_for_general_E}
Let $X$ be a qcqs scheme and $E\in\CAlg(\h\SH(X))$ equipped with an orientation $c_E:\tilde M_X(\Omega^\infty\Pic)\to E\otimes\bb T_X$. Then, for any rank $d$ projective bundle $\pi:P\to X$, and $j\in\bb Z$, the homotopy class of maps of spectra
\[\sum_{i=0}^d c_{1,E}(\roi_P(1))^{i}\pi^*:\bigoplus_{i=0}^d E(j-i)(X)[-2i] \To E(j)(P)\]
is an equivalence.
\end{theorem}
\begin{proof}
See, for example, \cite[Theorem~2.1.13]{Deglise2019}.
\end{proof}

\section{$\bb A^1$-motivic cohomology of schemes}\label{sec:a1-mot-schemes} The goal of this section is introduce the protagonist of this paper, namely \emph{$\bb A^1$-motivic cohomology}, in Definition~\ref{definition_A^1_mot-_coh}. In the language of Remark \ref{remark:traditional}, it assembles into a traditional cohomology theory whose motivic Eilenberg--MacLane spectrum is equivalent to the zero slice of the motivic spectrum $\KGL$ representing Weibel's homotopy $K$-theory $\KH$; we refer the reader to Remark~\ref{rem:motivic_is_traditional} for more details. The construction of $\bb A^1$-motivic cohomology uses the machinery that we have reviewed in the previous section. Furthermore, we use Voevodsky's slice filtration to equip $\KH(X)$, for any qcqs scheme $X$, with a natural, descending, multiplicative filtration $\Fil^{\star}_{\bb A}\KH(X)$, whose graded pieces correspond to the $\bb A^1$-motivic cohomology; we discuss this filtration in detail in \S\ref{subsub:slices-kgl}. Bott periodicity for $\KGL$ will show that the slices of $\KGL$ produce a periodization of its zero slice.

In contrast to the simplicity of the definition of $\bb A^1$-motivic cohomology, proving properties about it will be difficult. For example, it is not at all clear from its definition that it is a presheaf of complexes, as opposed to mere spectra. To establish this structure and produce its first Chern class in Construction~\ref{cons:low-weights}, we use an unstable variant of the slice filtration in \S\ref{subsec:low-weights}.

We also define a variant of $\bb A^1$-motivic cohomology, called $\bb A^1$-cdh-motivic cohomology. In fact, we will eventually see in Theorem \ref{thm:a1-a1cdh} that it is not a variant: the two theories coincide. But until that comparison has been established, we must separately keep track of, and independently establish properties of, both theories. Many such basic properties are summarized in Theorem~\ref{theorem:SH_mot_coh}. In \S\ref{sec:rational}, we study their rational structure via Adams operations, relying on previous work of Riou and Soul\'e.

\subsection{The motivic spectrum $\KGL$} \label{subsec:KGL}
Let $X$ be a qcqs scheme. We explain the characterisation (and construction) of the motivic spectrum $\KGL_X \in \SH(X)$, representing homotopy $K$-theory, via a simple universal property. Consider $\KH$ as a $\bb A^1$-invariant Nisnevich sheaf of $\bb E_{\infty}$-algebras on $\Sm_X$. Let $\scr K_X$ be the $\infty$-category of pairs $(E, \alpha)$, where $E \in \CAlg(\SH(X))$ and $\alpha: \omega^{\infty}E \simeq \KH$ is an $\bb E_{\infty}$-algebra equivalence, which is required to be \emph{Bott periodic}: that is, the element $1 - [\scr O(-1)] \in \KH_0(\bb P^1_X)$ gives rise, via the equivalence $\alpha$, to a map $\beta:\bb T_X \rightarrow E$ and we demand that the composition $E\otimes\bb T_X\xrightarrow{\id \otimes \beta}  E \otimes E\xto{\sub{mult}} E$ be an equivalence.\footnote{More formally, $\scr K_X$ is a subcategory of the pullback of the cospan of $\infty$-categories $\CAlg(\SH(X)) \xrightarrow{\omega^{\infty}} \CAlg(\opp{Sh}_{\Nis, \bb A^1}(\Sm_X,\Spt)) \leftarrow \{ \KH\},$ spanned by Bott periodic objects.}

\begin{lemma}
$\scr K_X$ is contractible
\end{lemma}
\begin{proof}
First, note that for $E \in \CAlg(\SH(X))$ a map of $\bb E_{\infty}$-algebras $\KH\to \omega^\infty E$ is the same as a map $\sigma^\infty(\KH) \to E$.
We thus see that $\scr K_X$ is equivalent to a full subcategory of $\CAlg(\SH(X))_{\sigma^\infty(\KH)/}$, namely the full subcategory on objects which are (1) Bott-periodic and (2) such that $\KH \to \omega^\infty E$ is an equivalence.
By \cite[Proposition 3.1]{Hoyois2020}, similarly to the proof of Proposition \ref{prop:P1-spectrum-from-graded}, the full subcategory of objects of $\CAlg(\SH(X))_{\sigma^\infty(\KH)/}$ satisfying (1) is equivalent to the full subcategory of $\CAlg(\opp{Sh}_{\Nis, \bb A^1}(\Sm_X,\Spt))_{\KH/}$ satisfying the analogous condition.
Thus $\scr K_X$ is equivalent to a full subcategory of $\CAlg(\opp{Sh}_{\Nis, \bb A^1}(\Sm_X,\Spt))_{\KH/}$, whose objects in particular satisfy the analog of (2); but that just means that the unit map $\KH \to E$ is an equivalence. Since the unit object is initial, this shows that $\scr K_X$ is indeed contractible.
\end{proof}

\begin{definition}
We denote by $(\KGL_X,\al) \in \cal K_X$ the unique (up to contractible choices) object. That is, $\KGL_X\in\CAlg(\SH(X))$ is the unique Bott periodic $\bb E_\infty$-algebra in $\SH(X)$ equipped with an equivalence of $\bb E_\infty$-algebras $\al:\omega^\infty\KGL_X\simeq\KH$. We will henceforth drop $\al$ from the notation and implicitly identify $\omega^\infty\KGL_X$ with $\KH$. The map $\beta:\bb T_X\to\KGL_X$, corresponding to $1-[\roi(-1)]\in\KH_0(\bb P_X^1)$ and witnessing Bott periodicity, will repeatedly appear.
\end{definition}

\begin{remark}[$\KGL$ as an absolute motivic spectrum]\label{rem_abs} 
Let $p_{\SH}:\int \SH \rightarrow \Sch^{\qcqs}$ be the cartesian fibration classified by $\SH: \Sch^{\qcqs,\op} \rightarrow \Cat_{\infty}$. Then there is an equivalence between the $\infty$-category of cartesian sections of $p_{\SH}$ and the value of $\SH$ at the terminal scheme $\Spec(\bb Z)$ via the correspondence established in \cite[Corol.~3.3.3.2]{Lurie2009}. We call such an object an \emph{absolute motivic spectrum}. Informally, the data of an absolute motivic spectrum consists of $E_X \in \SH(X)$ for all qcqs schemes $X$ and, for any morphism of qcqs schemes $f: X \rightarrow Y$, a comparison isomorphism $f^*E_Y\quis E_X$ subject to higher coherences.

Repeating the above construction of $\KGL_X$ fiberwise, we see that $\KGL$ defines a section of $\int \SH \rightarrow \Sch$ (we refer to \cite[\S 5]{Hoyois2020} for how to implement this formally). The fact that $\KGL$ is a cartesian section, i.e., an absolute motivic spectrum, is a result of Cisinski \cite{Cisinski2013}; see also \cite{Hoyois2020}.
\end{remark}

\subsubsection{Orientation of $\KGL$}\label{sss:orientation}
The following map will induce orientations of $\KGL$ and of motivic cohomology, and yield a first Chern class map in the latter:

\begin{definition}\label{definition:orientation_of_K}
We denote by \[1-(\cdot)^\vee:\Omega^\infty\Pic\To \Omega^\infty \K^\sub{cn}\] the unique homotopy class of maps of presheaves of pointed spaces on qcqs schemes such that, for any $m\ge1$, the composition \[(\bb P_\bb Z^m,\infty)\xto{[\roi(1)]}\Omega^\infty\Pic\xto{1-(\cdot)^\vee} \Omega^\infty \K^\sub{cn}\] classifies $1-[\roi(-1)]\in K_0(\bb P_\bb Z^m)$. By a sight abuse of notation we will sometimes to refer to the map $1-(\cdot)^\vee$ as the orientation of $K$-theory.
\end{definition}
\begin{proof}[Proof of existence and uniqueness]
The previous definition does require justification. To show that a unique such map exists up to homotopy we may restrict to smooth $\bb Z$-schemes, since $\Omega^\infty\Pic$ is Zariski-locally left Kan extended from smooth $\bb Z$-schemes\NB{ref?} and both $\Omega^\infty\Pic$ and $\Omega^\infty \K^\sub{cn}$ are Zariski sheaves of pointed spaces. Now set $(\bb P_\bb Z^{\infty},\infty) := \colim_m (\bb P_\bb Z^m,\infty)$ in $\opp{PSh}(\Sm_\bb Z,\opp{Spc}_*)$, where the filtered colimit is induced by the inclusion of the last $m+1$-coordinates $\bb P_\bb Z^m \rightarrow \bb P_\bb Z^{m+1}$. A calculation in $\bb A^1$-homotopy theory \cite[\S 4, Proposition 3.7]{MorelVoevodsky1999} states that the map $(\bb P_\bb Z^{\infty},\infty)\to \Omega^\infty\Pic$ in $\opp{PSh}(\Sm_\bb Z,\opp{Spc}_*)$,  classifying $\roi(1)$ on each $\bb P_\bb Z^m$, becomes an equivalence after applying the endofunctor $L_{\sub{Nis},\bb A^1}$. Therefore the canonical map \[\Map_{\opp{PSh}(\Sm_\bb Z,\opp{Spc}_*)}(\Omega^\infty\Pic,\Omega^\infty \K^\sub{cn})\To \opp{lim}_n\Map_{\opp{PSh}(\Sm_\bb Z,\opp{Spc}_*)}((\bb P_\bb Z^m,\infty),\Omega^\infty \K^\sub{cn})\] is an equivalence, and taking $\pi_0$ we obtain a short exact sequence of abelian groups \[0\To \opp{lim}_m^1\tilde{\K}_1(\bb P_\bb Z^m)\To \pi_0\Map_{\opp{PSh}(\Sm_\bb Z,\opp{Spc}_*)}(\Omega^\infty\Pic,\Omega^\infty \K^\sub{cn})\To \opp{lim}_n\tilde{\K}_0(\bb P_\bb Z^m)\To 0.\] Here $\tilde K_1$ and $\tilde K_0$ refer to reduced $K$-groups. The projective bundle formula shows that the transition maps in the left system are surjective, and so the $\opp{lim}_m^1$ vanishes. There is therefore a unique map $\Omega^\infty\Pic\to\Omega^\infty \K^\sub{cn}$ corresponding to the compatible system $1-[\roi(-1)]\in\tilde \K_0(\bb P^m_\bb Z)$, for $m\ge1$, as required.
\end{proof}

\NB{Maybe reinsert remark?}
%

By adjunction, the map $1-(\cdot)^\vee$ corresponds to a map $1-(\cdot)^\vee:\Sigma^{\infty}\Omega^{\infty}\Pic \rightarrow \K^\sub{cn}$ of presheaves of spectra on qcqs schemes. Furthermore, for any qcqs $X$, we may restrict to smooth $X$-schemes and compose with $\K^\sub{cn}\to\KH$; by the $\sigma^\infty$-$\omega^\infty$ adjunction this then corresponds to a map
\[
1 - (\cdot)^{\vee}: \tilde M(\Omega^\infty\Pic) \To \KGL_X
\]
in $\SH(X)$, which we continue to denote by the same notation.

\begin{lemma}
For any qcqs scheme $X$, the composition \[c_{\KGL,X}:\tilde M(\Omega^\infty\Pic) \xto{1 - (\cdot)^{\vee}} \KGL_X\stackrel{\beta^{-1}}{\quis}\KGL_X\otimes\bb T_X\] in $\SH(X)$ is an orientation of $\KGL_X$.
\end{lemma}
\begin{proof}
We must show that the diagram
\begin{equation*}
\begin{CD}
\tilde M(\Omega^\infty\Pic) @>{1-(\cdot)^\vee}>> \KGL_X \\
@A{[\scr O(1)]}AA              @A{\wequi}A{\beta}A                    \\
\tilde{M}(\bb P_X^1)=\bb T_X @>{\eta_{\KGL_X} \otimes \opp{id}_\bb T}>> \KGL_X \otimes \bb T 
\end{CD}
\end{equation*}
commutes up to homotopy. The composition from the bottom left to the top right, via the bottom right, classifies $1-[\roi(-1)]\in\KH_0(\bb P_X^1)$ by definition of $\beta$. But this is also true of the composition via the top left, by applying Definition \ref{definition:orientation_of_K} in the case $m=1$.
\end{proof}

\begin{remark}
Let $f: Y \to X \in \Sch$.
Under the equivalence $f^* \KGL_X \wequi \KGL_Y$, the orientation of $\KGL_Y$ induced by $f^* c_{\KGL,X}$ is $c_{\KGL,Y}$.
\end{remark}

\subsubsection{The slice filtration for $\KGL$} \label{subsub:slices-kgl}
The effective cover of $\KGL_X$ is denoted by \[ \kgl_X := \Fil^0_\slice \KGL_X \in \CAlg(\SH(X)).\] As recorded for example in \cite[Lemma~2.1]{RoendigsOestvaer2016}, for any $E\in\SH(X)$ there are functorial equivalences
\begin{equation}\label{eq:tate-slice}
\Fil^j_\sub{slice}(\bb T_X \otimes E) \simeq \bb T_X \otimes \Fil^{j-1}_\sub{slice}(E) \qquad s^j(\bb T_X \otimes E) \simeq \bb T_X \otimes s^{j-1}(E).
\end{equation}
Combining this with the Bott periodicity equivalence $\bb T \otimes \KGL_X \xrightarrow{\simeq} \KGL_X$ we obtain equivalences $\bb T \otimes \Fil^{j-1}_\sub{slice}\KGL_X \xrightarrow{\simeq} \Fil_\sub{slice}^j\KGL_X$ and similarly for the slices. This yields an alternative description of the slice filtration of $\KGL_X$ which we now explain. We have a map $\beta: \bb T_X \otimes \kgl_X \rightarrow \kgl_X$ which is induced via the universal property\footnote{One can also check that this map is given via multiplicative properties of $\kgl_X$ as follows. We note that $\kgl_X$ admits a $\bb E_{\infty}$-ring structure such that the map $\kgl_X \rightarrow \KGL_X$ is an $\bb E_{\infty}$-ring map and that the Bott element $\beta: \bb T_X \rightarrow \KGL_X$ lifts to $\beta:\bb T_X \rightarrow \kgl_X$. The present map is then homotopic to $\bb T_X \otimes \kgl_X \xrightarrow{\beta \otimes \mathrm{id}} \kgl_X \otimes \kgl_X \xrightarrow{\mathrm{mult}} \kgl_X$.} of $\Fil^0_\sub{slice}\KGL_X$ from the map $\bb T_X \otimes \kgl_X \xrightarrow{\beta \otimes \mathrm{can}} \KGL_X \otimes \KGL_X \xrightarrow{\mathrm{mult}} \KGL_X$. By construction, there is thus a commutative diagram
\[
\begin{tikzcd}
\bb T_X \otimes \kgl_X \ar{rr}{\simeq} \ar[swap]{dr}{\beta}  & & \Fil_\sub{slice}^1\KGL_X \ar{dl}{\mathrm{can}} \\
 & \kgl_X & .
\end{tikzcd}
\]
More generally there are, for all $j\ge0$, compatible such commutative diagram in which $\bb T_X$ is replaced by $\bb T_X^{\otimes j}$ and the Bott map $\beta$ is replaced by the $j$-fold iterate $\beta^j: \bb T_X^{\otimes j} \otimes \kgl_X \rightarrow \kgl_X$. We thus obtain the desired equivalence between the slice filtration on $\KGL_X$
\[
\Fil^{\star}_\sub{slice}\KGL_X=\qquad \cdots\leftarrow \Fil^{j-1}_\sub{slice}\KGL_X\leftarrow \Fil^{j}_\sub{slice}\KGL_X\leftarrow \Fil^{j+1}_\sub{slice}\KGL_X\leftarrow\cdots
\]
and the \emph{Bott filtration}
\[
\bb T_X^{\star} \otimes \kgl_X :=\qquad  \cdots \leftarrow\bb T_X ^{\otimes j-1} \otimes \kgl_X \xleftarrow{\id^{\otimes j-1} \otimes \beta} \bb T_X ^{\otimes j} \otimes \kgl_X \xleftarrow{\id^{\otimes j} \otimes \beta} \bb T_X ^{\otimes j+1}  \kgl_X  \leftarrow\cdots
\]
In particular, taking associated gradeds, we also obtain a description of the slices of $\KGL_X$ as Tate twists of the zeroth slice: $s^{j}\KGL_X \simeq \bb T^{\otimes j}\otimes s^0\KGL_X$ for $j\ge0$.

From this we can deduce convergence of the slice filtration in some cases:
\begin{lemma} \label{lemm:convergence-bootstrap}
Let $X$ be a regular Noetherian scheme of dimension $\le 1$.
Then $\kgl_X \in \SH(X)^\veff$ and $\omega^\infty \Fil^j_\slice \KGL$ is $(j-\dim(X))$-connective as a Nisnevich sheaf of spectra.
\end{lemma}
\begin{proof}
To check that $\kgl_X \in \SH(X)^\veff$, by Proposition \ref{prop:detecting-veff}(2), we need only verify that $\omega^{\infty}\kgl_X \simeq \KH|_{\Sm_X}$ is Nisnevich-locally connective. This holds since $\K(Y) = \KH(Y)$ and negative $K$-groups of $Y$ vanishes whenever $Y$ is a regular Noetherian scheme. The claim about its slices follow by the Bott filtration description of the slices and Proposition \ref{prop:detecting-veff}(1).
\end{proof}

\begin{remark}
Given $f: Y \to X \in \Sch$ there is a natural map $f^* \kgl_X \to \kgl_Y$, as discussed in Remark \ref{rem:naturality}. We shall eventually prove in Corollary \ref{cor:eff-base-change}(1) that this maps is an equivalence, but for most of the article we do not know (and so cannot use) this. It will follow in particular that $\kgl_X \in \SH(X)^\veff$ for every $X$ but, again, at the present point in the article we do not know this.
\end{remark}

\subsubsection{The motivic spectrum $\cal V$} \label{subsub:V}
One way to control effective motivic spectra is via the theory of framed motivic spectra as explained by Hoyois, Khan, Sosnilo, Yakerson, and the second author \cite{ElmantoHoyoisKhanSosniloYakerson2021} \cite{Hoyois2021}. One can consider this theory as a ``spectral counterpart'' to Voevodsky's derived $\infty$-category of motives, built out of the theory of finite correspondences. The topologically-minded reader might want to view it as a theory of motivic infinite loop spaces. 

Given a presheaf of spectra $F$ on smooth $X$-schemes with the additional structure of \emph{framed transfers}, the theory of framed motivic spectra produces an effective motivic spectrum $\Sigma^{\infty}_\sub{fr}F\in\SH(X)$. This construction assembles into a functor $\Sigma^{\infty}_\sub{fr}$ from presheaves with framed transfers to motivic spectra; as $X$ varies, this construction commutes with pullbacks by \cite[Lemma~16]{Hoyois2021}. To make use of this theory in practice, one often starts with a ``simple-looking'' presheaf with an obvious framed transfers structure, defined over $\Spec(\bb Z)$; applying $\Sigma^{\infty}_\sub{fr}$ often produces an interesting absolute motivic spectrum. Furthermore, the functor $\Sigma^{\infty}_\sub{fr}$ is symmetric monoidal and thus converts $\bb E_{\infty}$-algebras in presheaves with framed transfers to $\bb E_{\infty}$-algebras in motivic spectra.

For example \cite{HoyoisJelisiejewNardinTotaroYakerson2021}, the functor $\mathrm{Vect}_X$, which assigns to any smooth $X$-scheme $T$ the groupoid of finite-rank vector bundles on $T$, defines a  presheaf of $\bb E_{\infty}$-monoids (under direct sum) on $\Sm_X$. One then obtains a presheaf of spectra by group-completing, which is moreover an $\bb E_{\infty}$-algebra in framed presheaves via tensor product and pushforward of vector bundles. The associated motivic spectrum $\cal V_X=\Sigma^\infty_\sub{fr}\mathrm{Vect}_X := \Sigma^\infty_\sub{fr}\Vect^{\mathrm{gp}}_{X} \in \SH(X)$ is an effective $\bb E_{\infty}$-ring spectrum equipped with a multiplicative map $\Sigma^\infty_\sub{fr}\mathrm{Vect}_X \rightarrow \KGL_X$; since the domain is effective it factors through a multiplicative map
\begin{equation}\label{eq:fr-to-kgl}
\cal V_X  \To \kgl_X.
\end{equation}
We also have a map 
\begin{equation*}
\bb T_X\To\cal V_X,
\end{equation*} such that the composite to $\KGL_X$ classifies the Bott element; we refer to \cite[\S 5]{HoyoisJelisiejewNardinTotaroYakerson2021} for details. The assignment $X \mapsto \cal V_X$ defines an absolute motivic spectrum, the key point being that the formation of $\mathrm{Vect}_X$ is stable under pullbacks in $X$ (see e.g. \cite[Example A.0.6(1)]{ElmantoHoyoisKhanSosniloYakerson2020}).

We are interested in $\cal V_X$ in that it serves as a bridge between $\kgl_X$ and the motivic sphere. Its relation to $\kgl_X$ is through the map \eqref{eq:fr-to-kgl} (which will eventually be shown in Corollary~\ref{cor:eff-base-change} to be an equivalence). Its relation to the motivic sphere is through the next result, which reduces to the known case of fields:

\begin{proposition}\label{prop_1_vs_V}
For any qcqs scheme $X$, the unit map $1_{X} \rightarrow \scr V_{X}$ induces an equivalence $s^0(1_{X}) \xrightarrow{\simeq} s^0(\scr V_{X})$ in $\SH(X)$.
\end{proposition}
\begin{proof}
It suffices to show that the fibre of $1_{X} \rightarrow \scr V_{X}$ is $1$-effective, which may be checked on perfect fields $k$ (use that $\scr V_X$ is stable under base change and apply Proposition \ref{prop:detect-effective} to $\scr V_\Z$). But $1$ and $\scr V$ are stable under base change, so the problem reduces to showing that the fibre of $1_k \rightarrow \scr V_k$ is $1$-effective. This follows from the facts that the map of $\bb E_\infty$-algebras $\scr V_k \to \kgl_k$ is an equivalence by \cite[Corollary 5.2]{HoyoisJelisiejewNardinTotaroYakerson2021}, and the map $s^0(1_k)\to s^0(\kgl_k)$ induced by the unit is an equivalence by Levine \cite[Theorems 6.4.2 and 10.5.1]{Levine2008}.
\end{proof}

\begin{remark}[Motivic cohomology of smooth varieties]\label{rem_Levine}
We have just used the result of Levine and Voevodsky that, for any perfect field $k$, the map $s^0(1_k)\to s^0(\kgl_k)$ induced by the unit of $\kgl_k$ is an equivalence in $\SH(k)$. We refer the reader to the proof of Theorem \ref{thm:classical_BL_equiv} for some comments on the proof, where Bloch's cycle complex appears. We note that, by passing to a filtered colimit using \cite[Proposition~C.12]{Hoyois2014}, the hypothesis that the field be perfect can be removed.

We will use the equivalence $s^0(1_k)\quis s^0(\kgl_k)$, which is a deep result from classical motivic homotopy theory, multiple times and eventually show in Corollary~\ref{cor:conjectures} that in fact it holds for arbitrary qcqs schemes in place of $\Spec(k)$.
\end{remark}

\subsection{The cohomology theories $\bb Z(\star)^{\bb A}$ and $\bb Z(\star)^{\bb A, \sub{cdh}}$} 
In this subsection we explain how to use the slice filtration on $\KGL$ to construct ``two'' natural candidates for the $\bb A^1$-invariant motivic cohomology of qcqs schemes: \[\bb Z(j)^{\bb A,\sub{cdh}}\text{  and  } \bb Z(j)^\bb A : \Sch^\sub{qcqs}\To\D(\bb Z)\qquad j\in\bb Z.\] These cohomology theories are related by maps $\bb Z(j)^{\bb A,\sub{cdh}}\to\bb Z(j)^\bb A$, arising from the graded pieces of two compatible filtrations on $\KH$. The notation is chosen the reflect the fact that general results about motivic stable homotopy theory will imply that $\bb Z(j)^{\bb A,\sub{cdh}}$ and $\bb Z(j)^\bb A$ are both $\bb A^1$-invariant,\footnote{It is only for typographical reasons that the superscript is $\bb A$ rather than $\bb A^1$.} and that $\bb Z(j)^{\bb A,\sub{cdh}}$ is moreover a cdh sheaf.

The apparent difference between $\bb Z(j)^{\bb A,\sub{cdh}}$ and $\bb Z(j)^\bb A$ is illusory: once we prove in Corollary~\ref{cor:conjectures} that the slice filtration on $\KGL$ is stable under pullback, it will follow that the map $\bb Z(j)^{\bb A,\sub{cdh}}(X)\to\bb Z(j)^\bb A(X)$ is actually an equivalence for any qcqs scheme $X$ and $j\in\bb Z$. However, until that theorem has been proved, our arguments will require both candidates. 

\begin{remark}\label{rem_roles} Although $\bb Z(j)^{\bb A,\sub{cdh}}$ and $\bb Z(j)^\bb A$ will ultimately be shown to coincide, they have different flavours and play different roles in this paper.

One of our goals is to realize Voevodsky's program in \cite{Voevodsky2002, Voevodsky2002a} for constructing an $\bb A^1$-invariant theory of motivic cohomology of arbitrary schemes, arising as the graded pieces of a motivic filtration on $\KH$. The definition proposed by Voevodsky is $\bb Z(j)^{\bb A}$, namely for each qcqs scheme $X$ we use the graded pieces of the slice filtration of $\KGL_X$ to induce a filtration on $\KH(X)$ (see Definition \ref{definition_A^1_mot-_coh}). Owing to the elegance of this definition and its universality, we regard $\bb Z(j)^{\bb A}$ as the principal motivic cohomology theory of this paper. It automatically enjoys the property of vanishing in negative weights, i.e., $\bb Z(j)^{\bb A}$ for $j<0$, which is expected of motivic cohomology; see Remark~\ref{rem:eff-neg} for more details on this point. In Section~\ref{section_motivic_DD} we will also see, in the case of smooth schemes over fields and over mixed characteristic Dedekind domains, that $\bb A^1$-motivic cohomology satisfies a Beilinson--Lichtenbaum style description in terms of \'etale cohomology (see Theorem~\ref{theorem_BL1}).

On the other hand, it is a priori not clear that the theory $\bb Z(j)^{\bb A}$ for $j\ge0$ arises from an absolute motivic spectrum, as one would expect since this is true for $\KGL$. So we introduce the variant theory $\bb Z(j)^{\bb A, \sub{cdh}}$, which does have this property and consequently satisfies cdh descent (as any absolute motivic spectrum does); this refines cdh descent for $\KH$, first established by Cisinski in \cite{Cisinski2013}. Many techniques in this paper, which exploit cdh descent and calculations involving valuation rings, apply more directly to $\bb Z(j)^{\bb A, \sub{cdh}}$ than to $\bb Z(j)^{\bb A}$.
\end{remark}

\begin{remark}[Comparison to Spitzweck's motivic spectrum and the Bloch--Levine complex]\label{rem:Spitzweck}
There is another notion of motivic cohomology of general schemes, due to Spitzweck \cite{Spitzweck2018}; outside of this remark, which is not used in the rest of the article, we do not need Spitzweck's cohomology.

It is proved in \cite{Bachmann2022} that, over $\Z$, Spitzweck's spectrum $H\bb Z^\sub{Spi}\in\SH(\bb Z)$ coincides with $s^0\KGL_\Z$. Since Spitzweck's spectrum for a general scheme is by definition pulled back from $\bb Z$, it follows from Definition \ref{definition_A^1_cdh_mot_coh} that our $\bb A^1$-cdh-motivic cohomology is the same as Spitzweck's motivic cohomology. In particular, combining this identification with \cite[Theorem 7.18]{Spitzweck2018}, we obtain the following: for any smooth scheme $X$ over a mixed characteristic Dedekind domain, and $j\ge0$, there exists an equivalence $\bb Z(j)^{\bb A,\sub{cdh}}(X)\simeq z^j(X,\bullet)[-2j]$ where $z^j(X,\bullet)$ is the Bloch--Levine cycle complex \cite{Bloch1986b, Levine2001}. 
\end{remark}

We now turn to definitions, beginning with $\bb Z(j)^\bb A(X)$, which is defined in terms of the slice filtration on $\KGL_X\in\SH(X)$:

\begin{definition}\label{definition_A^1_mot-_coh}
For a qcqs scheme $X$, its {\em $\bb A^1$-motivic cohomology} is defined by \begin{equation}\bb Z(j)^{\bb A}(X):=\map_{\SH(X)}(1_X,s^j\KGL_X)[-2j]\in\opp{Sp}\label{eqn_ZjA}\end{equation} for $j\in\bb Z$. Write $H^i_{\bb A}(X,\Z(j)):=H^i(\Z(j)^\bb A(X))$, for $i,j\in\bb Z$, for the corresponding {\em $\bb A^1$-motivic cohomology groups}. 
\end{definition}

Alternatively, we may form the slices of $\KGL_\bb Z\in\SH(\bb Z)$, pull back to an arbitrary qcqs scheme, and then evaluate on the scheme by taking the mapping spectrum; this is how we define $\bb A^1$-cdh-motivic cohomology:

\begin{definition}\label{definition_A^1_cdh_mot_coh}
For a qcqs scheme $X$, with structure map to $\Spec(\bb Z)$ denoted by $f:X\to\Spec(\bb Z)$, its {\em $\bb A^1$-cdh-motivic cohomology} is defined by \[\bb Z(j)^{\bb A,\sub{cdh}}(X):=\map_{\SH(X)}(1_X,f^*s^j\KGL_\bb Z)[-2j]\in\opp{Sp}\] for $j\in\bb Z$. Write $H^i_{\bb A,\sub{cdh}}(X,\Z(j)):=H^i(\Z(j)^\bb A(X))$, for $i,j\in\bb Z$, for the corresponding {\em $\bb A^1$-cdh-motivic cohomology groups}. 
\end{definition}

\begin{remark}[Functoriality and comparison maps]\label{rem:functoriality}
It is clear that $\bb Z(j)^{\bb A,\sub{cdh}}$ defines a Sp-valued presheaf on qcqs schemes. The analogous statement about $\bb Z(j)^{\bb A}$ uses the transformations \eqref{eq:natural_graded}.
Furthermore, we have comparison maps 
\begin{equation}\label{eq:cdh-to-aone}
\bb Z(j)^{\bb A, \sub{cdh}} \To \bb Z(j)^{\bb A},
\end{equation}
induced by the transformation~\eqref{eq:natural_slice}.
\end{remark}

\begin{remark}[Negative weights]\label{rem:eff-neg}
If $j<0$ then $\bb Z(j)^\bb A=0$; indeed, this follows immediately from Definition \ref{definition_A^1_mot-_coh} since then $s^j\KGL_X$ is right orthogonal to $\SH^\sub{eff}(X)(j+1)\ni 1_{X}$.

However, at this point in the article we cannot yet prove that the $\bb A^1$-cdh-motivic cohomology also vanishes when $j<0$. In the notation of Definition \ref{definition_A^1_cdh_mot_coh}, this requires knowing that $f^*s^0\KGL\bb Z\in\SH(S)$ is {\em co-effective}, i.e., its $\Fil^1_\sub{slice}$ vanishes. But establishing this co-effectivity will be a key point in proving that $f^*s^0\kgl_\bb Z\quis s^0\kgl_S$, which we will not do until Corollary~\ref{prop:Fil1-vanishing}.
\end{remark}

\begin{remark}[Restricting to smooth schemes over a base]\label{rem_coh_over_base}
We have deliberately defined the $\bb A^1$- and $\bb A^1$-cdh-motivic cohomologies absolutely: they depend only on the input scheme, not any any chosen base. Nevertheless, note that if $S$ is a fixed qcqs base scheme then on the category of smooth $S$-schemes the motivic cohomologies, for $j\in\bb Z$, are represented in $\SH(S)$ via the formulae
\[\bb Z(j)^{\bb A}|_{\Sm_S}=\omega^\infty s^j\KGL_S[-2j]\qquad\mathrm{and}\qquad\bb Z(j)^{\bb A,\sub{cdh}}|_{\Sm_S}=\omega^\infty f^*s^j\KGL_\bb Z[-2j]\] where $f:S\to\Spec(\bb Z)$ denotes the structure map of $S$. These identifications follow from the fact that pulling back along a smooth morphism commutes with taking slices, as explained in Remark \ref{rem:naturality}.
\end{remark}

\begin{remark}[Multiplicative structure] \label{rmk:HZA-mult}
We have seen in Remark \ref{rem:mult} that, for any qcqs scheme $f:X \to \Spec(\bb Z)$, the slices of $\KGL_X$ naturally assemble to $s^\star \KGL_X \in \CAlg(\Gr\SH(X))$. By lax monoidality of $\omega^\infty$ and the pullback functor on $\SH$, we deduce that the presheaves \[ \bb Z^{\A} (\star)[2\star] \quad\text{and}\quad \bb Z^{\A,\cdh} (\star)[2\star]: \Sch^\op \To \Spt \] upgrade to $\bb E_\infty$-algebras in graded presheaves of spectra. Furthermore, we note that $s^\star \KGL_X$ (resp. $f^*s^\star \KGL_{\bb Z}$) is a periodization of $s^0\KGL_X$ (resp. $f^*s^0\KGL_{\bb Z}$). We will elaborate more on this structure in Remark~\ref{rem:motivic_is_traditional}.
\end{remark}

\begin{remark}[Orientation]\label{remark:orientation_motivic}
The standard orientation of $\KGL$, constructed in \S\ref{sss:orientation}, induces via Remark~\ref{rmk:orientation-morphisms} compatible orientations of $s^0(\KGL_X)$ and $f^*s^0(\KGL_\bb Z)$ for any qcqs scheme $f:X\to\Spec(\bb Z)$. By the equivalence of the slice filtration and Bott filtration as in \S\ref{subsub:slices-kgl}, these orientations may be rewritten as a natural map \[\tilde{M}_X(\Omega^{\infty}\Pic)\xto{1-(\cdot)^\vee} f^*s^1(\KGL_\bb Z)\To  s^1(\KGL_X).\] Via the $\sigma^\infty$-$\omega^\infty$ adjunction, this corresponds to a map
\begin{equation} \Sigma^\infty\Omega^\infty\Pic\To \Z^{\A,\cdh}(1)[2] \To \Z^{\A}(1)[2]\label{eqn:orientation_motivic}\end{equation}
of presheaves of spectra on qcqs schemes; for any line bundle $L$ on any qcqs scheme $X$, precomposing along $[L]:\Sigma^\infty_+X\to\Sigma^\infty\Omega^\infty\Pic$ classifies the {\em first Chern classes} $c_1^{\bb A,\sub{cdh}}(L)\in H^2_{\bb A,\sub{cdh}}(X,\bb Z(1))$ and $c_1^{\bb A}(L)\in H^2_{\bb A}(X,\bb Z(1))$. We will see in Lemma~\ref{lemma:HZA-orient-KGL-compat} below that the first map in \eqref{eqn:orientation_motivic} factors through the canonical map $\Sigma^\infty\Omega^\infty\Pic\to \Pic$.
\end{remark}

We immediately get the projective bundle formula from the theory of orientations:

\begin{theorem}[Projective bundle formula]\label{thm_pbf_Acdh}
For any qcqs scheme $X$, rank $d$ projective bundle $\pi:P\to X$, and $j\ge0$, the map \begin{equation}
\sum_{i=0}^d c_1^{\bb A,\sub{cdh}}(\roi_P(1))^{i}\pi^*:\bigoplus_{i=0}^d\bb Z(j-i)^{\bb A, \cdh}(X)[-2i]\To \bb Z(j)^{\bb A, \cdh}(P)\end{equation} induced by powers of the first motivic Chern class $c_1^{\bb A,\sub{cdh}}(\roi_P(1))\in H^2_{\bb A,\sub{cdh}}(P,\bb Z(1))$ of the canonical line bundle $\roi_{P}(1)$ is an equivalence. Similarly, the map
\begin{equation}
\sum_{i=0}^d c_1^{\bb A}(\roi_P(1))^{i}\pi^*:\bigoplus_{i=0}^d\bb Z(j-i)^{\bb A}(X)[-2i]\To \bb Z(j)^{\bb A}(P)
\end{equation} induced by powers of $c_1^{\bb A}(\roi_{P}(1))$ is an equivalence. 
\end{theorem}
\begin{proof}
This follows by applying Theorem \ref{thm_pbf_for_general_E} to $s^0(\KGL_X)$ and $f^*s^0(\KGL_\bb Z)$.
\end{proof}

\begin{remark}[Motivic cohomology as a traditional cohomology theory]\label{rem:motivic_is_traditional}
For any qcqs scheme, motivic cohomology on smooth $X$-schemes naturally assembles to form an object \[\bb Z(\star)^{\bb A}[2\star]\in\CAlg(\Gr\opp{Sh}_{\Nis, \bb A^1}(\Sm_X,\Spt)_\sub{c,pbf}).\] Indeed, we have already noted in Remark \ref{rmk:HZA-mult} that $\bb Z(\star)^{\bb A}[2\star]$ is an $\bb E_\infty$-algebra in graded presheaves (even of $\bb A^1$-invariant Nisnevich sheaves; see Theorem \ref{theorem:SH_mot_coh}), and the required map $c:\Sigma^\infty\bb P^1_X\to \bb Z(1)^{\bb A}[2]$ is given by precomposing \eqref{eqn:orientation_motivic} along the map $\Sigma^\infty\bb P^1_X\to \Sigma^\infty\Pic$ classifying $\roi(1)\in\opp{Pic}(\bb P^1_X)$. We record the following key identification about the interaction of $\bb A^1$-invariant motivic cohomology with the motivic Eilenberg--MacLane functor:
\begin{equation}H(\bb Z(\star)^{\bb A}[2\star])=s^0(\KGL_X).\label{eqn_mot_to_EM}\end{equation}
Indeed, we know from \S\ref{subsub:slices-kgl} that $s^0(\KGL_X)$ admits a periodization given by $s^\star(\KGL_X)$, so this follows from the discussions in Remark~\ref{rem:qx} and \ref{rem:recover}.
%

In fact, $\bb Z(\star)^{\bb A}[2\star]$ even forms a ``traditional cohomology theory'' in the sense of Remark \ref{remark:traditional}: that is, $\bb Z(\star)^{\bb A}$ is an $\bb E_\infty$-algebra in $\CAlg(\Gr\opp{Sh}_{\Nis, \bb A^1}(\Sm_X,\D(\bb Z))$ and the first Chern class of $\roi(1)$ arises from a first Chern class map $c_1^\bb A:R\Gamma_\sub{Nis}(-,\bb G_m)[1]\to \bb Z(1)^{\bb A}[2]$. See Theorem \ref{theorem:SH_mot_coh}(3) and Remark \ref{rem_A_1st_Chern_class}.

Similar comments apply to $\bb Z(\star)^{\bb A,\sub{cdh}}$.
\end{remark}

\subsubsection{Low weights and the first Chern class map}\label{subsec:low-weights}
To begin analysing motivic cohomology in weights $0$ and $1$, with goal being Construction~\ref{cons:low-weights}, we use the unstable analogue of the slice filtration. For any qcqs scheme $X$ and any $j \geq 0$, let
\[
\opp{Sh}_{\Nis, \bb A^1}(\Sm_X,\Spt)(j) \subset \opp{Sh}_{\Nis, \bb A^1}(\Sm_X,\Spt)
\]
be the subcategory generated under colimits by $(\Sigma^\infty \P^1)^{\otimes j} \otimes E$, for $E \in \opp{Sh}_{\Nis, \bb A^1}(\Sm_X,\Spt)$; the presheaves belonging to this subcategory will be called {\em unstably $j$-effective}. By the same reasoning as in the case of $\SH$ these inclusions admit right adjoints, and we display the adjunction using a similar notation:
\[
\iota^j: \opp{Sh}_{\Nis, \bb A^1}(\Sm_X,\Spt)(j) \rightleftarrows \opp{Sh}_{\Nis, \bb A^1}(\Sm_X,\Spt): r^j.
\]
In the same way as for the slice filtration, this induces a functorial, exhaustive multiplicative filtration for $E \in  \opp{Sh}_{\Nis, \bb A^1}(\Sm_X,\Spt)$ which is now $\bb N$-indexed. We denote this tower by:
\[
E \simeq \Fil^{0}_\sub{u-slice}E \leftarrow \cdots\leftarrow \Fil^{j-1}_\sub{u-slice}E\leftarrow \Fil^j_\sub{u-slice}E\leftarrow\Fil^{j+1}_\sub{u-slice}E\leftarrow \cdots 
\]
and write $s_u^jE$ for the graded pieces, which we call the {\em unstable slices} of $E$.

We now compare the unstable slice filtration to the slice filtration on $\SH(X)$. Observing that 
\[
\sigma^\infty(\opp{Sh}_{\Nis, \bb A^1}(\Sm_X,\Spt)(j)) \subset \SH(X)^{\eff}(j),
\] we obtain Beck--Chevalley style transformations $\iota^j \omega^{\infty} \rightarrow \omega^{\infty} \iota^j$.
They induce, for each $E \in \SH(X)$ and $j \geq 0$, natural maps
\[
\Fil^j_\sub{u-slice}\omega^{\infty}E \rightarrow \omega^{\infty}\Fil^j_\sub{slice}E \qquad\textrm{ and }\qquad s_u^j\omega^{\infty}E \rightarrow \omega^{\infty}s^jE.
\] 
Since all these functors and transformations are compatibly lax monoidal the transformation above upgrades to a lax monoidal transformation of $\bb N$-indexed filtered and graded objects respectively:
\begin{equation}\label{eq:slice-slice}
\Fil^{\star}_\sub{u-slice}\omega^{\infty} \rightarrow \omega^{\infty}\Fil^{\star}_\sub{slice} \qquad\textrm{ and }\qquad s_u^{\star}\omega^{\infty} \rightarrow \omega^{\infty}s^{\star}.
\end{equation}
In the case of $K$-theory of regular Noetherian schemes, we will compute the unstable slice filtration in low weights in Lemma \ref{lemma_unstable_slices_of_K}. First we need to record the following well-known descriptions of cohomology of $\bb Z$ and $\bb G_m$ on regular Noetherian schemes; we include a proof as we could not find a reference for part (2):

\begin{lemma}\label{lem:zarvsnis_on_regular}
Let $X$ be a regular Noetherian scheme.
\begin{enumerate}
\item There are natural isomorphisms
\[
H^n_{\Zar}(X, \bb Z) = \begin{cases}
\bb Z^{\pi_0(X)} & n= 0\\
0 & n > 0. 
\end{cases}
\]
 and
\[
H^n_{\Zar}(X, \bb G_m) = \begin{cases}
\roi(X)^{\times} & n = 0\\
\rm Pic(X) & n = 1\\
0 &  n>1,
\end{cases}
\]

\item The change-of-topology maps
\[R\Gamma_\sub{Zar}(X, \bb Z) \To R\Gamma_\sub{Nis}(X, \bb Z)\qquad\mathrm{and}\qquad R\Gamma_\sub{Zar}(X,\bb G_m)\To R\Gamma_\sub{Nis}(X,\bb G_m)\]
are equivalences.
\item The canonical maps $R\Gamma_\sub{Nis}(\bb A_X^1, \bb Z)\to R\Gamma_\sub{Nis}(X, \bb Z)$ and $R\Gamma_\sub{Nis}(\bb A_X^1, \bb G_m)\to R\Gamma_\sub{Nis}(X, \bb G_m)$ are equivalences.
\end{enumerate}
\end{lemma}
\begin{proof}
(1): For the cohomology of $\bb Z$, the key point is the vanishing of the higher cohomology of constant sheaves on irreducible schemes; see \cite[Tag 02UW]{Stacks}. Similarly, for $\bb G_m$, the non-trivial part is showing vanishing of $H^n_\sub{Zar}(X,\bb G_m)$ for $n>1$; but this holds because $X$ is normal and so $\bb G_m$ has its usual flasque resolution $0\to\bb G_m\to\oplus_{\eta\in X^0} i_{\eta*}k(\eta)^\times\to\oplus_{x\in X^1}i_{x*}\bb Z\to0$.

(2): It is enough to establish the equivalences whenever $X$ is the spectrum of a regular Noetherian local ring $R$, and by induction on dimension we may assume the equivalences are known for regular Noetherian schemes of dimension $<\dim R$. We will consider the ind-Nisnevich cover of $R$ given by its henselisation $R^\sub{h}$ and its punctured spectrum $\Spec(R)\setminus\{\frak m\}$. More precisely, the cartesian square
\[\xymatrix{
\Spec(R^\sub{h})\setminus\{\frak m R^\sub{h}\}\ar[d]\ar[r] &\Spec(R^h)\ar[d]\\
\Spec(R)\setminus\{\frak m\}\ar[r] &\Spec(R)
}\]
is a cofiltered limit of Nisnevich distinguished squares, and so is carried to a cartesian square by $R\Gamma_\sub{Nis}(-, F)$, where $F=\bb Z$ or $\bb G_m$. By the inductive hypothesis we know that the change-of-topology map $R\Gamma_\sub{Zar}(-, F)\to R\Gamma_\sub{Nis}(-, F)$ is an equivalence on the two punctured spectra; it is also an equivalence on $R^\sub{h}$. Therefore, to prove it is an equivalence on $R$ (as required to complete the proof of (2)), we must show that $R\Gamma_\sub{Zar}(-, F)$ carries the above square to a cartesian square in $\rm D(\bb Z)$. We prove this by treating various cases.

If $R$ is a field the schemes on the left of the diagram are empty and $R=R^\sub{h}$, so the claim is vacuous. If $R$ is a discrete valuation ring then the punctured spectra are the spectra of the fields of fractions of $R$ and $R^\sub{h}$ respectively, and applying $R\Gamma_\sub{Zar}(-, \bb Z)$ and $R\Gamma_\sub{Zar}(-, \bb G_m)$ yield squares
\[\xymatrix{
\bb Z\ar[r]\ar[d] & \bb Z\ar[d] \\
\bb Z\ar[r] & \bb Z
}
\qquad\mbox{and}\qquad
\xymatrix{
R^\times\ar[r]\ar[d] & \opp{Frac}(R)^\times\ar[d] \\
R^{\sub{h}\times}\ar[r] & \opp{Frac}(R^\sub{h})^\times
}
\]
respectively. Both squares are cartesian in $\D(\bb Z)$ (for the right square, the cokernels of the horizontal arrows are both $\bb Z$, given by the discrete valuation), as desired.

It remains to treat the case that $\dim R\ge 2$, where it is enough to check that the horizontal maps $R\Gamma_\sub{Zar}(R, F)\to R\Gamma_\sub{Zar}(\Spec(R)\setminus\{\frak m\}, F)$ and $R\Gamma_\sub{Zar}(R^h, F)\to R\Gamma_\sub{Zar}(\Spec(R^h)\setminus\{\frak mR^h\}, F)$ are equivalences, for both $F=\bb Z$ and $\bb G_m$. For the cohomology of $\bb Z$ this holds by the descriptions in (1) since the punctured spectra are still connected (as $\dim R=\dim R^h\ge 2$). For $\bb G_m$ we again use the descriptions in (1), and that $R^\times \isoto H^0_\sub{Zar}(\Spec(R)\setminus\{\frak m\},\bb G_m)$ (as $R$ is normal of dimension $\ge 2$) and $\opp{Pic}(\Spec(R)\setminus\{\frak m\})=0$ by Grothendieck \cite[Tag 0F2H]{Stacks}, and similarly for $R^h$ in place of $R$.

(3): By part (2) we may replace Nis by Zar, and now we use part (1) to compute the cohomologies. For the cohomology of $\bb Z$, note that the connected components of $X$ and $\bb A_X^1$ correspond. For $\bb G_m$, we use that $A^\times\isoto A[T]^\times$ whenever $A$ is a reduced ring, and $\opp{Pic}(A)\isoto \opp{Pic}(A[T])$ whenever $A$ is a seminormal ring \cite{Traverso1970}  \cite[Theorem 1]{Swan1980}.
\end{proof}

\begin{lemma}\label{lemma_unstable_slices_of_K}
Let $X$ be a regular Noetherian scheme. The following assertions hold in $\opp{Sh}_{\Nis, \bb A^1}(\Sm_X,\Spt)$:
\begin{enumerate}
\item The rank map $\mathrm{rk}:K\to R\Gamma_\sub{Nis}(-,\bb Z)$ induces an equivalence $s^0_uK\quis R\Gamma_\sub{Nis}(-,\bb Z)$, so that $\Fil^1_\sub{u-slice}K\simeq \K^{\sub{rk}=0}:=\fib(\mathrm{rk})$.
\item The map $\textrm{det}:\K^{\sub{rk}=0}\to R\Gamma_\sub{Nis}(-,\bb G_m)[1]$ (i.e., induced by $\det:K_1(-)\to H^0_\sub{Nis}(-,\bb G_m)$) induces an equivalence $s^1_uK\quis R\Gamma_\sub{Nis}(-,\bb G_m)[1]$.
\end{enumerate}
\end{lemma}
\begin{proof}
We first note, thanks to Lemma \ref{lem:zarvsnis_on_regular}(3), that both $R\Gamma_\Nis(\ph, \Z)$ and $R\Gamma_\Nis(\ph, \Gm)$ are $\A^1$-invariant on the category of smooth $X$-schemes.

Consequently there are fibre sequences \[ \K^{\sub{rk}=0} \to \K \to R\Gamma_\Nis(\ph, \Z) \quad\text{and}\quad F \to \K^{\sub{rk}=0} \to R\Gamma_\sub{Nis}(-,\bb G_m)[1]\] in $\opp{Sh}_{\Nis, \bb A^1}(\Sm_X,\Spt)$. We claim that $R\Gamma_\Nis(\ph, \Z)$ is right orthogonal to $\opp{Sh}_{\Nis, \bb A^1}(\Sm_X,\Spt)(1))$ and that $R\Gamma_\sub{Nis}(-,\bb G_m)$ is right orthogonal to $\opp{Sh}_{\Nis, \bb A^1}(\Sm_X,\Spt)(2)$.
Noting that $\opp{Sh}_{\Nis, \bb A^1}(\Sm_X,\Spt)(j)$ is generated under colimits and desuspensions by objects of the form $\Sigma^\infty_+\P^j_Y/\Sigma^\infty_+\P^{j-1}_Y \wequi (\Sigma^\infty \P^1)^{\otimes j} \otimes \Sigma^\infty_+ Y$, it is enough to show that $R\Gamma_\sub{Nis}(Y,\bb Z) \wequi R\Gamma_\sub{Nis}(\bb P_Y^1,\bb Z))$ and $R\Gamma_\sub{Nis}(\bb P^2_Y,\bb G_m)\quis R\Gamma_\sub{Nis}(\bb P_Y^1,\bb G_m)$ for all $Y\in\Sm_X$. This is straightforward, using Lemma \ref{lem:zarvsnis_on_regular} and \cite[Tag 0BXJ]{Stacks}.

It thus remains to show that \[  \K^{\sub{rk}=0} \in \opp{Sh}_{\Nis, \bb A^1}(\Sm_X,\Spt)(1) \quad\text{and}\quad F \in \opp{Sh}_{\Nis, \bb A^1}(\Sm_X,\Spt)(2). \]
We will use the following general observation\footnote{See also the proof of \cite[Lemma 5.2]{BachmannElmanto} for detecting \emph{stable} effectivity in motivic homotopy theory.} $(*)$: if $E \in \opp{Sh}_{\Nis, \bb A^1}(\Sm_X, \Spt)$ is Nisnevich locally connective and $L_{\A^1,\Nis}\Sigma^\infty\Omega^\infty E \in \opp{Sh}_{\Nis, \bb A^1}(\Sm_X,\Spt)(j)$, then $E \in \opp{Sh}_{\Nis,\bb A^1}(\Sm_X, \Spt)(j)$.
We prove this at the end, proceeding first with the rest of the proof.

Firstly, it is clear that $\K^{\sub{rk}=0}$ and $F$ are connective. We know $\Omega^{\infty} \K \wequi L_{\Nis,\A^1}(\Z \times \mathrm{BGL})$ \cite[Propositions 3.9 and 3.10]{MorelVoevodsky1999}, with the $\bb Z$ component corresponding to the rank map, and so $\Omega^{\infty} \K^{\sub{rk}=0} \wequi L_{\Nis,\A^1} \mathrm{BGL}$.
We now claim that $\Omega^{\infty} F \wequi L_{\Nis,\A^1} \mathrm{BSL}$.
To show this, it will suffice to prove that the fiber sequence of Nisnevich sheaves of pointed spaces $L_\Nis \mathrm{BSL} \to L_\Nis \mathrm{BGL} \to L_\Nis \mathrm{B}\Gm$ (where the second map is the determinant) remain a fibre sequence after applying $L_{\A^1}$. But $L_\Nis \mathrm{B}\Gm\simeq \Omega^\infty (R\Gamma_\sub{Nis}(-,\bb G_m)[1])$ is $\A^1$-invariant (as in the first paragraph) and Nisnevich locally connected, so it follows from \cite[Lemma 5.5.6.17]{LurieHA} that \[ L_{\A^1} L_\Nis \mathrm{BSL} \to L_{\A^1} L_\Nis \mathrm{BGL} \to L_{\A^1} L_\Nis \mathrm{B}\Gm (\simeq L_\Nis \mathrm{B}\Gm) \] is also a fiber sequence of presheaves of pointed spaces.
It now remains to note that the middle term in this sequence is already a Nisnevich sheaf, e.g. by \cite[Proof of Theorem 6.11]{AntieauElmanto}, so the same is true of the left term, and therefore the three terms are unchanged if we replace $L_{\bb A^1}L_\Nis$ by $L_{\Nis,\bb A^1}$.

Having described $K^\sub{rk=0}$ and $F$ geometrically, it is now enough thanks to (*) to show that $L_{\Nis,\bb A^1} \Sigma^\infty \mathrm{BGL}$ lies in $\opp{Sh}_{\bb A^1,\Nis}(\Sm_X,\Spt)(1)$ and that $L_{\Nis,\bb A^1}\Sigma^\infty \mathrm{BSL}$ lies in $\opp{Sh}_{\Nis, \bb A^1}(\Sm_X,\Spt)(2)$. 
One method is to use the Schubert cell decomposition, e.g.~\cite[proof of Proposition 3.7]{wendt2010more}, which shows for any split semisimple groups $G$ and standard parabolic $P\subset G$, that $L_{\Nis,\bb A^1}\Sigma^\infty G/P$ can be obtained as iterated extension of copies of $L_{\Nis,\bb A_1}(\Sigma^\infty \bb P_X^{1\,\otimes d_i})$, where the numbers $d_i$  are given by $\dim{G/P}$ minus the dimension of the cells.
Thus for $G/P$ to be $1$-effective, there must be a unique cell of maximal dimension, and for $G/P$ to be $2$-effective there must in addition not be a cell of dimension one below the maximal one.
In our case $G/P$ is a Grassmannian, respectively special linear Grassmannian, so this is well-known.\NB{ref?}

It remains to prove $(*)$. For any spectrum $A$ the canonical map of spectra $\colim_n  (\Sigma^\infty \Omega^\infty (A[n]))[-n]\to A$ is an equivalence and so, in particular, there is an equivalence $\colim_n  (\Sigma^\infty \Omega^\infty (E[n]))[-n]\quis E$ of presheaves of spectra. Moreover, since $E$ is Nisnevich locally connective, the deloopings $\Omega^\infty (E[n])$ are Nisnevich locally equivalent to the iterated bar constructions $B^n(\Omega^\infty E)$ (where we use that $\Omega^\infty E$ is a commutative monoid in presheaves of pointed spaces). Consequently $E \wequi \colim_n  (L_{\Nis,\A^1}\Sigma^\infty B^n (\Omega^\infty E))[-n]$, and so it suffices to show for each $n \geq 0$ that $L_{\Nis,\A^1}  \Sigma^\infty B^n(\Omega^\infty E)$ is unstably $j$-effective. Since the $\infty$-category $\opp{Sh}_{\Nis, \bb A^1}(\Sm_X,\Spt)(j)$ is closed under colimits and tensor products, and $B^n$ is obtained as an iteration of such operations (recall that $\Sigma^\infty(A_1 \times A_2) \wequi \Sigma^\infty A_1 \oplus \Sigma^\infty A_2 \oplus (\Sigma^\infty A_1 \otimes \Sigma^\infty A_2)$ whenever $A_1$, $A_2$ are pointed spaces), we conclude.
\end{proof}

We next explain how the orientation of $K$-theory, from Definition \ref{definition:orientation_of_K}, interacts with its unstable slices:

\begin{lemma} \label{lemma:HZA-orient-KGL-compat}
Let $X$ be a regular Noetherian scheme. Then the mapping space $\Map(\Sigma^\infty \Omega^\infty\Pic, R\Gamma_\sub{Nis}(-,\bb Z))$ in $\opp{PSh}(\Sm_X,\Spt)$ is contractible, and so the map $1-(\cdot)^\vee :\Sigma^\infty \Omega^\infty \Pic \to \K$ on smooth $X$-schemes factors uniquely through $\K^{\sub{rk}=0}$. The resulting diagram
\[\xymatrix{
\Sigma^\infty\Omega^\infty\Pic \ar[d]_{\sub{(\ref{eqn_Pic_mess})}}\ar[rr]^{1-(\cdot)^\vee}&&\K^{\sub{rk}=0}\ar[d]^{\sub{det}} \\
\Pic(-) \ar[r]^-{\cong}_-{\sub{(\ref{eqn_pic_as_Gm_coh})}}&(\tau^{\le 1}R\Gamma_\sub{Zar}(-,\bb G_m))[1]\ar[r]& R\Gamma_\sub{Nis}(-,\bb G_m)[1]
}\]
of presheaves on spectra on smooth $X$-schemes commutes up to homotopy.
\end{lemma}
\begin{proof}
Since $\Sigma^\infty \Omega^\infty\Pic$ is Nisnevich locally connected and $R\Gamma_\sub{Nis}(-,\bb Z)$ is coconnective, the mapping space is contractible.

Now we turn to the diagram. Since $X$ is regular, we appeal to Lemma \ref{lem:zarvsnis_on_regular} to identify Zariski and Nisnevich cohomology of $\bb G_m$ on smooth $X$-schemes, and moreover remove the $\tau^{\le 1}$. Then the claim becomes that the composition
\[\Sigma^\infty\Omega^\infty\Pic\xto{1-(\cdot)^\vee} \K^{\sub{rk}=0}\xto{\sub{det}}R\Gamma_\sub{Zar}(-,\bb G_m)[1]\xto{\sub{(\ref{eqn_pic_as_Gm_coh})}\cong}\Pic\] is homotopic to the counit map (\ref{eqn_Pic_mess}). Via the $\Sigma^\infty$-$\Omega^\infty$ adjunction, this is the same as showing that the composition 
\[\Omega^\infty\Pic\xto{1-(\cdot)^\vee} \Omega^\infty \K^{\sub{rk}=0}\xto{\sub{det}}\Omega^\infty (R\Gamma_\sub{Zar}(-,\bb G_m)[1])\xto{\sub{(\ref{eqn_pic_as_Gm_coh})}\cong}\Omega^\infty\Pic\]
is the identity. Arguing yet again as in the uniqueness part of Definition \ref{definition:orientation_of_K}, it is enough to show, for each $m\ge1$, that the composition \[(\bb P^m_X,\infty)\xto{[\roi(1)]}\Omega^\infty\Pic\xto{1-(\cdot)^\vee} \Omega^\infty \K^{\sub{rk}=0}\xto{\sub{det}}\Omega^\infty (R\Gamma_\sub{Zar}(-,\bb G_m)[1])\xto{\sub{(\ref{eqn_pic_as_Gm_coh})}\cong}\Omega^\infty\Pic\] classifies $\roi(1)$. Firstly, by the very definition of the orientation $1-(\cdot)^\vee$, the composition of the first two maps classifies $1-[\roi(-1)]\in K_0^{\sub{rk}=0}(\bb P^m_X)$. It remains to check that the determinant map $\opp{det}:K_0^{\sub{rk}=0}(\bb P^m_X)\to \opp{Pic}(\bb P_X^m)$ sends $1-[\roi(-1)]$ to $\roi(1)$. The rest of the proof is devoted to proving a more general assertion for arbitrary line bundles.

It will be convenient use Morel--Voevodsky's equivalence $L_{\Nis,\A^1}(\Z \times \mathrm{BGL})\quis \Omega^{\infty} \K $ of presheaves of pointed spaces on $\Sm_X$, as in the proof of Lemma \ref{lemma_unstable_slices_of_K}. This equivalence is defined so as to have the following property: given a smooth $X$-scheme $Y$ and a rank $d$ vector bundle $V$ on $Y$, then the composition \begin{equation}X\xto{[V]}L_\sub{Nis}\mathrm{BGL_d}\To L_{\Nis,\A^1}(\bb Z\times \mathrm{BGL})\quis \Omega^{\infty} \K\label{eqn:rankd}\end{equation} classifies $[V]\in K_0(X)$. Here the first map in \eqref{eqn:rankd} classifies $V$, while the middle map is given by $d$ on the first factor and induced by the usual inclusion $\GL_d\to\GL$ on the second factor.

Restricting to rank $0$ parts, we obtain an equivalence $L_{\Nis,\A^1} \mathrm{BGL}\simeq \Omega^{\infty} \K^{\sub{rk}=0}$ with the following property: for any smooth $X$-scheme $Y$ and rank $d$ vector bundle $V$ on $Y$, the composition \[X\xto{[V]}L_\Nis\mathrm{BGL_d}\To L_{\Nis,\bb A^1}\mathrm{BGL}\quis \Omega^{\infty} \K^{\sub{rk}=0}\] classifies $[V]-d\in \K_0^{\sub{rk}=0}(X)$. We now restrict to the case $d=1$, so that $V=L$ is a line bundle on $Y$, and we rewrite the previous line as the top of the following diagram:
\[
\xymatrix{
X\ar[r]^-{[L]}& L_\Nis \rm{B}\bb G_m\ar[r]& L_{\Nis,\bb A^1}\mathrm{BGL}\ar[r]^\sim\ar[d]_{L_{\Nis,\bb A^1}B(\det)} & \Omega^{\infty} \K^{\sub{rk}=0}\ar[d]^{\sub{det}}\\
&& L_\sub{Nis} \rm{B}\bb G_m\ar[r]_-{\sim} & \Omega^\infty(R\Gamma_\sub{Nis}(-,\bb G_m)[1])
 }\]
Furthermore, by the very definition of the determinant map in Lemma \ref{lemma_unstable_slices_of_K}(2) (which already implicitly used a the description of $K$-theory via $\BGL$), the new square
commutes (in which we drop the $\bb A^1$-localisation at the bottom left because it is already $\bb A^1$-invariant). The unwritten diagonal composition $L_\Nis \rm{B}\bb G_m\to L_\Nis \rm{B}\bb G_m$ is the identity. Since we have seen that the top line classifies $[L]-1$, it follows that indeed the determinant map $\det:\K_0^{\sub{rk}=0}(X)\to\opp{Pic}(X)$ sends $[L]-1$ to $L$; therefore $\det(1-[L^\vee])=\det([L^\vee]-1)^{-1}=(L^\vee)^{-1}=L$, as required.
\end{proof}

We are now equipped to define the $\bb Z$-linear structure and first Chern class in $\bb A^1$-motivic cohomology, at least in the case of regular Noetherian schemes; later we will left Kan extend from smooth $\bb Z$-schemes to arbitrary qcqs schemes (see Theorem \ref{theorem:SH_mot_coh}):

\begin{construction} \label{cons:low-weights}
Let $X$ be a regular Noetherian scheme. Let $R\Gamma_\sub{Nis}(-,\bb Z) \to \Z(0)^{\A}$ be the map of $\bb E_\infty$-algebras in presheaves of spectra on smooth $X$-schemes given as the composition
 \[ R\Gamma_\sub{Nis}(-,\bb Z) \stackrel{\sub{Lemma \ref{lemma_unstable_slices_of_K}(1)}}\simeq s^0_u(\K) \xto{\sub{\eqref{eq:slice-slice}}} \omega^\infty s^0 \KGL_X=\Z(0)^{\A}.\]
Similarly, let $c_1^{\bb A}:R\Gamma_\Nis(\ph, \Gm)[1] \to \Z(1)^{\A}[2]$ be the map of presheaves of spectra on smooth $X$-schemes given as the composition
\[ R\Gamma_\Nis(\ph, \Gm)[1] \stackrel{\sub{Lemma \ref{lemma_unstable_slices_of_K}(2)}} \simeq s^1_u(\K) \xto{\sub{\eqref{eq:slice-slice}}} \omega^\infty s^1 \KGL_X=\bb Z(1)^\bb A[2].\] Note that the following diagram 
\begin{equation}\label{eq:c1-det2}
\begin{tikzcd}
\K^{\sub{rk}=0} \ar{d}{\mathrm{det}} \ar{r} & \Fil^1_{\bb A}\KH \ar{d}{\mathrm{edge}}\\
R\Gamma_\sub{Nis}(-,\bb G_m)[1] \ar[swap]{r}{c_1^{\bb A}[2]} & \bb Z(1)^{\bb A}[2]
\end{tikzcd}
\end{equation}
commutes up to homotopy on smooth $X$-schemes, by Lemma \ref{lemma:HZA-orient-KGL-compat}.
\end{construction}

\begin{remark}[$\bb Z$-linear structure and shearing]
Construction \ref{cons:low-weights} defines in particular a map of $\bb E_\infty$-algebras $\bb Z\to \bb Z(0)^\bb A(\bb Z)$. By the multiplicative structure of Remark \ref{rmk:HZA-mult}, this upgrades $\bb Z^{\A} (\star)[2\star]$ and $\bb Z^{\A,\cdh} (\star)[2\star]$ to be $\bb E_\infty$-algebras in graded $\D(\bb Z)$-valued (rather than merely $\opp{Sp}$-valued) presheaves on qcqs schemes. We may therefore shear as in Remark \ref{rmk:shearing1} to replace these by $\bb Z^{\A} (\star)$ and $\bb Z^{\A,\cdh}(\star)$, while maintaining the $\bb E_\infty$-algebra structure.
\end{remark}

\begin{remark}[Motivic cohomology mod $p$]
One reason that we record the existence of the $\bb Z$-linear structure in the previous remark is that it implies the existence of natural and compatible $\bb E_\infty$-algebra structures on the graded presheaves $\Z(\star)^{\bb A,\sub{cdh}}/p$ and $\Z(\star)^{\bb A}/p$, obtained by reducing modulo any prime number $p$. Without the $\bb Z$-linear structure we would encounter the usual multiplicative issues with taking quotients in the world of ring spectra. This multiplicative structure mod-$p$ is useful for treating motivic cohomology ``one prime at a time''.
\end{remark}

\subsubsection{Further properties}
Here is a summary of further properties of $\bb A^1$-motivic cohomology and $\bb A^1$-cdh-motivic cohomology which follow from what we have already established and from general results in motivic homotopy theory:

\begin{theorem}\label{theorem:SH_mot_coh}
The $\mathbb{E}_{\infty}$-algebras in graded presheaves of complexes
\[
\Z(\star)^{\bb A,\sub{cdh}}\qquad\text{and}\qquad\bb Z(\star)^\bb A:\Sch^{\qcqs,\op} \rightarrow \Gr\D(\bb Z)\qquad 
\]
have the following properties:
\begin{enumerate}

\item For any qcqs scheme $X$ there exist functorial, multiplicative, exhaustive $\bb Z$-indexed filtrations 
\[
\mathrm{Fil}_{\bb A,\sub{cdh}}^{\star}\KH(X) \qquad\textrm{and}\qquad \mathrm{Fil}_\bb A^{\star}\KH(X)
\] on $\KH(X)$, such that the graded pieces are naturally and multiplicatively given by 
\[
\mathrm{gr}_{\bb A,\sub{cdh}}^j\KH(X)\simeq \Z(j)^{\bb A,\sub{cdh}}(X)[2j] \qquad\textrm{and}\qquad \mathrm{gr}_\bb A^j\KH(X)\simeq \Z(j)^\bb A(X)[2j]
\] for $j\in\bb Z$. The filtrations are related by a natural multiplicative map $\mathrm{Fil}_{\bb A,\sub{cdh}}^{\star}\KH(X)\to \mathrm{Fil}_{\bb A}^{\star}\KH(X)$.

In particular, there exist natural, multiplicative Atiyah--Hirzebruch spectral sequences
\[
E_2^{i,j}=H_{\bb A,\sub{cdh}}^{i-j}(X, \Z(-j)) \implies \KH_{-i-j}(X)\qquad\mathrm{and}\qquad E_2^{i,j}=H_{\bb A}^{i-j}(X, \Z(-j)) \implies \KH_{-i-j}(X)
\]
such that the left one maps to the right one.

\item For each $j\in\bb Z$, the presheaves $\Z(j)^{\bb A,\sub{cdh}}$ and $\mathrm{Fil}_{\bb A,\sub{cdh}}^{\star}\KH$ are finitary, $\bb A^1$-invariant, cdh sheaves; meanwhile, $\Z(j)^{\bb A}$ and $\mathrm{Fil}_{\bb A}^{\star}\KH$ are a priori only an $\bb A^1$-invariant, Nisnevich sheaves which commute with filtered colimits along smooth affine transition maps.

\item Weight zero and $\bb Z$-linear structure: cdh-locally left Kan extending the map from Construction \ref{cons:low-weights} defines a map of presheaves of $\bb E_\infty$-algebras 
\begin{equation} \label{eq:wt0-cdh}
R\Gamma_\sub{cdh}(-,\bb Z)\To \bb Z(0)^{\bb A,\sub{cdh}}
\end{equation}
on qcqs schemes, fitting into a commutative diagram
\begin{equation}\xymatrix{
R\Gamma_\sub{cdh}(-, \bb Z) \ar[r]& \bb Z(0)^{\bb A,\sub{cdh}}\\
\KH\ar[u]^{\sub{rank}}\ar[ur]_{\sub{edge}}&
}\label{eq:z-structure5}\end{equation}

\item Weight one and first Chern class: cdh-locally left Kan extending the map from Construction \ref{cons:low-weights} defines the \emph{motivic first Chern class} \[c_1^{\bb A,\sub{mot}}:R\Gamma_\sub{cdh}(-,\bb G_m)[-1]\to\bb Z(1)^{\bb A,\sub{cdh}},\] fitting into a commutative diagram
\begin{equation}\label{eq:c1-det}
\begin{tikzcd}
\KH^\sub{rk=0} \ar{d}{\mathrm{det}} \ar{r} & \Fil^1_{\bb A,\sub{cdh}}\KH \ar{d}{\mathrm{edge}}\\
R\Gamma_\sub{cdh}(-,\bb G_m)[1] \ar[swap]{r}{c_1^{\bb A,\sub{cdh}}[2]} & \bb Z(1)^{\bb A,\sub{cdh}}[2]
\end{tikzcd}
\end{equation}\NB{Font sizes in this diagram are a bit small?}
of presheaves of spectra on qcqs schemes.
\end{enumerate}
\end{theorem}
\begin{proof}
(1): For any scheme $X$, with structure map $f: X \rightarrow \Spec(\bb Z)$, the transformation~\eqref{eq:natural_slice} induces a natural multiplicative map of filtered objects, functorial in $X$:
\begin{equation}\label{eq:filtered-kh}
\Fil_{\bb A,\cdh}^{\star}\KH(X):=\map_{\SH(X)}(1_X, f^*\Fil^{\star}_\sub{slice}\KGL_{\bb Z}) \rightarrow \Fil_{\bb A}^{\star}\KH(X):=\map_{\SH(X)}(1_X, \Fil^{\star}_\sub{slice}\KGL_X)
\end{equation}
The exhaustiveness follows from the exhaustiveness of the slice filtration which was recalled earlier in Remark~\ref{rem:exhaustive} and the fact that $f^*$ preserves colimits. Taking graded pieces of~\eqref{eq:filtered-kh} then gives us the multiplicative map
\[
\bb Z(\star)^{\bb A, \sub{cdh}} \To \bb Z(\star)^{\bb A},
\]
already discussed in Remark~\ref{rem:functoriality}.

(2) For any fixed qcqs scheme $S$, Remark \ref{rem_coh_over_base} shows that the sheaves $\bb Z(j)^{\bb A}$ and $\bb Z(j)^{\bb A, \sub{cdh}}$, when restricted to smooth $S$-schemes, are representable in $\SH(S)$. They are therefore $\bb A^1$-invariant Nisnevich sheaves on smooth $S$-schemes; as this holds for all $S$, they are $\bb A^1$-invariant Nisnevich sheaves on qcqs schemes.
The claims about cdh descent and finitariness of $\bb Z(j)^{\bb A,\sub{cdh}}$ are immediate from Proposition \ref{prop:auto-cdh}, while the case of $\bb Z(j)^{\bb A}$ follows from that proposition and the final assertion of Remark \ref{rem:naturality}. The same arguments apply to the $\bb A^1$-motivic and $\bb A^1$-cdh-motivic filtrations.

(3) \& (4): Recall first $L_\sub{cdh}\K^\sub{cn}\quis\KH$; see the proof of Theorem \ref{prop:basic_props_of_cdh_mot}(1) below. So, by left Kan extending then cdh sheafifying, parts(3)\&(4) follow immediately from Construction \ref{cons:low-weights} and Lemma \ref{lemma:HZA-orient-KGL-compat} once we also note the following: the three functors $\opp{CAlg}_\bb Z^\sub{loc}\to\D(\bb Z)$, $A\mapsto\bb Z$, $A\mapsto A^\times$, and $A\mapsto \K^\sub{cn}(A)$ are left Kan extended from essentially smooth, local $\bb Z$-algebras. The first of these is clear: the diagram computing the left Kan extension is constantly $\bb Z$ and contractible. The second holds because $\bb G_{m,\bb Z}$ is a smooth affine $\bb Z$-scheme. The third follows from a result of Bhatt--Lurie \cite[Example~A.0.6]{ElmantoHoyoisKhanSosniloYakerson2020}.
\end{proof}

\begin{remark}[Completeness of filtrations]\label{remark_completeness_of_fil}
It is natural to ask for conditions under which the filtrations $\mathrm{Fil}_{\bb A,\sub{cdh}}^{\star}\KH(X)$ and $\mathrm{Fil}_\bb A^{\star}\KH(X)$ are complete. By Proposition \ref{prop:detecting-veff} and Lemma \ref{lemm:convergence-bootstrap} we know the following:
\begin{enumerate}
\item For any regular Noetherian scheme $X$ of dimension $\le 1$, the Nisnevich sheaf of spectra $\Fil_{\bb A}^j\KH$ on smooth $X$-schemes is Nisnevich locally $j-\dim(X)$-connective.
\item For any qcqs scheme $f:X\to\Spec(\bb Z)$, the motivic spectrum $f^*\kgl_\bb Z\in\SH(X)$ is very effective.
\item For any Noetherian scheme $X$ of dimension $\le d$, the Nisnevich sheaf of spectra $\Fil_{\bb A,\sub{cdh}}^j\KH$ on smooth $X$-schemes is Nisnevich locally $j-d$-connective.
\end{enumerate}
We will eventually see in Corollary~\ref{cor:bdd-a1} that, for any qcqs scheme $X$ of finite valuative dimension $\le d$, then $\mathrm{Fil}_{\bb A,\sub{cdh}}^{j}\KH(X)$ and $\mathrm{Fil}_\bb A^{j}\KH(X)$ agree and are $j-d$-connective (so that the filtrations are complete).
\end{remark}

\begin{remark}[First Chern class for $\bb A^1$-motivic cohomology]\label{rem_A_1st_Chern_class}
A variant of Theorem \ref{theorem:SH_mot_coh}(4), which we will also use, is obtained by considering the composition \[c_1^{\bb A}:R\Gamma_\sub{Nis}(-,\bb G_m)[-1]\To R\Gamma_\sub{cdh}(-,\bb G_m)[-1]\xto{c_1^{\bb A,\sub{cdh}}} \bb Z(1)^{\bb A,\sub{cdh}}\To \bb Z(1)^\bb A.\] On regular Noetherian schemes this agrees with the Chern class already defined in Construction \ref{cons:low-weights}.
\end{remark}

\subsection{Rational structure}\label{sec:rational}
The goal of this section is to address the rational structure of the $\bb A^1$-invariant motivic cohomologies introduced in the previous subsection. We identify the two theories after rationalisation and degenerate the Atiyah--Hirzebruch spectral sequence. In Corollary~\ref{cor:rational-smd} we moreover establish rational results which will later be shown to hold integrally.

\begin{remark}[Rationalisation]
To clearly fix notation, for any qcqs scheme $X$ and $j\in\bb Z$ we write 
 \[\bb Q(j)^{\bb A,\sub{cdh}}(X):=\bb Z(j)^{\bb A,\sub{cdh}}\otimes_{\bb Z}\bb Q\quad\text{and}\quad \bb Q(j)^\bb A(X):=\bb Z(j)^\bb A(X) \otimes_{\bb Z} \bb Q\] for the rationalisations of the $\bb A^1$- and $\bb A^1$-cdh-motivic cohomologies. Note that, thanks to Corollary~\ref{cor:cdh-cpt-gen} and Theorem \ref{theorem:SH_mot_coh}, the presheaf $\bb Q(j)^{\bb A,\sub{cdh}}$ is the rationalisation of $\bb Z(j)^{\bb A,\sub{cdh}}$ within the $\infty$-category of cdh sheaves on qcqs schemes, and $\bb Q(j)^\bb A$ is the rationalisation of $\bb Z(j)^{\bb A}$ within the $\infty$-category of Nisnevich sheaves of complexes on qcqs schemes.

Similarly we set 
\[
\Fil^j_{\bb A,\sub{cdh}}\KH(X)_\bb Q:=(\Fil^j_{\bb A,\sub{cdh}}\KH(X))_\bb Q \qquad \Fil^j_{\bb A}\KH(X)_\bb Q:=(\Fil^j_{\bb A}\KH(X))_\bb Q,
\] 
which define the rationalisations of $\Fil^j_{\bb A,\sub{cdh}}\KH$ (respectively~$\Fil^j_{\bb A}\KH$) as cdh (respectively~Nisnevich) sheaves of spectra. Note that this is compatible with rationalising motivic spectra in the sense that \[\map_{\SH(X)}(M_X(Y)\otimes\bb T^{\otimes -j},E_\bb Q)=\map_{\SH(X)}(M_X(Y)\otimes\bb T^{\otimes -j},E)_\bb Q\] for all $Y\in\Sm_X$ and $j\in\bb Z$, since $M_X(Y)\otimes\bb T^{\otimes -j}$ is a compact object; here the rationalisation of any $E\in\SH(X)$ (or in $\opp{Sp}$, etc.) is defined as in Example \ref{ex:rationalization-cdh}. We will use such identifications without further mention.
\end{remark}

The main theorem of this section is as follows.

\begin{theorem}\label{thm:rational} Let $X$ be a qcqs scheme.
\begin{enumerate}
\item The comparison maps 
\[
\bb Q(j)^{\bb A, \cdh}(X) \To \bb Q(j)^{\bb A}(X) \qquad \mathrm{Fil}_{\bb A,\sub{cdh}}^{j}\KH(X)_{\bb Q}\To \mathrm{Fil}_{\bb A}^{j}\KH(X)_{\bb Q}
\] (by rationalising (\ref{eq:cdh-to-aone}) and Theorem \ref{theorem:SH_mot_coh}(1)) are equivalences for all $j\in\bb Z$.
\item The filtration $\mathrm{Fil}_{\bb A}^{\star}\KH(X)_{\bb Q}$ naturally splits, i.e., there is a natural, multiplicative equivalence of filtered spectra:
\[
\mathrm{Fil}_{\bb A}^{\star}\KH(X)_{\bb Q} \simeq \bigoplus_{j \geq \star} \bb Q(j)^{\bb A}(X)[2j].
\]
\item The Atiyah--Hirzebruch spectral sequences of Theorem~\ref{theorem:SH_mot_coh}(1) rationally degenerate at the $E_2$-page and the abutement filtration on $\KH_n(X)_\bb Q$ is naturally split for each $n\in\bb Z$.
\item If $X$ is regular Noetherian then the splittings of part (3) identify $H^i_\bb A(X,\bb Q(j))$ with $\K_{2j-i}(X)_\bb Q^{(j)}$ for all $i\in\bb Z$ and $j\ge 0$.\footnote{We adopt the convention, for $X$ regular Noetherian, that $\K_{n}(X)_\bb Q^{(j)}=0$ if $n<0$.\label{footnote_neg_adams}}
\end{enumerate}
\end{theorem}

To prove Theorem \ref{thm:rational} we need some non-trivial inputs from rational motivic cohomology and $\bb A^1$-homotopy theory, in particular the Adams decomposition of $\KGL_X$ due to Riou:

\begin{theorem}[Riou \cite{Riou2010}]\label{thm_riou}
For any qcqs scheme $X$, there exist in $\SH(X)$ a natural direct sum decomposition
\[
(\KGL_X)_\bb Q\simeq\bigoplus_{j\in\bb Z}\KGL_X^{(j)}
\] and endomorphisms $\psi^k$ of $(\KGL_X)_\bb Q$, for $k\in\bb Z\setminus\{0\}$, with the following properties:
\begin{enumerate}
\item $\psi^k\psi^{\ell}$ is homotopic to $\psi^{k\ell}$ for all $k,\ell\in\bb Z\setminus\{0\}$;
\item each direct summand $\KGL_X^{(j)}$ of $\KGL_{\bb Q}$ is preserved up to homotopy by $\psi^k$, where it acts as multiplication by $k^m$;
\item for each $j\in\bb Z$, the rationalisation of the Bott isomorphism $\KGL_X\otimes\bb T\quis \KGL_X$ restricts to an equivalence $\KGL_X^{(j)}\otimes\bb T\quis\KGL_X^{(j+1)}$ (whence $\KGL_X^{(0)}\otimes\bb T^{\otimes j}\quis \KGL_X^{(j)}$);
\item if $X$ is a regular Noetherian scheme with an ample family of line bundles then, for any $n\ge0$, the endomorphism of 
\[
\pi_n\map_{\SH(X)}(1_X,(\KGL_X)_\bb Q)=K_n(X)_\bb Q
\] induced by $\psi^k$ is the classical $k^\sub{th}$ Adams operator of Hiller \cite{Hiller1981}, Kratzer \cite{Kratzer1980}, and Soul\'e \cite{Soule1985}.
\end{enumerate}
\end{theorem}

\begin{proof}[Proof overview]
As well as Riou \cite{Riou2010}, we recommend \cite[\S9]{AnnalaHoyoisIwasa2025} for a modern proof of such results in the more general, non-$\bb A^1$-invariant context (which also has the advantage that it does not require the separability assumption on schemes which is often present in \cite{Riou2010}).

Riou works with regular Noetherian schemes, but this is sufficient to prove the theorem in general. Indeed, by naturality and the fact that $X\mapsto\KGL_X$ is an absolute motivic spectrum, we see that his Adams decomposition is preserved under pulling back along any map of regular Noetherian schemes (more precisely, it provides a decomposition of $\KGL$ as a cartesian section of the cartesian fibration classified by $\SH:(\Sch^\sub{reg.~noeth.})^\sub{op} \rightarrow \Cat_{\infty}$). The Adams operators of Riou are similarly preserved under pulling back along any map of regular Noetherian schemes, as follows from their definition \cite[Definition~5.3.2]{Riou2010}. We extend the decomposition and operators to arbitrary qcqs schemes $X$ by pulling back along the structure map $X\to\Spec(\bb Z)$; by what we have just explained, this does not change the decomposition or operators in the regular Noetherian case. Furthermore, in this way properties (1)--(3) reduce to the case $X=\bb Z$.

We now give precise references to \cite{Riou2010} to prove (1)--(3) in the regular Noetherian case. The Adams operations are constructed in \cite[Definition~5.3.2]{Riou2010}, where (1) is explained; the decomposition is \cite[Definition~5.3.9 \& Theorem~5.3.10]{Riou2010}. Part (2) is \cite[Proposition~5.3.14]{Riou2010}; part (3) is explained after \cite[Definition~5.3.17]{Riou2010}. Part (4) is never explicitly stated in \cite{Riou2010} (nor, seemingly, elsewhere in the literature), so we provide a reasonably detailed explanation of the proof.

We begin by providing a modern account of Soul\'e's construction of his Adams operations; see also \cite[\S3.3]{Riou2010}. Let $X$ be a regular Noetherian scheme, and let $\Omega^\infty \K \in \cal P:=\opp{PSh}(\Sm_X,\opp{Spc}_*)$ be the $K$-theory presheaf of pointed spaces on smooth $X$-schemes. For each $d\ge1$, Soul\'e \cite[\S4.3]{Soule1985} defines a map $R_\bb Z(\GL_d)\to \pi_0\Map_\cal P(\BGL_d,\Omega^\infty \K)$
where $R_\bb Z(\GL_d)$ is Grothendieck's representation ring of the algebraic group $\GL_d$ over $\bb Z$. Putting $\BGL:=\colim_d\BGL_d$, we now let $d\to\infty$ to obtain \begin{equation}R_\bb Z(\GL):=\lim_dR_\bb Z(\GL_d)\To \pi_0\Map_\cal P(\BGL,\Omega^\infty \K).\label{eqn:Soule}\end{equation} (Here we have implicitly used that the canonical map $\pi_0\Map_\cal P(\BGL,\Omega^\infty \K)\to \lim_d\pi_0\Map_\cal P(\BGL_d,\Omega^\infty \K)$ is an isomorphism; indeed, the obstruction term $\opp{Rlim}^1_d\pi_1\Map_\cal P(\BGL_d,\Omega^\infty \K)$ vanishes because the transition maps as $d\to\infty$ are surjective, due to the standard calculation of the $K$-theory of $\BGL_d$ in terms of Chern classes.) Moreover, by the universal property of the $+$-construction, we may replace $\BGL$ by $\BGL^+$ in line (\ref{eqn:Soule}). Using the rank map of the representation to incorporate a copy of $\bb Z$\NB{I don't understand this step in Soule. What does one do to $\bb Z$?}, and identifying $\Omega^\infty \K$ with the Nisnevich sheafification of $\bb Z\times\BGL^+$, this defines Soul\'e's final map \begin{equation}R_\bb Z(\GL)\To \pi_0\Map_\cal P(\Omega^\infty \K,\Omega^\infty \K).\label{eqn:Soule2}\end{equation} Each of his Adams operators on the $K$-groups of smooth $X$-schemes is then induced by a certain endomorphism $\psi^k_\sub{S}:\Omega^\infty \K\to \Omega^\infty \K$ in the image of (\ref{eqn:Soule2}).

Meanwhile, via the functors $\omega^\infty:\SH(X)\to \opp{Sh}_{\sub{Nis},\bb A^1}(\Sm_X,\opp{Sp})$ and $\Omega^\infty:\opp{Sp}\to\opp{Spc}_*$, each of Riou's Adams operators $\psi^k$ on $\KGL_\bb Q$ induces an endomorphism $\psi^k_\sub{R}:\Omega^\infty \K_\bb Q\to \Omega^\infty \K_\bb Q$.

We now prove (4). So let us assume further that $X$ has an ample family of line bundles. Our goal is to show that $\psi^k_\sub{S}$ and $\psi^k_\sub{R}$ induce the same endomorphism on $\K_n(X)_\bb Q$, for all $n\ge0$. In fact we will show that the following diagram of presheaves of pointed spaces on $\Sm_X$ commutes up to homotopy:
\begin{equation*}
\begin{CD}
\Omega^\infty \K @>>> \Omega^\infty \K_\Q \\
@V{\psi^k_S}VV @V{\psi^k_R}VV \\
\Omega^\infty \K @>>> \Omega^\infty \K_\Q.
\end{CD}
\end{equation*}
In other words, we wish to show that two elements of $\pi_0\Map_\cal P(\bb Z\times\BGL,\Omega^\infty \K_\bb Q)$ coincide. By again writing $\BGL=\opp{colim}_d\BGL_d$, and using the motivic equivalence between $\BGL_d$ and the infinite Grassmannian $\mathrm{Gr}(d,\infty)=\opp{colim}_r\mathrm{Gr}(d,r)$ \cite[\S4.3 Proposition~3.7]{MorelVoevodsky1999}, and noting that the $\opp{Rlim}^1_r$ obstruction vanishes (similarly to above; see \cite[Lemma~1.2.10]{Riou2010}), we obtain an isomorphism \[\pi_0\Map_\cal P(\bb Z\times\BGL,\Omega^\infty \K_\bb Q)\isoto \lim_{d,r,n}\pi_0\Map_\cal P(\{-n,\dots,n\}\times \mathrm{Gr}(d,r),\Omega^\infty \K_\bb Q).\] Each $\pi_0\Map_\cal P$ on the right is $\K_0(\bigsqcup_{-n}^n \mathrm{Gr}(d,r))_\bb Q$, and so we have reduced the problem to showing that the diagram
\begin{equation*}
\begin{CD}
\K_0(Y) @>>> \K_0(Y)_\bb Q \\
@V{\psi^k_S}VV @V{\psi^k_R}VV \\
\K_0(Y) @>>> \K_0(Y)_\bb Q.
\end{CD}
\end{equation*}
commutes for each of the smooth $X$-schemes $Y=\bigsqcup_{-n}^n \mathrm{Gr}(d,r)$, for $d,r,n\ge0$. A priori this is only a diagram of sets. However, $\psi_\sub{R}^k$ was induced via $\Omega^\infty$ from a map of rational $K$-theory spectra, so in fact the endomorphism it induces on $\K_0(-)_\bb Q$ is additive. Soul\'e's $\psi^k_\sub{R}$ is also additive on $\K_0$: it is even an endomorphism of special $\lambda$-rings \cite[\S1.5]{Soule1985}.

We now apply the splitting principle (for each $Y=\bigsqcup_{-n}^n \mathrm{Gr}(d,r)$, for some $d,r,n$, and each vector bundle $P$ on $Y$, the complete flag variety $\mathrm{Fl}(P)$ is a smooth projective $Y$-scheme such that $\K_0(Y)\to \K_0(\mathrm{Fl}(P))$ is injective and $P$ becomes a sum of classes of line bundles in $\K_0(\mathrm{Fl}(P))$) and Jouanolou's trick (there is an affine scheme $Q$ together with a smooth map $Q\to \mathrm{Fl}(P)$ such that $\K_0(\mathrm{Fl}(P))\to \K_0(Q)$ is injective). To apply Jouanolou's trick as in \cite[Proposition 4.4]{Weibel1989a}, we note that $\mathrm{Fl}(P)$ is quasi-projective over $X$ and therefore admits an ample family of line bundles by \cite[Example 2.1.2(h)]{Thomason1990}. The problem has thus been reduced to showing, for each affine smooth $X$-scheme $Q$ and line bundle $L$ on $Q$, the elements $\psi^k_\sub{S}([L])$ and $\psi^k_\sub{R}([L])$ coincide in $\K_0(Q)_\bb Q$. We claim they are both given by $[L^{\otimes k}]$. In the case of Soul\'e's Adams operator this follows from general theory of lambda rings, since $[L]\in \K_0(Q)$ is rank one for the lambda ring structure \cite[\S4]{{Kratzer1980}}. For Riou's Adams operator it is easily deduced from the following facts: (1) $\psi^k_\sub{R}$ is additive on $\K_0(Q)_\bb Q$ (as already noted above), (2) any line bundle on $Q$ is the pullback of $\scr O_{\bb P_X^N}(1)$ along some map $Q\to\bb P_X^N$ for some $N\gg0$ (since $Q$ is affine, so any line bundle on it is globally generated), and (3) by the definition of $\psi^k_\sub{R}$ \cite[Definition~ 5.3.2]{Riou2010}, one has $\psi^k_\sub{R}([\scr O_{\bb P_X^N}(1)]-1)=[\scr O_{\bb P_X^N}(k)]-1$ in $\K_0(\bb P_X^N)$.
\end{proof}

We now present various consequences of the Adams decomposition in rational motivic homotopy theory.

\begin{corollary}\label{corol_Adams}
For any regular Noetherian scheme with an ample family of line bundles $X$, $n\ge0$, and $j\in\bb Z$, there is a natural isomorphism $\pi_n\map_{\SH(X)}(1_X,\KGL_X^{(j)})\cong \K_n(X)_{\bb Q}^{(j)}$ (where the right hand side refers to the $j^\sub{th}$ Adams eigenspace of $\K_n(X)_\bb Q$).
\end{corollary}
\begin{proof}
Theorem \ref{thm_riou} implies that $\K_n(X)_\bb Q$ decomposes into the direct sum of $\pi_n\map_{\SH(X)}(1_{X},\KGL_X^{(j)})$ and a complementary piece (represented by the sum of $\KGL_X^{(i)}$ for $i\neq j$), such that the usual Adams operator $\psi^k$ acts as $k^j$ on the first summand, and $\psi^k-k^j$ acts invertibly on the second summand. Therefore $\pi_n\map_{\SH(X)}(1_X,\KGL_X^{(j)})$ is the kernel of $\psi^k-k^j$ on $\K_n(X)_\bb Q$ (for any $k\ge1$), which is by definition $\K_n(X)_\bb Q^{(j)}$.
\end{proof}

\begin{corollary}\label{cor:rational-smd}
For any qcqs scheme $f:X\to \Spec(\bb Z)$, the following hold in $\SH(X)$:
\begin{enumerate}
\item The motivic spectrum $\KGL^{(0)}_X$ is effective and $\Fil^1_\sub{slice}\KGL^{(0)}_X=0$.
\item The map $\KGL^{(0)}_X\to(\KGL_X)_\bb Q$ induces an equivalence $\KGL^{(0)}_X\quis s^0(\KGL_X)_\bb Q$.
\item The canonical map $f^*s^0(\KGL_{\bb Z}) \rightarrow s^0(\KGL_X)$ is a rational equivalence.
\item The canonical map $f^*\kgl_{\bb Z} \To \kgl_{X}$ is a rational equivalence.
\item The canonical map $s^0(1_{X}) \To s^0(\KGL_X)$ is a rational equivalence.
\item The canonical map $f^*s^0(1_{\bb Z}) \To s^0(1_{X})$ is a rational equivalence.
\end{enumerate}
\end{corollary}
\begin{proof}
(1): Since $X\mapsto \KGL^{(0)}_X$ is an absolute motivic spectrum (as it is a direct summand of the absolute motivic spectrum $\KGL \otimes \bb Q$), to prove the effectivity assertion it is enough to prove that $\KGL^{(0)}_k$ is effective for all perfect fields $k$ (see Proposition \ref{prop:detect-effective}). Corollary \ref{corol_Adams} implies that the canonical map \[\pi_0\map_{\SH(k)}(1_{k},\KGL^{(0)}_k)\To \pi_0\map_{\SH(k)}(1_{k},(\KGL_k)_\bb Q)\] is an isomorphism, since it is given by $\K_0(k)_\bb Q^{(0)}\isoto \K_0(k)_\bb Q$; in particular, every map $1_{\SH(k)}\to(\KGL_k)_\bb Q$ factors through $\KGL_k^{(0)}$. It follows that the zero slice of the unit map $\eta:1_{\SH(k)}\to(\KGL_k)_\bb Q$ factors through $s^0(\KGL_k^{(0)})$. That is, we have a commutative diagram
\[\xymatrix{
s^0(1_{k})\ar[r]^{s^0(\eta)}\ar[dr]_{\exists} & s^0(\KGL_k)_\bb Q\\
& s^0(\KGL_k^{(0)})\ar[u]
}\]
Next we apply Levine's theorem \cite[Theorems 6.4.2 and 10.5.1]{Levine2008}, stating that the horizontal arrow in the above diagram is an equivalence; since the vertical arrow is split, we deduce that in fact all the arrows in the diagram are equivalences. By now decomposing $(\KGL_k)_\bb Q$ as in Theorem \ref{thm_riou}, it follows that $s^0(\KGL_k^{(j)})=0$ for all $j\in\bb Z\setminus\{0\}$, or in other words (using part (3) of the theorem) that $s^{-j}(\KGL_k^{(0)})=0$ for all such $j$. Since the slice filtration is exhaustive, this implies that $\KGL_k^{(0)}$ is effective\footnote{Up to convergence issues of the slice filtration it also shows that $\Fil^1_\sub{slice}\KGL_k^{(0)}=0$; but that does not formally imply the same for general $X$, whence the proof of part (1) is not finished at this point.} and so completes the proof that $\KGL_X^{(0)}$ is effective for all qcqs schemes $X$.

Next we show that $\Fil^1_\sub{slice}\KGL_X^{(0)} \simeq 0$; that is, we need to show for all smooth $X$-schemes $Y$ and $j>0$ that $\map_{\SH(X)}(M_X(Y)\otimes \bb T^{j}_X, \KGL_X^{(0)}) \simeq 0$. Replacing $X$ by $Y$ it suffices to treat the case $Y=X$, or in other words we must show that the presheaf $X\mapsto \map_{\SH(X)}(1_X, \KGL_X^{(0)}\otimes\bb T_X^{\otimes -j})$ is zero. But Theorem \ref{thm_riou} shows that this presheaf on qcqs schemes is a direct summand of the presheaf $X\mapsto \map_{\SH(X)}(1_X, \KGL_{X})_\bb Q\simeq \KH(X)_\bb Q$, which is the cdh sheafification of the left Kan extension of its restriction to smooth $\bb Z$-schemes (see the proof of Theorem \ref{prop:basic_props_of_cdh_mot}(1)). It follows that $X\mapsto \map_{\SH(X)}(1_X, \KGL_X^{(0)}\otimes\bb T_X^{\otimes -j})$ has the same property, and so to prove that it vanishes we may restrict attention to smooth affine $\bb Z$-schemes $X$.

We have reduced to showing that $\pi_n\map_{\SH(X)}(1_X, \KGL_X^{(0)}\otimes\bb T_X^{\otimes -j})=0$ for all smooth affine $\bb Z$-schemes $X$ and all $n\in\bb Z$. If $n<0$ this vanishing holds because the group is a direct summand of $\K_n(X)_\bb Q$ (by Theorem \ref{thm_riou}), which vanishes since $X$ is regular. Otherwise $n\ge 0$, in which case the group identifies with $\K_{n}(X)_{\bb Q}^{(-j)}$ by Corollary \ref{corol_Adams}, and so we must show that $\psi^k-k^{-j}$ acts invertibly on $\K_n(X)_\bb Q$ for any/all $k\in\bb Z\setminus\{0\}$. This holds because the $\bb N$-indexed gamma filtration $\K_n(X)_\bb Q=\Fil^0_\gamma\K_n(X)_\bb Q\supset \Fil^1_\gamma\K_n(X)_\bb Q\supset\cdots$ has finite length (for $n>0$ see \cite[Corollary~1]{Soule1985}; for $n=0$ see \cite[Exp.~VI, Theorem~6.9]{SGA_VI}), and $\psi^k-k^{-j}$ acts invertibly on the $i^\sub{th}$-graded piece as $k^i-k^{-j}\neq0$ (recall that $j>0$).

(2) \& (4): Part (1) implies that $\KGL_X^{(0)}$ agrees with its own zero slice, and so $s^0(\KGL_X^{(j)})=0$ for all $j\neq 0$. So equivalence (2) follows by taking zero slices of the Adams decomposition of Theorem \ref{thm_riou}. Moreover, taking effective covers yields a natural decomposition $(\kgl_X)_\bb Q\simeq\bigoplus_{j\ge0}\KGL_X^{(j)}$; so $(\kgl_X)_\bb Q$ is naturally a direct summand of $(\KGL_X)_\bb Q$ and part (4) now follows from the fact that $X\mapsto(\KGL_X)_\bb Q$ is an absolute motivic spectrum.

(3): This follows from (2) since we have already noted that $\KGL^{(0)}$ is an absolute motivic spectrum.

(5): Let $F_X\in\SH(X)$ be the fibre of the unit map $1_X \rightarrow \kgl_X$; we must show that $F_X$ is rationally $1$-effective. But note that $X\mapsto (F_X)_\bb Q$ is an absolute motivic spectrum (since this is true for the unit and we have shown in (4) that it holds rationally for $\kgl$) so it is enough to check that $F_k$ is rationally $1$-effective whenever $k$ is a perfect field. But this is even known integrally, thanks to Levine's theorem that $s^0(1_{\SH(k)}) \quis s^0(\kgl_k)$ which we already used in the proof of (1).

(6): We use part (5) to replace $s^0(1)$ by $s^0(\KGL)$ and then apply part (3).
\end{proof}

\begin{definition}
The absolute motivic spectrum $X\mapsto \KGL_X^{(0)}$ (by Corollary \ref{cor:rational-smd}) is known as the {\em Beilinson motivic spectrum}, and often denoted by $H\bb Q$.
\end{definition}

\begin{remark}[Properties of $H\bb Q$]\label{rem_CD}
If we point the Beilinson motivic spectrum via the composition $\eta:1_X\to\KGL_X\to H\bb Q_X$, then it is idempotent (i.e., $\opp{id}\otimes\eta:H\bb Q_X\to H\bb Q_X\otimes H\bb Q_X$ is an equivalence) and so $H\bb Q$ admits a unique $\bb E_\infty$-algebra structure  such that the map $\KGL_X\to H\bb Q_X$ is multiplicative. Under this structure the equivalence $s^0(1_X)_\bb Q\quis H\bb Q_X$ of Corollary \ref{cor:rational-smd} is multiplicative and so we prefer to view the Beilinson motivic spectrum as the $\bb E_\infty$-algebra $s^0(1_X)_\bb Q$. Furthermore, thanks to the idempotency, a motivic spectrum can be a module over $H\bb Q_X$ in at most one way (that is, the forgetful map $\Mod_{H\bb Q_X}(\SH(X))\to \SH(X)$ is fully faithful). See Cisinski--Déglise \cite[\S14.1\&14.2]{CisinskiDeglise2019} for proofs of these facts, which we will only need for multiplicativity in Theorem \ref{thm:rational}(2).
\end{remark}

\begin{remark}[Uniqueness of the Adams decomposition]\label{rem_adams_unique}
Theorem \ref{thm_riou} only stated the existence of some natural decomposition of $(\KGL_X)_\bb Q$; it did not assert the uniqueness of the decomposition nor provide any definition or characterisation of it. However, Corollary \ref{cor:rational-smd} has shown that, given any decomposition satisfying the conditions of Theorem \ref{thm_riou}, then there are natural equivalences $s^0(1_X)_\bb Q\quis\KGL_X^{(0)}\quis s^0(\KGL_X)_\bb Q$ for any qcqs scheme $X$. The Adams decomposition is then necessarily given by \[(\KGL_X)_\bb Q\simeq \bigoplus_{j\in\bb Z}s^0(1_X)_\bb Q\otimes\bb T^{\otimes j}_X,\] where $s^0(1_X)_\bb Q\to\KGL_\bb Q$ is the unique map of $\bb E_\infty$-algebras afforded by the previous remark, and the maps $\bb T^{\otimes j}_X\to (\KGL_X)_\bb Q$ are given by powers of the Bott element for $\KGL_X$.
\end{remark}

We are now prepared to prove the main theorem of the subsection about rational motivic cohomology:

\begin{proof}[Proof of Theorem~\ref{thm:rational}]
(1): Using the presentation of Remark \ref{rem_coh_over_base}, the desired equivalence $\bb Q(j)^{\bb A, \cdh} \quis \bb Q(j)^{\bb A}$ follows from Corollary \ref{cor:rational-smd}(3). Arguing inductively down the filtration, the desired equivalence $\mathrm{Fil}_{\bb A,\sub{cdh}}^{j}\KH_{\bb Q}\quis \mathrm{Fil}_{\bb A}^{j}\KH_\bb Q$ similarly follows from Corollary \ref{cor:rational-smd}(4).

(2): To prove part (2) without worrying about multiplicative structures we just observe that, in the Adams decomposition $(\KGL_X)_\bb Q \simeq\bigoplus_{j\in\bb Z}\KGL_X^{(j)}$, the slice filtration $\Fil^\star_\sub{slice}$ on the left hand side corresponds to the filtration $\bigoplus_{j\ge \star}$ on the right hand side. Indeed, from Theorem \ref{thm_riou} and Corollary \ref{cor:rational-smd}(1) we know that each $\KGL_X^{(j)}$ is a $j^\sub{th}$ slice. This naturally splits the slice filtration on $(\KGL_X)_\bb Q$, and the desired natural splitting of the filtered spectrum $\Fil^\star_\bb A\KH(X)_\bb Q$ follows by applying $\omega^\infty$.

To ensure multiplicativity of the splitting we use the results of Cisinski--Déglise from Remark \ref{rem_CD}: namely, the inclusion of the summand $H\bb Q_X=\KGL_X^{(0)}\to\KGL_X$ is one of $\bb E_\infty$-algebras. Then the Bott element defines a map of $\bb E_\infty$-$H\bb Q_X$-algebras $\bigoplus_{j\in\bb Z}H\bb Q_X\otimes\bb T^{\otimes j}_X\to (\KGL_X)_\bb Q$, where the left side is the free $\mathbb{E}_{\infty}$-$H\bb Q_X$-algebra on the invertible $H\bb Q_X$-module $H\bb Q_X\otimes\bb T$;\footnote{To know that the underlying $H\bb Q_X$-module of the free algebra is as claimed we need only know that the switch map on $\H \bb Q_X \otimes \bb T^{\otimes 2}$ is homotopic to the identity. Indeed, this follows from the identification of $H \bb Q_X$ with the $+$-part of the rational motivic sphere spectrum (defined to be the summand where the switch map on $\bb T^{\otimes 2}$ acts by the identity; see \cite[16.2.1]{CisinskiDeglise2019}) as proved in \cite[Theorem 16.2.13]{CisinskiDeglise2019}. We also refer the reader to the proof of \cite[Theorem 3.4.8]{arndt2017abstract}.} this map is an equivalence of $\mathbb{E}_{\infty}$-$H\bb Q_X$-algebras thanks to the Adams decomposition (in the form written in Remark \ref{rem_adams_unique}). As in the previous paragraph this refines to an equivalence $\bigoplus_{j\ge\star}H\bb Q_X\otimes\bb T^{\otimes j}_X\quis \Fil^\star_\sub{slice}(\KGL_X)_\bb Q$ of $\mathbb{E}_{\infty}$-$H\bb Q_X$-algebras in filtered motivic spectra, thereby establishing the desired multiplicative splitting of the slice filtration on $(\KGL_X)_\bb Q$; again the multiplicative splitting of $\Fil^\star_\bb Q\KH(X)_\bb Q$ follows by applying $\omega^\infty$.

Finally, part (3) of the theorem follows immediately from part (2) (even without any multiplicative structure), and part (4) is a consequence of the compatibility of the constructed splitting with Soul\'e's Adams operators (Theorem \ref{thm_riou}(4) and Corollary \ref{corol_Adams}).
\end{proof}

We finish this subsection with some descriptions of rational motivic cohomology in low weights.

\begin{proposition}\label{prop:low-wts-ration}
For any regular Noetherian scheme $X$, the compositions
\begin{equation}R\Gamma_\sub{Zar}(X, \bb Z) \To R\Gamma_\sub{cdh}(X, \bb Z)\xto{\sub{(\ref{eq:wt0-cdh})}} \bb Z(0)^{\bb A,\sub{cdh}}(X)\label{eqn:low-wts0}\end{equation}
and
\begin{equation}R\Gamma_\sub{Zar}(X,\bb G_m)[-1]\To R\Gamma_\sub{cdh}(X,\bb G_m)[-1]\xto{c_1^{\bb A,\sub{cdh}}} \bb Z(1)^{\bb A,\sub{cdh}}\label{eqn:low-wts1}\end{equation}
are rational equivalences.
\end{proposition}
\begin{proof}
We recall some descriptions of low weight Adams eigenspaces, due to Kratzer \cite[Corollary~6.8]{Kratzer1980} and Soul\'e \cite[Corollary~1]{Soule1985}. For any regular Noetherian ring $R$, we have that $\K_n(R)^{(0)}_\bb Q=0$ for all $n>0$ and $\K_n(R)^{(1)}_\bb Q=0$ for all $n>1$; using Theorem \ref{thm:rational}(4) this means that $\bb Q(0)^\bb A(R)$ is supported in degree $0$ and $\bb Q(1)^\bb A(R)$ is supported in degrees $[1,2]$. In particular, on the Zariski site of any regular Noetherian scheme, both $\bb Q(0)^\bb A$ and $\bb Q(1)^\bb A$ are Postnikov complete (even without assuming that the scheme be of finite Krull dimension), and consequently it is enough to prove the desired equivalences on (Zariski) stalks. So henceforth let $R$ be a regular, Noetherian local ring.

We first treat the case of weight zero, where we consider the diagram:
\[\xymatrix{
\bb Q\ar[r]^-{\sub{(\ref{eqn:low-wts0})}} & H^0_{\bb A,\sub{cdh}}(R,\bb Q(0))\\
\K_0(R)_\bb Q\ar[u]^{\sub{rank}}\ar[ur]_{\sub{can}}&\\
\K_0(R)^{(0)}\ar[u]&
}\]
Note, since $R$ is local, both vertical arrows are isomorphisms: the rank isomorphism is well-known, while $\K_0(R)_\bb Q$ is entirely of Adams weight zero since it is additively generated by the class of the trivial line bundle $R$, on which each Adams operator acts as the identity. Furthermore, the composition $\K_0(R)^{(0)}_\bb Q\to H^0_{\bb A,\sub{cdh}}(R,\bb Q(0))$ through can is an isomorphism by Theorem \ref{thm:rational}(4). Therefore the horizontal arrow is an isomorphism, as desired (recall that we have already shown that $\bb Q(0)^{\bb A,\sub{cdh}}(R)$ is supported in degree $0$).

Next we treat the case of weight one, where we use an analogous diagram
\begin{equation}
\begin{tikzcd}
\K_1(R)_\bb Q^{(1)}\ar{d}&\\
\K_1(R)_\bb Q \ar{d}{\mathrm{det}} \ar{r}{\mathrm{can}} & \pi_1(\Fil^1_{\bb A,\sub{cdh}}\KH(R))_\bb Q \ar{d}\\
(R^\times)_\bb Q \ar[swap]{r}{\sub{(\ref{eqn:low-wts1})}} & H^1_{\bb A,\sub{cdh}}(R,\bb Q(1))
\end{tikzcd}
\end{equation}
Here the bottom commutative square is easily obtained by taking $\pi_1$ in (\ref{eq:c1-det}). Again, since $R$ is local, both vertical arrows on the left are isomorphisms: the det isomorphism is well-known, while $\K_1(R)_\bb Q$ is entirely of Adams weight one by the previous citations to Kratzer and Soul\'e. The composition $\K_1(R)^{(1)}_\bb Q\to H^1_{\bb A,\sub{cdh}}(R,\bb Q(1)))$ through can is an isomorphism; again, this is precisely what Theorem \ref{thm:rational}(4) means in this case. It follows that the bottom horizontal arrow is an equivalence; but this is the desired equivalence, since we saw in the previous paragraph that $\K_0(R)_\bb Q$ has no Adams summand of weight one, and so $\bb Q(1)^{\bb A,\sub{cdh}}(R)$ is supported in degree $1$.
\end{proof}

\begin{corollary}\label{corol:low-wts-ration}
\begin{enumerate}
\item For any qcqs scheme $X$, the maps \[R\Gamma_\sub{cdh}(X, \bb Z) \xto{\sub{(\ref{eq:wt0-cdh})}} \bb Z(0)^{\sub{cdh},\bb A}(X)\qquad \mathrm{and}\qquad c_1^{\sub{cdh},\bb A}:R\Gamma_\sub{cdh}(X,\bb G_m)[-1]\xto{(\ref{eq:c1-det})} \bb Z(1)^{\sub{cdh},\bb A}(X)\]
are rational equivalences.
\item For any regular Noetherian scheme $X$, the canonical change-of-topology maps 
\[R\Gamma_\sub{Nis}(X, \bb Z) \To R\Gamma_\sub{cdh}(X, \bb Z)\qquad\mathrm{and}\qquad R\Gamma_\sub{Nis}(X,\bb G_m)\To R\Gamma_\sub{cdh}(X,\bb G_m)\]
are rational equivalences.
\end{enumerate}
\end{corollary}
\begin{proof}
For any $j\ge0$, the presheaf $\bb Q(j)^{\sub{cdh},\bb A}$ is the cdh sheafification of the left Kan extension of its restriction to smooth $\bb Z$-schemes; indeed by Theorem \ref{thm:rational}(2) it is a direct summand of $\KH_\bb Q$, which has this property (see the proof of Theorem \ref{prop:basic_props_of_cdh_mot}(1)). Part (1) thus easily follows from left Kan extending and cdh sheafifying Proposition \ref{prop:low-wts-ration}. Part (2) follows by combining that proposition with part (1).
\end{proof}

We point out one last corollary which concerns a  rational ``Gersten vanishing'' bound which we can deduce from Soul\'e's work.

\begin{corollary}[Soul\'e]\label{cor:rational-gersten}
Let $R$ be a regular Noetherian local ring, or more generally a filtered colimit of such rings. Then, in the range $i>j$, the groups $H^i_\bb A(R,\bb Z(j))$ vanish rationally.
\end{corollary}

\begin{proof} If $R$ is a regular Noetherian ring, then Theorem~\ref{thm:rational}(4) gives an identification between motivic cohomology groups and Adams eigenspaces: $H^i_\bb A(R,\bb Q(j)) \cong \K_{2j-i}(R)_{\bb Q}^{(j)}$. If $R$ is furthermore local, then the vanishing holds in the $i > j$ range for Adams eigenspaces by \cite[Theorem 1]{Soule1985}, noting that the stable rank of a local ring is at most $1$. Since $\bb Q(j)^{\bb A}$ is finitary (by combining Theorems \ref{theorem:SH_mot_coh}(2) and \ref{thm:rational}(1)), the vanishing range persists even after taking filtered colimits. 
\end{proof} 

\begin{remark}
We expect Proposition \ref{prop:low-wts-ration}, Corollary \ref{corol:low-wts-ration}(2), and Corollary \ref{cor:rational-gersten} to hold integrally, but we cannot prove them in general. For the case of smooth schemes over fields and over mixed characteristic Dedekind domains, see Corollaries \ref{corollary:BL_over_B} and \ref{corol:low-wts-integral}; the latter also shows that Corollary \ref{corol:low-wts-ration}(1) holds integrally.

In addition, we do not know any proof of Corollary \ref{corol:low-wts-ration}(2) which does not use $K$-theory.
\end{remark}

\section{Syntomic cohomology of schemes}\label{sec:syntomic}
Let $p$ be a prime number. Most of this section is devoted to a presentation of {\em $p$-adic syntomic cohomology} $\bb Z_p(j)^\sub{syn}(X)\in \D(\bb Z)$, defined for all qcqs schemes $X$ and integers $j\in \bb Z$. Whenever $p$ is invertible in $X$ there will be a natural identification \[\bb Z_p(j)^\sub{syn}(X) \simeq R\Gamma_\sub{\'et}(X,\bb Z_p(j))\] (where the right hand side is simply defined to be $\lim_rR\Gamma_\sub{\'et}(X,\mu_{p^r}^{\otimes j})$, or alternatively in terms of pro-\'etale cohomology) and so $\bb Z_p(j)^\sub{syn}$ should be seen as a well-behaved extension of $p$-adic \'etale cohomology to general schemes where $p$ is no longer necessarily invertible. In particular, the open inclusion of the $p$-adic generic fibre $X[\tfrac1p]\to X$ induces a natural {\em syntomic-to-\'etale comparison map} \begin{equation}\bb Z_p(j)^\sub{syn}(X)\To R\Gamma_\sub{\'et}(X[\tfrac1p],\bb Z_p(j))=\bb Z_p(j)^\sub{syn}(X[\tfrac1p]),\label{eqn_syn_to_et}\end{equation} which will play a key role in our arguments. We will write $H^*_\sub{syn}(X,\bb Z_p(j)):=H^*(\bb Z_p(j)^\sub{syn}(R))$ for the associated syntomic cohomology groups, and similarly with finite coefficients.

From a motivic point of view, the syntomic cohomology $\bb Z_p(j)^\sub{syn}(X)$ provides a candidate for the $p$-adic \'etale motivic cohomology of $X$, and so the truncations $\tau^{\le j}(\bb Z_p(j)^\sub{syn}/p^r)$ are locally candidates for motivic cohomology with mod-$p^r$ coefficients. The foundations of this ``Beilinson--Lichtenbaum cohomology'' are developed in \S\ref{ss_BL}--\ref{ss_BL2} without appealing to any prior work concerning motivic cohomology of schemes in mixed characteristic. To simplify the notation when working with mod-$p$ coefficients, we will systematically adopt the notation $\bb F_p(j)^{\syn}=\bb Z_p(j)^{\syn}/p$ (we will sometimes use analogous notation for other cohomology theories, e.g., $\bb F_p(j)^\bb A:=\bb Z(j)^\bb A/p$).

As a brief historical comment, the complexes $\bb Z_p(j)^\sub{syn}(R):=\bb Z_p(j)^\sub{syn}(\Spec (R))$ were introduced in \cite{BhattMorrowScholze2} in the case of $p$-completely quasisyntomic rings $R$. They were subsequently extended to all derived $p$-complete animated rings $R$ in \cite{AntieauMathewMorrowNikolaus2022}. By gluing contributions from \'etale cohomology after inverting $p$, the syntomic cohomology was extended to general rings (even animated rings) and to schemes by Bhatt--Lurie \cite{BhattLurie2022}. In carrying out this extension, and in the intermediate study of the first Chern class, the paper [op.~cit.] uses their development of the prismatic logarithm and absolute prismatic cohomology to bypass some results of \cite{BhattMorrowScholze2} \cite{AntieauMathewMorrowNikolaus2022} depending on $K$-theory and topological cyclic homology. In \S\ref{ss_syn}--\ref{ss:weight_one_syn} we present an overview of the construction of syntomic cohomology of general rings and schemes, freely using $K$-theory and topological cyclic homology but which is otherwise close to that of \cite{BhattLurie2022}. The reader familiar with syntomic cohomology should probably jump forward to \S\ref{ss_BL} and refer back to earlier results if necessary.

The main properties of syntomic cohomology which we will need are contained in the following theorem,\footnote{Observe that, perhaps up to higher coherence issues which are implicitly resolved by arguments in Section \ref{section_motivic_DD}, the theorem uniquely determines syntomic cohomology: indeed, parts (7) and (8) determine the theory on smooth $\bb Z$-algebras (as in Corollary \ref{corol_BM_corol}), then part (2) extends it to all rings, and finally part (1) extends it by Zariski descent to all qcqs schemes.} summarising results of various collaborations between Antieau, Bhatt, L\"uders, Lurie, Mathew, the third author, Nikolaus, and Scholze.

\begin{theorem}[Existence of $p$-adic syntomic cohomology]\label{theorem_syntomic_properties}
There exists an $\mathbb{E}_{\infty}$-algebra in graded presheaves of complexes
\[
\bb Z_p(\star)^\sub{syn}:\opp{Sch}^\sub{qcqs,op}\to \Gr\D(\bb Z),
\] with the following properties:
\begin{enumerate}
\item $\bb Z_p(j)^\sub{syn}$ satisfies fpqc descent and takes values in derived $p$-complete complexes.
\item For any $r\ge1$, the functor $\bb Z_p(j)^\sub{syn}/p^r:\opp{CAlg}\to \D(\bb Z)$ is left Kan extended from smooth $\bb Z$-algebras.\footnote{In case of confusion, we recall that CAlg means simply the category of commutative rings.}
\item On the category of qcqs $\bb Z[\tfrac1p]$-schemes, there is a multiplicative family of equivalences \[\bb Z_p(j)^\sub{syn}\simeq R\Gamma_\sub{\'et}(-,\bb Z_p(j)).\]
\item On the category of smooth algebras over any field of characteristic $p$, there is a multiplicative family of equivalences \[
\Z_p(j)^{\sub{syn}}/p^r \simeq R\Gamma_{\et}(-, W_r\Omega^{j}_\sub{log})[-j].
\] for any $r\ge0$.
\item Low weights and first Chern class: there are equivalences of presheaves \[R\Gamma_\sub{\'et}(-,\bb Z_p)\quis\bb Z_p(0)^\sub{syn}\qquad{and} \qquad c_1^\sub{syn}:R\Gamma_\sub{\'et}(-,\bb G_m)[-1]_p^\comp\quis \bb Z_p(1)^\sub{syn}\] where the restriction of $c_1^\sub{syn}$ to $R\Gamma_\sub{\'et}(-,\bb G_m)[-1]$ is called the first syntomic Chern class.
\item Projective bundle formula: For any qcqs scheme $X$ and rank $d$ projective bundle $P\to X$ the map
\[\bigoplus_{i=0}^d\bb Z_p(j-i)^\sub{syn}(X)[-2i]\To \bb Z_p(j)^\sub{syn}(P)\]
induced by powers of the first syntomic Chern class $c_1^\sub{syn}(\roi(1))\in H^2_\sub{syn}(P,\bb Z_p(1))$ and by multiplicativity is an equivalence.
\item High degrees: For any henselian local ring $A$, the first Chern class map and multiplicativity induces isomorphisms \[\hat K_j^M(A)/p^r\isoto H^j_\sub{syn}(R,\bb Z/p^r(j))\qquad j\ge0,\] where $\hat K_j^M(A)$ denotes Gabber--Kerz' improved Milnor $K$-groups of $A$. Moreover, if $A$ is in fact strictly henselian then $\bb Z_p(j)^\sub{syn}(A)$ is supported in degrees $\le j$.
\item For any regular Noetherian scheme $X$ which is flat over $\bb Z$, and $r\ge1$, $j\in\bb Z$, the cofibre of the natural map \[\bb Z_p(j)^\sub{syn}(X)/p^r\To \bb Z_p(j)^\sub{syn}(X[\tfrac1p])/p^r\stackrel{\sub{part (3)}}=R\Gamma_\sub{\'et}(X[\tfrac1p],\mu_{p^r}^{\otimes j})\] is supported is cohomological degrees $\ge j$.
\end{enumerate}
\end{theorem}
\begin{proof}
(1) \& (2): Syntomic cohomology will be defined as a pullback of $p$-complete complexes in Definition \ref{def:bms-decomplete}, so is $p$-complete. We will explain in Proposition \ref{prop:syn} that it satisfies fpqc descent, and that it is left Kan extended from smooth algebras modulo any power of $p$.

Part (3) is Lemma \ref{lemma_syn_p-hens}, and part (4) is Example \ref{example_syn_in_char_p}.

(5): We will describe the weight-$0$ syntomic cohomology in Remark \ref{rem_syn_weight_0}, while the weight-$1$ theory and first Chern class are the subject of \S\ref{ss:weight_one_syn}.

A proof of part (6) can be found in \cite[Theorem 9.9.1]{BhattLurie2022}, by reduction to a projective bundle formula for powers of the cotangent complex. This reduction passes through diffracted Hodge cohomology, but we observe that one could instead carry it out more $\TC$-theoretically using Remark \ref{rem_syn_via_cotangent}.

Part (7) will be discussed in Remark \ref{rmk_syntomic_props} and Theorem \ref{thm_syntomic_Milnor}. Part (8) is part of a theorem of Bhatt--Mathew \cite[Theorem~1.8]{BhattMathew2023}.
\end{proof}

We now record here a convenient uniform description of syntomic cohomology on regular Noetherian scheme, as a gluing of \'etale cohomology and Milnor $K$-theory, obtained by combining parts (7) and (8) of the theorem:

\begin{corollary}\label{corol_BM_corol}
For any local regular Noetherian\footnote{More generally, ``regular Noetherian'' could be replaced by F-smooth in the sense of \cite{BhattMathew2023}.} ring $A$ which is strictly Henselian, there are natural pullback squares
\[\xymatrix{
\bb Z_p(j)^\sub{syn}(R)/p^r\ar[r]\ar[d] & \tau^{\le j}R\Gamma_\sub{\'et}(R[\tfrac1p],\mu_{p^r}^{\otimes j})\ar[d]\\
\hat K_j^M(R)/p^r[-j]\ar[r] & H^j_\sub{\'et}(R[\tfrac1p],\mu_{p^r}^{\otimes j})[-j]
}\]
for all $r,j\ge0$.
\end{corollary}
\begin{proof}
Using Theorem \ref{theorem_syntomic_properties}(7), the left vertical arrow is defined to be the natural map from $\bb Z_p(j)^\sub{syn}(R)/p^r$ to its top degree cohomology $H^j_\sub{syn}(R,\bb Z_p(j)/p^r)[-j]$. So, looking at the fibres of the vertical arrows, the claim is that the natural map \begin{equation}\bb Z_p(j)^\sub{syn}(R)/p^r\To \bb Z_p(j)^\sub{syn}(R[\tfrac1p])/p^r=R\Gamma_\sub{\'et}(R[\tfrac1p],\mu_{p^r}^{\otimes j})\label{eqn_BM_corol}\end{equation} is an isomorphism in degrees $<j$.

If $p$ is invertible in $R$ then (\ref{eqn_BM_corol}) is even an equivalence. If $R$ has characteristic $p$ then, by a filtered colimit argument using N\'eron--Popescu, we may assume that $R$ is smooth over $\bb F_p$; the right hand side of (\ref{eqn_BM_corol}) vanishes, and the left hand side also vanishes in degrees $<j$ by Theorem \ref{theorem_syntomic_properties}(4). In the final case that $R$ is $p$-torsion-free but $p$ is not necessary invertible, we apply Bhatt--Mathew's result recorded as Theorem \ref{theorem_syntomic_properties}(8).
\end{proof}

\subsection{Syntomic cohomology of $p$-complete rings}\label{ss_syn}
We begin with a discussion of syntomic cohomology $\bb Z_p(j)^\sub{syn}(R)$ in the generality introduced in \cite{BhattMorrowScholze2}, namely for $p$-completely quasisyntomic rings $R$.

Let us recall the site of $p$-completely quasisyntomic rings in the sense of \cite[\S4]{BhattMorrowScholze2} (just called quasisyntomic in \cite[Definition~4.9(1)]{BhattMorrowScholze2}). A ring $R$ is called {\em $p$-completely quasisyntomic} if it is $p$-complete, has bounded $p^{\infty}$-torsion, and the cotangent complex $\bb L_{A/\bb Z}$ has $p$-complete Tor amplitude in $[-1,0]$ (i.e., $\bb L_{A/\bb Z}\otimes_AN$ is supported in degrees $[-1,0]$ for all $A/pA$-modules $N$); the category of such rings (and all maps between them) is denoted by $\text{QSyn}$. We turn $\text{QSyn}^\sub{op}$ into a site by declaring a map $A\to B$ in $\text{QSyn}$ to be cover if and only if it is $p$-completely faithfully flat (i.e., $B\otimes^L_AA/pA$ is discrete and equal to a faithfully flat $A/pA$-module) and $\bb L_{B/A}$ has $p$-complete Tor amplitude in $[-1,0]$. The site $\text{QSyn}^\sub{op}$ admits base changes along covers (but not arbitrary fibre products, which is a source of technical difficulties).

A $p$-completely quasisyntomic ring is said to be \emph{quasiregular semiperfectoid} if it is the quotient of a perfectoid ring. The collection of quasiregular semiperfectoid rings forms a basis for the site $\text{QSyn}^\sub{op}$, whence the following definition can be shown to make sense:

\begin{definition}[{\cite[\S7.4]{BhattMorrowScholze2}}]\label{def:syn_coh_p_complete}
For $j\in\bb Z$, let $\bb Z_p(j)^\sub{syn}:\text{QSyn}^\sub{op}\to D(\bb Z)$ be the sheaf which is equal to $(\tau^{[-2j,-2j+1]}\TC(-;\bb Z_p))[-2j]$ on quasiregular semiperfectoids.

Note that this vanishes if $j<0$, since $\TC(-;\bb Z_p)$ of any ring is supported in cohomological degrees $\le 1$.
\end{definition}

\begin{remark}[Syntomic cohomology as quasisyntomic cohomology]\label{remark:odd_vanishing}
The trace map from $K$-theory to topological cyclic homology induces a natural morphism \[R\Gamma_{\text{QSyn}^\sub{op}}(R,\K_{2j}(-;\bb Z_p))\To \bb Z_p(j)^\sub{syn}(R)\] for any $R\in \text{QSyn}$, where the left hand side denotes cohomology on the site $\text{QSyn}^\sub{op}$ with coefficients in $\K_{2j}(-;\bb Z_p)$ (or more precisely in its sheafification). This map is in fact an equivalence for any $j\ge0$ by \cite[Theorem~C]{ClausenMathewMorrow2021} and the odd vanishing conjecture \cite[Theorem~14.1]{BhattScholze2022} (the cases $j=0,1$, which are sufficient for us, may be found in \cite[Propositions~7.16 \& 7.17]{BhattMorrowScholze2}). In short, $\bb Z_p(j)^\sub{syn}$ is calculated as the quasisyntomic cohomology of a presheaf of $p$-complete $K$-groups.
\end{remark}

\begin{remark}[Prismatic point of view]
We recall the prismatic point of view on $\bb Z_p(j)$; see \cite[\S5.1]{AntieauMathewMorrowNikolaus2022} for a similar summary. In terms of the Nygaard completed prismatic cohomology $\hat\Prism_R$, its Breuil--Kisin twists $\hat\Prism_R\{n\}$ for $n\ge0$, and their Nygaard filtration $\cal N^{\ge j}\hat\Prism_R\{n\}$, the syntomic cohomology of any $R\in\text{QSyn}$ is the fibre
\begin{equation}\bb Z_p(j)^\sub{syn}(R)=\opp{fib}\left(\cal N^{\ge j}\hat\Prism_R\{j\}\stackrel{\phi_j-1}\To \hat\Prism_R\{j\}\right)\label{eqn_div_frob}\end{equation} where $\phi_j:\cal N^{\ge j}\hat\Prism_R\{j\}\to\hat\Prism_R\{j\}$ is the divided Frobenius. The formula (\ref{eqn_div_frob}) remains valid if we work instead with non-Nygaard completed prismatic cohomology \cite[\S 7.4]{BhattLurie2022}.
\end{remark}

\begin{example}[Smooth algebras in characteristic $p$]\label{example_syn_in_char_p}
If $R$ is smooth over a field $k$ of characteristic $p$, then the complexes $\Z_p(j)^\sub{syn}(R)/p$ naturally compute the \'etale cohomology of the logarithmic de Rham sheaf $\Omega^j_\sub{log}$ on the \'etale site of $\Spec(R)$ \cite[Corollary~8.21]{BhattMorrowScholze2} :
\[
\bb F_p(j)^{\sub{syn}}(R) \simeq R\Gamma_{\et}(\Spec (R), \Omega^{j}_\sub{log})[-j].
\]
This equivalence holds more generally on the class of Cartier smooth $\bb F_p$-algebras, such as smooth algebras over valuation rings of characteristic $p$; see \cite{KellyMorrow2021} \cite{LuedersMorrow2023}.

Here, for any qcqs $\bb F_p$-scheme $X$, we denote by $\Omega^j_\sub{log}$ the subsheaf of $\Omega^j_{X/\bb F_p}$ generated Zariski (or, equivalently, Nisnevich or \'etale, by \cite[Corollary 4.2]{Morrow_pro_GL2}) locally by elements of the form $\tfrac{df_1}{f_1}\wedge\cdots\wedge \tfrac{df_j}{f_j}$ where $f_1,\dots,f_j\in\bb G_{m,X}$.
\end{example}

\begin{example}[Perfectoids]\label{example_syn_of_perfectoid}
If $R$ is a perfectoid ring, then there are natural multiplicative equivalences \[\bb Z_p(j)^\sub{syn}(R)\quis\begin{cases} R\Gamma_\sub{\'et}(\Spec(R), \bb Z_p) & j=0, \\ R\Gamma_\sub{\'et}(\Spec (R[\tfrac1p]),\bb Z_p(j)) & j>0,\end{cases}\] by \cite[Theorem~9.4]{BhattScholze2022}. The origin of these equivalences, in terms of symbols and the trace map from $K$-theory, will be discussed in \S\ref{ss:weight_one_syn}.
\end{example}

Having defined syntomic cohomology on the site of $p$-completely quasisyntomic rings, the next step is to extend $\bb Z_p(j)^\sub{syn}$ to all derived $p$-complete animated rings. The key is \cite[Theorem~5.1(2)]{AntieauMathewMorrowNikolaus2022}, stating that $\bb Z_p(j)^\sub{syn}:\text{QSyn}\to D(\bb Z)$ coincides with the $p$-completion of the left Kan of its restriction to $\text{CAlg}_{p,\Sigma}^\comp$, where $\text{CAlg}_{p,\Sigma}^\comp\subset \text{QSyn}$ is the category of $p$-completions of finitely generated polynomial $\bb Z$-algebras. Therefore, as in \cite[Construction 5.33]{AntieauMathewMorrowNikolaus2022}, we may extend $\bb Z_p(j)^\sub{syn}$ from $\text{QSyn}$ to all derived $p$-complete animated rings by $p$-completed left Kan extension from $\text{CAlg}_{p,\Sigma}^\comp$.

\begin{remark}[Basic properties of syntomic cohomology]\label{rmk_syntomic_props}
Here we collect two basic properties of syntomic cohomology of derived $p$-complete animated rings $R$ which will be used.
\begin{enumerate}
\item Firstly, note that $\bb Z_p(\star)^\sub{syn}(R)$ is naturally an $\bb E_\infty$-algebra in $\bb N$-graded complexes.
\item Secondly, $\bb Z_p(j)^\sub{syn}(R)$ is supported in degrees $\le j+1$ \cite[Theorem~5.1 \& Construction~5.33]{AntieauMathewMorrowNikolaus2022} and there is a natural isomorphism $H^{j+1}_\sub{syn}(R,\bb F_p(j))\cong \tilde\nu(j)(\pi_0(R)/p)$ \cite[Corollary~5.43]{AntieauMathewMorrowNikolaus2022} where, for any $\bb F_p$-algebra $A$, we write \[\tilde\nu(j)(A):=\text{coker}(C^{-1}-1:\Omega^j_{A}\to \Omega^j_{A}/d\Omega^j_{A});\] this {\em weight-$j$ Artin--Schreier obstruction} vanishes if $A$ is a strictly henselian $\bb F_p$-algebra. In fact, more generally, for any $\bb F_p$-algebra $A$ there are natural isomorphisms $\tilde\nu(j)(A)\cong H^1_\sub{\'et}(A,\Omega^j_\sub{log})$, obtained by sheafifying the previous line (so that the kernel of $C^{-1}-1$ becomes $\Omega^j_\sub{log}$) then taking cohomology.
\end{enumerate}
\end{remark}

The remaining goal of this subsection is to define, for any derived $p$-complete ring $R$ (here $R$ is discrete, not animated, just to avoid defining \'etale cohomology of animated rings), natural comparison maps
\[\gamma_\sub{syn}^\sub{\'et}\{j\}:\bb Z_p(j)^\sub{syn}(R)\To R\Gamma_\sub{\'et}(R[\tfrac1p],\bb Z_p(j))\] which will be special instances of (\ref{eqn_syn_to_et}). This is done by extending the implicit comparison map of Example~\ref{example_syn_of_perfectoid} from perfectoids to derived $p$-complete rings $R$. There are several ways to do this; we adopt the following right Kan extension technique, which will also be used numerous times later (sometimes implicitly):

\begin{proposition}\cite[Proposition A.6]{Aoki2023} \label{prop:RKE_from_basis}
Let $\cal C$ be a Grothendieck site, $\cal B\subset\cal C$ a basis, $D$ a presentable $\infty$-category, and $F:\cal C\to D$ a hypersheaf. Then $F$ is right Kan extended from $\cal B$.
\end{proposition}

Let $\opp{CAlg}_p^\comp$ denote the category of discrete commutative rings which are derived $p$-complete, with subcategory $\opp{Perfd}$ of perfectoid rings.

\begin{corollary}\label{corollary:RKE_perfectoid}
For any $j\in \bb Z$, the functor $R\Gamma_\sub{\'et}(-[\tfrac1p],\bb Z_p(j)):\opp{CAlg}_p^\comp\to D(\bb Z)$ is right Kan extended from $\opp{Perfd}$.
\end{corollary}
\begin{proof}
We equip $\opp{CAlg}_p^\comp$ with the $p$-complete arc topology $\opp{arc}_p$ of \cite{BhattMathew2021} \cite[Definition~8.7]{BhattScholze2022}. That is, a cover for the pretopology is a map $A\to B$ (or rather its opposite) such that, given any map from $A$ to a $p$-adically complete valuation ring $V$ of rank $\le 1$, then there is an extension of $p$-adically complete valuation rings $V\subset W$ of rank $\le 1$ and a map $B\to W$ extending the map $A\to V$.

We claim the following:
\begin{enumerate}
\item $\opp{CAlg}_p^\comp$ is a site admitting fiber products;
\item the perfectoid rings form a basis;
\item the presheaf \[\opp{CAlg}_p^\comp\to D(\bb Z),\quad A\mapsto R\Gamma_\sub{\'et}(A[\tfrac1p],\bb Z_p(j))\] is a hypercomplete arc$_p$-sheaf for each $j,r\ge0$.
\end{enumerate}

For (1), given maps $A\to B$ and $A\to B'$ of derived $p$-complete rings, then $H_0(\hat{B\otimes_AB'})$ is easily checked to be the pushout in $\opp{CAlg}_p^\comp$, where the hat denoted derived $p$-adic completion. 

(2): The argument of \cite[Lemma~8.8]{BhattScholze2022}, which is written for classical $p$-complete rings, works verbatim for derived $p$-complete rings to show that $\opp{Perfd}$ forms a basis of the site $(\opp{CAlg}_p^\comp)^\sub{op}$.

Finally, for (3), \cite[Corollary~6.17]{BhattMathew2021} states that $R\mapsto R\Gamma_\sub{\'et}(R^\sub{cl}_p[\tfrac1p],\bb Z_p(j))$ is an arc$_p$ sheaf, where $R^\sub{cl}_p$ denotes the classical $p$-completion of $R$. But the canonical surjection $R\to R^\sub{cl}$ has nilpotent kernel by \cite[Corollary~7.3.6.2]{LurieSAG}, whence $R\Gamma_\sub{\'et}(R[\tfrac1p],\bb Z_p(j))\quis R\Gamma_\sub{\'et}(R^\sub{cl}_p[\tfrac1p],\bb Z_p(j))$ by nil-invariance of \'etale cohomology away from the characteristic. The sheaf is hypercomplete simply because it takes coconnective values.

This completes the proofs of the claims. The desired right Kan extension property now follows from Proposition \ref{prop:RKE_from_basis}.
\end{proof}

\begin{construction}\label{cons:syn_to_etale_complete}
Using Proposition \ref{prop:RKE_from_basis} to right Kan extend Example \ref{example_syn_of_perfectoid}, we obtain a natural comparison map \[\gamma^\sub{\'et}_\sub{syn}\{\star\}:\bb Z_p(\star)^\sub{syn}(R)\To R\Gamma_\sub{\'et}(R[\tfrac1p],\bb Z_p(\star))\] of graded complexes, for every derived $p$-complete ring $R$. Note that this is a morphism of $\bb E_\infty$-algebras in graded complexes since restriction to the basis of perfectoids is symmetric monoidal, and so its right adjoint is lax symmetric monoidal. We call this the \emph{syntomic-to-\'etale} comparison map.
\end{construction}

\subsection{Syntomic cohomology of schemes}\label{ss_syn_of_schemes}
We may now define the syntomic cohomology of any (not necessarily $p$-complete) ring:

\begin{definition}\label{def:bms-decomplete}
For a commutative ring $R$ and $j\in\bb Z$, we define $\Z_p(j)^{\mathrm{syn}}(R)$ to be the pullback
\begin{equation}\label{eq:syn}
\xymatrix{
\Z_p(j)^{\mathrm{syn}}(R) \ar@{-->}[rr]^{\gamma_\sub{syn}^\sub{\'et}\{j\}} \ar@{-->}[d] && R\Gamma_{\et}(R[\tfrac{1}{p}];\bb Z_p(j)) \ar[d]\\
\Z_p(j)^{\syn}(R^{\comp}_p)\ar[r] &\Z_p(j)^{\syn}(\pi_0(R^{\comp}_p)) \ar[r]_{\gamma_\sub{syn}^\sub{\'et}\{j\}} & R\Gamma_{\et}(\pi_0(R_p^\comp)[\tfrac{1}{p}];\bb Z_p(j)).
}
\end{equation}
where $R_p^\comp$ denotes the derived $p$-adic completion of $R$ (note that the animated ring $R_p^\comp$ is not necessarily discrete, but $\pi_0(R^{\comp}_p)$ is still derived $p$-complete by \cite[Theorem~7.3.4.1]{LurieSAG} and so its syntomic-to-\'etale comparison map was defined in Construction \ref{cons:syn_to_etale_complete}). Observe that $\bb Z_p(\star)^\sub{syn}(R)$ is naturally an $\bb E_\infty$-algebra in $\Gr\rm D(\bb Z)$.

The resulting top horizontal map in (\ref{eq:syn}) is by definition the syntomic-to-\'etale comparison map (\ref{eqn_syn_to_et}) for the arbitrary commutative rings $R$.
\end{definition}

\begin{remark}[Coconnectivity]\label{rem:coconnective}
Suppose that $R$ has bounded $p^\infty$-torsion and that the cotangent complex $L_{R/\bb Z}$ has $p$-complete Tor amplitude in $[-1,0]$ (i.e., $L_{R/\bb Z}\otimes_R^{\bb L}N$ is supported in cohomological degrees $[-1,0]$ for every $R/pR$-module $N$). Then $R_p^\comp$ is $p$-completely quasisyntomic and so $\bb Z_p(j)^\sub{syn}(R_p^\comp)$ is coconnective; it now follows from Definition \ref{def:bms-decomplete} that $\bb Z_p(j)^\sub{syn}(R)$ is also coconnective.
\end{remark}

The appearance of $\pi_0(R^{\comp}_p)$ in~\eqref{eq:syn} might seem a little strange. It would perhaps be more appealing to replace the bottom right of the diagram by the equivalent cohomology $R\Gamma_{\et}(R^{\comp}_p[\tfrac{1}{p}],\bb Z_p(j))$, but this would require introducing \'etale cohomology of animated rings (as is done in \cite{CesnaviciusScholze2024, BhattLurie2022}). Instead we justify that the previous definition is reasonable by the following lemma and proposition.

\begin{lemma}\label{lemma_syn_p-hens}
Let $R$ be a commutative ring and $j\in\bb Z$.
\begin{enumerate}
\item If $R$ is a $\bb Z[\tfrac1p]$-algebra then the map $\gamma_\sub{syn}^\sub{\'et}\{j\}:\bb Z_p(j)^\sub{syn}(R) \to R\Gamma_\sub{\'et}(R[\tfrac1p],\bb Z_p(j))$ is an equivalence.
\item If $R$ is a $p$-henselian ring then the natural map $\bb Z_p(j)^\sub{syn}(R)\to \Z_p(j)^{\syn}(R^{\comp}_p)$ is an equivalence.
\item For any commutative ring $R$, the square
\begin{equation}
\xymatrix@=1.5cm{
\Z_p(j)^{\mathrm{syn}}(R) \ar[r]^{\gamma_\sub{syn}^\sub{\'et}\{j\}} \ar[d] & R\Gamma_{\et}(R[\tfrac{1}{p}];\bb Z_p(j)) \ar[d]\\
\Z_p(j)^{\syn}(R^h_p) \ar[r]_{\gamma_\sub{syn}^\sub{\'et}\{j\}} & R\Gamma_{\et}(R_p^h[\tfrac{1}{p}];\bb Z_p(j)),
}\label{eqn_syn_via_p-hens}
\end{equation}
is cartesian; the bottom left term vanishes if $j<0$.
\end{enumerate}
\end{lemma}
\begin{proof}
(1) is obvious from the definition, since the derived $p$-completion of a $\bb Z[\tfrac1p]$-algebra vanishes.

(2): Now let $R$ be a $p$-henselian ring. Then the maps $R\to \pi_0(R_p^\comp)\to R_p^\sub{cl}$ (as in the proof of Proposition \ref{prop:RKE_from_basis}, the latter denotes the classical $p$-completion of $R$) induce equivalences on $R\Gamma_{\et}(-[\tfrac{1}{p}];\bb Z_p(j))$. For the second arrow, this is by nil-invariance argument already used in the proof of Corollary~\ref{corollary:RKE_perfectoid}. For the composition it is because of the Fujiwara--Gabber theorem \cite[Corollary 1.18(2)]{BhattMathew2021} \cite{Fujiwara1995}. Therefore the right vertical arrow of (\ref{eq:syn}) is an equivalence, and so we deduce the same for the left vertical arrow, as desired.

(3): This follows from the definition and the identifications established in the proof of part (2).
\end{proof}

One next extends $\bb Z_p(j)^\sub{syn}$ and the comparison map (\ref{eqn_syn_to_et}) beyond the affine case via Zariski, or even fpqc, descent using part (1) of the following result:

\begin{proposition} \label{prop:syn}
Let $j\in\bb Z$ and let $B$ be any base ring. Then the functor $\bb Z_p(j)^\sub{syn}:\opp{CAlg}_B\to \D(\bb Z)$ has the following properties:
\begin{enumerate}
\item it is an fpqc sheaf;
\item modulo any power of $p$, it is left Kan extended from smooth $B$-algebras (in particular, modulo any power of $p$ it commutes with filtered colimits, c.f., Proposition \ref{prop:finitary_conditions}(1)).
\end{enumerate}
\end{proposition}
\begin{proof}
It suffices to prove the claims modulo $p$, namely for $\bb F_p(j)^\sub{syn}$, which we do following \cite[Propositions~8.4.6 \& 8.4.10]{BhattLurie2022}. As in \cite[Remark~8.4.4]{BhattLurie2022}, there is a natural fibre sequence \begin{equation}R\Gamma_\sub{\'et}(R,j_!\mu_p^{\otimes j})\To \bb F_p(j)^\sub{syn}(R)\To \bb F_p(j)^\sub{syn}(R_p^\comp)\label{eqn_j!p}\end{equation} for any $B$-algebra $R$, where $j_!\mu_p$ is the extension by zero of the \'etale sheaf $\mu_p^{\otimes j}$ along the open inclusion $j:\Spec(R[\tfrac1p])\to\Spec(R)$. 

We first claim that properties (1) and (2) hold for the functor $R\mapsto R\Gamma_\sub{\'et}(R,j_!\mu_p^{\otimes j})$. It satisfies arc descent by \cite[Theorem~5.4]{BhattMathew2021} hence, in particular, satisfies fpqc descent. To see that it is left Kan extended from smooth $B$-algebras, it suffices by Lemma \ref{lemma_rigid_implies_lke} to check that it is is rigid. Let $R$ be a $B$-algebra and $I\subset R$ be a henselian ideal, and consider the cartesian square:
\[
\begin{tikzcd}
\Spec(R/I[\tfrac{1}{p}]) \ar{r}{j'} \ar{d}{i'} & \Spec(R/I) \ar{d}{i} \\ 
\Spec(R[\tfrac{1}{p}]) \ar{r}{j} & \Spec(R).
\end{tikzcd}
\]
Then we get equivalences
\[
R\Gamma_{\et}(R,j_!\mu_p^{\otimes j})  \quis  R\Gamma_{\et}(R/i, i^*j_!\mu_p^{\otimes j}) \quis  R\Gamma_{\et}(R/I,j'_!i^{'*}\mu_p^{\otimes j}) \quis R\Gamma_{\et}(R/I, j'_!\mu_p^{\otimes j})
\]
where the first equivalence is again Gabber's affine analog of proper base change \cite{Gabber1994a}, the second equivalence is base change for \'etale cohomology, and the last equivalence follows from the fact that $\mu_p^{\otimes j}$ is an \'etale group scheme on $\Z[\tfrac{1}{p}]$-algebras and thus is stable under base change.

It remains to establish that properties (1) and (2) also hold for the functor $R\mapsto \bb F_p(j)^\sub{syn}(R_p^\comp)$, where we may assume $j\ge0$ since otherwise the syntomic cohomology vanishes (Lemma \ref{lemma_syn_p-hens}(3)). As a functor on all rings, it is left Kan extended from finitely generated polynomial $\bb Z$-algebras by construction \cite[Theorem~5.1(2) \& Construction~5.33]{AntieauMathewMorrowNikolaus2022}, hence its restriction to $B$-algebras is left Kan extended from finitely generated polynomial $B$-algebras, so a fortiori from smooth $\bb Z$-algebras. To prove that it satisfies fpqc descent, we use Remark  \ref{rem_syn_via_cotangent} below to reduce to fpqc descent of the functors $R\mapsto( L^i_{R_p^\comp/\bb Z_p})_p^\comp\otimes^L_{\bb Z_p}\bb F_p$ for $i\ge0$. But the latter is equivalent to $L^i_{R/\bb Z}\otimes^L_{\bb Z}\bb F_p$, which indeed has fpqc descent as a functor of $R$ by \cite[Theorem~3.1]{BhattMorrowScholze2}.
\end{proof}

We record the following consequence of the proof of the previous proposition:

\begin{corollary}\label{corol_syn_rigid1}
For any ring $R$ and henselian ideal $I\subset R$, the canonical map on relative syntomic cohomology $\bb Z_p(j)^\sub{syn}(R,I)\to\bb Z_p(j)^\sub{syn}(R_p^h,IR_p^h)$\footnote{For a $D(\bb Z)$ or $\Spt$-valued functor on rings, we write $F(R, I) := \sub{fib}(F(R) \rightarrow F(R/I))$.} is an equivalence for any $j\in\bb Z$.
\end{corollary}
\begin{proof}
It is enough to prove the claim modulo $p$. In the proof of Proposition \ref{prop:syn} we showed that the leftmost term of fibre sequence (\ref{eqn_j!p}) is rigid, and in Lemma \ref{lemma_syn_p-hens}(2) we saw that the rightmost term is the same as $\bb F_p(j)^\sub{syn}(R_p^h)$. So the claim follows by applying the fibre sequence to both $R$ and $R/I$, and using the identity $(R/I)_p^h=R_p^h/IR_p^h$.
\end{proof}

\begin{remark}[Controlling $\bb Z_p(j)^\sub{syn}$ via the cotangent complex]\label{rem_syn_via_cotangent}
The proof of fpqc descent in Proposition \ref{prop:syn} reduced to the analogous property of the cotangent complex; we recall the technique from \cite[\S5]{AntieauMathewMorrowNikolaus2022}. In the approach of \cite{BhattLurie2022}, the following would be replaced by arguments with diffracted Hodge cohomology.

Let $\bb S\in\Spt$ denote the sphere spectrum; considering $\Z_p$ as a $\bb S[z]$-algebra via $z \mapsto p$ we consider the functor from animated rings to $p$-complete spectra $R \mapsto \THH(R/\bb S[z];\Z_p)$. For $R\in \text{QSyn}$, there exists a natural filtration on $\THH(R/\bb S[z];\Z_p)$ defined via extension from quasiregular semiperfectoids \cite[\S11]{BhattMorrowScholze2}, and we need the following two properties of the graded pieces $\mathrm{gr}^n\THH(R/\bb S[z];\Z_p)$ of this filtration \cite[Corollaries~5.19\&5.21]{AntieauMathewMorrowNikolaus2022}: firstly, $\mathrm{gr}^n\THH(R/\bb S[z];\Z_p)$ itself carries a finite, increasing filtration with graded pieces $(L^i_{R/\Z_p})^{\comp}_p[2i-n]$ for $i=0,\dots,n$; secondly, they fit into natural fibre sequences
\begin{equation}
\scr N^n\widehat{\Prism}_R\{ n\}[2n] \rightarrow \mathrm{gr}^n \THH(R/\bb S[z];\Z_p) \rightarrow \mathrm{gr}^{n-1}\THH(R[z];\Z_p)[2]
\label{eqn:rel-to-absolute}
\end{equation}
where $\cal N^j\hat\Prism_R\{n\}:=\opp{gr}_\cal N^j\hat\Prism_R\{n\}$ denotes $j^\sub{th}$ graded piece of the Nygaard filtration on $\hat\Prism_R\{n\}$. The Breuil--Kisin twists may in fact be trivialised on any graded piece of the Nygaard filtration \cite[Thm 5.9(1)]{AntieauMathewMorrowNikolaus2022}, so that the left term in the fibre sequence is naturally equivalent to $\scr N^n\widehat{\Prism}_R$, or equivalently $\scr N^n\widehat{\Prism}_R\{j\}$ for any other $j\ge0$.

Furthermore, (\ref{eqn_div_frob}) admits a ``finite level'' strengthening: for some $N\gg0$ (independent of $R\in\opp{QSyn}$) the mod-$p$ syntomic cohomology fits into a fibre sequence \begin{equation}\bb F_p(j)^\sub{syn}(R)\To \cal N^{\ge j}\hat\Prism_R\{j\}/\cal N^{\ge j+N}\hat\Prism_R\{j\}\otimes^L_{\bb Z}\bb F_p\stackrel{\phi_j-1}\To \hat\Prism_R\{j\}/\cal N^{\ge j+N}\hat\Prism_R\{j\}\otimes^L_{\bb Z}\bb F_p\label{eqn_div_frob2}.\end{equation}

By left Kan extension, the fibre sequence (\ref{eqn:rel-to-absolute}), the finite filtration on the middle and right terms by $p$-completed wedge powers of the cotangent complex, and the fibre sequence (\ref{eqn_div_frob2}) continue to hold for any derived $p$-complete animated ring $R$.

In conclusion, for any derived $p$-complete animated ring $R$, if the $p$-completed wedge powers of the cotangent complex $(L^i_{R/\Z_p})^{\comp}_p$ are ``understood'' for all $i\ge0$, then using the finite filtration we first get information about the graded pieces $\mathrm{gr}^n \THH(R/\bb S[z];\Z_p)$ for each $n\ge0$, then about the graded pieces $\scr N^n\widehat{\Prism}_R\{ n\}$, and then by induction about the quotients $\cal N^{\ge j}\hat\Prism_R\{j\}/\cal N^{\ge j+N}\hat\Prism_R\{j\}$ for all $j,N\ge0$, and finally from  (\ref{eqn_div_frob2}) we obtain information about $\bb F_p(j)^\sub{syn}(R)$. As well as the previous proof of fpqc descent, see Proposition \ref{prop:nygaard-prism-exc} for another typical application of this argument.
\end{remark}

Thanks to the fpqc descent of Proposition \ref{prop:syn} (with $B=\bb Z$), the syntomic cohomology of Definition~\ref{def:bms-decomplete} for rings uniquely extends to an fpqc sheaf \[\bb Z_p(j)^\sub{syn}:\text{Sch}^\sub{qcqs,op}\To \D(\bb Z)\] for each $j\in\bb Z$. This defines syntomic cohomology $\bb Z_p(\star)^\sub{syn}$ in our desired degree of generality, namely as an $\bb E_\infty$-algebra in graded presheaves of complexes on qcqs schemes.

From Lemma \ref{lemma_syn_p-hens}(1) we see that we have natural multiplicative equivalences of fpqc sheaves \begin{equation}\label{eq:syn-to-et}
\bb Z_p(j)^\sub{syn}\simeq R\Gamma_\sub{\'et}(-,\bb Z_p(j))\qquad j\in\bb Z
\end{equation} on the category of qcqs $\bb Z[\tfrac1p]$-schemes. For general qcqs schemes, this induces the natural, multiplicative syntomic-to-\'etale comparison map \begin{equation}\bb Z_p(j)^\sub{syn}(-)\To R\Gamma_\sub{\'et}(-[\tfrac1p],\bb Z_p(j))=\bb Z_p(j)^\sub{syn}(-[\tfrac1p])\qquad j\in\bb Z\label{eqn_syn_to_et_general}\end{equation} which is the counit $\bb Z_p(j)^\sub{syn}\To j_*j^*\bb Z_p(j)^\sub{syn}$ associated with the open inclusion $j:\text{Sch}^\sub{qcqs}_{\bb Z[1/p]}\to \text{Sch}^\sub{qcqs}$.

\begin{remark}\label{rem_comparison_map}
The comparison map (\ref{eqn_syn_to_et_general}) is fundamentally important to our arguments. By definition, in the affine case it is equal to the map $\gamma^\sub{\'et}_\sub{syn}\{j\}$ which was defined for derived $p$-complete rings in Construction \ref{cons:syn_to_etale_complete} and for general rings in Definition \ref{def:bms-decomplete}. In the case of a perfectoid ring $R$, it is by construction the equivalence $\bb Z_p(j)^\sub{syn}(R)\quis R\Gamma_\sub{\'et}(R[\tfrac1p],\bb Z_p(j))$ of Example \ref{example_syn_of_perfectoid} when $j>1$, and identifies with the canonical map $R\Gamma_\sub{\'et}(R,\bb Z_p(0))\to R\Gamma_\sub{\'et}(R[\tfrac1p],\bb Z_p(0))$ when $j=0$.

Henceforth we will often drop the notation $\gamma^\sub{\'et}_\sub{syn}\{j\}$ for the comparison map, viewing it as the natural map in syntomic cohomology induced by the open inclusion $X[\tfrac1p]\to X$.
\end{remark}

\begin{remark}[Weight zero]\label{rem_syn_weight_0}
For any qcqs scheme $X$ there is a natural map \begin{equation}R\Gamma_\sub{\'et}(X,\bb Z_p)\To\bb Z_p(0)^\sub{syn}(X)\label{eqn_weight0}\end{equation} induced by the fact that $\bb Z_p(0)^\sub{syn}$ is an \'etale sheaf of $\bb E_\infty$-$\bb Z_p$-algebras. We claim that this map is an equivalence.

We may check this modulo $p$ and, by descent, we may assume that $X=\Spec(R)$ is affine. Since both sides are then left Kan extended from smooth $\bb Z$-algebras (\'etale cohomology by rigidity; syntomic cohomology by Proposition \ref{prop:syn}), we may assume further that $R$ is a smooth $\bb Z$-algebra. Comparing (\ref{eq:syn}) to the pullback square at the start of the proof of Proposition \ref{prop:syn}, it is enough to show that (\ref{eqn_weight0}) is an equivalence mod $p$ for $X=\Spec(R_p^\comp)$. But $R_p^\comp$ is $p$-completely quasisyntomic, so by descent (note that $R\Gamma_\sub{\'et}(-,\bb Z/p\bb Z)$ satisfies descent on $\text{QSyn}$, as there it identifies with $R\Gamma_\sub{\'et}(-/p,\bb Z/p\bb Z)$ and \'etale cohomology satisfies fpqc descent) we then reduce further to the case of quasiregular semiperfectoids; for the desired equivalence in that case see the proof of \cite[Proposition~7.16]{BhattMorrowScholze2} (which we remark uses $K$-theory and topological cyclic homology).
\end{remark}

We finish this subsection by establishing some properties of the syntomic-to-\'etale comparison map
\[
\gamma_\sub{syn}^\sub{\'et}\{j\}: \Z_p(j)^{\mathrm{syn}} \To R\Gamma_{\et}(-[\tfrac{1}{p}],\bb Z_p(j))
\]
under $h$-sheafification and $\bb A^1$-localisation. We will need the latter for some arguments in \S\ref{sec:a1-comparison}.  The target of the comparison map, namely $R\Gamma_\sub{\'et}(-[\tfrac1p],\mu_{p^r}^{\otimes j})$, is an $\bb A^1$-invariant $h$-sheaf. There are therefore induced maps of presheaves of qcqs schemes
\begin{equation}
L_{\bb A^1}\bb Z_p(j)^\sub{syn}/p^r\To R\Gamma_\sub{\'et}(-[\tfrac1p],\mu_{p^r}^{\otimes j})
\label{eqn:syn_A1_comparison}
\end{equation}
and
\begin{equation}
\qquad L_{\h}\bb Z_p(j)^\sub{syn}/p^r\To R\Gamma_\sub{\'et}(-[\tfrac1p].\mu_{p^r}^{\otimes j})\label{eqn:syn_h_comparison}
\end{equation}
for each $j\in\bb Z$. We now prove that the first map is an equivalence, and describe $L_{\h}\bb Z_p(j)^\sub{syn}$:

\comment{
For future use we define \emph{$p$-adic $h$-motivic cohomology by}:
\begin{equation}\label{eq:p-adic-h}
\Z_p(j)^{\h} := L_{\h}\bb Z^{\syn}_p(j) \qquad j \in \bb Z,
\end{equation}
and denote its mod-$p$ variant by $\bb F_p(j)^h := L_{\h}\bb Z_p(j)/p \simeq L_{\h}\bb F(j)^{\syn}$. We will see an integral version of this theory in Definition~\ref{def:h-eh-mot}.
}

\begin{theorem}\label{thm:a1-comparison}
\begin{enumerate}
\item For any $j\in\bb Z$, the map of presheaves of qcqs schemes \eqref{eqn:syn_A1_comparison} is an equivalence.

\item For any $j\ge1$, the map of presheaves of qcqs schemes \eqref{eqn:syn_h_comparison} is an equivalence. Meanwhile, for any $j\le 0$, the presheaf $\bb Z_p(j)^\sub{syn}$ is already an $h$-sheaf (and the map \eqref{eqn:syn_h_comparison} is not necessarily an equivalence).
\end{enumerate}
\end{theorem}
\begin{proof}
For this proof, we may assume $r=1$.

(1) First we treat the $\bb A^1$-localisation of syntomic cohomology. By $\bb A^1$-localising the mod-$p$ version of (\ref{eqn_syn_via_p-hens}), it is enough to show that the $\bb A^1$-localisations of the two functors \[R\mapsto \bb F_p(j)^{\syn}(R_p^h),\quad R\Gamma_\sub{\'et}(R_p^h[\tfrac1p],\mu_p^{\otimes j}) \] are zero. Note that, for any $j \in \bb Z$ the complexes $\bb F_p(j)^{\syn}(R_p^h)$ and $R\Gamma_\sub{\'et}(R_p^h[\tfrac1p],\mu_p^{\otimes j})$ are modules over $\Z_p(0)^{\syn}(R^h_p)$ and $R\Gamma_\sub{\'et}(R_p^h,\bb Z/p)$ respectively. By the usual argument that any module over the zero ring is zero (see, for example, \cite{Elmanto2021}), and Remark \ref{rem_syn_weight_0}, it is enough to prove the claim that the $\bb A^1$-localisation of $R\mapsto R\Gamma_\sub{\'et}(R_p^h,\bb Z/p)$ is zero. This agrees with $R\Gamma_\sub{\'et}(R/p,\bb Z/p\bb Z)$ by Gabber's affine analogue of the proper base change theorem \cite{Gabber1994a}, so it is enough to establish this claim on the category of $\bb F_p$-algebras. This follows from the Artin--Schreier sequence and fact that the functor $R\mapsto R$ has vanishing $\bb A^1$-localisation; see the proof of \cite[Proposition A.3.1]{CisinskiDeglise2016} for details. 

(2) We next treat the $\h$-sheafification of syntomic cohomology.  When $j<0$ the third term of \eqref{eqn_j!p} vanishes and so $\bb F_p(j)^\sub{syn}\simeq R\Gamma_\sub{\'et}(-,j_!\mu_p^{\otimes j})$, which we already noted in the course of that proof is an $h$-sheaf (even an arc sheaf). Meanwhile, when $j=0$, we saw in Remark~\ref{rem_syn_weight_0} that $\bb F_p(0)^\sub{syn}\simeq R\Gamma_\sub{\'et}(-,\bb Z/p)$, which is again an arc sheaf. So now let $j > 0$.

We claim that $L_{\h}\bb F_p(j)^\sub{syn}:\text{Sch}^\sub{qcqs}\to D(\bb Z)$ is finitary and takes coconnective values. Writing $L_{\h}=L_{\h}L_\sub{cdh}$ and applying Proposition \ref{prop:finitary_conditions}(2), it is enough to establish the same properties about $L_\sub{cdh}\bb F_p(j)^\sub{syn}$. Firstly, $L_\sub{cdh}\bb F_p(j)^\sub{syn}$ is finitary by Proposition \ref{prop:syn} and another application of Proposition \ref{prop:finitary_conditions}(2). Next, for coconnectivity we appeal to Proposition \ref{proposition_checking_on_points}(1), which tell us it is enough to check coconnectivity of $\bb F_p(j)^\sub{syn}(V)$ for every henselian valuation ring $V$ (even of finite rank). But the cotangent complex $L_{V/\bb Z}$ has Tor amplitude in $[-1,0]$ by Gabber--Ramero,\footnote{More precisely, let $\ell$ be the residue characteristic of $V$ and consider two cases, depending on the characteristic of $F:=\opp{Frac} V$. If $\opp{char}F=0$ then $L_{V/\bb Z}=\Omega^1_{V/\bb Z_{(\ell)}}$ by \cite[Theorem~6.5.12(2)]{GabberRamero2003}. If $\opp{char}F=\ell$ then the transitivity sequence for $\bb Z\to\bb F_\ell\to V$ looks like $V[1]\to L_{V/\bb Z}\to L_{V/\bb F_\ell}=\Omega^1_{V/\bb F_\ell}[0]$, the final identification being another application of [loc.~cit.]. In both cases it now suffices to check that the $\Omega^1$ term has Tor dimension in $[-1,0]$. But {\em any} module over a valuation ring has such Tor dimension as it is a filtered colimit of finitely presented modules, each of which admits a decomposition into cyclic modules.} so the desired coconnectivity was explained in Remark \ref{rem:coconnective}.

Having proved that $L_{\h}\bb F_p(j)^\sub{syn}$ is finitary and takes coconnective values, we note that the same is true of the $h$-sheaf $R\Gamma_\sub{\'et}(-[\tfrac1p],\mu_p^{\otimes j})$. Therefore, by Proposition \ref{proposition_checking_on_points}, to show that $L_{\h}\bb F_p(j)^\sub{syn}\to R\Gamma_\sub{\'et}(-[\tfrac1p],\mu_p^{\otimes j})$ is an equivalence, it is enough to check on points, namely that the syntomic-to-\'etale comparison map $\bb F_p(j)^\sub{syn}(V)\to R\Gamma_\sub{\'et}(V[\tfrac1p],\mu_p^{\otimes j})$ is an equivalence for every absolutely integrally closed valuation ring $V$. If $p\in V^\times$ then this follows from Lemma \ref{lemma_syn_p-hens}; otherwise it suffices to prove the equivalence after replacing $V$ by $V_p^\comp$ (thanks to Definition \ref{def:bms-decomplete} itself), so now $V$ is an absolutely integrally closed, $p$-adically complete, valuation ring of residue characteristic $p$. Therefore $V$ is perfectoid and the desired equivalence is exactly Example \ref{example_syn_of_perfectoid}.
\end{proof}

\subsection{Weight one syntomic cohomology and the Milnor range}\label{ss:weight_one_syn} In this subsection we begin by describing weight one syntomic cohomology, complementing the discussion for weight zero syntomic cohomology in Remark~\ref{rem_syn_weight_0}. The goal is to prove an equivalence \begin{equation}c_1^\sub{syn}:R\Gamma_\sub{\'et}(-,\bb G_m)_p^\comp[-1]\quis\bb Z_p(1)^\sub{syn}(-).\label{comparison:j=1-intro}\end{equation} 
We proceed by establishing the equivalence first for the quasisyntomic case, then by bootstrapping to $p$-henselisation of the input ring and eventually to all rings by gluing in the $p$-invertible part; the proof is finished in Proposition~\ref{proposition_1st_Chern_class_syn}. We then use this to analyse the Milnor range (this is when the cohomological degree is equal to the weight) of syntomic cohomology as in Theorem~\ref{thm_syntomic_Milnor}.

By a slight abuse of notation we will also write $c_1^\sub{syn}$ for the composition \[R\Gamma_\sub{\'et}(-,\bb G_m)[-1]\To R\Gamma_\sub{\'et}(-,\bb G_m)_p^\comp[-1]\xto{c_1^\sub{syn}}\bb Z_p(1)^\sub{syn}(-),\] which will be called the {\em syntomic first Chern class}. We begin by analysing the left side of \eqref{comparison:j=1-intro} in the $p$-henselian case:

\begin{proposition}\label{prop:G_m_p_henselian}
\begin{enumerate}
\item The functor $\opp{CAlg}\to \rm D(\bb Z)$, $R \mapsto R\Gamma_\sub{\'et}(R_p^h,\bb G_m)/p$ is left Kan extended from smooth $\bb Z$-algebras.
\item For any $p$-henselian ring $R$ with bounded $p$-power torsion, the canonical map \[R\Gamma_\sub{\'et}(R,\bb G_m)_p^\comp\To \lim_sR\Gamma_\sub{\'et}(R/p^s,\bb G_m)_p^\comp\] is an equivalence.
\item The functor $\opp{QSyn}^\sub{op}\to\rm  D(\bb Z)$, $R \mapsto R\Gamma_\sub{\'et}(R,\bb G_m)_p^\comp$ is an fpqc sheaf (for the quasisyntomic topology defined at the beginning of \S\ref{ss_syn}).
\end{enumerate}
\end{proposition}
\begin{proof}
We begin by noting, for any $p$-henselian ring $R$, that $R\Gamma_\sub{\'et}(R,\bb G_m)/p$ admits a natural finite, decreasing, exhaustive filtration with graded pieces \[\opp{Br}(R)[p][-1],\qquad\opp{Pic}(R)[-1]\otimes_{\bb Z}^L\bb F_p,\qquad(R)^\times\otimes_{\bb Z}^L\bb F_p\] (where $[p]$ denotes $p$-torsion and $[-1]$ denotes a cohomological shift). Indeed the \'etale cohomology of $p$-torsion sheaves on $\Spec(R)$ vanishes in degrees $\ge 2$ \cite[Lemma 6.2]{ClausenMathewMorrow2021}, whence the same vanishing bound holds for $R\Gamma_\sub{\'et}(R,\bb G_m)/p$. Moreover, since $\tau^{\le1}R\Gamma_\sub{Zar}(-,\bb G_m)\quis \tau^{\le 1}R\Gamma_\sub{\'et}(-,\bb G_m)$ and $H^2_\sub{\'et}(-,\bb G_m)=\opp{Br}(-)$, we have a natural fibre sequence \[(\tau^{\le1}R\Gamma_\sub{Zar}(R,\bb G_m))/p\To R\Gamma_\sub{\'et}(R,\bb G_m)/p\To \opp{Br}(R)[p][-1],\] from which the claimed finite filtration is immediate.

We now prove (1). Since $\opp{Pic}(-)$ and $\opp{Br}(-)$ are rigid, the same is true of $R\mapsto \opp{Pic}(R_p^h)$ and $R\mapsto \opp{Br}(R_p^h)[p]$ (note that $R\mapsto R_p^h$ preserves henselian surjections) and so they are left Kan extended from smooth $\bb Z$-algebras by Lemma \ref{lemma_rigid_implies_lke}. So to complete the proof of (1), using the above finite filtration it remains only to show the same about $R\mapsto (R_p^h)^\times$; but this follows from Mathew's criterion \cite[Proposition~A.0.1]{ElmantoHoyoisKhanSosniloYakerson2020}.

Next we prove (2). It is enough to check the desired continuity equivalence modulo $p$, in which case we will prove the stronger result that, for any $p$-henselian ring $R$ with bounded $p$-power torsion, $R\Gamma_\sub{\'et}(R,\bb G_m)/p\to\{R\Gamma_\sub{\'et}(R/p^s,\bb G_m)/p\}_s$ is a pro equivalence. Using the finite filtration from the start of the proof, and the rigidity of $\opp{Br}$ and $\opp{Pic}$, the problem reduces to showing that the canonical map $R^\times\otimes_\bb Z^L\bb F_p\to\{(R/p^s)^\times\otimes_\bb Z^L\bb F_p\}$ is a pro equivalence; in other words we must show that the pro abelian groups $\{(1+p^sR)/p\}_s$ and $\{(1+p^sR)[p]\}_s$ vanish.

For $\{(1+p^sR)/p\}_s$ we recall the following form of Hensel's lemma \cite[Lemma~II.2]{Elkik1973}: given $f(X)\in R[X]$, $n>2$, and $a\in R$ such that $f(a)\in p^nR$ and $p\in f'(a)R$, then there exists $a'\in a+p^{n-1}R$ such that $f(a')=0$. Applying this to the approximate root $1+p^sa$ of $X^p-(1+p^{s+1}a)$, we see that raising to the $p^\sub{th}$-power defines a surjection $1+p^sR\to 1+p^{s+1}R$ for any $s\ge2$; so the transition map $(1+p^{s+1}R)/p\to (1+p^sR)/p$ is zero.

For $\{(1+p^sR)[p]\}_s$, the transition maps are injective and so we must in fact show that $(1+p^sR)[p]=0$ for $s\gg0$. Consider the maps \[(1+p^sR)[p]\to \mu_p(R[\tfrac1p])\to\mu_p(R_p^\sub{cl}[\tfrac1p]).\] The second map is an isomorphism by Fujiwara--Gabber \cite[Corollary 1.18(2)]{BhattMathew2021} \cite{Fujiwara1995}. Picking $c\ge1$ such that $R[p^\infty]=R[p^c]$, so that $p^cR\to R/R[p^\infty]$ is injective, then the first map is injective for any $s\ge c$. So the composition, which factors through $(1+p^sR_p^\sub{cl})[p]$, is injective; therefore $(1+p^sR)[p]\to (1+p^sR_p^\sub{cl})[p]$ is injective for $s\ge c$. So it is enough to show that $(1+p^sR_p^\sub{cl})[p]=0$ for $s\gg0$; but an elementary approximation argument shows that any element of $(1+p^3R_p^\sub{cl})[p]$ is in fact congruent to $1$ modulo all powers of $p$, hence vanishes since $R_p^\sub{cl}$ is $p$-adically separated. This completes the proof that $\{(1+p^sR)[p]\}_s=0$ and so of part (2).
 
(3) is an immediate consequence of (2) and Grothendieck's theorem that $R\Gamma_\sub{\'et}(-,\bb G_m)$ is an fpqc sheaf.
\end{proof}

We are now equipped to define the equivalence (\ref{comparison:j=1-intro}) in the $p$-henselian case. The construction begins by reviewing \cite[Proposition 7.17]{BhattMorrowScholze2}, namely Remark \ref{remark:odd_vanishing} in the case $j=1$. For any quasiregular semiperfectoid ring $S$, the trace map $\opp{tr}:\K(S;\bb Z_p)\to \TC(S;\bb Z_p)$ is an equivalences in homotopical degrees $\ge0$ \cite[Corollary~6.9]{ClausenMathewMorrow2021}, and so in particular induces an equivalence
 \[\tau^{[-2,-1]}\K(S;\bb Z_p)\stackrel{\sub{tr}}\quis \tau^{[-2,-1]}\TC(S;\bb Z_p)=\bb Z_p(1)^\sub{syn}(S)[2].\] Now assume in addition that $S$ is strictly $w$-local with $p$-divisible unit group. Then  arguments with symbols show that the canonical split inclusion $S^\times\to \K_1(S)$ induces an equivalence $(T_pS^\times)[2]\quis \tau^{[-2,-1]}\K(S;\bb Z_p)$. Moreover, $S$ being strictly $w$-local implies that $R\Gamma_\sub{\'et}(\Spec(S),\bb G_m)\simeq S^\times[0]$, whose derived $p$-completion is $(T_pS^\times)[1]$. We have thus defined a natural equivalence \begin{equation}R\Gamma_\sub{\'et}(S,\bb G_m)_p^\comp[-1]\quis\bb Z_p(1)^\sub{syn}(S)\label{eqn:j=1_basis}\end{equation} for any such $S$, where the hat denotes derived $p$-completion.

Such $S$ (that is, strictly $w$-local, quasiregular semiperfectoids with $p$-divisible unit group) form a basis of quasiregular semiperfectoids (to be careful, the proof of \cite[Proposition 7.17]{BhattMorrowScholze2} uses $w$-local, rather than strictly $w$-local, so the references to \cite{BhattScholze2015} in the proof should be replaced by Lemmas 2.2.13 and 2.2.15 of \cite{BhattScholze2015}), hence for all of $\text{QSyn}$. Since both sides of (\ref{eqn:j=1_basis}) are sheaves on $\text{QSyn}$ (the right hand side by definition; the left hand side by Proposition \ref{prop:G_m_p_henselian}(3)), right Kan extending as in Proposition \ref{prop:RKE_from_basis} defines a natural equivalence \begin{equation}R\Gamma_\sub{\'et}(R,\bb G_m)_p^\comp[-1]\quis \bb Z_p(1)^\sub{syn}(R)\label{comparison:j=1_qsyn}\end{equation} for any $R\in \text{QSyn}$.

In particular, for any smooth $\bb Z$-algebra $R$, we have a natural equivalence \[R\Gamma_\sub{\'et}(R_p^h,\bb G_m)_p^\comp\quis R\Gamma_\sub{\'et}(R_p^\comp,\bb G_m)_p^\comp\stackrel{\sub{\ref{comparison:j=1_qsyn}}}{\quis} \bb Z_p(1)^\sub{syn}(R_p^\comp)\stackrel{\sub{Lemma~\ref{lemma_syn_p-hens}(2)}}\simeq \bb Z_p(1)^\sub{syn}(R_p^h),\] where the first equivalence follows by applying Proposition \ref{prop:G_m_p_henselian}(2) to both $R_p^h$ and $R_p^\comp$. Left Kan extending from smooth $\bb Z$-algebras using Proposition \ref{prop:G_m_p_henselian}(1) at last defines the desired natural equivalence \begin{equation}R\Gamma_\sub{\'et}(R_p^h,\bb G_m)_p^\comp[-1]\quis \bb Z_p(1)^\sub{syn}(R_p^h)\label{comparison:j=1}\end{equation} for any ring $R$.

The description (\ref{comparison:j=1}) of weight one syntomic cohomology is compatible with the syntomic-to-\'etale comparison map, as follows:

\begin{lemma}\label{lemma_first_Chern_p_hens}
For any $p$-henselian ring $R$, the diagram
\[\xymatrix{
\bb Z_p(1)^\sub{syn}(R)\ar[d]_{\gamma_\sub{syn}^\sub{\'et}\{1\}} &\ar[l]^\simeq _{\sub{(\ref{comparison:j=1})}}R\Gamma_\sub{\'et}(R,\bb G_m)_p^\comp[-1]\ar[d]\\
R\Gamma_\sub{\'et}(R[\tfrac1p],\bb Z_p(1)) & R\Gamma_\sub{\'et}(R[\tfrac1p],\bb G_m)_p^\comp[-1]\ar[l]_\simeq
}\]
naturally commutes, where the bottom horizontal arrow is the Kummer isomorphism.
\end{lemma}
\begin{proof}
Since the target of the maps, namely $R\Gamma_\sub{\'et}(-[\tfrac1p],\bb Z_p(1))$, is right Kan extended from perfectoid rings by Corollary \ref{corollary:RKE_perfectoid}, it is enough to check that the diagram naturally commutes on $\text{Perfd}$. In fact, by the same argument, we may even restrict attention to the standard basis of $\text{Perfd}$ for the $p$-complete arc topology, namely products $R=\prod_iV_i$ of absolutely integrally closed valuation rings $V_i$ of residue characteristic $p$.

In that case $R$ is strictly $w$-local with $p$-divisible unit group, and the four corners of the diagram are discrete; the problem reduces to checking that the composition \[T_pR^\times\quis \bb Z_p(1)^\sub{syn}(R)\stackrel{\gamma_\sub{syn}^\sub{\'et}\{1\}}{\to} T_pR[\tfrac1p]^\times\] is the canonical map (i.e. the natural map induced by $R\to R[\tfrac1p]$), where the first map is the isomorphism (\ref{eqn:j=1_basis}) defined in terms of $K$-theory and the trace map (using \cite[Prop 7.17]{BhattMorrowScholze2}). This may be checked on each individual $V_i$ occurring on the product. If $V_i$ has characteristic $p$ then the question is vacuous as the target is zero. If $V_i$ is $p$-torsion-free then Bhatt--Scholze define the map $\gamma_\sub{syn}^\sub{\'et}\{1\}$ in \cite[Theorem~9.4]{BhattScholze2022} (denoted by $\alpha_1$ in [loc.~cit.]) exactly to make the composition the canonical map.
\end{proof}

For an arbitrary ring $R$, we may now use Lemma \ref{lemma_syn_p-hens}(3) to define \[c_1^\sub{syn}:R\Gamma_\sub{\'et}(R,\bb G_m)_p^\comp[-1]\To \bb Z_p(1)^\sub{syn}(R)\] as the unique map fitting into a commutative diagram
\[\xymatrix{
R\Gamma_\sub{\'et}(R,\bb G_m)_p^\comp[-1]\ar@{-->}[dr]^{c_1^\sub{syn}}\ar[dd]\ar[rr]&&R\Gamma_\sub{\'et}(R[\tfrac1p],\bb G_m)_p^\comp[-1]\ar[d]^{\simeq\sub{ Kummer}\,c_1^\sub{\'et}}\\
&\bb Z_p(1)^\sub{syn}(R)\ar[r]\ar[d] & R\Gamma_\sub{\'et}(R[\tfrac1p],\bb Z_p(1))\\
R\Gamma_\sub{\'et}(R_p^h,\bb G_m)_p^\comp[-1]\ar[r]^-{\sub{(\ref{comparison:j=1})}} & \bb Z_p(1)^\sub{syn}(R_p^h) &
}\]
That is, on $p$-henselian rings $c_1^\sub{syn}$ is as defined as at (\ref{comparison:j=1}), and on rings on which $p$ is invertible it is given by the Kummer isomorphism.

\begin{proposition}\label{proposition_1st_Chern_class_syn}
For any ring $R$, the map \[c_1^\sub{syn}:R\Gamma_\sub{\'et}(R,\bb G_m)_p^\comp[-1]\To \bb Z_p(1)^\sub{syn}(R)\] is an equivalence.
\end{proposition}
\begin{proof}
Noting that the square
\[\xymatrix{
R\Gamma_\sub{\'et}(R,\bb G_m) \ar[r]\ar[d] & R\Gamma_\sub{\'et}(R[\tfrac1p],\bb G_m)\ar[d]\\
R\Gamma_\sub{\'et}(R_p^h,\bb G_m)\ar[r] & R\Gamma_\sub{\'et}(R_p^h[\tfrac1p],\bb G_m)
}\]
is cartesian (and remains so after $p$-completing and shifting), this follows by glueing the $p$-henselian case (Lemma \ref{lemma_first_Chern_p_hens}) and the Kummer isomorphism via Lemma \ref{lemma_syn_p-hens}(3).
\end{proof}

\begin{remark}[Comparison to Bhatt--Lurie]
For any ring $R$, the maps $c_1^\sub{syn}$ and $\gamma^\sub{\'et}_\sub{syn}\{j\}$ agree with those defined by Bhatt--Lurie \cite{BhattLurie2022}.

We start with $c_1^\sub{syn}$. Since Bhatt--Lurie's $c_1^\sub{syn}$ is also given by the Kummer map when $p$ is invertible, we reduce to the case that the ring is $p$-henselian. By left Kan extension (to reduce to $p$-henselisations of smooth algebras) and right Kan extension (to reduce to a basis of $\text{QSyn}$), it is enough to check, for each $p$-torsion-free, strictly w-local, quasiregular semiperfectoid $S$ with $p$-divisible units, that the map (\ref{eqn:j=1_basis}) (constructed using the trace map $\K(S)\to \TC(S)$) is the same as Bhatt--Lurie's $c_1^\sub{syn}$ (constructed using their prismatic logarithm). Both sides of (\ref{eqn:j=1_basis}) are discrete, given by the Tate module $T_pS^\times$, and the canonical map to $T_p(S/p)^\times$ is injective (indeed, the Tate module of $1+pS$ vanishes since $S$ is $p$-adically separated). Note that $\overline S:=S/p$ is a strictly w-local, quasiregular semipefect ring with $p$-divisible units; we have reduced to checking that (\ref{eqn:j=1_basis}) for $\overline S$ is the same as Bhatt--Lurie's $c_1^\sub{syn}$. Composing with the canonical map $\TC_2(\overline S)=\bb Z_p(1)^\sub{syn}(\overline S)\hookrightarrow A_\sub{crys}(\overline S)=\mathrm{TP}_2(\overline S)$, Bhatt--Lurie's map becomes the crystalline logarithm \cite[Proof of Proposition~7.5.5]{BhattLurie2022}. But so does the trace map: this is a special case of Ansch\"utz--Le Bras' identification of the trace map with a q-logarithm \cite[Theorem~1.2]{AnschutzLeBras2020}.

Next, for $\gamma_\sub{syn}^\sub{\'et}\{j\}$, note that Bhatt--Lurie also define their comparison map for general rings by pulling back from the $p$-complete case \cite[Construction 8.4.1]{BhattLurie2022}, so we reduce to checking that our map agrees with theirs on derived $p$-complete rings. Since we have checked the agreement of first Chern classes, this agreement follows by quoting the uniqueness assertion in \cite[Theorem 8.3.1]{BhattLurie2022}, but we sketch the necessary arguments for the sake of completeness. Firstly, by Corollary \ref{corollary:RKE_perfectoid} we reduce to the case of perfectoid $R$, or even to the finer basis for the arc$_p$-topology consisting of products of rings of integers $V$ of algebraically closed, spherically complete, non-archimedean fields of residue characteristic $p$. We thus reduce to checking the agreement on such $V$, which we may further assume have characteristic $0$ (as otherwise the target of the comparison map vanishes). Then we see from Example \ref{example_syn_of_perfectoid} that $\bb F_p(j)^\sub{syn}(V)$ multiplicatively identifies with $\bb F_p(1)^\sub{syn}(V)$ for any $j\ge1$;\footnote{In fact, this special case of Example \ref{example_syn_of_perfectoid} is used to treat the case of general perfectoid rings; see the proof of \cite[Theorem~9.4]{BhattScholze2022}.} thus we reduce finally to the case $j=1$, where the desired agreement is precisely the agreement of the first syntomic Chern classes.
\end{remark}

We finish this subsection by describing the relation between syntomic cohomology and Milnor $K$-theory, originally due to Bhatt--Mathew \cite{BhattMathew2023} and to L\"uders and the third author \cite{LuedersMorrow2023}. Working modulo $p^r$, the first syntomic Chern class in degree $1$ is a map of abelian groups $R^\times/p^m\to H^1_\sub{syn}(R,\bb Z/p^m(1))$; by multiplicativity of syntomic cohomology this then induces multiplicatively compatible ``symbol'' maps \[\underbrace{R^\times\otimes_{\bb Z}\cdots\otimes_{\bb Z}R^\times}_{j\sub{ times}}\To H^1_\sub{syn}(R,\bb Z/p^m(1))\] for all $j\ge0$. 

\begin{theorem}\label{thm_syntomic_Milnor}
For any $r,j\ge0$ and henselian local ring $R$, the symbol map factors (necessary uniquely) as \[(\underbrace{R^\times\otimes_{\bb Z}\cdots\otimes_{\bb Z}R^\times}_{j\sub{ times}})/p^r\onto \hat K_j^M(A)/p^r\isoto H^j_\sub{syn}(R,\bb Z/p^r(j)),\] where $\hat K_j^M(A)$ denotes Gabber--Kerz' improved Milnor $K$-group of $A$ \cite{Kerz2010}.
\end{theorem}
\begin{proof}
In the case of $p$-henselian rings this is \cite[Theorem~3.1]{LuedersMorrow2023}, relying on Bhatt--Mathew \cite[Theorem~1.8]{BhattMathew2023}.

Otherwise $R$ is a henselian local ring in which $p$ is invertible. Letting $k$ denote the residue field, the maps $\hat K_j^M(A)/p^r\to K_j^M(k)/p^r$ and $H^j_\sub{syn}(R,\bb Z/p^r(j))=H^j_\sub{\'et}(R,\mu_{p^r}^{\otimes j})\to H^j_\sub{\'et}(k,\mu_{p^r}^{\otimes j})$ are isomorphisms; indeed, the first isomorphism is a rigidity property of Milnor $K$-theory away from the characteristic \cite[Proposition~10(7)]{Kerz2010}, and the second is Gabber's rigidity of \'etale cohomology \cite{Gabber1994a}. These isomorphisms reduce the problem to checking that the symbol map $(k^\times)^{\otimes j}/p^r\to H^j_\sub{\'et}(k,\mu_{p^r}^{\otimes j})$, induced by Kummer theory and multiplicativity of \'etale cohomology, descends to an isomorphism $K_j^M(k)/p^r\isoto H^j_\sub{\'et}(k,\mu_{p^r}^{\otimes j})$; but that is precisely the Bloch--Kato conjecture established by Voevodsky.
\end{proof}

\subsection{Beilinson--Lichtenbaum cohomology: $\bb A^1$-invariance and projective bundle formula}\label{ss_BL}
At this point we have defined syntomic cohomology as a multiplicative family of presheaves of complexes on qcqs schemes \[\bb Z_p(j)^\sub{syn}:\Sch^\sub{qcqs}\To{\rm D}(\bb Z),\qquad j\in\bb Z,\] and established its main properties as stated in Theorem~\ref{theorem_syntomic_properties}. It is supposed to serve the role of a general theory of $p$-adic \'etale motivic cohomology. According to the classical ideas of Beilinson and Lichtenbaum, one therefore expects that the truncations $\tau^{\le j}\bb F_p(j)^\sub{syn}$, or more precisely their suitable sheafifications, should be related to mod-$p$ (non-\'etale) motivic cohomology in certain contexts. We therefore adopt the following definition:

\begin{definition}\label{def:bl-cohomology}
We define Nisnevich sheaves on qcqs schemes
 \[\bb F_p(j)^\sub{BL}:=L_\sub{Nis}\tau^{\le j}\bb F_p(j)^\sub{syn}:\Sch^\sub{qcqs,op}\To\rm D(\bb Z),\] for $j\in\bb Z$, which we call {\em mod-$p$ Beilinson--Lichtenbaum cohomology}.\footnote{The notation BL is doubly convenient: on smooth schemes over fields and over mixed characteristic Dedekind domains, $\bb F_p(j)^\sub{BL}$ identifies with the mod-$p$ reduction of the Bloch--Levine cycle complex; see \S\ref{section_motivic_DD}.} We will also need $p$-adic and mod-$p^r$ variants, so set \[\bb Z_p(j)^\sub{BL}:=\opp{lim}_r L_\sub{Nis}\tau^{\le j}(\bb Z_p(j)^\sub{syn}/p^r):\Sch^\sub{qcqs,op}\To\rm D(\bb Z).\] Observe that $\bb Z_p(\star)^\sub{BL}$ and $\bb F_p(\star)^\sub{BL}$ are $\bb E_\infty$-algebras in graded presheaves of complexes on qcqs schemes (and similarly modulo any power of $p$), by Remark \ref{rem:BL_truncation}.
\end{definition}

The following properties are simple but crucial:

\begin{lemma}\label{lemma_BL_mod_p}
Let $j\in\bb Z$.
\begin{enumerate}
\item $\bb Z_p(j)^\sub{BL}=0$ if $j<0$.
\item The canonical map $\bb Z_p(j)^\sub{BL}/p^r \to L_\sub{Nis}\tau^{\le j}(\bb Z_p(j)^\sub{syn}/p^r)$ is an equivalence of presheaves on qcqs schemes for all $r \geq 1$; in particular, $\bb Z_p(j)^\sub{BL}/p\quis \bb F_p(j)^\sub{BL}$.
\item For any $r\ge1$, the presheaf $\bb Z_p(j)^\sub{BL}/p^r:\opp{Sch}^\sub{qcqs,op}\to \D(\bb Z)$ is finitary.
\item ``Beilinson--Lichtenbaum equivalence'': the canonical map $\bb Z_p(j)^\sub{BL}/p^m\to L_\sub{Nis}\tau^{\le j}(L_\sub{\'et}\bb Z_p(j)^\sub{BL}/p^r)$ is an equivalence.
\end{enumerate}
\end{lemma}
\begin{proof}
(1): If $j<0$ then the analogue of \eqref{eqn_j!p} for arbitrary powers of $p$ shows that $\bb F_p(j)^\sub{syn}/p^r\simeq R\Gamma_\sub{\'et}(-,j_!\mu_{p^r}^{\otimes j})$, which is coconnective and so vanishes after applying the truncation $\tau^{\le j}$. 

The remaining parts of the lemma are now clear if $j<0$, so we henceforth suppose $j\ge0$. Part (2) follows from the fact that, for any henselian local ring $R$, the canonical mod-$p$ reduction map $H^j_\sub{syn}(R,\bb Z/p^{r+N}(j))\to H^j_\sub{syn}(R,\bb Z/p^r(j))$ is surjective, by Theorem \ref{thm_syntomic_Milnor}. (Indeed this implies that $L_\sub{Nis}\tau^{\le j}(\bb Z_p(j)^\sub{syn}/p^{N-r}) \xrightarrow{p^r} L_\sub{Nis}\tau^{\le j}(\bb Z_p(j)^\sub{syn}/p^N) \to L_\sub{Nis}\tau^{\le j}(\bb Z_p(j)^\sub{syn}/p^r)$ is a cofiber sequence, and now take the limit over $N$.) Part (3) follows from the finitariness of syntomic cohomology (Proposition \ref{prop:syn}) and Proposition \ref{prop:finitary_conditions}(2).

Part (4) is equivalent to the assertion that $\tau^{>j}(\bb Z_p(j)^\sub{syn}/p^r)$ vanishes after \'etale sheafification, or equivalently vanishes on all strictly henselian local rings $A$. If $A$ is $p$-henselian this holds by Remark \ref{rmk_syntomic_props}, while if $p$ is invertible in $A$ then this holds by Lemma \ref{lemma_syn_p-hens}(1).
\end{proof}

\begin{example}\label{example_BL_cohomology}
\begin{enumerate}
\item On the category of qcqs $\bb Z[\tfrac1p]$-schemes, the syntomic-to-\'etale equivalence of Lemma~\ref{lemma_syn_p-hens} induces an equivalence $\bb F_p(j)^\sub{BL}(\star)\simeq L_\sub{Nis}\tau^{\le \star}R\Gamma_\sub{\'et}(-,\mu_p^{\otimes \star})$ of $\bb E_\infty$-algebras in graded presheaves of complexes.
\item On the category of smooth schemes over a field of characteristic $p$, Example \ref{example_syn_in_char_p} induces an equivalence $\bb F_p(\star)^\sub{BL}(X)\simeq R\Gamma_\sub{Nis}(X,\Omega^\star_\sub{log})[-\star]$.
\item On the category of qcqs $\bb F_p$-schemes, there is a fibre sequence \[\bb F_p(j)^\sub{BL}\to\bb F_p(j)^\sub{syn}\to R\Gamma_\sub{Nis}(-,\tilde\nu(j))[-j-1]\] for each $j\ge0$, resulting from Remark \ref{rmk_syntomic_props}(2).
\item In low weights, we have $\bb Z_p(0)^\sub{BL}\simeq R\Gamma_\sub{Nis}(-,\bb Z)_p^\comp$ and $\bb Z_p(1)^\sub{BL}\simeq R\Gamma_\sub{Nis}(-,\bb G_m)_p^\comp[-1]$ (the latter being induced by the first Chern class $c_1^\sub{BL}$ below). These identifications easily follow from Remark \ref{rem_syn_weight_0} and Proposition \ref{proposition_1st_Chern_class_syn}.
\end{enumerate}
\end{example}

Next note that the syntomic first Chern class $c_1^\sub{syn}:R\Gamma_\sub{\'et}(-,\bb G_m)[-1]\to\bb Z_p(1)^\sub{syn}$ of \S\ref{ss:weight_one_syn} induces (by applying $\opp{lim}_mL_\sub{Nis}\tau^{\le 1}(-/p^r)$ to both sides) a first Chern class map \[c_1^\sub{BL}:R\Gamma_\sub{Nis}(-,\bb G_m)[-1]\To \bb Z_p(1)^\sub{BL}\] in Beilinson--Lichtenbaum cohomology. This allows us to formulate the question of the $\bb P^1$-bundle formula: namely, given a qcqs scheme $X$, is the map 
\begin{equation}\bb Z_p(j)^\sub{BL}(X) \oplus \bb Z_p(j-1)^\sub{BL}(X)[-2] \xrightarrow{\pi^* \oplus c_1^\sub{BL}(\scr O(1))\pi^*} \bb Z_p(j)^\sub{BL}(\P^1_{X})\label{eqn_PBF_for_BL}\end{equation}
an equivalence (where $\pi:\bb P_X^1\to X$ denotes the structure map)? This should only be expected to hold when $X$ is sufficiently regular.

The goal of this subsection is to prove as directly as possible the following known motivic properties of Beilinson--Lichtenbaum cohomology, including the $\bb P^1$-bundle formula, in the smooth case:

\begin{theorem}\label{thm_BL_coh}
Let $B$ be a field or a mixed characteristic Dedekind domain and let $j \in \Z$. Then the following hold for any smooth $B$-scheme $X$:
\begin{enumerate}
\item $\bb A^1$-invariance: the canonical map $\bb Z_p(j)^\sub{BL}(\bb A_X^1)\to\bb Z_p(j)^\sub{BL}(X)$ is an equivalence.
\item Projective bundle formula: the map (\ref{eqn_PBF_for_BL}) is an equivalence.
\end{enumerate}
\end{theorem}

Over fields the proof of the theorem uses presheaves with transfers, which are crucial to Voevodsky's construction of a derived $\infty$-category of motives. Here we summarise what we need from Voevodsky's theory:

\begin{theorem}[Voevodsky]\label{thm:PST}
Let $k$ be a perfect field and $F:\opp{Cor}_k\to\opp{Ab}$ a homotopy invariant presheaf with transfers. Then the following hold:
\begin{enumerate}
\item Strict $\bb A^1$-invariance: The canonical map $H^n_\sub{Nis}(\bb A_X^1,a_\sub{Nis}F)\to H^n_\sub{Nis}(X,a_\sub{Nis}F)$ is an isomorphism for all $n\ge0$ and all $X\in\Sm_k$.
\item The canonical map $a_\sub{Nis}(F(\bb G_{m,\ph}))\to (a_\sub{Nis}F)(\bb G_{m,\ph})$ is an isomorphism of presheaves of abelian groups on $\Sm_k$. (The left hand side is the Nisnevich sheafification of the presheaf $F(\bb G_{m,\ph})$; the right hand side is defined by first Nisnevich sheafifying $F$ then evaluating on $\bb G_{m,\ph}$.)
\item For any $X\in\Sm_k$, writing $\pi:\bb G_{m,X}\to X$ for the structure map, we have $R^n\pi_*(a_\sub{Nis}F)=0$ for all $n>0$.
\item The canonical map $L_\sub{Nis}(F(\bb G_{m,-}))\to R\Gamma_\sub{Nis}(\bb G_{m,-},a_\sub{Nis}F))$ is an equivalence of sheaves of complexes on $\Sm_k$. (The left hand sides denotes $L_\sub{Nis}$ of the presheaf $F(\bb G_{m,\ph})$.)
\end{enumerate}
\end{theorem}
\begin{proof}
These results were originally established by Voevodsky in \cite{Voevodsky2000}, but it is well-known that there is an error in \cite[Proposition~4.23]{Voevodsky2000}; a weaker, but correct and sufficient, statement was given by Voevodsky in his IAS lecture course \cite[Lemma~22.10]{MazzaWeibelVoevodsky2006}, so we will use to the latter reference.

Part (1) is \cite[Theorem~24.1]{MazzaWeibelVoevodsky2006}. For part (2) we quote the statement ``$a_\sub{Nis}F_{-1}=(a_\sub{Nis}F)_{-1}$'' of \cite[Proposition~23.5]{MazzaWeibelVoevodsky2006}; this means that the canonical map \begin{equation}a_\sub{Nis}\opp{coker}(F(\bb A^1_{\ph})\to F(\bb G_{m,\ph}))\To\opp{coker}((a_\sub{Nis}F)(\bb A_{\ph}^1)\to (a_\sub{Nis}F)(\bb G_{m,\ph}))\label{eqn_contract}\end{equation} of Nisnevich presheaves (here the cokernels are taken as presheaves) is an isomorphism. By homotopy invariance of $F$ the arrow inside the first coker is split and so we may swap the $a_\sub{Nis}$ and the coker. Furthermore, both $F$ and $a_\sub{Nis}F$ are homotopy invariant (the latter by part (1)), so that the canonical map of presheaves $a_\sub{Nis}(F(\bb A_{\ph}^1))\to (a_\sub{Nis}F)(\bb A_{\ph}^1)$ is clearly an isomorphism (they both agree with $a_\sub{Nis}F$). The content of the isomorphism (\ref{eqn_contract}) is thus exactly that $a_\sub{Nis}(F(\bb G_{m,\ph}))\isoto (a_\sub{Nis}F)(\bb G_{m,\ph})$, as desired to prove (2).

Part (3) is obtained by applying \cite[Corollary~24.5]{MazzaWeibelVoevodsky2006} to $a_\sub{Nis}F$, which is a homotopy invariant presheaf with transfers by \cite[Theorem~22.15]{MazzaWeibelVoevodsky2006}. Finally, part (4) is a formal consequence of (2) and (3).
\end{proof}

\begin{proof}[Proof of Theorem \ref{thm_BL_coh} over fields]
Now we prove Theorem \ref{thm_BL_coh} when $B=k$ is a field; if $k$ has finite characteristic then, by a filtered colimit argument using N\'eron--Popescu, we may assume that $k$ is perfect (even a prime field, but this is not necessary). Thanks to Lemma \ref{lemma_BL_mod_p}(2) it is enough to prove the statements of the theorem for mod-$p$ Beilinson--Lichtenbaum cohomology.

First we suppose that $p$ is invertible in $k$, in which case we saw in Example \ref{example_BL_cohomology}(1) that $\bb F_p(j)^\sub{BL}=L_\sub{Nis}\tau^{\le j}R\Gamma_\sub{\'et}(-,\mu_p^{\otimes j})$. Let $\ep$ denote the Nisnevich-to-\'etale change of topology map, so that for each $i\ge0$ the abelian sheaf $R^i\ep_*\mu_p^{\otimes j}$ is the Nisnevich sheafification of the presheaf of abelian groups \[F^i(j):=H^i_\sub{\'et}(-,\mu_p^{\otimes j}):\Sm_k\To\opp{Ab}.\] The key to the proof in this case is that each presheaf $F^i(j)$ is homotopy invariant and upgrades to a presheaf with transfers \cite[\S3.4]{Voevodsky2000}, so that the previous theorem applies. Indeed, to prove part (1) it is enough to prove $\bb A^1$-invariance of $L_\sub{Nis}H^i_\sub{\'et}(-,\mu_p^{\otimes j})$ on $\Sm_k$ for $i=0,\dots,j$; but for any given $i$ that is equivalent to $\bb A^1$-invariance of the presheaves of abelian groups $H^n_\sub{Nis}(-,R^i\ep_*\mu_p^{\otimes j})$ for all $n\ge0$, which follows by applying Theorem \ref{thm:PST}(1) to $F^i(j)$.

Next we treat part (2) (still in the case that $B=k$ is a perfect field and $p$ is invertible in $k$). By Nisnevich descent it is enough to prove that (\ref{eqn_PBF_for_BL}) is an equivalence whenever $X=\Spec(R)$ is the spectrum of the henselisation of any essentially smooth, local $k$-algebra. Since $\bb F_p(j)^\sub{BL}$ is now known to be $\bb A^1$-invariant on ind-smooth $k$-schemes (by the previous paragraph and Lemma \ref{lemma_BL_mod_p}(3)) we see, by covering $\bb P_R^1$ by the usual two copies of $\bb A^1_R$ (denote the inclusion of one of the copies by $u$), that the projective bundle formula (\ref{eqn_PBF_for_BL}) is equivalent to proving that sequence \begin{equation}\bb F_p(j-1)^\sub{BL}(R)[-2]\xto{u^*c_1^\sub{BL}(\roi(1))\pi^*} \bb F_p(j)^\sub{BL}(\bb A^1_R)\To \bb F_p(j)^\sub{BL}(\bb G_{m,R})\label{eqn:G_m:form_syn}\end{equation} (which is a null sequence via the fact that $\roi(1)$ is trivial on $\bb A^1_R$ ) is a fibre sequence. But, by the same argument, because \'etale cohomology is $\bb A^1$-invariant and satisfies the projective bundle formula, we have a fibre sequence
\begin{equation}
R\Gamma_\sub{\'et}(R,\mu_p^{\otimes j-1})[-2]\xto{u^*c_1^\sub{\'et}(\roi(1))\pi^*} R\Gamma_\sub{\'et}(\bb A^1_R,\mu_p^{\otimes j})\To R\Gamma_\sub{\'et}(\bb G_{m,R},\mu_p^{\otimes j}).
\label{eqn:G_m:form_et}
\end{equation}
Moreover, the second arrow in (\ref{eqn:G_m:form_et}) is injective on $H^j$ (indeed, on all cohomology groups) by $\bb A^1$-invariance of \'etale cohomology, so it remains a fibre sequence if we apply $\tau^{\le j+1}$ to the first term and $\tau^{\le j}$ to the second and third terms. We claim that the resulting fibre sequence 
\[(\tau^{\le j-1}R\Gamma_\sub{\'et}(R,\mu_p^{\otimes j-1})[-2]\xto{u^*c_1(\roi(1))^\sub{\'et}\pi^*} \tau^{\le j}R\Gamma_\sub{\'et}(\bb A^1_R,\mu_p^{\otimes j})\To \tau^{\le j}R\Gamma_\sub{\'et}(\bb G_{m,R},\mu_p^{\otimes j})\] is exactly the desired (\ref{eqn:G_m:form_syn}). Firstly, by $\bb A^1$-invariance of both \'etale cohomology and $\bb F_p(j)^\sub{BL}$ the second term is $\bb F_p(j)^\sub{BL}(\bb A^1_R)$. Secondly, for the third term, Theorem \ref{thm:PST}(4) implies that the canonical map $F^i(j)(\bb G_{m,R})\to R\Gamma_\sub{Nis}(\bb G_{m,R},a_\sub{Nis}F^i(j))$ is an equivalence for all $i,j$, whence the canonical map \[\tau^{\le j}R\Gamma_\sub{\'et}(\bb G_{m,R},\mu_p^{\otimes j})\To \bb F_p(j)^\sub{BL}(\bb G_{m,R})\] is also an equivalence for all $j$. The first Chern classes $c_1^\sub{\'et}$ and $c_1^\sub{syn}$ are compatible by the diagram before Proposition \ref{proposition_1st_Chern_class_syn}. This completes the proof of Theorem \ref{thm_BL_coh} in the case that $B=k$ is a field and $p$ is invertible in $k$.

We next treat the case that $B=k$ is a field (still assumed perfect) of characteristic $p>0$, so that on smooth $k$-schemes we have $\bb F_p(j)^\sub{BL}=L_\sub{Nis}H^0_\sub{\'et}(-,\Omega^j_\sub{log})[-j]=R\Gamma_\sub{Nis}(-,\Omega^j_\sub{log})[-j]$. By appealing again to Voevodsky's Theorem \ref{thm:PST}(1), we see that part (1) of Theorem \ref{thm_BL_coh} will follow if we can show that $G(j):=H^0_\sub{\'et}(-,\Omega^j_\sub{log}):\Sm_k\to\opp{Ab}$ is $\bb A^1$-invariant and upgrades to to a presheaf with transfers. To show that we note that Gros--Suwa's Gersten resolution \cite[Corollary~1.6]{GrosSuwa1988a}, together with the Bloch--Kato-Gabber isomorphism for fields \cite[Corollary~2.8]{Bloch1986} (see also \cite[(4.21)]{GrosSuwa1988a}), implies that for any smooth $k$-scheme $X$ one has \[H^0_\sub{\'et}(X,\Omega^j_\sub{log})=\opp{ker}\left(\bigoplus_{\eta\in X^0} K^M_j(k(\eta))/p\xto{\sub{tame symbols}}\bigoplus_{y\in X^1} K^M_{j-1}(k(y))/p\right)\] In the terminology of Rost \cite{Rost1996}, this means that $H^0_\sub{\'et}(-,\Omega^j_\sub{log})$ is the twist-$j$ summand of the $0^\sub{th}$ Chow group with coefficients in the cycle module $K_*^M/p$; therefore it admits the structure of a presheaf with transfers by D\'eglise \cite[Proposition~6.9]{Deglise2006}, and so is homotopy invariant by \cite[Proposition~8.6]{Rost1996}.

To prove part (2), let $R$ be the henselisation of any essentially smooth, local $k$-algebra; we must show that the maps
\begin{equation}
R\Gamma_{\Nis}(R, \Omega^j_{\log}) \oplus R\Gamma_{\Nis}(R, \Omega^{j-1}_{\log})[-1] \xto{\pi^* \oplus c_1^\sub{BL}(\roi(1))\pi^*} R\Gamma_{\Nis}(\bb P_R^1, \Omega^j_{\log})
\label{eqn_BLp}\end{equation}
are equivalences for all $j \geq 0$. To do so, we first observe that $R\Gamma_{\Nis}(\bb P_R^1, \Omega^j_{\log})$ is in concentrated in degrees $\leq 1$. Indeed, this follows by covering $\bb P^1_R$ by the usual two copies of $\bb A^1_R$, the $\bb A^1$-invariance of $R\Gamma_\sub{Nis}(-,\Omega^j_\sub{log})$ (which has been established in the previous paragraph) and Theorem~\ref{thm:PST}(4). Hence we need only check that the map \eqref{eqn_BLp} is an isomorphism in degrees $0$ and $1$. In degree zero, the pullback map induces an isomorphism $H^0_{\Nis}(R, \Omega^j_{\log}) \stackrel{\cong}{\to} H^0_{\Nis}(\bb P^1_R, \Omega^j_{\log})$ since we can replace $\Nis$ by $\et$ in degree zero (by definition of Beilinson--Lichtenbaum cohomology) and the fact that syntomic cohomology satisfies the projective bundle formula (recalled in Theorem \ref{theorem_syntomic_properties}(6)). 

To check that \eqref{eqn_BLp} is an isomorphism in degree $1$, observe that the comparison map between Nisnevich and \'etale cohomologies induces a commutative diagram
\[
\begin{tikzcd}
 & 0 \\
 & H^0_{\Nis}(\bb P^1_R,\tilde\nu(j)) \ar{u} \\
H^1_\sub{\'et}(R, \Omega_\sub{log}^{j})  \oplus H^0_\sub{\'et}(R, \Omega_\sub{log}^{j-1}) \ar{r}{\cong} & H^1_{\et}(\bb P^1_R, \Omega_\sub{log}^{j}) \ar{u}  \\
H^0_{\Nis}(R, \Omega_\sub{log}^{j-1}) \ar{r}{c_1^\sub{BL}(\roi(1))\pi^*} \ar{u}{(0,\sub{can})} & H^1_{\Nis}(\bb P^1_R, \Omega_\sub{log}^{j}) \ar{u}\\
& 0 \ar{u}
\end{tikzcd}
\]
where exactness of the column follows from Remark \ref{example_BL_cohomology}(3) and the vanishing of $H^2_{\Nis}(\bb P^1_R, \Omega_\sub{log}^{j})$. Recall here the presheaf $\tilde\nu(j)$ from Remark \ref{rmk_syntomic_props}. The top horizontal arrow is an isomorphism thanks to the projective bundle formula for syntomic cohomology, and the left vertical map is an isomorphism onto the second summand, once again because Nisnevich and \'etale global sections of $\Omega^j_\sub{log}$ coincide. We need to show that the bottom horizontal arrow is an isomorphism; by what we have just explained about the diagram, this is equivalent to showing that the composition \[H^1_\sub{\'et}(R, \Omega_\sub{log}^{j})\To H^1_{\et}(\bb P^1_R, \Omega_\sub{log}^{j})\To H^0_{\Nis}(\bb P^1_R,\tilde\nu(j))\] is an isomorphism. The diagram shows this composition is surjective; it is injective because it is the same as the composition \[H^1_\sub{\'et}(R, \Omega_\sub{log}^{j})\isoto \tilde\nu(j)(R)\To H^0_{\Nis}(\bb P^1_R,\tilde\nu(j)),\] where the second arrow is split injective.
\end{proof}

\begin{remark}[Nisnevich vs Zariski]\label{rem:nisnevich-zariski}
We defined Beilinson--Lichtenbaum cohomology using Nisnevich, rather than Zariski, sheafification so that Lemma \ref{lemma_BL_mod_p}(2) held: Theorem \ref{thm_syntomic_Milnor} is currently unknown for arbitrary non-henselian local rings.

However, on sufficiently regular schemes, we expect that the apparent difference between Zariski and Nisnevich sheafification is illusory: namely, we expect that the canonical map $L_\sub{Zar}\tau^{\le j}\bb F_p(j)^\sub{syn}\to L_\sub{Nis}\tau^{\le j}\bb F_p(j)^\sub{syn}$ is an equivalence. In the rest of this remark we sketch a proof of this equivalence for smooth schemes over a field $k$.

We first need the following addition to Theorem \ref{thm:PST}, assuming $k$ perfect: for any presheaf with transfers $F:\opp{Cor}_k\to\opp{Ab}$, the change of topology map $a_\sub{Zar}F\to a_\sub{Nis}F$ is an isomorphism of presheaves, and more generally the canonical  map $H^n_\sub{Zar}(X,a_\sub{Zar}F)\to H^n_\sub{Nis}(X,a_\sub{Nis}F)$ is an isomorphism for all $n\ge0$ and all $X\in\Sm_k$. Indeed, the first claim is \cite[Theorem~22.2]{MazzaWeibelVoevodsky2006}, and the second claim then follows by combining the first claim with \cite[Proposition~13.9]{MazzaWeibelVoevodsky2006} and \cite[Theorem~22.15]{MazzaWeibelVoevodsky2006}. It follows that in the statement of Theorem \ref{thm:PST}, all occurrences of $a_\sub{Nis}$ and $L_\sub{Nis}$ may be replaced by $a_\sub{Zar}$ and $L_\sub{Zar}$ respectively.

In particular, in the notation of the above proof of Theorem \ref{thm_BL_coh} over fields, we have $R\Gamma_\sub{Zar}(X,a_\sub{Zar}F^i(j))\quis R\Gamma_\sub{Nis}(X,a_\sub{Nis}F^i(j))$ for all $i,j$ and all $X\in\Sm_k$ when $\opp{char} k\neq p$, and similarly for $G(j)$ when $\opp{char} k=p$. This implies $L_\sub{Nis}\tau^{\le j}\bb F_p(j)^\sub{syn}\quis L_\sub{Nis}\tau^{\le j}\bb F_p(j)^\sub{syn}$ on $\Sm_k$, as claimed.
\end{remark}

We now turn to proving Theorem \ref{thm_BL_coh} in mixed characteristic. We we will reduce it to the treated case of smooth varieties over a field by establishing a localisation sequence in Beilinson--Lichtenbaum cohomology; see Theorem~\ref{thm:localisation_BL}. In turn, the key to this localisation sequence is a certain vanishing result in \'etale cohomology, which we establish via a Gabber--Quillen presentation lemma due to Schmidt--Strunk \cite[Theorem~B]{SchmidtStrunk2018}. We begin by recalling some terminology and observations from \cite[\S6]{ElmantoMorrow2023}, replacing the base field from that context by a discrete valuation ring.

\begin{definition}
Let $B$ be a discrete valuation ring and $F:\Sm_B^\sub{op}\to\Spt$ a presheaf of spectra on smooth $B$-schemes; for $X \in \Sm_{B}$, we will write $F^X$ for the presheaf $U \mapsto F(U \times_B X)$. There are two morphisms of presheaves
\[
j^*, \pi^*\infty^*: F^{\bb P^1_B} \rightarrow F^{\bb A^1_B},
\]
where:
\begin{enumerate}
\item $\pi$ is induced the projection $\A^1_B \times_B X \rightarrow X$,
\item $\infty$ is the closed immersion $\Spec(B)\to\P^1_B$ of the point at $\infty$,
\item $j: \A^1_B \hookrightarrow \P^1_B$ is the open immersion complementary to the point at $\infty$.
\end{enumerate}
We say that $F$ is \emph{deflatable} if the maps $j^*$ and $\pi^*\infty^*$ are homotopic.
\end{definition}

\begin{example}
Let $B$ be a discrete valuation ring and $F:\Sm_{B}^\sub{op} \to \Spt$ a presheaf. Here are two situations where $F$ is deflatable.
\begin{enumerate}
\item If $F$ is $\bb A^1$-invariant, then it is deflatable: indeed, in that case the map $\pi^*$ is an equivalence and there is a natural $\bb A^1$-homotopy between $j^*$ and $\infty^*$.
\item Let $F(\star)$ be a $\bb E_\infty$-algebra in graded sheaves of spectra on $\Sm_B$, equipped with a first Chern class map $c_1:\left( \tau^{\leq 1}R\Gamma_{\Zar}(-,\Gm) \right)[1]\to F(1)$, and assume that the associated $\bb P^1$-bundle formula holds. Then each $F(j)$ is deflatable, by arguing as in the proof of  \cite[Lemma~6.12]{ElmantoMorrow2023}.
\end{enumerate}
\end{example}

The next lemma is an analogue over discrete valuation rings $B$ of \cite[Lemma~6.11]{ElmantoMorrow2023}, which is a minor variant of a theorem of Colliot-Th\'el\`ene--Hoobler--Kahn \cite{ColliotThelene-Hoobler-Kahn1997}; it establishes a Gersten injectivity for deflatable Nisnevich sheaves on smooth $B$-schemes, which may be used to reduce questions to the case of discrete valuation rings. Such reduction arguments go back to Gillet--Levine \cite{GilletLevine1987} and have also appeared in work of Bouis--Kundu \cite{BouisKundu2025}, Geisser \cite{Geisser2004}, L\"uders \cite{Lueders2022, Lueders2024a, Lueders2024}, Schmidt--Strunk \cite{SchmidtStrunk2018}, and others. To first clarify the cumbersome notation, given a smooth $B$-scheme $X$ and a point $x\in X$ on the special fibre, the ring $\roi_{X,x}^h$ is as usual the henselian local ring at $x$, while $(\roi_{X,x}^h)_{\frak m\roi_{X,x}^h}$ denotes the discrete valuation ring obtained by localising $\roi_{X,x}^h$ at its prime ideal generated by the maximal ideal of $B$. We allow ourselves to evaluate any presheaf of spectra on $\Sm_B$ on these rings, by taking filtered colimits of smooth $B$-algebras.

\begin{lemma}\label{lemma_Gillet_Levine}
Let $B$ be a henselian discrete valuation ring with infinite residue field, $X$ a smooth $B$-scheme, and $x\in X$ a point on the special fibre. Then, for any deflatable Nisnevich sheaf of spectra $F:\Sm_B^\sub{op}\to\Spt$ and $j\in\bb Z$, the map \[\pi_jF(\roi_{X,x}^h)\To \pi_jF((\roi_{X,x}^h)_{\frak m\roi_{X,x}^h})\] is injective.
\end{lemma}
\begin{proof}

Let $s\in \opp{ker}(\pi_jF(\roi_{X,x}^h)\To \pi_jF((\roi_{X,x}^h)_{\frak m\roi_{X,x}^h}))$; possibly after first replacing $X$ by a Nisnevich neighbourhood of $x$ we may suppose that $s$ is defined on $X$ and vanishes on $X\setminus Z$, where $Z\into X$ is some closed subscheme which does not contain the whole special fibre of $X$.

Now we appeal to Schmidt--Strunk's presentation result \cite[Theorem~B]{SchmidtStrunk2018}: possibly after yet again replacing $X$ by a Nisnevich neighbourhood of $x$, there exists a smooth $B$-scheme $V$ and an \'etale morphism $\pi:X\to\bb A_V^1$ with the following properties:
\begin{enumerate}
\item The composition $Z\into X\xto{\pi}\bb A_V^1$ is a closed embedding, and $\pi^{-1}\pi(Z)=Z$;
\item the composition $Z\into X\xto{\pi}\bb A_V^1\to V$ is finite (which implies that $Z\into X\xto{\pi}\bb A_V^1\to\bb P_V^1$ is also a closed embedding).
\end{enumerate}
The first condition furnishes us with a Nisnevich distinguished square:
\[\xymatrix{
X\setminus Z\ar[r]  \ar[d] &X\ar[d]^\pi\\
\bb A_V^1\setminus\pi(Z)\ar[r]&\bb A_V^1
}\]
By now arguing exactly as in \cite[Lemma~6.11]{ElmantoMorrow2023} one deduces that $s=0$, as desired.
\end{proof}

The vanishing application which we require of the previous lemma is the following:\footnote{More generally the argument of the proposition works verbatim for any homotopy invariant presheaf of abelian groups and so shows that, for any homotopy invariant presheaf with transfers $F:\opp{Cor}_K^\sub{op}\to\opp{Ab}$, with Nisnevich sheafification $a_\sub{Nis}F:\Sm_K^\sub{op}\to\opp{Ab}$, we have that $H^n_\sub{Nis}(\roi_{X,x}^h\otimes_BK,a_\sub{Nis}F)=0$ for all $n>0$.}

\begin{proposition}\label{prop_BL_Gersten_vanishing}
Let $B$ be a discrete valuation ring, $X$ a smooth $B$-scheme, and $p$ a prime number invertible in $K:=\opp{Frac}(B)$. Then for any point $x\in X$ we have:
\begin{enumerate}
\item $H^n_\sub{Nis}(\roi_{X,x}^h\otimes_BK,R^i\ep_*\mu_p^{\otimes j})=0$ for all $i,j\ge0$ and $n>0$, where $\ep$ denotes the Nisnevich-to-\'etale change of topology map.
\item $\bb F_p(j)^\sub{BL}(\roi_{X,x}^h\otimes_BK)$ is supported in degrees $\le j$.
\end{enumerate}
\end{proposition}
\begin{proof}
Since $\roi_{X,x}^h\otimes_BK$ is a smooth algebra over the field $K$ of characteristic $\neq p$, where Beilinson--Lichtenbaum cohomology is given by $L_\sub{Nis}\tau^{\le j}R\Gamma_\sub{\'et}(-,\mu_p^{\otimes j})$, one sees that (1) implies (2). Moreover, (1) is clear if $x$ lies on the generic fibre of $X$, as then $\roi_{X,x}^h\otimes_BK=\roi_{X,x}^h$ is a henselian local ring.

The non-trivial part of the proposition is thus part (1) when $x$ lies on the special fibre of $X$, which we treat in the rest of the proof. We may replace $B$ by its henselisation $B^h$ and $X$ by $X\otimes_BB^h$, as this changes neither the special fibre nor $\roi_{X,x}^h\otimes_BK=\roi_{X\otimes_BB^h,x}^h\otimes_{B^h}\opp{Frac}(B^h)$. That is, we may henceforth assume that $B$ is a henselian discrete valuation ring.

In this paragraph we assume further that the residue field of $B$ is infinite. We consider the Nisnevich sheaf $R\Gamma_\sub{Nis}(-\otimes_BK,R^i\ep_*\mu_p^{\otimes j}):\Sm_B^\sub{op}\to\Spt$, which we saw in the proof of Theorem \ref{thm_BL_coh} over fields is $\bb A^1$-invariant. We may therefore apply Lemma \ref{lemma_Gillet_Levine} to deduce that the map \[H^n_\sub{Nis}(\roi_{X,x}^h\otimes_BK,R^i\ep_*\mu_p^{\otimes j})\To H^n_\sub{Nis}((\roi_{X,x}^h)_{\frak m\roi_{X,x}^h}\otimes_BK,R^i\ep_*\mu_p^{\otimes j})\] is injective. But the ring $(\roi_{X,x}^h)_{\frak m\roi_{X,x}^h}\otimes_BK$ appearing on the right hand side is a field (namely the fraction field of $\roi_{X,x}^h$), and so its higher Nisnevich cohomology vanishes. This completes the proof when $B$ has infinite residue field.

Now suppose $B$ has finite residue field $k$ (and is still henselian, and $x$ still lies on the special fibre); we will deduce the desired vanishing from the established case of infinite residue field by a standard transfer argument. Namely, write the residue field $k(x)$ as a finite separable extension of some rational function field $k(\underline{t})=k(t_1,\dots,t_d)$, and let $\ell$ be a prime number $\neq p$ and prime to the degree $|k(x):k(\underline{t})|$. Let $k_\ell$ be the infinite algebraic extension of $k$ corresponding to the quotient $\bb Z_\ell$ of $\opp{Gal}(k^\sub{sep}/k)=\hat{\bb Z}$; thus $k_\ell$ is an infinite tower of finite $\ell$-primary extensions of $k$. Let $B_\ell$ be the corresponding henselian discrete valuation ring, ind finite \'etale over $B$, with residue field $k_\ell$. Since we chose $\ell$ to be prime to $|k(x):k(\underline{t})|$, we see that $k(x)\otimes_kk_\ell$ is a field and therefore there is a unique point of $X\otimes_BB_\ell$ over $x$; moreover, the henselian local ring at this point is $\roi_{X,x}^h\otimes_BB_\ell$. Since $B_\ell$ has infinite residue field, the infinite residue field case of the proposition tells us that $H^n_\sub{Nis}(\roi_{X,x}^h\otimes_B\opp{Frac}(B_\ell),R^i\ep_*\mu_p^{\otimes j})=0$ for all $i,j\ge0$ and $n>0$. Since the \'etale cohomology is finitary, therefore any element $c\in H^n_\sub{Nis}(\roi_{X,x}^h\otimes_BK,R^i\ep_*\mu_p^{\otimes j})$ vanishes in $H^n_\sub{Nis}(\roi_{X,x}^h\otimes_BK',R^i\ep_*\mu_p^{\otimes j})$ for some finite extension $K'$ of $K$ of degree being a power of $\ell$. By the existence of transfers on $H^n_\sub{Nis}(-,R^i\ep_*\mu_p^{\otimes j}):\Sm_K\to\opp{Ab}$ (recalled in the proof of Theorem \ref{thm_BL_coh} over fields), and the fact that $\ell\neq p$, it follows that $c=0$ as desired to complete the proof.
\end{proof}

We are now prepared to establish the desired localisation sequence:

\begin{theorem}[Localisation sequence in Beilinson--Lichtenbaum cohomology]\label{thm:localisation_BL}
Let $B$ be a discrete valuation ring of mixed characteristic, with residue field $k$ and fraction field $K$, and let $j\ge0$. Then, for any smooth $B$-scheme $X$, there exists a natural fibre sequence \[\bb Z_p(j)^\sub{BL}(X)\To \bb Z_p(j)^\sub{BL}(X\otimes_BK)\xto{\partial}\bb Z_p(j-1)^\sub{BL}(X\otimes_Bk)[-1],\] where the boundary map $\partial:\bb Z_p(\star)^\sub{BL}(X\otimes_BK)\to \bb Z_p(\star-1)^\sub{BL}(X\otimes_Bk)[-1]$ naturally upgrades to a morphism of $\bb Z_p(\star)^\sub{BL}(X)$-modules in $\opp{Gr}(\D(\bb Z))$.
\end{theorem}
\begin{proof}
We will prove the theorem modulo $p$, which in any case is all we need for our applications. To get the full result one instead works modulo any power of $p$ and then passes to the limit.

We begin with the case where $B$ has residue characteristic $\neq p$, so that $p$ is invertible on $X$. A consequence of absolute purity in \'etale cohomology \cite{Fujiwara2002} is the existence of a natural localisation sequence \[R\Gamma_\sub{\'et}(X\otimes_Bk,\mu_p^{\otimes j-1})[-2]\To R\Gamma_\sub{\'et}(X,\mu_p^{\otimes j})\To R\Gamma_\sub{\'et}(X\otimes_BK,\mu_p^{\otimes j})\] in which the first map is multiplication by the refined cycle class in $H^2_\sub{\'et}(X \textrm{ on } X\otimes_Bk,\mu_p)$ of the divisor $X\otimes_Bk$ (i.e., the image of any uniformiser of $K^\times$ under the maps $H^0_\sub{\'et}(X\otimes_BK,\bb G_m)\xto{\sub{Kum}} H^1_\sub{\'et}(X\otimes_BK,\mu_p)\to H^2_\sub{\'et}(X \textrm{ on } X\otimes_Bk,\mu_p)$). Rearranging the localisation sequence, we thus have a fibre sequence of \'etale sheaves on smooth $B$-schemes 
\[R\Gamma_\sub{\'et}(-,\mu_p^{\otimes j})\To R\Gamma_\sub{\'et}(-\otimes_BK,\mu_p^{\otimes j})\xto\partial R\Gamma_\sub{\'et}(-\otimes_Bk,\mu_p^{\otimes j-1})[-1],\] 
to which we apply $L_\sub{Nis}\tau^{\le j}$. The first and third terms becomes $\bb F_p(j)^\sub{syn}$ and $\bb F_p(j-1)^\sub{syn}(-\otimes_Bk)[-1]$ respectively, and the two non-trivial steps of the proof are to verify firstly that the middle term becomes $\bb F_p(j)^\sub{syn}(-\otimes_BK)$ (even though the inclusion of the generic fibre is not exact for the Nisnevich topology), and secondly that the resulting sequence of Nisnevich sheaves is still a fibre sequence.

The first step is the claim that the canonical map $\tau^{\le j}\bb F_p(j)^\sub{syn}(\roi_{X,x}^h\otimes_BK)\to \bb F_p(j)^\sub{BL}(\roi_{X,x}^h\otimes_BK)$ is an equivalence for every point $x\in X$. But that is precisely Proposition \ref{prop_BL_Gersten_vanishing}(2). The second step is the claim that the boundary map $\partial:H^j_\sub{\'et}(\roi_{X,x}^h\otimes_BK,\mu_p^{\otimes j})\to H^{j-1}_\sub{\'et}(\roi_{X,x}^h\otimes_Bk,\mu_p^{\otimes j-1})$ is surjective, again for every point $x\in X$. This is deduced from the fact that the target is generated by symbols, thanks to rigidity of \'etale cohomology and the Bloch--Kato conjecture.

The boundary map in this case is one of $\bb F_p(\star)^\sub{BL}$-modules because it was ultimately induced by multiplication by the refined cycle class, which is clearly a morphism of $R\Gamma_\sub{\'et}(-,\mu_p^{\otimes\star})$-modules in graded presheaves.

Now we treat the case that $B$ has mixed characteristic $(0,p)$. Bhatt--Mathew have proved a comparison theorem identifying the syntomic cohomology of smooth $B$-schemes with the Sato--Schneider complex \cite[Example 5.9]{BhattMathew2023}; that is, on the category of smooth $B$-schemes there is a fibre sequence of \'etale sheaves \begin{equation}\bb F_p(j)^\sub{syn}\To L_\sub{\'et}\tau^{\le j}R\Gamma_\sub{\'et}(-\otimes_BK,\mu_p^{\otimes j})\xto{\partial} \bb F_p(j-1)^\sub{syn}(-\otimes_Bk)[-1]\label{eqn:Kato_sequence}\end{equation} where $\partial$ is induced by Kato's residue map (see Remark \ref{rem:Katos_map} below). We now argue exactly as in the previous case: namely, we apply $L_\sub{Nis}\tau^{\le j}$. The first and third terms immediately become $\bb F_p(j)^\sub{BL}$ and $\bb F_p(j-1)^\sub{syn}(-\otimes_Bk)[-1]$ respectively, while the middle term becomes $\bb F_p(j)^\sub{BL}(-\otimes_BK)$ thanks to Proposition \ref{prop_BL_Gersten_vanishing}(2), and we obtain a fibre sequence of Nisnevich sheaves because, for every point $x\in X$, Kato's residue map $H^j_\sub{\'et}(\roi_{X,x}^h\otimes_BK,\mu_p^{\otimes j})\to H^{0}_\sub{\'et}(\roi_{X,x}^h\otimes_Bk,\Omega^{j-1}_\sub{log})$ is surjective. Indeed, the target is again generated by symbols by assembling theorems of Bloch--Kato--Gabber, Geisser--Levine, and Kerz, c.f., \cite[Theorem~1.2]{Morrow_pro_GL2}.

It remains to prove in this mixed characteristic case that the boundary map $\partial:\bb F_p(\star)^\sub{BL}(X\otimes_BK)\to \bb F_p(\star-1)^\sub{BL}(X\otimes_Bk)[-1]$ is naturally one of $\bb \F_p(\star)^\sub{BL}$-modules. By construction, $\partial$ is the Nisnevich sheafification over smooth $B$-schemes of the composition \[\tau^{\le \star}R\Gamma_\sub{\'et}(-\otimes_BK,\mu_p^{\otimes \star})\To H^\star_\sub{\'et}(-\otimes_BK,\mu_p^{\otimes \star})[-\star]\xto{\kappa}H^0_\sub{\'et}(-\otimes_Bk,\Omega^{\star-1}_\sub{log})[-\star].\] Combined with the fact that the canonical map $\tau^{\le \star}\bb F_p(\star)^\sub{syn}\to H^*_\sub{syn}(-,\bb F_p(\star))[-\star]$ is one of $\bb E_\infty$-algebras (by \cite[Lemma 8.12]{AntieauNikolaus2021}, using the $t$-structure of Remark \ref{rem:BL_truncation}), it is enough to check that the above map $\kappa$ of graded presheaves of abelian groups is one of $H^*_\sub{syn}(-,\bb F_p(\star))[-\star]$-modules. This is a property, rather than being extra structure, which requires checking the commutativity of diagrams terminating in $H^0_\sub{\'et}(-\otimes_B k,\Omega^{\star-1}_\sub{log})[-\star]$. Now using the notation of Remark \ref{rem:Katos_map}, we may test the commutativity of these diagrams after mapping injectively  to the stalk $\Omega^j_{k(\zeta),\sub{log}}$, and in this way we reduce to the case of the ind-smooth $B$-scheme $\roi_{X,\zeta}^h$ in place of $X$. That is, we must show that \[\kappa:H^\star_\sub{\'et}(F,\mu_p^{\otimes \star})\cong K_\star^M(F)/p\xto{\partial}K_\star^M(k(\zeta))/p\xto{\sub{dlog}} \Omega^\star_{k(\zeta),\sub{log}}\] is a map of $H^\star_\sub{syn}(\roi_{X,\zeta}^h,\bb F_p(\star))$-modules. But, identifying $H^\star_\sub{syn}(\roi_{X,\zeta}^h,\bb F_p(\star))$ with $\hat K_\star^M(\roi_{X,\zeta}^h)/p$ by Theorem~\ref{thm_syntomic_Milnor}, this is precisely the well-known (and easily checked) fact that the boundary map in Milnor $K$-theory associated with a discrete valuation is linear over the Milnor $K$-theory of the ring of integers of the valuation.
\end{proof}

\begin{remark}[Kato's residue map]\label{rem:Katos_map}
We recall Kato's residue map (c.f., \cite[Lemma~1.5]{LuedersMorrow2023} for a similar discussion). Continue to suppose that $B$ is a discrete valuation ring of mixed characteristic $(0,p)$. For any smooth (and connected for simplicity of notation) $B$-scheme $X$, let $\roi_{X,\zeta}$ be the discrete valuation ring obtained by localising $X$ at the generic point $\zeta$ of its special fibre, let $\roi_{X,\zeta}^h$ be its henselisation, and let $F$ be the fraction field of the latter. Thus $F$ is a henselian discrete valuation field of mixed characteristic $(0,p)$ (with residue field $k(\zeta)$), and so the cohomological symbol $K_j^M(F)/p\to H^j_\sub{\'et}(F,\mu_p^{\otimes j})$ was shown to be an isomorphism by Bloch--Kato \cite[Theorem~5.12]{Bloch1986}.\footnote{Of course, thanks to Voevodsky, this Bloch--Kato isomorphism is now known to hold for arbitrary fields.} The composition
\[H^j_\sub{\'et}(X\otimes_BK,\mu_p^{\otimes j})\to H^j_\sub{\'et}(F,\mu_p^{\otimes j})\cong K_j^M(F)/p\xto{\partial}K_j^M(k(\zeta))/p\xto{\sub{dlog}} \Omega^j_{k(\zeta),\sub{log}},\] where $\partial$ is the boundary map in Milnor $K$-theory associated with the discrete valuation on $F$, can then be shown to land inside $H^0_\sub{\'et}(X\otimes_Bk,\Omega^j_\sub{log})\subset \Omega^j_{k(\zeta),\sub{log}}$. This is shown either as in Bloch--Kato \cite[(6.6)]{Bloch1986} or via an argument with Gersten complexes \cite[Lemma~3.2.4]{Sato2007}. This defines \[\kappa:H^j_\sub{\'et}(X\otimes_BK,\mu_p^{\otimes j})\To H^0_\sub{\'et}(X\otimes_Bk,\Omega^j_\sub{log}),\] natural in $X$. Precomposing the $[-j]$ shift of this map with the canonical map $\tau^{\le j}R\Gamma_\sub{\'et}(X\otimes_BK,\mu_p^{\otimes j})\to H^j_\sub{\'et}(X\otimes_BK,\mu_p^{\otimes j})[-j]$, and then \'etale sheafifying as $X$ varies, defines $\kappa$ as it appears in \eqref{eqn:Kato_sequence}.
\end{remark}

We are now equipped to complete the proof of Theorem \ref{thm_BL_coh}.

\begin{proof}[Proof of Theorem \ref{thm_BL_coh} over mixed characteristic Dedekind domains]
Given a mixed characteristic Dedekind domain $B$ and a smooth $B$-scheme $X$, we are required to prove that the maps $\bb Z_p(j)^\sub{BL}(\bb A_X^1)\to\bb Z_p(j)^\sub{BL}(X)$ and (\ref{eqn_PBF_for_BL}) are equivalences. It is enough to prove the claims modulo $p$ since the cohomologies are $p$-complete. Via the localisation sequence of Theorem \ref{thm:localisation_BL}, $\A^1$-invariance immediately reduces to verifying the same for the special fibre $X\otimes_Bk$ and generic fibre $X\otimes_BK$. But we have already proved Theorem \ref{thm_BL_coh}(1) over fields.

The projective bundle formula for $X$ similarly reduces, via the localisation sequence of Theorem \ref{thm:localisation_BL}, to the analogous assertions for the special fibre and generic fibre. Here we implicitly use that the diagram
\[\xymatrix{
\bb Z_p(j)^\sub{BL}(\bb P^1_{X\otimes_BK})\ar[r]^-\partial & \bb Z_p(j-1)^\sub{BL}(\bb P^1_{X\otimes_Bk})[-1]\\
\bb Z_p(j-1)^\sub{BL}(\bb P^1_{X\otimes_BK})[-2]\ar[r]_-\partial\ar[u]^{c_1^\sub{BL}(\roi(1)} & \bb Z_p(j-2)^\sub{BL}(\bb P^1_{X\otimes_Bk})[-3]\ar[u]_{c_1^\sub{BL}(\roi(1)}
}\]
commutes, which is a special case of the linearity of the boundary map from Theorem \ref{thm:localisation_BL}.
\end{proof}

Although we will not use it, we finish the section by recording the following version of Corollary \ref{corol_BM_corol} for Beilinson--Lichtenbaum cohomology in the smooth case:

\begin{corollary}\label{corol_BM_corol2}
Let $R$ be a henselian local ring which is ind-smooth over a field or over a mixed characteristic Dedekind domain. Then there are natural pullback squares
\[\xymatrix{
\bb Z_p(j)^\sub{BL}(R)/p^r\ar[r]\ar[d] & \bb Z_p(j)^\sub{BL}(R[\tfrac1p])/p^r\ar[d]\\
\hat K_j^M(R)/p^r[-j]\ar[r] & H^j_\sub{\'et}(R[\tfrac1p],\mu_{p^r}^{\otimes j})[-j]
}\]
for all $r,j\ge0$.
\end{corollary}
\begin{proof}
The top right term of the diagram is supported in degrees $\le j$: indeed, either $R[\tfrac1p]=R$, or $R[\tfrac1p]=0$, or $R$ is a filtered colimit of rings $\roi_{X,x}^h$ as in Proposition \ref{prop_BL_Gersten_vanishing}(2).

Both vertical maps are thus the canonical map of a complex onto its top degree, and so the claim is that the natural map $\tau^{<j}\bb Z_p(j)^\sub{BL}(R)/p^r\to \tau^{<j}\bb Z_p(j)^\sub{BL}(R[\tfrac1p])/p^r$ is an equivalence. But $\tau^{<j}\bb Z_p(j)^\sub{BL}=\tau^{<j}\bb Z_p(j)^\sub{syn}$, so the equivalence holds by the same argument as in the proof of Corollary \ref{corol_BM_corol}.
\end{proof}

\section{Motivic cohomology over fields and Dedekind domains}\label{section_motivic_DD}
The goal of this section is to study $\bb A^1$-motivic cohomology and the zero slices of $1$ and $\KGL$, over fields and mixed characteristic Dedekind domains,\footnote{The main theorems of this section consequently hold over arbitrary regular Noetherian base schemes of dimension $\le 1$ (using N\'eron--Popescu to reduce the case of equal characteristic Dedekind domains to that of fields); however we find it more conceptual to formulate the results in terms of fields and mixed characteristic Dedekind domains, as these are the separate cases to treat.} by relating them to the Beilinson--Lichtenbaum cohomology of \S\ref{ss_BL}. We begin by overviewing the slightly convoluted logic of the section.

Firstly in \S\ref{ss_coherent_BL} we present a list of conditions under which the $\bb A^1$-invariant motivic cohomology with finite coefficients identifies with Beilinson--Lichtenbaum cohomology. These conditions have the advantage of being relatively ``soft'' in the sense that they mainly assert the vanishing of various motivic cohomology groups, and other $1$-categorical conditions; nevertheless, the resulting identification with Beilinson--Lichtenbaum cohomology, obtained in Proposition \ref{prop:graded} and Theorem \ref{thm:axiomatic_BL}, is multiplicative in the highly coherent sense, and is suitably unique.

Classical theorems of motivic homotopy theory show that smooth varieties over a perfect field $k$ satisfy the conditions. Furthermore, the resulting highly coherent Beilinson--Lichtenbaum equivalence allows us to describe $s^0(1_k)\in\SH(k)$ in terms of Beilinson--Lichtenbaum cohomology. By a pullback argument in Proposition \ref{prop_1_to_BL}, we then obtain a similar description of $s^0(1_B)$ for any mixed characteristic Dedekind domain $B$; this depends on knowing that Beilinson--Lichtenbaum cohomology behaves well under pullback in $\SH$, which is covered in \S\ref{ss_BL2}. Another pullback argument in Proposition \ref{prop:specz_slices}, this time using the motivic spectrum $\cal V$, implies that $s^0(1_{B}) \xrightarrow{\simeq} s^0(\KGL_{B})$. Altogether, this describes $s^0(\KGL_{B})$ in terms of Beilinson--Lichtenbaum cohomology.

The latter description is enough to show that smooth schemes over mixed characteristic Dedekind domains satisfy the original list of conditions, whence another application of Theorem \ref{thm:axiomatic_BL} yields multiplicative, highly coherent description of their $\bb A^1$-motivic cohomology in terms of Beilinson--Lichtenbaum cohomology; see Theorem \ref{theorem_BL1}. This description will be cdh-locally left Kan extended in \S\ref{sec:syn-comparison} and will be crucial to the proofs of the main theorems of the paper.


\subsection{Coherent Beilinson--Lichtenbaum equivalence}\label{ss_coherent_BL}
In this subsection we study the $\bb A^1$-motivic cohomology of ``motivically regular'' schemes. We present a list of conditions, inspired by the conjectures of Beilinson and Lichtenbaum, under which we can relate motivic cohomology to \'etale and syntomic cohomology.

For a prime number $p$, which will be fixed for this entire subsection, let $\mathfrak{BL}^p$ denote the full subcategory of qcqs schemes $X$ such that the following conditions hold for every $j,m\ge0$ and every \'etale $X$-scheme $U$:
\begin{enumerate}[(A)]
\item for every $x\in U$, the complex $\bb Z(j)^\bb A(\roi_{U,x}^\sub{h})/p^m$ is supported in cohomological degrees $\le j$;
\item for every $x\in U$ of residue characteristic $\neq p$, the complex $\bb Z(j)^{\bb A}(\roi_{U,x}^\sub{sh})/p^m$ is discrete;
\item if $p$ is invertible on $U$ then the canonical map $\bb Z(j)^{\bb A}/p^m\to L_\sub{Nis}\tau^{\le j}(L_\sub{\'et}\bb Z(j)^{\bb A}/p^m)$, restricted to $U_\sub{Nis}$, is an equivalence of presheaves on $U_\sub{Nis}$;
\item for every $x\in U$ of residue characteristic $p$, the map $\bb Z(j)^\bb A(\roi_{U,x}^\sub{h})/p^m\to \bb Z(j)^\bb A(\roi_{U,x}^\sub{h}[\tfrac1p])/p^m$ is an isomorphism on cohomology groups of degree $<j$;
\item for every $x\in U$ of residue characteristic $p$, the motivic first Chern class induces $\hat K_j^M(\roi_{U,x}^\sub{h})/p^m\isoto H^j_\bb A(\roi_{U,x}^h,\bb Z/p^m(j))$;\footnote{We explain this condition in more detail. For any local ring $A$ the first Chern class of Remark \ref{rem_A_1st_Chern_class} induces on $H^1$ a map $A^\times\to H^1_\bb A(A,\bb Z(1))$. Using the multiplicative structure on $\bb A^1$-invariant motivic cohomology, this then induces a map $(A^\times)^{\otimes j}\to H^j_\bb A(A,\bb Z(j))$ for each $j\ge1$, and passing modulo $p^m$ induces $(A^\times)^{\otimes j}/p^m\to H^j_\bb A(A,\bb Z/p^m(j))$. Condition (E) asks that, for each Nisnevich local ring $A$ of $U$, this latter map factors (necessarily uniquely) as $(A^\times)^{\otimes j}/p^m\onto \hat K_j^M(A)/p^m \isoto H^j_\bb A(A,\bb Z/p^m(j))$.}
\item for every $x\in U$ of residue characteristic $p$, the map $\bb Z_p(j)^\sub{syn}(\roi_{U,x}^\sub{h})/p^m\to \bb Z_p(j)^\sub{syn}(\roi_{U,x}^\sub{h}[\tfrac1p])/p^m$ is an isomorphism on cohomology groups of degree $<j$;
\item if $p$ is invertible on $U$ then the complex $\bb Z(j)^{\bb A}(U)/p^m$ is coconnective;
\item for every $x\in U$, the filtered spectrum $\Fil_\bb A^\star\KH(\roi_{U,x}^\sub{h})/p^m$ is complete (note: under condition (A), this is equivalent to asking that $\Fil_\bb A^j\KH(\roi_{U,x}^\sub{h})/p^m$ be $j$-connective for all $j\ge0$).
\end{enumerate}

That is, condition (A) states that $\bb Z(j)^\bb A/p^m$ is Nisnevich locally supported in degrees $\le j$ (as expected on sufficiently regular schemes), and so the map in (C) exists. Condition (B) ensures that the target of the map in (C) has a similar form to $L_\sub{Nis}\tau^{\le j}R\Gamma(-,\mu_{p^m}^{\otimes j})$, as this is how Beilinson--Lichtenbaum predict the motivic cohomology of sufficiently regular $\bb Z[\tfrac1p]$-schemes looks. Meanwhile, at residue characteristic $p$, condition (D) states that most of the motivic cohomology agrees with that of the generic fibre (similarly to Lichtenbaum--Quillen's conjecture for $K$-theory of sufficiently regular schemes), and (E) states that the missing data is given by Milnor $K$-theory through a Nesterenko--Suslin-style isomorphism. Condition (F) is needed to ensure that the syntomic cohomology also behaves as though the scheme is regular (c.f. Corollary \ref{corol_BM_corol}). The connectivity condition (G) is included for technical reasons: it would follow from the other conditions if it were known that the \'etale sheaf $L_\sub{\'et}\bb Z(j)^{\bb A}/p^m$ were Postnikov complete, and we may use it implicitly. Similarly, the completeness condition (H) will eventually be shown to hold whenever $\roi_{X,x}^\sub{h}$ has finite valuative dimension (see Corollary~\ref{cor:bdd-a1}), but before then it will technically enter the proofs.

An important property of the axioms is that they can be checked at the level of cohomology/homotopy groups. Furthermore, multiplicative structure only appears in a mild way, namely in axiom (E). Nevertheless, one of the main results of this section is Theorem \ref{thm:axiomatic_BL}, which upgrades these conditions to coherent, functorial, and multiplicative comparison maps between motivic cohomology and Beilinson--Lichtenbaum cohomology.

\begin{remark}
By Theorem \ref{theorem:SH_mot_coh}(2), given a point $x\in U$, the spectrum $\Fil^j_{\bb A} \KH(\roi_{U,x}^h)$ is the stalk at $x$ in the Nisnevich topology of the presheaf $\Fil^j_{\bb A} \KH$. Similarly, $\Fil^j_{\bb A} \KH(\roi_{U,x}^\sub{sh})$ is its stalk at $x$ in the \'etale topology, and analogous statements hold for the $\bb A^1$-motivic complexes $\bb Z(j)^{\bb A}$. We will frequently use these identifications without explicit mention in what follows.
\end{remark}

The class of schemes $\mathfrak{BL}^p$ enjoys the following nice closure properties:

\begin{lemma}\label{lemma_BL_properties}
Let $X$ be a qcqs scheme.
\begin{enumerate}
\item If $X$ belongs to $\mathfrak{BL}^p$ then any \'etale $X$-scheme also belongs to $\mathfrak{BL}^p$.
\item If $X$ admits a Nisnevich cover $X_1,\dots,X_n$, where each $X_i$ belongs to $\mathfrak{BL}^p$, then so does $X$.
\item If $X$ belongs to $\mathfrak{BL}^p$ then so do $\roi_{X,x}^h$ and $\roi_{X,x}^{sh}$ for every point $x\in X$.
\item The collection of qcqs schemes $\mathfrak{BL}^p$ is closed under cofiltered limits along smooth affine transition maps.
\end{enumerate}
\end{lemma}
\begin{proof}
Parts (1) and (2) are clear. For part (4), recall from Theorem \ref{theorem:SH_mot_coh}(2) and Proposition \ref{prop:syn}(2) that $\Fil^j_{\bb A} \KH/p^m$, $\bb Z(j)^{\bb A}/p^m$, and $\bb Z_p(j)^\sub{syn}/p^m$ commute with cofiltered limits along smooth affine transition maps, and that improved Milnor $K$-theory commutes with filtered colimits of local rings \cite{Kerz2010}; therefore conditions (A), (B), (D), (E), and (F) behave well under filtered colimits. For (H) we use the parenthetical observation: in light of (A), the completeness is equivalent to the seemingly stronger condition that $\Fil_\bb A^j\KH(\roi_{U,x}^\sub{h})/p^m$ is $j$-connective for all $j$; that condition behaves well under filtered colimits. For (C) the key observation is that, since $\bb Z(j)^{\bb A}/p^m$ takes coconnective values on $\mathfrak{BL}^p$ by (G), its \'etale sheafification $L_\sub{\'et}\bb Z(j)^{\bb A}/p^m$ is still finitary on $\mathfrak{BL}^p$; so both sides of the map in (C) are finitary on $\mathfrak{BL}^p$, as required to complete the proof of (4). Part (3) now follows from (1) and (4).
\end{proof}

The fact that smooth varieties belong to $\mathfrak{BL}^p$ is the content of the Beilinson--Lichtenbaum conjecture:

\begin{theorem}[Geisser, Levine, Suslin, Voevodsky]\label{thm:classical_BL_equiv}
Any smooth variety over any field belongs to $\mathfrak{BL}^p$.
\end{theorem}
\begin{proof}
%
The Nisnevich sheaf of spectra $\Fil^j_\bb A\KH$ is $j$-connective for the Nisnevich Postnikov $t$-structure by Voevodsky's connectedness conjecture \cite[Proposition 5.5]{Voevodsky2002}, proved originally by Levine \cite[\S7.5]{Levine2008}. Alternatively, see Remark \ref{remark_completeness_of_fil}(1) where we observed that this follows from considerations about very effective motivic spectra and stable $\bb A^1$-connectivity. This proves both conditions (A) and (H). The remainder of the proof consists of recalling results in the theory of motivic cohomology of smooth algebraic varieties, many of which are formulated in terms of Bloch's cycle complex.

Firstly, Levine \cite[\S6.5]{Levine2008} constructed Nisnevich sheaves of complexes $\cal Z^j:\Sm^\sub{op}_k\to \mathrm{D}(\bb Z)$, whose restriction to the wide subcategory of smooth morphisms agrees in the homotopy category of presheaves with Bloch's cycle complex $z^j(-,\bullet)$ \cite[\S6.5]{Levine2008}. He assembled them into the ``Bloch motivic cohomology'' spectrum $\cal Z\in\SH(k)$, characterised by the fact that $\omega^{\infty,\sub{gr}}(\cal Z)^j=\cal Z^j$ for $i\in\bb Z$ with bonding maps induced by the localisation sequence for Bloch's cycle complex. The motivic spectrum $\cal Z$ is equivalent to Voevodsky's motivic spectrum built from correspondences \cite[Example 8.2.2(3)]{Levine2008}.

Levine then shows that the canonical map $1_k\to\cal Z$ induces an equivalence $s^0(1_k)\quis\cal Z$ in $\SH(k)$ \cite[\S10.4]{Levine2008}. He then establishes an equivalence $s^0(\KGL_k)\simeq s^0(1_k)$ \cite[Theorem.~10.5.1]{Levine2008}; for a different proof of the latter equivalence, we refer to work of the first two authors \cite[\S4]{BachmannElmanto}.

In particular, for any fixed smooth variety $X$ and $j\ge0$, we deduce that there exists an equivalence \begin{equation}\bb Z(j)^{\bb A}\simeq z^j(-,\bullet)[-2j]\label{eqn_mot_is_bloch}\end{equation} of presheaves on the \'etale site $X_\sub{\'et}$. Note that, since we are restricting to $X_\sub{\'et}$, we are only claiming functoriality with respect to \'etale maps and so do not need to worry about any choice or naturality of the object $\cal Z^j$. We moreover do not assert any compatibilities of the equivalences (\ref{eqn_mot_is_bloch}) as $j$ varies, either in terms of bonding maps or multiplicative structure. Using this soft description of $\bb A^1$-motivic cohomology as Bloch's cycle complex, we may now complete the proof by quoting the main theorems about the latter.

Suppose first that $p$ is invertible in the base field (whence (D), (E), and (F) are vacuous). Then by a theorem of Suslin--Voevodsky \cite{SuslinVoevodsky2000} (see Geisser--Levine \cite[Theorem~1.6]{GeisserLevine2001} for the exact statement we need in terms of Bloch's cycle complex), the Bloch--Kato conjecture implies that $z^j(-,\bullet)[-2j]/p\simeq L_\sub{Zar}\tau^{\le j}R\Gamma_\sub{\'et}(-,\mu_p^{\otimes j})$, as Zariski presheaves of complexes on the \'etale site of $X$. We may replace $L_\sub{Zar}$ by $L_\sub{Nis}$ as in Remark \ref{rem:nisnevich-zariski}, and so combined with (\ref{eqn_mot_is_bloch}) we have an equivalence of Nisnevich sheaves $\bb Z(j)^{\bb A}/p\simeq L_\sub{Nis}\tau^{\le j}R\Gamma_\sub{\'et}(-,\mu_p^{\otimes j})$ on $X_\sub{\'et}$. Again, we assert no canonicity/multiplicativity/etc of this equivalence; nevertheless, properties (B), (C), and (G) now clearly follow.

In the rest of the proof we treat the case where the base field has characteristic $p$; then (B) and (C) are vacuous, and (F) follows from Example \ref{example_syn_in_char_p}. The remaining necessary results are due to Geisser--Levine \cite{GeisserLevine2000} which we review in this paragraph. Firstly, \cite[Thm.~8.3]{GeisserLevine2000} states that $z^j(-,\bullet)[-2j]/p^m$ is Zariski locally supported in cohomological degree $j$, whence we obtain the same for $\bb Z(j)^\bb A/p^m$ from (\ref{eqn_mot_is_bloch}). Secondly, from \cite[Theorem~8.3]{GeisserLevine2000} and a Gersten complex argument (using Kerz for Milnor $K$-theory \cite{Kerz2009}) one knows that, for any regular local $\bb F_p$-algebra $A$, the symbol map $\hat K_j^M(A)/p^m\to K_j(A;\bb Z/p^m)$ is an isomorphism; see \cite[Theorem~5.1]{Morrow_pro_GL2} for details.

In particular, for each Nisnevich local ring $\roi_{U,x}^h$ of any smooth variety $U$, this show that $\bb Z(j)^{\bb A}(\roi_{U,x}^\sub{h})/p^m$ is supported in degree $j$ (which verifies conditions (A) and (D); it also shows that that $\Fil^0_\bb A\KH(\roi_{U,x}^h)/\Fil^j_\bb A\KH(\roi_{U,x}^h)$ is supported in homological degrees $[0,\dots,j-1]$, which we will need in a moment), and that the symbol map $\hat K_j^M(\roi_{U,x}^h)/p^m\to K_j(\roi_{U,x}^h,\bb Z/p^m)$ is an isomorphism. We may now check the remaining condition (E). We consider the series of isomorphisms
\begin{equation}
H^j_\bb{A}(\roi_{U,x}^h,\bb Z/p^m(j))\stackrel\cong\longleftarrow \pi_{j}(\Fil_\bb A^j \KH(\roi_{U,x}^h)/p^m)\stackrel{\cong}{\To} \KH_j(\roi_{U,x}^h;\bb Z/p^m) \stackrel{\cong}{\longleftarrow} \K_j(\roi_{U,x}^h;\bb Z/p^m) \stackrel{\sub{sym}\cong }{\longleftarrow}K_j^M(\roi_{U,x}^h)/{p^m}
\label{eqn:BLp}
\end{equation}
The first isomorphism holds because $\Fil^{j+1}_\bb A\KH(\roi_{U,x}^h)$ has already been shown to be $j+1$-connective; the second isomorphism holds because it has also been noted that $\Fil^0_\bb A\KH(\roi_{U,x}^h)/\Fil^j_\bb A\KH(\roi_{U,x}^h)$ is supported in homological degrees $<j$. The composition of the maps (\ref{eqn:BLp}) thus defines isomorphisms $K_j^M(\roi_{U,x}^h)/{p^m}\isoto H^j_\bb{A}(\roi_{U,x}^h,\bb Z/p^m(j))$ for $j\ge0$. As $j$ varies these isomorphisms are compatible with multiplication (possibly after increasing $m$, as in the proof of Lemma \ref{lemma_Gabber_Suslin_etc} below, all the maps in (\ref{eqn:BLp}) are multiplicative; alternatively work with $p$-complete theories as in the proof of Proposition \ref{prop:graded} below), and when $j=1$ it is easily seen using the diagram (\ref{eq:c1-det}) that the isomorphism agrees with the first Chern class $c_1^\bb A:(\roi_{U,x}^h)^{\times}/p^m\to H^1_\bb A(\roi_{U,x}^h,\bb Z/p^m(1))$. So the isomorphisms are indeed those induced by multiplicativity via the first Chern class, thus completing the proof of condition (E).
\end{proof}

\begin{remark}\label{rem:cohere} As emphasized throughout the proof of Theorem~\ref{thm:classical_BL_equiv}, we only deploy Bloch's cycle complex to verify certain unstructured properties of $\bb Z(j)^{\bb A}$. In particular, it is not even used to verify property (E) which is the only one that uses the product structure on $\bb A^1$-motivic cohomology to formulate. The forthcoming generalisation of Theorem~~\ref{thm:classical_BL_equiv}, namely Theorem~\ref{theorem_BL1}(2), will follow a similar pattern in that we will first establish an unstructured comparison between $\bb A^1$-motivic cohomology with Beilinson--Lichtenbaum cohomology in order to deduce properties (A)-(H). 
\end{remark}

In the case of $\bb Z[\tfrac1p]$-schemes in the class $\frak{BL}^p$, we relate motivic cohomology to \'etale cohomology by passing through algebraic $K$-theory as follows. It is well-known since Gabber \cite{Gabber1992}, Suslin \cite{suslin1983}, and Thomason \cite{Thomason1985} that, \'etale locally on $\bb Z[\tfrac1p]$-schemes, the mod-$p^m$ $K$-groups $\K_n(-;\bb Z/p^m \bb Z$) are given by $\mu_{p^m}^{\otimes j}$, if $n=2j>0$, and vanish if $n>0$ is odd. The purpose of the next lemma is to fix this identification, paying attention to the ``mess'' (to quote Thomason) caused by the fact that the mod-$p$ $K$-theory of a ring is not necessarily a homotopy commutative and associative ring spectrum. 

Firstly, given a separably closed field $k$ of characteristic $\neq p$, the isomorphism $k^\times\isoto \K_1(k)$ and symbolic generation of $\K_2(k)$ induce an isomorphism $H^0_\sub{\'et}(k,\bb Z_p(1))=T_pk^\times\isoto \K_2(k;\bb Z_p)$. Since $\K(k;\bb Z_p)$ is an $\bb E_\infty$-algebra, its homotopy groups form a graded ring and so the latter isomorphism induces a map of graded rings $H^0_\sub{\'et}(k,\bb Z_p(\star))\to \K_{2\star}(k;\bb Z_p)$; this is an isomorphism by Suslin. Moreover, the $K$-groups $\K_{n}(k;\bb Z_p)$ vanish if $n$ is odd.

Similarly, but now working with mod-$p^m$ coefficients on the \'etale site of qcqs $\bb Z[\tfrac1p]$-schemes, the isomorphism $\bb G_m\isoto a_\sub{\'et}\K_1$ and symbolic generation of $a_\sub{\'et}\K_2$ (this isomorphism and symbolic generation even hold Zariski locally) induce an isomorphism of \'etale sheaves of abelian groups $\mu_{p^m}\isoto a_\sub{\'et}\pi_2(\K/p^m)$. This extends in a unique way to higher $K$-groups in a manner compatible with the isomorphism of the previous paragraph:

\begin{lemma}\label{lemma_Gabber_Suslin_etc}
On the \'etale site of qcqs $\bb Z[\tfrac1p]$-schemes, and for any $j\ge0$ and $m\ge1$, there is a unique map of abelian \'etale sheaves $\lambda_j:\mu_{p^m}^{\otimes j}\to a_\sub{\'et}\pi_{2j}(\K/p^m)$ such that, for any separably closed field $k$ of characteristic $\neq p$, the following diagram commutes:
\[\xymatrix{
\mu_{p^m}^{\otimes j}(k)\ar[r]^-{\lambda_j} & \K_{2j}(k;\bb Z/p^m)\\
H^0_\sub{\'et}(k,\bb Z_p(j))\ar[u]\ar[r]_{\sub{Suslin}}^{\cong} & \K_{2j}(k;\bb Z_p)\ar[u]
}\]
Furthermore, $\lambda_j$ is an isomorphism for each $j\ge0$, and $a_\sub{\'et}\pi_{n}(\K/p^m)=0$ for $n>0$ odd.
\end{lemma}
\begin{proof}
Pick an auxiliary power $p^{m'}\ge p^m$ of $p$ such that the mod-$p^{m'}$ $K$-theory of any scheme naturally carries the structure of a homotopy commutative and associative ring spectrum compatible with the $\bb E_\infty$-structure on the $p$-adic $K$-theory. Indeed, it is known (see \cite{Oka1984}) that if $p>3$ then any $m'\ge 1$ is already fine; if $p=3$ then we need $m'\ge 2$; if $p=2$ then any $m'\ge4$ does the trick. Then, as in the paragraph before the lemma, we again have an isomorphism of \'etale sheaves of abelian groups $\mu_{p^m}\isoto a_\sub{\'et}\pi_2(\K/p^m)$, which by multiplicativity does now induce maps $\lambda_j':\mu_{p^m}^{\otimes j}\to a_\sub{\'et}\pi_{2j}(\K/p^m)$ of abelian \'etale sheaves for each $j\ge0$. In the case of a separably closed field $k$, these maps fit by construction into commutative diagrams
\[\xymatrix{
\mu_{p^m}^{\otimes j}(k)\ar[r]^-{\lambda'_j} & \K_{2j}(k;\bb Z/p^m)\\
H^0_\sub{\'et}(k,\bb Z_p(j))\ar[u]\ar[r]_{\sub{Suslin}}^{\cong} & \K_{2j}(k;\bb Z_p)\ar[u]
}\]
Since Suslin's map is an isomorphism and the odd $p$-adic $K$-groups of $k$ vanish, the right vertical arrow is going-mod-$p^m$; but the same is true of the left vertical arrow, so we deduce that $\lambda'_j$ is an isomorphism in the case of the field $k$. By rigidity of both sides (Gabber in the case of $K$-theory \cite{Gabber1992}), it is therefore also an isomorphism for any strictly henselian local ring, and so the map of abelian \'etale sheaves $\lambda'_j$ is in fact an isomorphism. Going mod $p^m$ now defines the desired isomorphism of abelian sheaves $\lambda_j$ as in the statement of the lemma.

For uniqueness, suppose we were given another map of abelian sheaves $\kappa_j:\mu_{p^m}^{\otimes j}\to a_\sub{\'et}\pi_{2j}(\K/p^m)$ making the diagram commute. Then the composition $\lambda_j^{-1}\kappa_j$ would be an endomorphism of the abelian \'etale sheaf $\mu_{p^m}^{\otimes j}$ which is the identity on any separably closed field, hence also on any strictly henselian local ring; so it would be the identity and therefore $\kappa_j=\lambda_j$, as required to prove uniqueness.
\end{proof}

\comment{
\begin{remark}[Warning]
Note that one cannot in general make sense of any statement that the isomorphisms $\lambda_j$ of Lemma \ref{lemma_Gabber_Suslin_etc} are compatible with multiplication: if $p=2$ then the mod-$p$ $K$-groups carry no multiplicative structure. Instead we will only use the fact that, in the case of separably closed fields, they are quotients of Suslin's map which is multiplicative.
\end{remark}
}

The main theorem of this subsection, namely Theorem \ref{thm:axiomatic_BL} below, is a highly coherent and multiplicative Beilinson--Lichtenbaum equivalence across the entire category of schemes $\mathfrak{BL}^p$. As far as we are aware, such functoriality has not been recorded previously in the literature even if we restrict attention to algebraic varieties over a fixed field of characteristic $\neq p$.\footnote{The subtlety being that the classical such comparisons usually pass through calculations with Bloch's cycle complex, which is a priori neither functorial for non-flat maps nor multiplicative.} Before we can state the theorem we recall that a consequence of Theorem \ref{theorem:SH_mot_coh}(3) is that the mod-$p^m$ $\bb A^1$-motivic cohomology $\bb Z^\bb{A}(\star)/p^m$ is equipped with the structure of an $\bb E_\infty$-algebra in presheaves of complexes; indeed, we equipped $\bb Z^\bb{A}(\star)$ with such structure, which is then inherited by going mod $p$.\footnote{Had we not equipped motivic cohomology with any $\bb Z$-linear structure then a priori it would only be an $\bb E_\infty$-algebra in presheaves of spectra, and there would be no canonical $\bb E_\infty$-structure on the mod-$p^m$ quotient because of the usual problem that $\bb S/p^m$ is not an $\bb E_\infty$-algebra.}

Before establishing the full Beilinson--Lichtenbaum equivalence, we begin with a description of \'etale motivic cohomology with coefficients away from the characteristic. Let $\mathfrak{BL}^p_{1/p}$ denote the full subcategory of $\mathfrak{BL}^p$ consisting of schemes on which $p$ is invertible (note that a qcqs $\bb Z[\tfrac1p]$ scheme $X$ belongs to $\mathfrak{BL}^p_{1/p}$ if and only if it satisfies conditions (A)--(C) and (G); the other conditions (D)--(F) are vacuous in this case).

\begin{proposition}\label{prop:graded}
Let $m\ge1$. There is a unique map of $\bb E_\infty$-algebras in graded presheaves of complexes on $\mathfrak{BL}^p_{1/p}$ \[\delta_\sub{mot}^\sub{\'et}\{\star\}:L_\sub{\'et}\bb Z(\star)^{\bb A}/p^m\To R\Gamma_\sub{\'et}(-,\mu_{p^m}^{\otimes \star})\] such that the following diagram commutes
\begin{equation}\xymatrix{
\bb Z(1)^\bb A/p^m\ar[r] & L_\sub{\'et}\bb Z(1)^\bb A/p^m\ar[r]^{\delta_\sub{mot}^\sub{\'et}\{1\}} & R\Gamma_\sub{\'et}(-,\mu_{p^m})\\
&R\Gamma_\sub{Nis}(-,\bb G_m)[-1]\ar[ul]^{c_1^{\bb A}}\ar[ur]_{c_1^\sub{\'et}}&
}\label{eqn_1st_Chern_compatability_et}\end{equation}
Moreover, $\delta_\sub{mot}^\sub{\'et}\{\star\}$ is an equivalence.
\end{proposition}
\begin{proof}
The proof will consist of three steps: first we construct a multiplicative comparison isomorphism $\delta_\sub{mot}^\sub{\'et}$, secondly we prove that it makes diagram (\ref{eqn_1st_Chern_compatability_et}) commute, and thirdly we show it is the only such comparison map with this property.

{\bf Step 1.} The key idea is that we will pass through $K$-theory in order to relate motivic cohomology to \'etale cohomology. We consider the following maps of abelian presheaves on $\mathfrak{BL}^p_{1/p}$:
\begin{equation}
H^0_\bb A(-,\bb Z/p^m(j))\longleftarrow \pi_{2j}(\Fil_\bb A^j \KH/p^m)\To \pi_{2j}(\KH/p^m) \stackrel{\cong}{\longleftarrow} \pi_{2j}(\K/p^m)\To a_\sub{\'et}\pi_{2j}(\K/p^m) \stackrel{\lambda_j\cong }{\longleftarrow}\mu_{p^m}^{\otimes j}
\label{eqn:BL}
\end{equation}
(where the first two maps are the induced by the canonical maps $\bb Z(j)^\bb A[2j]\leftarrow\Fil_\bb A^j\KH\to\KH$). The isomorphism in the middle is the fact that $\K/p^m=\KH/p^m$ on $\bb Z[\tfrac1p]$-schemes \cite[Proposition 1.6]{Weibel1989a}. Recall that $\lambda_j$ was defined and shown to be an isomorphism in Lemma \ref{lemma_Gabber_Suslin_etc}.

We claim that the first two maps in (\ref{eqn:BL}) are isomorphisms \'etale locally. To this end, we check that they are isomorphisms on $\roi_{X,x}^\sub{sh}$, for any $X\in\mathfrak{BL}^p_{1/p}$ and $x\in X$; recall that $\roi_{X,x}^\sub{sh}$ also lies in $\mathfrak{BL}^p_{1/p}$, by Lemma \ref{lemma_BL_properties}(3). The proof is similar to the final paragraph of proof of Theorem \ref{thm:classical_BL_equiv}. Indeed, conditions (B) 
and (H) imply that $\Fil^j_\bb A\KH(\roi_{X,x}^\sub{sh})/p^m$ is $2j$-connective; so the canonical map $\Fil^j_\bb A\KH(\roi_{X,x}^\sub{sh})/p^m\to \bb Z(j)^{\bb A}(\roi_{X,x}^\sub{sh})/p^m[2j]$ has fibre which is $2j+2$-connective, and in particular it is an isomorphism in degree $2j$. But it also follows from (B) that the map $\Fil_\bb A^j \KH(\roi_{X,x}^\sub{sh})/p^m)\to \KH(\roi_{X,x}^\sub{sh};\bb Z/p^m)$ has cofibre homologically supported in degrees $\le 2j-2$, so that it is isomorphism in degree $2j$. This completes the proof of the claimed isomorphisms.

In conclusion, \'etale sheafifying the maps of abelian presheaves in (\ref{eqn:BL}) yields isomorphisms of abelian \'etale sheaves \begin{equation}a_\sub{\'et}H^0_\bb A(-,\bb Z/p^m(j))\cong \mu_{p^m}^{\otimes j}\label{eqn:BL2}\end{equation} on $\frak{BL}^p_{1/p}$ for each $j\ge0$; by construction these isomorphisms are compatible as $m$ varies.

We claim that when we vary $j$ the isomorphisms (\ref{eqn:BL2}) are compatible with the multiplicative structures on each side.\footnote{Note that the isomorphism from~\eqref{eqn:BL2} passes through homotopy groups of mod-$p^m$ reductions of spectra which might not even admit any multiplicative structure.} It is enough to check this on stalks, or even on fields by rigidity (note that, for $X\in\mathfrak{BL}^p_{1/p}$ and $x\in X$, a consequence of the isomorphism (\ref{eqn:BL2}) is that $H^0_\bb A(\roi_{X,x}^\sub{sh},\bb Z/p^m(j))\isoto H^0_\bb A(k(x)^\sub{sep},\bb Z/p^m(j))$. So let $k$ be any separably closed field of characteristic $\neq p$. Then evaluating (\ref{eqn:BL}) on $k$ yields a line of maps of abelian groups which receives a map from the following analogous line of abelian groups associated with the $p$-adic theories:
\begin{equation}
H^0_\bb A(k,\bb Z_p(j))\longleftarrow \pi_{2j}(\Fil_\bb A^j \KH(k)_p^\comp)\To \KH_{2j}(k;\bb Z_p) \stackrel\cong{\longleftarrow} \K_{2j}(k;\bb Z_p)\stackrel={\To} \K_{2j}(k;\bb Z_p)\stackrel{\sub{Suslin}\cong}\longleftarrow H^0_\sub{\'et}(k,\bb Z_p(j))
\label{eqn:BL3}
\end{equation}
Similarly to the third paragraph of proof, we claim that the first two maps in (\ref{eqn:BL3}) are isomorphisms. We saw in that paragraph that the fibre of $\Fil^j_\bb A\KH(k)_p^\comp\to \bb Z_p(j)^{\bb A}(k)[2j]$ has the property that modulo $p$ it is $2j+2$-connective; since it is $p$-complete therefore it is $2j+1$-connective, and so the first map in (\ref{eqn:BL3}) is an isomorphism. We similarly saw that the cofibre of $\Fil_\bb A^j \KH(k)_p^\comp\to \KH(k)_p^\comp$ is homologically supported in degrees $\le 2j-2$ (mod $p$, whence $p$-adically), so that the second map in (\ref{eqn:BL3}) is also an isomorphism. In conclusion, there is a commutative diagram
\[\xymatrix{
H^0_\bb A(k,\bb Z/p^m(j))\ar@{-}[r]^-{\sub{(\ref{eqn:BL2})}\cong}& \mu_{p^m}^{\otimes j}\\
H^0_\bb A(k,\bb Z_p(j))\ar[u]\ar@{-}[r]^-\cong& H^0_\sub{\'et}(k,\bb Z_p(j))\ar[u]
}\]
in which the top isomorphism is obtained by evaluating (\ref{eqn:BL2}) on $k$ (which we are trying to show is multiplicative as $j$ varies), the bottom isomorphism is the composition of the isomorphisms (\ref{eqn:BL3}), and the vertical arrows are the natural ones. Since the right vertical arrow is surjective, so is the left vertical arrow, and the question of multiplicativity reduces to checking that the bottom isomorphism is multiplicative as $j$ varies. But this is now formal: $p$-adic completion is lax monoidal, so all terms in (\ref{eqn:BL3}) carry multiplicative structures and all the maps are compatible with these multiplicative structures. This completes the proof that the isomorphisms of abelian \'etale sheaves (\ref{eqn:BL2}) are compatible with multiplication.

Next note that the map of abelian presheaves $H^0_\bb A(-,\bb Z/p^m(j))\to \pi_0(L_\sub{\'et}\bb Z(j)^\bb A)/p^m)$ is an isomorphism \'etale locally (indeed, this holds for any presheaf in place of $\bb Z(j)^\bb A/p^m$). Finally, conditions (B) and (G) imply that $L_\sub{\'et}\bb Z(j)^\bb A/p^m$ is discrete as an \'etale sheaf on $\frak{BL}^p_{1/p}$, so that there is a natural equivalence $L_\sub{\'et}\bb Z(j)^\bb A/p^m\simeq R\Gamma_\sub{\'et}(-,a_\sub{\'et}\pi_0(L_\sub{\'et}\bb Z(j)^\bb A)/p^m))$ of \'etale sheaves; this equivalence is moreover compatible with the multiplicative structure as $j$ varies, since the inclusion of abelian \'etale sheaves into \'etale sheaves is lax monoidal.

Assembling (\ref{eqn:BL2}) with the isomorphism and equivalence of the previous paragraph, we have established equivalences $\delta_\sub{mot}^\sub{\'et}\{j\}:L_\sub{\'et}\bb Z(j)^\bb A/p^m\isoto R\Gamma_\sub{\'et}(-,\mu_{p^m}^{\otimes j})$ of \'etale sheaves on $\frak{BL}_p$, compatible with the multiplicative structure as $j$ varies.

{\bf Step 2.} We now claim that the diagram (\ref{eqn_1st_Chern_compatability_et}) commutes. It is enough to check this \'etale locally on $H^0$; by rigidity of \'etale cohomology we even reduce to separably closed fields $k$ of characteristic $\neq p$. In that case the desired commutativity states that the composition \[k^\times[p^m]\xto{c_1^\bb A}H^0_\bb A(k,\bb Z/p^m(1))\xto{\sub{(\ref{eqn:BL2})}\cong}H^0_\sub{\'et}(k,\mu_{p^m})\] is the identity map, which is indeed true essentially because both $c_1^\bb A$ and $\lambda_1$ are compatible with the determinant map; the details are as follows. We compare construction (\ref{eqn:BL}) when $j=1$ to analogous maps for $p^m$-torsion in $\pi_1$ of the integral theories:
\[\xymatrix@C=7mm{
k^\times[p^m]\ar[r]^-{c_1^\bb A}\ar[d]_= & H^0_\bb A(k,\bb Z(1)/p^m)\ar[d]_{(1)} & \pi_{2}(\Fil_\bb A^1 \KH(k)/p^m)\ar[l]_{\cong}\ar[d]_{(2)}\ar[r]^{\cong} \ar[d]& \pi_{2}(\KH(k)/p^m)\ar[d]_{(3)} & \pi_{2}(\K(k)/p^m)\ar[l]_{\cong} \ar[r]^-{\lambda_1\cong }\ar[d]_{(4)}&\mu_{p^m}^{\otimes j}(k)\\
k^\times[p^m]\ar[r]_-{c_1^\bb A} & H^1_\bb A(k,\bb Z(j))[p^m] & (\pi_{1}\Fil_\bb A^1 \KH(k))[p^m]\ar[l]\ar[r] & \KH_1(k)[p^m] & \K_1(k)[p^m]\ar[l]_{\cong} \ar[ur]_{\sub{det}}&
}\]
The vertical maps in this diagram are the natural ones appearing in the Bockstein sequences for cohomology/homotopy with finite coefficients; in particular all the squares in the diagram commute. The triangle commutes by definition of $\lambda_1$. We saw when discussing (\ref{eqn:BL}) that the horizontal maps on the top row are isomorphisms. We now claim that the surjective maps (1)--(4) are all isomorphisms. To prove that (1) is an isomorphism it is enough to show that $H^0_\bb A(k,\bb Z(1))$ is $p$-divisible; but this group is both $p$-torsion free (as $H^{-1}_\bb A(k,\bb Z/p(1))=0$ by axiom (B), as fields belong to $\frak{BL}^p$ by Theorem \ref{thm:classical_BL_equiv}) and torsion (as $H^0_\bb A(k,\bb Q(1))=\K_2(k)^{(1)}_\bb Q$ by Theorem \ref{thm:rational}, which vanishes by symbolic generation of $\K_2(k)$), hence $p$-divisible. From the commutative square involving maps (1) and (2) it now follows that (2) is also injective, whence an isomorphism. Finally, (3) and (4) are isomorphisms because $\K_1(k)=k^\times$ is $p$-divisible.

Therefore, apart from $c_1^\bb A$, we deduce that the other two remaining horizontal maps in the diagram are also isomorphisms. So, to prove the desired commutativity (namely that the composition from the top left of the diagram to the top right, via the top row, is the identity), we may instead show that the composition via the bottom row is the identity. But this follows from the integral description of $c_1^{\bb A}$ in Remark \ref{rem_A_1st_Chern_class}; in particular the commutative diagram of~\eqref{eq:c1-det2}.

{\bf Step 3.} It is now easy to prove the uniqueness part of the theorem. Indeed, by composing with the inverse of $\delta_\sub{mot}^\sub{syn}\{\star\}$, any multiplicative comparison map $L_\sub{\'et}\bb Z(\star)^{\bb A}/p^m\To R\Gamma_\sub{\'et}(-,\mu_{p^m}^{\otimes j})$ making (\ref{eqn_1st_Chern_compatability_et}) commute induces a multiplicative endomorphism of $R\Gamma_\sub{\'et}(-,\mu_{p^m}^{\otimes \star})$ given by the identity when $\star=1$, hence given by the identity.
\end{proof}

\begin{remark}\label{rem_nis_version_of_delta}
It was convenient to formulate and prove Proposition \ref{prop:graded} in terms of the \'etale sheafification of $\bb A^1$-motivic cohomology, but (to be compatible with the next theorem) we will also allow ourselves to write $\delta_\sub{mot}^\sub{\'et}\{j\}$ for the composition \[\bb Z(j)^{\bb A}/p^m\To L_\sub{\'et}\bb Z(j)^{\bb A}/p^m\xto{\delta_\sub{mot}^\sub{\'et}\{j\}\cong} R\Gamma_\sub{\'et}(-,\mu_{p^m}^{\otimes j})\] on $\frak{BL}^p_{1/p}$. Note that this motivic-to-\'etale comparison map is an equivalence in degrees $\le j$ by axiom (C).
\end{remark}

We are now prepared to state and prove our full Beilinson--Lichtenbaum comparison theorem, identifying the $\bb A^1$-motivic cohomology of schemes in $\frak{BL}^p$ with their Beilinson--Lichtenbaum cohomology in a highly structured manner:

\begin{theorem}\label{thm:axiomatic_BL}
Let $m\ge1$. There is a unique map of $\bb E_\infty$-algebras in graded presheaves of complexes on $\mathfrak{BL}^p$ \[\delta_\sub{mot}^\sub{BL}\{\star\}:\bb Z(\star)^{\bb A}/p^m\To \bb Z_p(\star)^\sub{BL}/p^m\] such that the diagram \begin{equation}\xymatrix{
\bb Z(1)^\bb A/p^m\ar[rr]^{\delta_\sub{mot}^\sub{BL}\{1\}} & & \bb Z(1)^\sub{BL}/p^m\\
&R\Gamma_\sub{Nis}(-,\bb G_m)[-1]\ar[ul]^{c_1^{\bb A}}\ar[ur]_{c_1^\sub{syn}}&
}\label{eqn_1st_Chern_compatability}\end{equation}
commutes. Moreover, $\delta_\sub{mot}^\sub{BL}\{\star\}$ is an equivalence.
\end{theorem}
\begin{proof}
We consider the following diagram of abelian presheaves on $\frak{BL}^p$:
\[\xymatrix{
&\bb G_m^{\otimes j}/p^m\ar[dl]_{c_1^\bb A}\ar[dr]^{c_1^\sub{syn}}&\\
H^j_{\bb A}(-,\bb Z/p^m(j))\ar[d]&&H^j_\sub{syn}(-,\bb Z/p^m(j))\ar[d]^{\gamma^\sub{\'et}_\sub{syn}\{j\}}\\
H^j_{\bb A}(-[\tfrac1p],\bb Z/p^m(j))\ar[rr]^{\cong}_-{\delta_\sub{mot}^\sub{\'et}\{j\}}&&H^j_\sub{\'et}(-[\tfrac1p],\bb Z/p^m(j))
}\]
The diagonal arrows are those induced by the respective first Chern class maps and by multiplicativity. The left vertical arrow is the canonical one. The right vertical arrow is the syntomic-to-\'etale comparison map. For the horizontal isomorphism see Remark \ref{rem_nis_version_of_delta}.

The diagram commutes when $j=1$ because $\delta_\sub{mot}^\sub{\'et}\{1\}$ was shown to be compatible with the first \'etale Chern class map, which itself is compatible with the first syntomic Chern class map; so the diagram commutes for all $j\ge0$ by multiplicativity.

Next we claim that Nisnevich locally the diagonal arrows are surjective and have the same kernel. At points of residue characteristic $\neq p$ this is clear, as then both vertical arrows are isomorphisms and  $H^j_\sub{\'et}(-,\bb Z/p^m(j))$ is Nisnevich locally generated by symbols by Bloch--Kato (and local rigidity of \'etale cohomology). Meanwhile, at points of residue characteristic $p$ this follows from condition (E) and Theorem \ref{thm_syntomic_Milnor}, which state that Nisnevich locally both arrows identify the targets with mod-$p^m$ improved Milnor $K$-theory of the corresponding henselian local ring. Having proved the claim, there is therefore an induced isomorphism $a_\sub{Nis}H^j_{\bb A}(-,\bb Z/p^m(j))\isoto a_\sub{Nis}H^j_\sub{syn}(-,\bb Z/p^m(j))$ of Nisnevich abelian sheaves on $\frak{BL}^p$, which on $\frak{BL}^p_{1/p}$ is the same as the isomorphism induced by $\gamma_\sub{syn}^\sub{\'et}\{j\}^{-1}\circ \delta_\sub{mot}^\sub{\'et}\{j\}$. Furthermore, by construction these isomorphisms are multiplicative as $j$ varies.

So we have constructed, for each $j\ge0$, a commutative diagram of presheaves of complexes on $\frak{BL}^p$:
\[\xymatrix@C=1mm{
&\tau^{\le j}\bb Z(j)^\bb A(-[\tfrac1p])\ar[d]\ar@/^/[drr]^{\delta_\sub{mot}^\sub{\'et}\{j\}}_\cong&&\\
a_\sub{Nis}H^j_\bb A(-,\bb Z/p^m(j))[-j]\ar[r]\ar@/_/[drr]_\cong & H^j_{\bb A}(-[\tfrac1p],\bb Z/p^m(j))[-j] \ar[drr]^{\delta_\sub{mot}^\sub{\'et}\{j\}}_\cong&& \tau^{\le j}R\Gamma_\sub{\'et}(-[\tfrac1p],\mu_{p^m}^{\otimes j})\ar[d]\\
&&a_\sub{Nis}H^j_\sub{syn}(-,\bb Z/p^m(j))[-j]\ar[r] & H^j_\sub{\'et}(-[\tfrac1p],\mu_{p^m}^{\otimes j}) [-j]
}\] 

After applying $L_\sub{Nis}$, the pullback of the top left part of the diagram is precisely $\bb Z(j)^\bb A/p^m$; this follows easily from conditions (C) and (D). Similarly, $L_\sub{Nis}$ of the pullback of the bottom right part of the diagram is $\bb Z(j)^\sub{BL}/p^m$, by condition (F). We have thus defined the desired equivalences \[\delta_\sub{mot}^\sub{BL}\{j\}:\bb Z(j)^\bb A/p^m\quis \bb Z(j)^\sub{BL}/p^m\] of presheaves on $\frak{BL}^p$. Varying $j$, these assemble to an equivalence of $\bb E_\infty$-algebras in graded presheaves, since the multiplicativity is preserved by pulling back and by $L_\sub{Nis}$.

The compatibility of $\delta_\sub{mot}^\sub{BL}\{1\}$ with the corresponding first Chern classes follows from the original commutative diagram of the proof.

The uniqueness claim of the theorem is clear from the proof: any other comparison map, when composed with the inverse of $\delta_\sub{mot}^\sub{BL}\{\star\}$, would yield an endomorphism of $\bb Z(\star)^\bb A/p^m$ which is locally the identity in degrees $< j$ (using condition (D) to reduce to the uniqueness assertion of Proposition \ref{prop:graded}) and in degree $j$ (where motivic cohomology is generated by symbols by condition (E)), hence the identity by the pullback argument in the course of the proof above.
\end{proof}

\subsection{Beilinson--Lichtenbaum motivic spectrum $H\bb Z_p^\sub{syn}$: base change}\label{ss_BL2}
In this subsection we introduce the Beilinson--Lichtenbaum motivic spectrum and prove a base change theorem for it. This will allow us in \S\ref{ss:slice-dedekind} to bootstrap results from fields to Dedekind domains; for example we will prove in Theorem~\ref{theorem_BL1} that schemes which are smooth over a Dedekind domain are in $\mathfrak{BL}^p$ for all primes $p$. As a historical comment, we mention that a similar base change result is  proved by Spitzweck \cite[\S8]{Spitzweck2018} while constructing his motivic cohomology spectrum; his arguments via Bloch--Kato's filtration are implicitly captured in our context by the fact that syntomic cohomology is left Kan extended from smooth $\bb Z$-algebras, or more precisely by Proposition \ref{prop:syn-trunc-lke}(2) below.

We first explain how to view Beilinson--Lichtenbaum cohomology from the point of view of motivic homotopy theory. Given any fixed field or mixed characteristic Dedekind domain $B$, we consider the restriction of Beilinson--Lichtenbaum cohomology $\bb Z_p(\star)^\sub{BL}[2\star]$ to smooth $B$-schemes, together with the class $c_1^\sub{BL}(\roi(1))\in H^2_\sub{BL}(\bb P_B^1,\bb Z_p(1))$; thanks to Theorem \ref{thm_BL_coh} we see that this forms an $\bb E_\infty$-algebra in the category of $\bb A^1$-invariant graded Nisnevich sheaves of complexes, satisfying the $\bb P^1$-bundle formula; that is, as in Remark \ref{remark:traditional}, it defines an object $\bb Z_{p,B}^\sub{BL}$ of the category $\CAlg(\Gr\opp{Sh}_{\Nis, \bb A^1}(\Sm_B,\Spt)_{c,\sub{pbf}})$ of Construction \ref{cons:lin-cohom-theories}. Applying the motivic Eilenberg--MacLane construction $H$ of Definition \ref{definition:graded-p1-mult} thus defines an $\bb E_\infty$-algebra \[H\bb Z_{p,B}^\sub{BL}\in\CAlg(\SH(B))\] such that, for each $j\in\bb Z$, the presheaf $H\bb Z_{p,B}^\sub{BL}(j)=\map_{\SH(B)}(M_B(-), H\bb Z_{p,B}^\sub{BL}\otimes\bb T_B^{\otimes j}[-2j])$ on $\Sm_B$ agrees with $\bb Z_p(j)^\sub{BL}|_{\Sm_B}$. This identification is moreover compatible with the bigraded multiplicative structure on cohomology groups, by Remark~\ref{rem:recover}(1). We will call $H\bb Z_{p,B}^\sub{BL}$ the {\em $p$-adic Beilinson--Lichtenbaum motivic spectrum}. Note that $H\bb Z_{p,B}^\sub{BL}$ is $p$-complete as an object of $\SH(B)$.\footnote{Indeed, in the notation of Proposition~\ref{prop:P1-spectrum-from-graded}, the equivalence $(\omega^{\infty,\gr}_\otimes)^{-1}$ and the forgetful functor both commute with limits.}

Similarly, applying $H$ to $\bb F_p(\star)^\sub{BL}[2\star]$, together with the mod-$p$ reduction of the class $c_1^\sub{BL}(\roi(1))$, defines the {\em mod-$p$ Beilinson--Lichtenbaum motivic spectrum} $H\bb F_{p,B}^\sub{BL}\in\SH(B)$, which is an $\bb E_\infty$-algebra equipped with an equivalence of $1_B$-pointed motivic spectra $H\bb F_{p,B}^\sub{BL}\simeq H\bb Z_{p,B}^\sub{BL}/p$; in particular this equips $H\bb Z_{p,B}^\sub{BL}/p$ with an $\bb E_\infty$-algebra structure. 

The proof of the base change theorem below will require the following properties of the Beilinson--Lichtenbaum truncation of syntomic cohomology, which hold even before Nisnevich sheafification:

\begin{proposition}\label{prop:syn-trunc-lke}
Let $j\in\bb Z$.
\begin{enumerate}
\item Rigidity: For any ring $R$ and henselian ideal $I\subset R$, the relative syntomic cohomology $\bb F_p(j)^\sub{syn}(R,I)$ is supported in cohomological degrees $\le j$.
\item Left Kan extended: For any base ring $B$, the functor $\tau^{\le j}\bb F_p(j)^\sub{syn}:{\rm CAlg}_B\to\D(\bb Z)$ is left Kan extended from smooth $B$-algebras.
\end{enumerate}
\end{proposition}
\begin{proof}
For part (1), we use Corollary \ref{corol_syn_rigid1} to assume that $R$ is $p$-henselian; by taking a filtered colimit using Proposition \ref{prop:syn}, we may also assume that $R$ is Noetherian. We then use Lemma \ref{lemma_syn_p-hens}(2), for both $R$ and $R/I$, to identify $\bb F_p(j)^\sub{syn}(R,I)$ with the fibre of the map \[\bb F_p(j)^\sub{syn}(R_p^{\comp})\to \bb F_p(j)^\sub{syn}((R/I)_p^{\comp}).\] But since $R$ is Noetherian we have $(R/I)_p^{\comp}=R_p^{\comp}/IR_p^{\comp}$, where $IR_p^{\comp}$ is an Henselian ideal of $R_p^{\comp}$ since $I$ generates a Henselian ideal of $R_p^{\comp}/pR_p^{\comp}=R/pR$. We may therefore appeal to the rigidity theorem for syntomic cohomology of $p$-complete rings \cite[Theorem~5.3]{AntieauMathewMorrowNikolaus2022} to see that the fibre is supported in degrees $\le j$.

Part (1) shows in particular that the functor $\tau^{>j}\bb F_p(j)^\sub{syn}:{\rm CAlg}_B\to\D(\bb Z)$ is rigid, hence left Kan extended from smooth $B$-algebras by Lemma~\ref{lemma_rigid_implies_lke}. The proof is now completed by recalling that $\bb F_p(j)^\sub{syn}$ itself is left Kan extended from smooth $B$-algebras by Proposition \ref{prop:syn}(2).
\end{proof}

The following is the desired base change property (in which the most important case to have in mind is that $B$ is a mixed characteristic Dedekind domain and $k$ is a field):

\begin{theorem}\label{thm_syntomic_SH_base_change}
Let $B$ be a mixed characteristic Dedekind domain or a field; also let $k$ be a mixed characteristic Dedekind domain or a field, and let $i:B\to k$ be a map of rings. Then there exists an equivalence 
$\bb E_\infty$-algebras \[(i^*H\bb Z_{p,B}^\sub{BL})_p^\comp\simeq H\bb Z_{p,k}^\sub{BL}\] in $\SH(k)$.\footnote{At present we are not claiming that this equivalence is canonical in any sense. However, once Proposition \ref{prop_1_to_BL} has been established, it will follow that there is in fact a unique such equivalence.}
\end{theorem}
\begin{proof}
We will use the commutative diagram
\[\xymatrix{
\Gr\opp{Sh}_{\Nis, \bb A^1}(\Sm_k,\Spt)_\sub{c,pbf}\ar[r]^-H&\SH(k)\\
\Gr\opp{Sh}_{\Nis, \bb A^1}(\Sm_k,\Spt)_\sub{c}\ar[u]^{L_\sub{pbf}} & \\
\Gr\opp{Sh}_{\Nis, \bb A^1}(\Sm_B,\Spt)_\sub{c,pbf}\ar[r]_-H\ar[u]^{L_{\Nis,\bb A^1}L^\sub{sm}}&\SH(B)\ar[uu]_{i^*}
}\]
of Remark \ref{rem:pullbackQ}.

First, by naturality of Beilinson--Lichtenbaum cohomology there is a comparison map $L_{\sub{Nis}}L^\sub{sm}\bb Z_{p,B}^\sub{BL}\to \bb Z_{p,k}^\sub{BL}$ of $\bb E_\infty$-algebras in $\Gr\opp{Sh}_{\Nis, \bb A^1}(\Sm_k,\Spt)_\sub{c}$, which is an equivalence modulo $p$ by Propositions \ref{prop:lke-restrict} and \ref{prop:syn-trunc-lke}. Since the target satisfies $\bb A^1$-invariance and the $\bb P^1$-bundle formula by Theorem \ref{thm_BL_coh}, there is then an induced map $L_\sub{pbf}L_{\sub{Nis},\bb A^1}L^\sub{sm}\bb Z_{p,B}^\sub{BL}\to \bb Z_{p,k}^\sub{BL}$ which is again an equivalence modulo $p$ (since we have not changed the domain modulo $p$). By now applying the commutative diagram we obtain a comparison map of $\bb E_\infty$-algebras $i^*H\bb Z_{p,B}^\sub{BL}\to H\bb Z_{p,k}^\sub{BL}$ in $\SH(k)$, which is an equivalence modulo $p$. This completes the proof since the target is $p$-complete.
\end{proof}

\subsection{Voevodsky's slice conjectures over Dedekind domains and the Beilinson--Lichtenbaum equivalence}\label{ss:slice-dedekind}
We now apply the results of \S\ref{ss_coherent_BL}--\ref{ss_BL2} to Voevodsky's conjecture relating the motivic sphere and $\KGL$, in the special case of Dedekind domains and fields. We also show that all schemes which are smooth over a Dedekind domain are in $\mathfrak{BL}^p$. We begin by relating the motivic sphere to Beilinson--Lichtenbaum cohomology:

\begin{proposition}\label{prop_1_to_BL}
Let $B$ be a mixed characteristic Dedekind domain or a field. Then, for any prime number $p$, the unit map $1_B\to H\bb Z_{p,B}^\sub{BL}$ in $\SH(B)$ induces an equivalence $s^0(1_{B})_p^\comp\quis H\bb Z_{p,B}^\sub{BL}$.
\end{proposition}
\begin{proof}
Suppose first that $B=k$ is a field. Combining Theorems \ref{thm:classical_BL_equiv} and \ref{thm:axiomatic_BL}, and taking the limit as $m\to\infty$, yields a map $\bb Z(\star)^{\bb A}\to \bb Z_p(\star)^\sub{BL}$ of $\bb E_\infty$-algebras in graded presheaves of complexes on $\Sm_k$, compatibly with first Chern classes, which is an equivalence modulo any power of $p$. By shearing and then applying the functor $H$ (and using Remark \ref{rem:recover}(2)), we obtain an equivalence of $\bb E_\infty$-algebras $s^0(\KGL_k)_p^\comp\simeq H\bb Z_{p,k}^\sub{BL}$ in $\SH(k)$. We also know, by Levine's theorem (see Remark \ref{rem_Levine}), that the unit map $1_k\to\KGL_k$ induces an equivalence on $0$-slices. Composing shows that the unit map of $H\bb Z_{p,k}^\sub{BL}$ indeed induces $s^0(1_B)_p^\comp\quis H\bb Z_{p,k}^\sub{BL}$, as desired.

Now we treat the case that $B$ is a mixed characteristic Dedekind domain. It is enough to prove the desired equivalence modulo $p$. Note first that $H\bb F_{p,B}^\sub{BL}$ is effective: indeed, this may be checked by pulling back to fields by Proposition~\ref{prop:detect-effective}, where it follows from Theorem \ref{thm_syntomic_SH_base_change} and the fact that the proposition has already been proved over fields. Next, $\Fil^1_\sub{slice}H\bb F_{p,B}^\sub{BL}=0$ since Beilinson--Lichtenbaum cohomology vanishes in negative weights (Lemma \ref{lemma_BL_mod_p}). So it remains to prove that the mod-$p$ reduction of the fibre of the unit map $1_{B}\to H\bb Z_{p,B}^\sub{BL}$ is $1$-effective. But this may be checked by pulling back to fields by Proposition~\ref{prop:detect-effective} again, where it again follows from Theorem \ref{thm_syntomic_SH_base_change} and the treated case of fields.
\end{proof}

\begin{corollary}\label{corol_base_change_of_unit}
Let $B$ be a mixed characteristic Dedekind domain or a field; also let $k$ be a mixed characteristic Dedekind domain or a field, and let $i:B\to k$ be a map of rings. Then the canonical map of $\bb E_\infty$-algebras \[i^*(s^01_B)\To s^01_k\] in $\SH(k)$ is an equivalence.
\end{corollary}
\begin{proof}
The map is a rational equivalence by Corollary \ref{cor:rational-smd}, so it remains to prove that it is an equivalence modulo $p$ for all prime numbers $p$. But that follows by combining Theorem \ref{thm_syntomic_SH_base_change} and Proposition \ref{prop_1_to_BL}.
\end{proof}

Using the absolute effective motivic spectrum $\cal V$ of \S\ref{subsub:V}, we may now relate the motivic sphere to $\KGL$; the following theorem is not new when $B$ is a field, and a version of it was proved by the first author via Spitzweck's version of motivic cohomology \cite{Bachmann2022}.
\begin{proposition}\label{prop:specz_slices}
Let $B$ be a mixed characteristic Dedekind domain or a field. Then the following hold in $\SH(B)$.
\begin{enumerate}
\item The map to the zero-th slice $\scr V_{B}/\beta \rightarrow s^0(\scr V_{B}/\beta)$ is an equivalence.
\item The map $\scr V_{B} \rightarrow \kgl_{B}$ of \eqref{eq:fr-to-kgl} is an equivalence.
\item The unit map $1_{B} \rightarrow \KGL_{B}$ induces an equivalence $s^0(1_{B}) \xrightarrow{\simeq} s^0(\KGL_{B})$
\end{enumerate}
\end{proposition}
\begin{proof}
In the case of a field $k$, part (3) is Levine's theorem (see Remark \ref{rem_Levine}), part (2) is \cite[Corollary 5.2]{HoyoisJelisiejewNardinTotaroYakerson2021}, and then part (1) follows since $\kgl_k/\beta\simeq s^0(\kgl_k)$ is a zero slice (as recalled in \S\ref{subsub:slices-kgl}). So now let $B$ be a mixed characteristic Dedekind domain.

First note that for any qcqs scheme $X$, the unit map $s^0(1_X)\to s^0(\cal V_X/\beta)$ is an equivalence: indeed, the map $1_X\to\cal V_X$ induces an equivalence on zero slices by Proposition \ref{prop_1_vs_V}, and the same is true of $\cal V_X\to\cal V_X/\beta$ since the cofibre $\cal V_X\otimes\bb T_X$ is $1$-effective.

We may now prove part (1) for $B$. Since equivalences may be checked on points, it is enough to show that $i^*(\scr V_{B}/\beta) \rightarrow i^*s^0(\scr V_{B}/\beta)$ is an equivalence for all fields $i:B\to k$. But the left hand side is $\cal V_k/\beta$, and we claim the right hand side is $s^0(\cal V_k/\beta)$; indeed, this follows by applying the previous paragraph to both $B$ and $k$ and then appealing to Corollary \ref{corol_base_change_of_unit}. So after pulling back to $k$ the map obtained is the canonical one $\cal V_k/\beta\to s^0(\cal V_k/\beta)$, which is an equivalence by part (1) for fields.

(2) now follows by the same argument as in \cite[Theorem 1.1(1)]{Bachmann2022}, which we briefly recall.
We know that $\KGL_B \wequi \cal V_B[\beta^{-1}]$ and so it suffices to show that $\Fil^0_\sub{slice} \bb T^{\otimes j} \otimes \cal V_B \xrightarrow{\beta} \Fil^0_\sub{slice} \bb T^{\otimes j-1} \otimes \cal V_B$ is an equivalence for every $j \le 0$.
But the cofiber of this map is $\Fil^0_\sub{slice} \bb T^{\otimes j-1} \otimes \scr V_B/\beta$, which vanishes by (1).

(3) clearly follows from Proposition \ref{prop_1_vs_V} and part (2).
\end{proof}

We reach the main theorem of the section, describing $\bb A^1$-motivic cohomology with finite and $p$-adic coefficients over Dedekind domains in terms of Beilinson--Lichtenbaum cohomology:

\begin{theorem}\label{theorem_BL1}
Let $B$ be a mixed characteristic Dedekind domain or a field, and let $p$ be any prime number.
\begin{enumerate}
\item There is a unique map of $\bb E_\infty$-algebras $s^0(\KGL_B)\to H\bb Z_{p,B}^\sub{BL}$ in $\SH(B)$; after $p$-adic completion it is an equivalence \[s^0(\KGL_B)_p^\comp\quis H\bb Z_{p,B}^\sub{BL}.\]
\item All smooth $B$-schemes belong to the class $\frak{BL}^p$.
\item On the category of smooth $B$-schemes, for each $m\ge1$ there is a unique map of $\bb E_\infty$-algebras in graded presheaves of complexes $\bb Z(\star)^{\bb A}/p^m\to \bb Z_p(\star)^\sub{BL}/p^m$ which is compatible with first Chern classes in the sense of \eqref{eqn_1st_Chern_compatability}; moreover, the map is an equivalence.
\end{enumerate}
\end{theorem}
\begin{proof}
Part (1) follows from Propositions \ref{prop_1_to_BL} and \ref{prop:specz_slices}(3).


(2): We assume that $B$ is a Dedekind domain, since the case of fields was already treated in Theorem \ref{thm:classical_BL_equiv}. We must check conditions (A)--(H) for all smooth $B$-schemes. The completeness condition (H) was already explained in Remark \ref{remark_completeness_of_fil}(1), and condition (F) follows from Theorem \ref{theorem_syntomic_properties}(8).

Next we apply $\omega^{\infty,\gr}$ to the equivalence from part (1) to obtain an equivalence \begin{equation}\bb Z(\star)^{\bb A}/p^m\quis \bb Z_p(\star)^\sub{BL}/p^m\label{eqn:mot_vs_BL_not_mult}\end{equation} of graded presheaves of complexes. (Warning: although both sides are $\bb E_\infty$-algebras in graded presheaves of complexes, we saw in \S\ref{ss:omega} that $\omega^{\infty,\gr}$ is not lax symmetric monoidal, so a priori it does not output an equivalence compatible with the $\bb E_\infty$-algebra structures.) Therefore checking conditions (A)--(D) and (G) reduce to the analogous statements with $\bb Z_p(j)^\sub{BL}/p^m$ in place of $\bb Z_p(j)^\bb A/p^m$. Condition (A) is clear by definition of Beilinson--Lichtenbaum cohomology; conditions (B) and (G) follow from Remark \ref{example_BL_cohomology}; condition (C) follows from Lemma \ref{lemma_BL_mod_p}(4); condition (D) becomes condition (F), which we already explained follows from Theorem \ref{theorem_syntomic_properties}(8).

It remains to prove (E), which requires upgrading \eqref{eqn:mot_vs_BL_not_mult} to being mildly multiplicative and compatible with first Chern classes on units. Firstly, Remark \ref{rem:recover}(1) implies that \eqref{eqn:mot_vs_BL_not_mult} upgrades to an equivalence of $\bb E_1$-algebras in graded presheaves of complexes, and that it is moreover compatible with first Chern classes of $\roi(1)$ in the sense that the induced isomorphism $H^2_{\bb A}(\bb P_B^1,\bb Z(1)/p^m)\isoto H^2_\sub{BL}(\bb P_B^1,\bb Z(1)/p^m)$ sends $c_1^{\bb A}(\roi(1))$ to $c_1^\sub{BL}(\roi(1))$. We may now apply Lemma~\ref{lem:units-Chern} to $F=\bb Z(\star)^\bb A/p^m[2\star]$, equipped with two possible first Chern class maps \[c_1:R\Gamma_{\Nis}(-,\Gm)[1] \xto{c_1^{\bb A}} \bb Z(1)^\bb A/p^m[2]\] and \[c_1':R\Gamma_{\Nis}(-,\Gm)[1]\xto{c_1^\sub{BL}} \bb Z_p(1)^\sub{BL}/p^m[2]\stackrel{\eqref{eqn:mot_vs_BL_not_mult}}\simeq \bb Z(1)^\bb A/p^m[2].\] The hypotheses of that lemma are satisfied: (1) by Remark \ref{rem:eff-neg}, (2) by immediately above, and (3) by Theorem \ref{thm_pbf_Acdh}. The lemma therefore implies that the diagram of presheaves of abelian groups on $\Sm_X$
\begin{equation}\label{eq:c1-compatible}
\begin{tikzcd}
 & \bb G_m \ar{dr}{c_1^\sub{BL}} \ar[swap]{dl}{c_1^{\bb A}} & \\
 H_{\bb A}^1(-, \bb Z/p^m(1)) \ar{rr}{\eqref{eqn:mot_vs_BL_not_mult}\cong} & &  H_\sub{BL}^1(-, \bb Z/p^m(1))
\end{tikzcd}
\end{equation}
is commutative. But since \eqref{eqn:mot_vs_BL_not_mult} has been explained to be an equivalence of $\bb E_1$-algebras, it induces an isomorphism of presheaves of bigraded rings $\bigoplus_{i\in\bb Z,\, j\ge0}H_\bb A^i(-,\bb Z/p^m(j))\isoto \bigoplus_{i\in\bb Z,\, j\ge0}H^i_\sub{BL}(-,\bb Z_p/p^m(j))$. Putting all this together, we see that, for any smooth $B$-scheme $X$, point $x\in X$,  and $j\ge1$, the map $(\roi_{X,x}^{\sub{h}\,\times})^{\otimes j}\to H^j_{\bb A}(\roi_{X,x}^{\sub{h}},\bb Z/p^m(j))$ (induced by $c_1^\bb A$ and the multiplicative structure on $\bb A^1$-motivic cohomology) agrees with the map $(\roi_{X,x}^{\sub{h}\,\times})^{\otimes j}\to H^j_\sub{BL}(\roi_{X,x}^{\sub{h}},\bb Z_p/p^m(j))$ (induced by $c_1^\sub{BL}$ and the multiplicative structure on Beilinson--Lichtenbaum cohomology). In this way condition (E) reduces to the analogous assertion in Beilinson--Lichtenbaum cohomology, or equivalently in syntomic cohomology since the two cohomologies agree in degrees $\le$ weight; therefore Theorem \ref{thm_syntomic_Milnor} completes the proof of condition (E).

(3): The existence of the map follows from part (2) thanks to Theorem \ref{thm:axiomatic_BL}. The uniqueness follows from the proof of the latter theorem. Indeed, the proof really showed that, for any full category $\frak C\subset\frak{BL}_p$ closed under \'etale maps (i.e., $X\in \frak C$ and $Y\to X$ \'etale implies $Y\in\frak C$), for example $\frak C=\Sm_B$, there is a unique map of $\bb E_\infty$-algebras $\bb Z(\star)^{\bb A}/p^m\to \bb Z_p(\star)^\sub{BL}/p^m$ on $\frak C$ compatible with first Chern classes.
\end{proof}

We deduce the following form of the Beilinson--Lichtenbaum conjecture:

\begin{corollary}[Hilbert 90 and Beilinson--Lichtenbaum equivalence]\label{corollary:BL_over_B}
Let $B$ be a mixed characteristic Dedekind domain or a field, and $j\ge0$. Then, on the category of smooth $B$-schemes, the presheaf of complexes $\bb Z(j)^\bb A$ is Nisnevich locally supported in cohomological degrees $\le j$, and the canonical map \[\bb Z(j)^\bb A\To L_\sub{Nis}\tau^{\le j+1}L_\sub{\'et}\bb Z(j)^\bb A\] is an equivalence.\footnote{In fact, this equivalence also holds if we replace $L_\sub{Nis}$ by $L_\sub{Zar}$: see \cite[Theorem 1.2(2)]{Geisser2004}. We have not tried to prove this stronger statement using our methods, as we use the Nisnevich topology throughout.}
\end{corollary}
\begin{proof}
For any prime number $p$, we have shown in Theorem \ref{theorem_BL1}(2) that all smooth $B$-schemes belong to the class $\frak{BL}^p$. In particular, for any smooth $B$-scheme $X$ and point $x\in X$, condition (A) states that $\bb Z(j)^\bb A(\roi_{X,x}^\sub{h})/p$ is supported in cohomological degrees $\le j$, and condition (E) shows that $H^j_\bb A(\roi_{X,x}^\sub{h},\bb Z/p(j))$ is generated by symbols, which lift to $H^j_\bb A(\roi_{X,x}^\sub{h},\bb Z(j))$. It follows that $H^n_\bb A(\roi_{X,x}^\sub{h},\bb Z(j))$ is torsion-free for all $n>j$. But these groups are also rationally zero by Theorem \ref{thm:rational} and Soul\'e's vanishing bound of the Adams eigenspaces of $K$-groups of local rings \cite[Corollaire 1]{Soule1985}, so they are zero.

The canonical map $\bb Z(j)^\bb A\to L_\sub{Nis}\tau^{\le j+1}L_\sub{\'et}\bb Z(j)^\bb A$ is now well-defined, and we must show it is an equivalence. It is an equivalence rationally since $\bb Q(j)^\bb A$ is both Nisnevich locally supported in degrees $\le j$ (so also $\le j+1$), by the previous paragraph, and already an \'etale sheaf, since it is a direct summand of the \'etale sheaf $\KH_\bb Q$ by Theorem \ref{thm:rational}(2). Here we have hid the subtlety that rationalisation does not necessarily commute with \'etale sheafification, but it does in this case since $\tau^{<0}\bb Z(j)^\bb A$ takes rational values on $\Sm_B$, as follows from Theorem \ref{theorem_BL1}(3) and the vanishing of negative degree syntomic cohomology on $\Sm_B$; see the proof of Theorem \ref{def:h-eh-mot}(1) for the technical details of the argument in a cdh-local context.

It remains to prove the equivalence modulo $p$ for every prime number $p$. But we claim that ``Hilbert 90'' holds, i.e., for $X$ any smooth $B$-scheme and $x\in X$, we have $H^{j+1}(L_\sub{\'et}\bb Z(j)^\bb A(\roi_{X,x}^\sub{h}))=0$. Once that claim has been proved, it will be enough to show that the canonical map \begin{equation}\bb Z(j)^\bb A/p\to L_\sub{Nis}\tau^{\le j}L_\sub{\'et}(\bb Z(j)^\bb A/p)\label{eqn_preH90}\end{equation} is an equivalence; but \eqref{eqn_preH90} is indeed an equivalence since we may replace $\bb Z(j)^\bb A/p$ by $\bb F_p(j)^\sub{BL}$ using Theorem~\ref{theorem_BL1} and then appeal to Lemma \ref{lemma_BL_mod_p}(4).

It remains to establish the Hilbert 90 vanishing claim, but this follows by repeating arguments we have already made. Indeed, the group is rationally zero since $(L_\sub{\'et}\bb Z(j)^\bb A)_\bb Q=\bb Q(j)^\bb A$ by the second paragraph, which locally vanishes in degrees $>j$ by the first paragraph. But it is also torsion-free, as is seen from the commutative diagram
\[\xymatrix{
H^j_\bb A(\roi_{X,x}^\sub{h},\bb Z(j))\ar[r]\ar[d]  & H^j(L_\sub{\'et}\bb Z(j)^\bb A(\roi_{X,x}^\sub{h}))\ar[d] \\
H^j_\bb A(\roi_{X,x}^\sub{h},\bb Z/p(j))\ar[r] & H^j(L_\sub{\'et}\bb Z(j)^\bb A(\roi_{X,x}^\sub{h})/p)
}\]
where the bottom horizontal arrow is an isomorphism (since \eqref{eqn_preH90} has been shown to be an equivalence), and the left vertical arrow is surjective by the first paragraph, so that the right vertical arrow is also surjective.
\end{proof}

As another corollary, we obtain the following description of motivic cohomology in low weights for schemes which are smooth over a field or Dedekind domain:

\begin{corollary}\label{cor:low-wts-ddk-fields} Let $X$ be a scheme, smooth over a field or a mixed characteristic Dedekind domain. Then the maps
 \[R\Gamma_\sub{Nis}(X, \bb Z) \To \bb Z(0)^{\bb A}(X)\qquad\mathrm{and}\qquad c_1^{\bb A}:R\Gamma_\sub{Nis}(X,\bb G_m)[-1]\To \bb Z(1)^\bb A(X),\] defined in Construction \ref{cons:low-weights} are equivalences.
 \end{corollary}
\begin{proof}
The maps are equivalences rationally by Proposition \ref{prop:low-wts-ration}, so it is enough to check they are equivalences modulo each prime number $p$.

Firstly, on the category of smooth $B$-schemes we consider the maps \[R\Gamma_\sub{Nis}(-, \bb Z/p\bb Z) \To \bb Z(0)^{\bb A}/p\simeq \bb F_p(0)^\sub{BL}\simeq R\Gamma_\sub{Nis}(-,\bb Z/p\bb Z),\] where both equivalences are ones of $\bb E_\infty$-algebras in presheaves of complexes: the first is Theorem \ref{theorem_BL1}(3), and the second follows from Remark \ref{rem_syn_weight_0}. The composition is an endomorphism of $R\Gamma_\sub{Nis}(-, \bb Z/p\bb Z)$ as an $\bb E_\infty$-algebra in presheaves of complexes; but any such endomorphism is an automorphism (by checking on stalks), as required.

Next, for the $\bb G_m$ claim, we use the equivalence of Theorem \ref{theorem_BL1}(3) to reduce the problem to checking that $c_1^\sub{BL}:R\Gamma_\sub{Nis}(X,\bb G_m)/p[-1]\to \bb F_p(1)^\sub{BL}(X)$ is an equivalence. But we already noted this in \S\ref{ss_BL} when we introduced $c_1^\sub{BL}$.
\end{proof}

\section{cdh-motivic cohomology}\label{sec:cdh-mot}
In this section, we introduce \emph{cdh-motivic cohomology} as a third option for a motivic cohomology theory which appears as the graded pieces of a natural filtration on $\KH$. This theory will be used to give better control over $\Z(j)^{\bb A, \cdh}$. For example, it will ultimately be used to prove in Corollary \ref{prop:Fil1-vanishing} that $\Z(j)^{\bb A, \cdh}=0$ when $j<0$; recall that this is not obvious (see Remark~\ref{rem_roles} for more discussion), since pulling back in $\SH$ is difficult to control. In fact, we will ultimately prove a key comparison theorem in Theorem~\ref{thm:comparison}, stating that $\Z(j)^{\bb A, \cdh}$ is simply the $\bb A^1$-localisation of the cdh-motivic cohomology; this comparison will also underlie our desired comparison with $\bb Z(j)^{\bb A}$. 

An advantage of $\Z(j)^{\cdh}$ is that it is more elementary than $\Z(j)^{\bb A, \cdh}$, in the sense that it does not require the pullback functoriality of $\SH$.  Furthermore, due to its simple definition, one can relate it more easily to syntomic cohomology: this is made precise by Theorem~\ref{thm:singular-bl}, which is a singular analog of the Beilinson-Lichtenbaum equivalence. However, the corresponding downside of $\Z(j)^{\cdh}$ is that it can be difficult to prove its instrinsic properties. For example, although we expect it to satisfy $\bb A^1$-invariance and the projective bundle formula in general, we can only prove partial and conditional results; this is the subject of \S\ref{sec:pbf-a1}.

\subsection{Lisse motivic cohomology}
Throughout this whole subsection $B$ denotes a field or a mixed characteristic Dedekind domain. We begin by introducing the left Kan extended version of motivic cohomology, which was christened ``lisse'' in \cite{ElmantoMorrow2023}. As in \S\ref{subsec:lke}, we left Kan extend presheaves of spectra and of complexes from $\Sm_B$ to $\Sch_B^\sub{qcqs}$.

\begin{definition}\label{def:lisse} For any $j \in \Z$, we define the \emph{weight-$j$, lisse motivic cohomology} of qcqs $B$-schemes to be the left Kan extension to $\Sch_B^\sub{qcqs}$ of the restriction of $\bb A^1$-motivic cohomology to smooth $B$-schemes:
\[
\Z(j)_B^\sub{lse}:=L_B^\sub{sm}(\Z(j)^{\bb A}|_{\Sm_B}): \Sch_B^\sub{qcqs,op} \To \rm D(\Z).
\]
We will sometimes drop the restriction notation $|_{\Sm_B}$ for the sake of readability. In the case $B=\bb Z$ we drop the subscript $B$ and refer to $\bb Z(j)^\sub{lse}(X)$, for any qcqs scheme $X$, as its {\em weight-$j$ lisse motivic cohomology}.
\end{definition}

\begin{remark}[Negative weights] \label{rem:negative-lisse}
As noted in Remark~\ref{rem:eff-neg}, $\Z(j)^{\bb A} = 0$ for $j < 0$. Therefore the same is true for lisse motivic cohomology: $\Z(j)^\sub{lse}_B = 0$ for $j < 0$. 
\end{remark}

The next lemma collects some elementary properties of the lisse motivic cohomology over $B$; the first part recalls that its value on affines is left Kan extended from smooth $B$-algebras (rather than from arbitrary smooth $B$-schemes):\footnote{In fact, the lisse motivic cohomology is probably not a useful invariant on arbitrary qcqs $B$-schemes; one should restrict attention to $B$-schemes with ``enough vector bundles'', e.g., quasi-projective. For us, the case of affine $B$-schemes will be sufficient, as we will soon only be interested in the theory after suitable sheafification.}

\begin{lemma}\label{lemma_lke_from_affines} Let $B$ be a field or a mixed characteristic Dedekind domain. Then for any $j \in \bb Z$:
\begin{enumerate}
\item The functor $\Z(j)_B^\sub{lse}(\Spec(-)):\opp{CAlg_B}\to \rm D(\Z)$ is left Kan extended from smooth $B$-algebras.
\item For any $B$-algebra $R$, there exists a natural, $\N$-indexed, multiplicative filtration $\Fil^\star_\sub{B-lse}\K^\sub{cn}(R)$ on the connective $K$-theory $\K^\sub{cn}(R)$ with graded pieces
\[
\mathrm{gr}^j_\sub{B-lse}\K^\sub{cn}(R) \simeq  \Z(j)_B^\sub{lse}(R)[2j]
\]
for $j\ge0$. In particular, $\Z(\star)_B^\sub{lse}$ assembles into a presheaf of graded $\bb E_{\infty}$-algebras in $\D(\Z)$.
\item For any henselian local $B$-algebra $R$, then $\bb Z(j)_B^\sub{lse}(R)$ is supported in cohomological degrees $\le j$ and the filtration is bounded (indeed, $\Fil^j_\sub{B-lse}\K^\sub{cn}(R)$ is supported in cohomological degrees $\le -j$).
\end{enumerate}
\end{lemma}

\begin{proof}
(1) Immediate from the commutativity of the left-most square in Proposition~\ref{prop:lke-restrict}.

(2) The filtration is obtained by left Kan extension of the filtration $\Fil_{\bb A}^\star\KH$ on $\Sm_B$ from Theorem \ref{theorem:SH_mot_coh}. It is indeed a filtration on connective $K$-theory since the functor $\K^\sub{cn}: \CAlg_{B} \rightarrow \Spt$ is left Kan extended from smooth $B$-algebras \cite[Example~A.0.6]{ElmantoHoyoisKhanSosniloYakerson2020}.

(3) The claimed bounds hold Nisnevich locally for $\bb Z(j)^\bb A$ of smooth $B$-algebras by Corollary \ref{corollary:BL_over_B}, and Nisnevich locally for $\Fil_\bb A^j\KH$ since the latter filtration is complete by Remark \ref{remark_completeness_of_fil}(1). These bounds are then preserved by left Kan extension (use Proposition~\ref{prop:lke-restrict}, in particular the right-most square with ?loc$=$hloc).
\end{proof}

We expect that the lisse motivic cohomology does not depend on $B$: more precisely, we expect that the restriction of $\bb Z(j)^\sub{lse}$ to $B$-algebras agree with $\bb Z(j)^\sub{lse}_B$. However, we can only prove this expectation Nisnevich locally (Theorem \ref{thm:base-indep}):

\begin{proposition} \label{prop:base-indep-bootstrap}
Let $B$ be a field or mixed characteristic Dedekind domain and $j \in \Z$.
Then the canonical map \[ L_\Nis L_\Z^\sub{sm} \Z(j)^{\A} \To \Z(j)^\A \] becomes an equivalence after restriction to $\Sm_B$.
\end{proposition}
\begin{proof}
Both sides being Nisnevich sheaves, it suffices to check the equivalence on stalks, that is, on henselisations of essential smooth local $B$-algebras. Moreover, it suffices to prove the equivalence rationally and mod $p$ for all primes $p$.

Rationally we prove a stronger statement, namely that the canonical map is an equivalence rationally on any regular Noetherian ring $R$ (not necessarily local). As explained in the proof of Lemma~\ref{lemma_lke_from_affines}, $\bb Z(j)^{\lse}$ is the graded pieces of a filtration on $\K^\sub{cn}$ obtained by left Kan extension of $\Fil_{\bb A}^{\star}\KH (= \Fil_{\bb A}^{\star}\K^\sub{cn})$ from smooth $\bb Z$-algebras. Rationally, the latter filtration splits by Theorem~\ref{thm:rational}(2) and therefore the left Kan extended filtration splits as well. On the other hand, for any ring $R$, we have that $L^{\sm}(\K^\sub{cn}|_{\Sm_{\bb Z}})(R) \xrightarrow{\simeq} \K^\sub{cn}(R) \xrightarrow{\simeq} \KH(R)$; the first equivalence is that of Bhatt--Lurie \cite[Example~A.0.6]{ElmantoHoyoisKhanSosniloYakerson2020} and the second equivalence follows from the fact that $K$-theory of a regular Noetherian rings is connective and is $\bb A^1$-invariant. In particular, the same equivalences holds rationally. On the split graded pieces we therefore obtain $(\Z (j)^{\lse} \otimes \bb Q)(R) \quis L^{\sm}\bb Q(j)^{\bb A}(R) \quis \bb Q(j)^{\bb A}(R)$\footnote{For the first equivalence we have used that rationalisation commutes with colimits which holds because colimits commute.} for any regular Noetherian ring $R$.

Now let $p$ be a prime. Then Theorems  \ref{thm:axiomatic_BL} and \ref{theorem_BL1}(3) imply that on the category of schemes $\frak{BL}^p$, which includes both all smooth $\bb Z$-schemes and all smooth $B$-schemes, there is an equivalence $\bb Z(j)^\bb A/p\simeq \bb F_p(j)^\sub{BL}$. So we must show that the canonical map $L_\sub{Nis}L_\Z^\sub{sm} \bb F_p(j)^\sub{BL} \to \bb F_p(j)^\sub{BL}$ is an equivalence after restriction to $\Sm_B$. But this follows from Proposition \ref{prop:syn-trunc-lke}(2) (and was already used in the proof of Theorem \ref{thm_syntomic_SH_base_change}).
\end{proof}

\begin{theorem}[Nisnevich-local base independence]\label{thm:base-indep}
Let $B$ be a field or a mixed characteristic Dedekind domain.
There is a canonical, multiplicative equivalence of Nisnevich sheaves on $\Sch^{\qcqs}_B$ \[ L_\Nis \Z(\star)^\sub{lse}|_{\Sch_B} \wequi L_\Nis \Z(\star)_B^\sub{lse}. \]
\end{theorem}
\begin{proof}
Applying Lemma \ref{lem:lke-b-b'} to the functor $F'=\bb Z(j)^\bb A|_{\Sm_\bb Z}$, this is an immediate consequence of Proposition~\ref{prop:base-indep-bootstrap}.
\end{proof}

\begin{remark}[Base independence of the lisse-motivic filtration]
Arguing inductively down the filtration, we see that Proposition \ref{prop:base-indep-bootstrap} and Theorem \ref{thm:base-indep} remain valid for the filtration of Lemma \ref{lemma_lke_from_affines}(2). The identification of Theorem \ref{thm:base-indep} thus arises, by taking graded pieces, from an identification of the Nisnevich sheafification on qcqs $B$-schemes of the filtrations $\Fil^\star_\sub{lse}\K^\sub{cn}|_{\Sch_B}$ and $\Fil^\star_\sub{B-lse}\K^\sub{cn}$.
\end{remark}

\begin{remark}[Comparison with the lisse motivic cohomology of \cite{ElmantoMorrow2023}]\label{rem:lisse-in-EM}
Lisse motivic cohomology was defined in \cite[Definition 3.1]{ElmantoMorrow2023} on $k$-algebras, where $k=\bb F_p$ or $\bb Q$, by left Kan extending from smooth $k$-algebras. It was explained in  \cite[Remark~3.4]{ElmantoMorrow2023} how lisse motivic cohomology over a field is a kind of ``cycle complex.'' Lemma \ref{lemma_lke_from_affines} and Theorem \ref{thm:base-indep} show that, on henselian local $k$-algebras, the definitions of \cite{ElmantoMorrow2023} and this paper agree. They in fact agree on arbitrary local $k$-algebras (not necessarily henselian), which we mention since certain results about the lisse motivic cohomology over fields were established in \cite[\S7]{ElmantoMorrow2023} in such generality; a proof of this Zariski-local agreement will be recorded elsewhere.
\end{remark}

\subsection{Cdh-motivic cohomology}\label{sec:lis-cdh}
Whereas lisse motivic cohomology plays only an auxiliary role in the present paper, its cdh sheafification is a major tool which we use to analyse the $\bb A^1$-invariant motivic cohomologies:

\begin{definition}\label{def:cdh}
For each $j \in \Z$, define {\em weight-$j$, cdh-motivic cohomology} $\bb Z(j)^\sub{cdh}$ to be the cdh-sheafification of $\bb Z(j)^\sub{lse}$:
\[
\Z(j)^{\cdh}:=L_{\cdh}\Z(j)^\sub{lse}: \Sch^{\qcqs,\op} \To \D(\Z).
\]
In other words, $\bb Z(j)^\sub{cdh}$ is the cdh sheaf of complexes on qcqs schemes whose restriction to smooth $\bb Z$-schemes is equipped with a map from the presheaf $\bb Z(j)^{\bb A}|_{\Sm_Z}$.
\end{definition}

By Remark~\ref{rem:negative-lisse} we see that $\Z(j)^{\cdh} = 0$ for $j < 0$. Since it is the sheafification of the $\bb E_{\infty}$-algebra $\Z(\star)^\sub{lse}_B$ in graded presheaves, it assembles into an $\bb E_\infty$-algebra in graded cdh sheaves of complexes. We will soon see that it is in fact the graded pieces of a multiplicative filtration on $\KH$. To simplify the notation, we will write $\bb F_p(j)^{\cdh}$ for $\bb Z(j)^{\cdh}/p$. Given a qcqs scheme $X$, we infrequently write $H^i_\sub{cdh}(X,\bb Z(j)):=H^i(\bb Z(j)^\sub{cdh}(X))$ for the corresponding {\em cdh-motivic cohomology groups}.

\begin{remark}[Friedlander--Voevodsky motivic cohomology] \label{rmk:FV}
Letting $B$ be a field or mixed characteristic Dedekind domain, the restriction of $\bb Z(j)^\sub{cdh}$ to qcqs $B$-schemes is given by $L_\sub{cdh}\bb Z(j)^\sub{lse}_B$, by cdh sheafifying Theorem \ref{thm:base-indep}. In particular, given a field $k$, the restriction of $\bb Z(j)^\sub{cdh}$ to qcqs $k$-schemes is the cdh-local left Kan extension of motivic cohomology from smooth $k$-schemes. At least if $k$ has characteristic zero, it follows that the restriction of $\bb Z(j)^\sub{cdh}$ to finite type $k$-schemes is the same as Friedlander--Voevodsky's motivic cohomology of \cite[Definition 9.2]{FriedlanderVoevodsky2000}.
\end{remark}

\begin{remark}[Relationship with $\Z(j)^{\bb A,\cdh}$]\label{rem_compar_map_cdh_to_Acdh}
Note that there is a comparison map of $\bb E_{\infty}$-algebras of graded cdh sheaves of complexes
\[
\Z(\star)^{\cdh} \To \Z(\star)^{\bb A, \cdh}.
\]
Indeed, by the universal property of $\Z(\star)^{\cdh}$ and the fact that $\Z(\star)^{\bb A, \cdh}$ is a cdh sheaf (Theorem \ref{theorem:SH_mot_coh}(2)) it suffices to produce a map $\Z(\star)^{\A}|_{\Sm_\Z} \to \Z(\star)^{\A,\cdh}|_{\Sm_\Z}$; but on smooth $\bb Z$-schemes the comparison \eqref{eq:cdh-to-aone} is an equivalence, and so we take the inverse map. See also Theorem \ref{prop:basic_props_of_cdh_mot}(2).
\end{remark}

\begin{remark}[Rationalisation] \label{rem:rationalise} As already used several times, the presheaf rationalisation of a cdh sheaf is a cdh sheaf (Example~\ref{ex:rationalization-cdh}); therefore we set
\[
\bb Q(j)^{\cdh} := \bb Z(j)^{\cdh}(-) \otimes_{\bb Z} \bb Q
\]
to be the rational cdh-motivic cohomology. We will see in Theorem~\ref{prop:basic_props_of_cdh_mot}(7) that it appears as a summand of the rationalisation of $\KH$.
\end{remark}

Let us now collect together some properties of the cdh-motivic cohomology which follow relatively formally from its definition and from analogous properties of the $\bb A^1$-invariant theories.

\begin{theorem}\label{prop:basic_props_of_cdh_mot}
Cdh-motivic cohomology forms an $\mathbb{E}_{\infty}$-algebra in graded presheaves of complexes
\[
\bb Z(\star)^\sub{cdh}:\Sch^{\qcqs,\op} \To \Gr \D(\Z)
\]  which satisfies the following properties:
\begin{enumerate}
\item For any qcqs scheme $X$, there exists a natural, multiplicative, $\bb N$-indexed filtration 
\[
\mathrm{Fil}_\sub{cdh}^{\star}\KH(X),
\] such that the graded pieces are naturally and multiplicatively given by 
\[
\mathrm{gr}_\sub{cdh}^j\KH(X)\simeq \Z(j)^\cdh[2j],
\]  
for $j \geq 0$. Furthermore, this is the initial $\CAlg(\Fil\Spt)$-valued cdh sheaf whose restriction to $\Sm_{\Z}$ is equipped with a map from $\Fil_{\bb A}^{\star}\KH|_{\Sm_{\Z}}$.

In particular, there exist a natural, multiplicative \emph{cdh-local Atiyah--Hirzebruch} spectral sequence
\begin{equation}\label{eq:cdh-ahss}
E_2^{i,j}=H_\sub{cdh}^{i-j}(X, \Z(-j)) \implies \KH_{-i-j}(X).
\end{equation}
\item There is a natural multiplicative map of filtered spectra
\[
\mathrm{Fil}_\sub{cdh}^{\star}\KH(X) \To \mathrm{Fil}_{\bb A, \cdh}^{\star}\KH(X)
\]
which, on graded pieces, gives a multiplicative map
\begin{equation}\label{eq:cdh-a1}
\Z(\star)^{\cdh}(X)[2\star] \To \Z(\star)^{\bb A, \cdh}(X)[2\star].
\end{equation}
\item If $X$ has finite valuative dimension $\le d$ then this filtration is bounded. More precisely, $\mathrm{Fil}_{\cdh}^j\KH(X)$ is supported in cohomological degrees $\le d-j$ (and so $\bb Z(j)^\sub{cdh}(X)$ is supported in cohomological degrees $\le j+d$) for all $j\ge0$.

\item $\Fil^j_\sub{cdh}\KH$ and $\Z(j)^{\cdh}$ are finitary, cdh sheaves for all $j\ge0$.

\item Weight zero: for any qcqs scheme $X$ there is a natural equivalence of $\bb E_\infty$-algebras in $\rm D(\bb Z)$
\[
R\Gamma_{\cdh}(X,\Z)	\quis \Z(0)^{\cdh}(X).
\]

\item Weight one: for any qcqs scheme $X$ there is a natural equivalence
\[
c_1^\sub{cdh}:R\Gamma_{\cdh}(X,\Gm)[-1]\quis \Z(1)^{\cdh}(X),
\]
which we will call the first Chern class for the cdh-motivic cohomology.

\item Rational structure: the filtration $\mathrm{Fil}_\sub{cdh}^{\star}\KH(X)_{\bb Q}$ naturally splits, i.e., there is a natural, multiplicative equivalence of filtered spectra
\[
\Fil^\star_\sub{cdh}\KH(X)_{\bb Q} \simeq \bigoplus_{j \geq \star} \bb Q(j)^{\cdh}(X)[2j]
\]
compatible with Theorem \ref{thm:rational}(2).
\end{enumerate}
\end{theorem}

\begin{proof} 
(1): We define 
\[
\mathrm{Fil}_{\cdh}^\star\KH := L_{\cdh}L^\sub{sm}\Fil^{\star}_{\bb A}\KH: \Sch^{\qcqs} \rightarrow \CAlg(\Fil\Spt).
\]
to be the cdh sheafification of the left Kan extension of the filtration on $\KH|_{\Sm_\bb Z}$. The assertions about the universal property and graded pieces follow at once by construction. The only thing we need to check is that $\mathrm{Fil}_{\cdh}^\star\KH$ is indeed a filtration $\KH$. This follows from the equivalences $L_{\cdh}L^\sub{sm}\K \simeq L_{\cdh}L^\sub{sm}\K^\sub{cn}  \simeq \KH$, which is a consequence of \cite[Theorem~6.3]{KerzStrunkTamme2018} (see also \cite[Remark~3.4]{KellyMorrow2021}) and the fact that connective $K$-theory is left Kan extended from smooth $\bb Z$-algebras \cite[Example~A.0.6]{ElmantoHoyoisKhanSosniloYakerson2020}. 

(2):  We argue just as in Remark \ref{rem_compar_map_cdh_to_Acdh}. Namely, on $\Sm_{\Z}$ the filtered spectra $\Fil^{\star}_{\bb A}\KH$ and $\Fil^{\star}_{\bb A, \cdh}\KH$ are equivalent, and so the desired comparison map follows by the universal property of cdh sheafification and left Kan extension: indeed, $\Fil^{\star}_{\bb A, \cdh}\KH$ is an example of a $\CAlg(\Fil\Spt)$-valued $\cdh$-sheaf whose restriction to $\Sm_{\bb Z}$ is $\Fil^{\star}_{\bb A}\KH$. The assertion on motivic cohomologies then follows by taking graded pieces. 

(3): We know from Lemma \ref{lemma_lke_from_affines} that the filtration $\Fil^j_{\bb Z\sub{-lse}}\K^\sub{cn}$ is Nisnevich-locally $j$-connective on qcqs schemes. For $X$ of finite valuative dimension $\le d$, the $\infty$-topos of cdh sheaves on finitely presented $X$-schemes has homotopy dimension $\le d$ \cite[Theorem 2.4.15]{ElmantoHoyoisIwasaKelly2021}, and so the desired bound for $\mathrm{Fil}_{\cdh}^j\KH(X)$ follows by taking cdh cohomology.

(4): They are cdh sheaves by design, while finitariness follows from Proposition~\ref{prop:finitary_conditions}.

(5) and (6) follow from Corollary~\ref{cor:low-wts-ddk-fields} by cdh-locally left Kan extending, and using that $A\mapsto\bb Z$ and $A\mapsto A^\times$ are left Kan extended from essentially smooth, local $\bb Z$-algebras, as in the proof of Theorem~\ref{theorem:SH_mot_coh}(3)\&(4).

(7): The decomposition of $\KH(X)_{\bb Q}$ follows immediately from Theorem~\ref{thm:rational} since $L_{\cdh}$ and left Kan extensions preserve the direct sum decompositions.
\end{proof}

Similarly to Theorem \ref{thm:rational}(1), we may already deduce that the cdh-motivic cohomology agrees with the $\bb A^1$-theories rationally. We also take the opportunity to record a descent property of rational motivic cohomology, which will be relevant in Proposition~\ref{prop_eh_motivic cohomology} where we study in more detail \'eh and h-local versions of motivic cohomology. 

\begin{corollary}\label{corol:qj2}
\begin{enumerate}
\item For any qcqs scheme $X$, the comparison map \eqref{eq:cdh-a1} is a rational equivalence.
\item For $j\in\bb Z$, the presheaves 
\[
\bb Q(j)^{\cdh}\simeq\bb Q(j)^{\bb A,\cdh}\simeq\bb Q(j)^{\bb A}: \Sch^{\qcqs} \To \D(\bb Q)
\]
are $h$-sheaves.
\end{enumerate}
\end{corollary}
\begin{proof}
(1): Theorems \ref{thm:rational}(2) and \ref{eq:cdh-ahss}(7) show that, for each $j$, the comparison map $\Q(j)^{\cdh}(X) \to \Q(j)^{\bb A, \cdh}(X)$ is, up to shift, a direct summand of the equivalence $L_\sub{cdh}\K^\sub{cn}(X)_\bb Q\quis \KH(X)_\bb Q$; therefore it is itself an equivalence.

(2):  Since $\bb Q(j)^{\cdh}$ is a summand of $\KH_{\bb Q} = \KH(-) \otimes \bb Q$ by Theorem~\ref{prop:basic_props_of_cdh_mot}(8), it suffices to prove that the latter is an $h$-sheaf. By Remark~\ref{remark:h_descent}, it suffices to verify that $\KH_{\bb Q}$ has finite flat descent. This, in turn, follows from a standard transfer argument; we refer the reader to \cite[Theorem 5.1 \& Proposition 5.4]{ClausenMathewNoelNaumann2024} for a modern account which also verifies finite flat descent for higher chromatic localisations (the case of rationalization being the case of height $0$). 
\end{proof}

\subsection{Syntomic comparison}\label{sec:syn-comparison}
In this subsection we fix a prime number $p$. We shall show that cdh-motivic cohomology with finite coefficients agrees with the cdh-sheafification of Beilinson-Lichtenbaum cohomology, as defined in \S\ref{ss_BL}. We will use this to establish many key properties of $\cdh$-motivic cohomology, including an integral Beilinson--Lichtenbaum comparison in Theorem~\ref{thm:singular-bl}.

For $j\in\bb Z$, let \[L_\sub{cdh}\bb Z_p(j)^\sub{BL}\quad\mbox{and}\quad L_\sub{cdh}\bb Z_p(j)^\sub{BL}/p^r=L_\sub{cdh}\tau^{\le j}(\bb Z_p(j)^\sub{syn}/p^r):\Sch^\sub{qcqs,op}\To\rm D(\bb Z)\] denote the cdh-sheafification of $p$-adic and mod-$p^r$ Beilinson--Lichtenbaum cohomology. Note that $L_\sub{cdh}\bb Z_p(\star)^\sub{BL}$ is an $\bb E_\infty$-algebra in graded presheaves of complexes on qcqs schemes (and similarly modulo any power of $p$).\footnote{Warning: $L_\sub{cdh}\bb Z_p(j)^\sub{BL}$ need not take $p$-complete values. For a specific example in the case $j=1$, so that $L_\sub{cdh}\bb Z_p(j)^\sub{BL}$ is the cdh sheafification of $R\Gamma_\sub{Nis}(-,\bb G_m)_p^\comp[-1]$ by Remark \ref{example_BL_cohomology}(4), see \cite[Footnote 7]{ElmantoMorrow2023}.} The first Chern class for Beilinson--Lichtenbaum cohomology (which was itself induced by the first Chern class for syntomic cohomology) induces, upon cdh-sheafification, a first Chern class map \[c_1^\sub{cdh-BL}:R\Gamma_\sub{cdh}(-,\bb G_m)[-1]\To L_\sub{cdh}\bb Z_p(1)^\sub{BL}.\]

Our cdh-local Beilinson--Lichtenbaum comparison with finite coefficients is as follows:

\begin{theorem}\label{theorem_syn_comp}
There is a unique map of $\bb E_\infty$-algebras in graded presheaves of complexes on $\opp{Sch}^\sub{qcqs}$
\begin{equation}\label{eq:cdh-syn}
\bb Z(\star)^\sub{cdh}\To  L_\sub{cdh}\bb Z_p(\star)^\sub{BL}
\end{equation}
which is compatible with Theorem \ref{theorem_BL1}(3) (in the obvious sense which we will make precise at the beginning of the proof). This comparison map is an equivalence modulo any power of $p$, and is compatible with first Chern classes in the sense that the diagram
\begin{equation}\xymatrix{
\bb Z(1)^\sub{cdh}\ar[rr] & & L_\sub{cdh}\bb Z(1)^\sub{BL}\\
&R\Gamma_\sub{cdh}(-,\bb G_m)[-1]\ar[ul]^{c_1^\sub{cdh}}\ar[ur]_{c_1^\sub{cdh-BL}}&
}\label{eqn_cdh_1st_Chern_compatability}\end{equation}
commutes.
\end{theorem}
\begin{proof}
Taking the limit over powers of $p$, Theorem \ref{theorem_BL1}(3) provided us with a canonical comparison map $\delta_\sub{mot}^\sub{BL}\{\star\}:\bb Z(\star)^\bb A\to \bb Z_p(\star)^\sub{BL}$ on $\Sm_\bb Z$, which is compatible with first Chern classes and an equivalence modulo all powers of $p$. Left Kan extending to qcqs schemes and cdh sheafifying defines our desired comparison map $\bb Z(\star)^\sub{cdh}\to  L_\sub{cdh}\bb Z_p(\star)^\sub{BL}$ on $\opp{Sch}^\sub{qcqs}$; by adjunctions, this map is characterised uniquely by the fact that, upon restricting back to $\Sm_\bb Z$, the diagram
\[\xymatrix{
\bb Z(\star)^\sub{cdh}|_{\Sm_\bb Z}\ar[r] & L_\sub{cdh}\bb Z_p(\star)^\sub{BL}|_{\Sm_\bb Z}\\
\bb Z(\star)^\bb A|_{\Sm_\bb Z}\ar[u]\ar[r]_{\delta_\sub{mot}^\sub{BL}\{\star\}} & \bb Z_p(\star)^\sub{BL}|_{\Sm_\bb Z}\ar[u]
}\]
commutes.

Our comparison map is compatible with first Chern classes, because the $c_1^\sub{cdh}$ was defined by cdh-locally left Kan extending $c_1^\bb A$, which is itself compatible with $c_1^\sub{BL}$.

We must prove that the comparison map is an equivalence modulo $p$. Given that $\delta_\sub{mot}^\sub{BL}\{\star\}$ is already known to be an equivalence modulo $p$ it remains only to show, for each $j\in\bb Z$, that the canonical map
\[L_{\cdh}L^\sub{sm}_\bb Z(\Z_p(j)^\sub{BL}|_{\Sm_\bb Z})\To L_{\cdh}\Z_p(j)^\sub{BL}\] of presheaves on qcqs schemes is an equivalence modulo $p$. But this is even true if we replace $L_\sub{cdh}$ by $L_\sub{Nis}$, by Proposition \ref{prop:syn-trunc-lke}(2).
\end{proof}

In particular, the $p$-adic cdh-motivic complexes admit the following explicit description in characteristic $p$, analogous to the Geisser--Levine theorem \cite{GeisserLevine2000}:

\begin{corollary}\label{cor:singular_gl}
On the category of qcqs $\bb F_p$-schemes there are compatible equivalences of $\bb E_\infty$-algebras in graded presheaves of complexes \[\bb Z(\star)^\sub{cdh}/p^r\simeq R\Gamma_\sub{cdh}(-,W_r\Omega^\star_\sub{log})[-\star]\] for all $r\ge0$.
\end{corollary}
\begin{proof}
The functor $\tau^{\le j}(\bb Z_p(j)^\sub{syn}/p^r):{\rm CAlg}_{\bb F_p}\to\D(\bb Z)$ is left Kan extended from smooth $\bb F_p$-algebras; this follows by combining Proposition~\ref{prop:syn-trunc-lke}(2) and Lemma \ref{lem:lke-b-b'}. Moreover, by Example \ref{example_syn_in_char_p}, for any smooth $\bb F_p$-algebra $R$ there is a natural equivalence $\tau^{\le j}(\bb Z_p(j)^\sub{syn}(R)/p^r)\simeq W_r\Omega^j_{R,\sub{log}}[-j]$. So left Kan extending defines a natural counit map $\tau^{\le j}(\bb Z_p(j)^\sub{syn}(A)/p^r) \to W_r\Omega^j_{A,\sub{log}}[-j]$ for any $\bb F_p$-algebra $A$. Moreover this counit is an equivalence not only for smooth $\bb F_p$-algebras, but more generally for Cartier smooth $\bb F_p$-algebra in the sense of \cite{KellyMorrow2021} (again, see Example \ref{example_syn_in_char_p}); in particular, it is an equivalence whenever $A$ is a valuation ring of characteristic $p$. So cdh sheafifying defines equivalences $L_\sub{cdh}\bb Z_p(j)^\sub{BL}/p^r\simeq R\Gamma_\sub{cdh}(-,W_r\Omega^j_\sub{log})$. Combined with Theorem \ref{theorem_syn_comp}, this completes the proof.
\end{proof}

In an orthogonal direction, we have the following result when $p$ is invertible; this should be thought of as a singular extension of the mod-$p^r$ Beilinson--Lichtenbaum equivalence relating motivic and \'etale cohomology for such schemes:

\begin{corollary}\label{cor:singular_bl}
On the category of qcqs $\bb Z[\tfrac1p]$-schemes there are compatible equivalences of $\bb E_\infty$-algebras in graded presheaves of complexes \[\bb Z(\star)^\sub{cdh}/p^r\simeq L_{\cdh}\tau^{\leq \star}R\Gamma_{\et}(-,\mu_{p^r}^{\otimes \star})\] for all $r\ge0$.
\end{corollary}
\begin{proof}
Using the comparison Theorem \ref{theorem_syn_comp}, this follows by cdh sheafifying Example \ref{example_BL_cohomology}(1).
\end{proof}

\subsection{Milnor excision}\label{ss:Milnor_excision}
Cdh-motivic cohomology is a finitary cdh sheaf, hence it is controlled by its values on henselian valuation rings. To prove our main theorems in \S\ref{sec:pbf-a1} and \S\ref{sec:a1-comparison} we will need to further reduce to rank one valuation rings; this is a technical but crucial reduction step. To carry this out, we need to establish Milnor excision for cdh-motivic cohomology:

\begin{theorem} \label{theorem:exc-cdarc} The presheaves 
\[
\Z(j)^{\cdh}:\opp{Sch}^\sub{qcqs,op} \To\D(\bb Z)
\]
satisfy Milnor excision for all $j\ge0$.
\end{theorem}

The theorem is a motivic refinement of Milnor excision for $L_{\cdh}\K$, which follows from the fact it is equivalent to $\KH$; indeed, $\KH$ was shown to satisfy Milnor excision by Weibel \cite{Weibel1989a}. It is also a step towards proving, at least under suitable conditions, that the cdh-motivic cohomology is represented by an absolute motivic spectrum; indeed, according to \cite{ElmantoHoyoisIwasaKelly2021}, any such cohomology theory automatically satisfies Milnor excision.

Our proof of Theorem \ref{theorem:exc-cdarc} will depend on some finer structure of syntomic and prismatic cohomology. We begin by noting that, since $\Z(j)^{\cdh}$ is already known to be a finitary cdh sheaf (Theorem~\ref{prop:basic_props_of_cdh_mot}(4)), Theorem \ref{thm:ehik} reduces the problem of Milnor excision to proving that $\Z(j)^{\cdh}$ satisfies henselian $v$-excision. With finite coefficients we have seen in Theorem \ref{theorem_syn_comp} that $\bb Z(j)^\sub{cdh}$ is controlled by the syntomic cohomology of \S\ref{sec:syntomic}. To prove henselian $v$-excision for the latter on $p$-complete valuation rings, we will use standard arguments (c.f., \cite{AntieauMathewMorrowNikolaus2022}) to reduce to the cotangent complex: for a morphism of rings $A \rightarrow B$, we denote by $\bb L_{B/A}$ the associated cotangent complex, and by $\bb L^j_{B/A} := \bigwedge^i_B \bb L_{B/A}$, its derived $j$-th exterior power as a $B$-module. The proof of the next lemma is inspired by arguments of \cite[Lemma 3.14]{HuberKelly2018} where $v$-excision for differential forms was established.

\begin{lemma} \label{lemm:L-v-exc}
Let $A$ be a ring. Then the functors \[\bb L_{-/A}^j: \opp{CAlg}_A \to \D(A) \] satisfy $v$-excision for all $j\ge0$.
\end{lemma}
\begin{proof}
The case $j=0$ being tautological, we begin by treating the case $j=1$ and then proceed by induction.

Let $V$ be a valuation ring over $A$ and $\frak p\subset V$ a prime ideal. For reference we recall that the corresponding Milnor square is
\begin{equation}\label{eq:v-p}
\begin{tikzcd}
V \ar{r} \ar{d} & V_{\mathfrak{p}} \ar{d} \\
V/\mathfrak{p} \ar[swap]{r} & \kappa(\mathfrak{p}),
\end{tikzcd}
\end{equation}
and that our goal is to prove that applying $\bb L_{-/A}$ to the above square yields a cartesian square. Thanks to the transitivity sequence \cite[Tag 08QR]{Stacks} for $A\to V\to -$ we see that there is a fibre sequence of squares \[
\bb L_{V/A} \otimes_A \eqref{eq:v-p} \rightarrow \bb L_{\eqref{eq:v-p}/A} \rightarrow \bb L_{\eqref{eq:v-p}/V},
\]
where the left term is a cartesian square since \eqref{eq:v-p} is a bicartesian square. Therefore the square $\bb L_{\eqref{eq:v-p}/A}$ is cartesian if and only if $\bb L_{\eqref{eq:v-p}/V}$ is. The latter square 
\comment{
To this end, let $R$ be \emph{any} ring equipped with a map $R\to V$; we claim that
\begin{itemize}
\item[($\ast$)] the functor $L_{-/\Z}$ converts~\eqref{eq:v-p} to a cartesian square if and only if the functor $L_{-/R}$ does.
\end{itemize}
Indeed there is a fiber sequence of squares coming from the transitivity sequence
\[
L_{R/\Z} \otimes_R \eqref{eq:v-p} \rightarrow L_{\eqref{eq:v-p}/\Z} \rightarrow L_{\eqref{eq:v-p}/R},
\]
and, since \eqref{eq:v-p} is a bicartesian square, the left term is a cartesian square and so the claim $(\ast)$ follows. Therefore, setting $R =A$, we have reduced to proving the claim for $A=\bb Z$

Now, to prove the claim over $\Z$ we can also prove the claim after setting $R = V$ thanks to $(\ast)$. The square of interest }
takes the form
\[
\begin{tikzcd}
0 = \bb L_{V/V} \ar{r} \ar{d} & \bb L_{V_{\mathfrak{p}}/V}=0\ar{d} \\
\bb L_{(V/\mathfrak{p})/V} \ar[swap]{r} & \bb L_{\kappa(\mathfrak{p})/V},
\end{tikzcd}
\]
and so the goal is now to prove that the map $\bb L_{(V/\mathfrak{p})/V} \rightarrow \bb L_{\kappa(\mathfrak{p})/V}$ is an equivalence. Since this map is precisely the localisation map $\bb L_{(V/\mathfrak{p})/V}\to \bb L_{(V/\mathfrak{p})/V}\otimes_{V/\frak p}k(\frak p)$, we must show that, for any $f\in V\setminus \frak p$, the multiplication map $\cdot f:\bb L_{(V/\mathfrak{p})/V}\to \bb L_{(V/\mathfrak{p})/V}$ is an equivalence. This follows from 
a computation of the cofiber $\bb L_{(V/\mathfrak{p})/V}/f$:
\begin{eqnarray*}
\bb L_{(V/\mathfrak{p})/V}/f & \simeq & \bb L_{(V/\mathfrak{p})/V} \otimes_V V/f\\
& \simeq & \bb L_{(V/\mathfrak{p} \otimes_V V/f)/V/f}\\
& \simeq & \bb L_{(V/\mathfrak{p}+(f))/(V/f)}\\
& \simeq & \bb L_{(V/f)/(V/f)}\\
& \simeq & 0.
\end{eqnarray*}
Here we used base change \cite[Tag 08QQ]{Stacks} for the second equivalence, and the fact that $V$ is a valuation ring for the fourth equivalence (so that $\mathfrak{p} +fV = fV)$. This completes the proof in the case $i=1$.

Proceeding inductively to the case $j>1$, the goal is to show that $\bb L_{\eqref{eq:v-p}/A}^j$ is a cartesian square. Recall that, for any $V$-algebra $R$, the transitivity sequence for $A\to V\to R$ implies the existence of a natural, finite filtration on $\bb L^j_{R/A}$ whose graded pieces are $\mathrm{gr}^i \bb L^j_{R/A} \simeq \bb L^i_{V/A} \otimes_V \bb L^{j-i}_{R/V}$ for $i = 0, \cdots, j$ \cite[\S V.4]{Illusie1971}. Applying this with $R$ being the terms in the square \eqref{eq:v-p}, the inductive hypothesis reduces the goal to proving that $\bb L_{\eqref{eq:v-p}/V}^j$ is cartesian. As in the case $j=1$, this is equivalent to showing that $\bb L_{(V/\mathfrak{p})/V}^j \rightarrow \bb L_{\kappa(\mathfrak{p})/V}^j$ is an equivalence; but this identifies with the localisation map $\bb L_{(V/\mathfrak{p})/V}^j\to \bb L_{(V/\mathfrak{p})/V}^j\otimes_{V/\frak p}k(\frak p)$, which is indeed an equivalence since we saw in the proof of the case $j=1$ that any non-zero $f\in V/\frak p$ acts invertibly on $\bb L_{(V/\mathfrak{p})/V}$, hence also on its wedge powers.
\comment{
\[
\bigwedge^j_{V/\mathfrak{p}} L_{V/\mathfrak{p}/V}  \To (\bigwedge^j_{V/\mathfrak{p}}  L_{V/\mathfrak{p}/V}) \otimes_V V_{\mathfrak{p}} \simeq \bigwedge^j_{V/\mathfrak{p}}  (L_{V/\mathfrak{p}/V} \otimes_V V_{\mathfrak{p}}),
\]
which we saw was an equivalence in the proof of the case $j=1$.
}
\end{proof}

In the next proposition we bootstrap the $v$-excision results about the cotangent complex to Nygaard-complete prismatic cohomology $\hat\Prism_R$, as well as its Nygaard filtration $\cal N^{\ge n}\hat\Prism_R$, their Breuil--Kisin twists $\cal N^{\ge n}\hat\Prism_R\{j\}$, and the associated graded pieces $\cal N^n\hat\Prism_R\{j\}$. These invariants were first introduced in \cite{BhattMorrowScholze2}, for $p$-completely quasisyntomic rings $R$; see also \cite[\S 5]{AntieauMathewMorrowNikolaus2022} for a summary.

\begin{remark} \label{rmk:locn-compn-valn-ring}
We shall use frequently in the sequel that if $V$ is a $p$-complete valuation ring and $\frak p \subset V$ is a prime ideal containing $p$, then $V/\frak p$, $V_\frak p$, and $k(\frak p)$ are all $p$-complete valuation rings.
(The only subtle point is $p$-completeness of the localisation $V_\frak p$.
Note that given an ideal $I$ containing $p$ in any ring $A$, then $I$ is derived $p$-complete if and only if $A$ is, the cofiber being $p$-torsion.
Now apply this to $\frak p \subset V$ and $\frak p \wequi \frak p V_{\frak p} \subset V_{\frak p}$.)
\end{remark}

We begin with a result in the $p$-complete context:

\begin{proposition}\label{prop:nygaard-prism-exc}
Let $V$ be a $p$-complete valuation ring and $\frak p\subset V$ a prime ideal containing $p$. Then, for all $j,n,r\ge0$, the functors \[\cal N^n\hat\Prism\{j\},\quad\cal N^{\ge n}\hat\Prism\{j\}, \quad\bb Z_p(j)^\sub{syn},\quad \tau^{\le j}(\bb Z_p(j)^\sub{syn}/p^r)\] satisfy $v$-excision for the pair $(V,\frak p)$.
\end{proposition}
\begin{proof}
First note that, since $V, V/\frak p$, $V_\frak p$, and $k(\frak p)$ are all $p$-complete valuation rings, they are $p$-completely quasisyntomic by \cite[Theorem 6.5.8]{GabberRamero2003}.

The $p$-completed exterior powers $\hat{\bb L_{-/\bb Z_p}^i}$ of the cotangent complex satisfy $v$-excision for the pair $(V,\frak p)$ by $p$-completing Lemma \ref{lemm:L-v-exc}. We now bootstrap, using Remark \ref{rem_syn_via_cotangent}, to prove the proposition. Indeed, from the second paragraph of that remark we successively deduce that the graded pieces of $\THH(-/\bb S[z];\Z_p)$, and $N^n\widehat{\Prism}\{j\}$ for all $n,j\ge0$, satisfy $v$-excision. 
%
%
By induction on $i$ we then obtain the desired $v$-excision for any $\cal N^{\ge n}\widehat{\Prism}\{j\}/\scr N^{\ge n+i}\widehat{\Prism}\{j\}$, and by passing to the limit we obtain it for $\cal N^{\ge n}\widehat{\Prism}\{j\}$. We next appeal to the fibre sequence $\bb Z_p(j)^\sub{syn}\to\cal N^{\ge j}\Prism\{j\}\stackrel{\phi_j-1}\to \cal \Prism\{j\}$ to prove it for $\bb Z_p(j)^\sub{syn}$, and pass modulo $p^r$ for $\bb Z_p(j)^\sub{syn}/p^r$.

Finally, for $\tau^{\le j}(\bb Z_p(j)^\sub{syn}/p^r)$ we must show that the square
\[\xymatrix{
\bb Z_p(j)^\sub{syn}(V)/p^r\ar[r]\ar[d] & \bb Z_p(j)^\sub{syn}(V_\frak p)/p^r\ar[d] \\
\bb Z_p(j)^\sub{syn}(V/\frak p)/p^r\ar[r] & \bb Z_p(j)^\sub{syn}(k(\frak p))/p^r
},\]
which we now know is cartesian, remains so after applying $\tau^{\le j}$. It is enough to show that the right vertical arrow is surjective on $H^{j}$.
By Example~\ref{example_syn_in_char_p}, the target group $H^j_\sub{syn}(k(\frak p),\bb Z_p/p^r(j))$ is generated by symbols, i.e., products of elements in the image of the syntomic Chern character (see \S\ref{ss:weight_one_syn}).\footnote{More excessively, Theorem \ref{thm_syntomic_Milnor} shows that the right vertical arrow identifies with the map $\hat K_j^M(V_\frak p)/p^r\to \hat K_j^M(k(\frak p))/p^r.$}
The desired surjectivity thus follows from the surjectivity of $V_\frak p^\times \to k(\frak p)^\times$.
%
\end{proof}

The following lemma will be required to analyse different cases in the proof of Proposition \ref{proposition_synt_coh_hv_excision}:

\begin{lemma}\label{lemma_p_compl_of_hv_square}
Let $V$ be a valuation ring, $\frak p\subset V$ a prime ideal, and $p$ a prime number. Let
\begin{equation}\xymatrix{
\hat V\ar[r]\ar[d] & \hat{V_{\frak p}}\ar[d]\\
\hat{V/\frak p}\ar[r] & \hat{k(\frak p)}
}\label{equation_p-comp_hv}
\end{equation}
denote the $p$-adic completion of the square (\ref{eqn_hv_square}).
\begin{enumerate}
\item If $p\notin\frak p$ then the vertical arrows in (\ref{equation_p-comp_hv}) are isomorphisms of rings.
\item If $p\in\frak p$ then (\ref{equation_p-comp_hv}) identifies with the Milnor square associated with the pair $(\hat V,\frak p\hat V)$
\end{enumerate}
\end{lemma}
\begin{proof}
Case 1: if $p\notin\frak p$ then $\frak p$ is uniquely $p$-divisible (as it is an ideal of $V_\frak p$, in which $p$ is invertible), so we see that $\hat V\isoto\hat{V/\frak p}$ and $\hat{V_\frak p}=0=\hat{k(\frak p)}$.

Case 2: From now on we assume that $p\in\frak p$. Since $\hat V$ is also the derived $p$-adic completion of $V$ ($V$ being $p$-torsion free), we see that the quotient $\hat V/V$ is uniquely $p$-divisible and so $M\isoto M\otimes_V\hat V$ for any $V/p$-module $M$. This holds in particular for $M=V/\frak p$, thereby giving the isomorphism \[\hat{V/\frak p}=V/\frak p\isoto V/\frak p\otimes_V\hat V=\hat V/\frak p\hat V.\] 

It remains to show that the canonical map $\hat V_{\frak p\hat V}\to\hat{V_\frak p}$ is an isomorphism. Since it is clearly an isomorphism modulo $p$, it is enough to show that $\hat V_{\frak p\hat V}$ is actually $p$-adically complete. In the cofiber sequence $\frak p\hat V_{\frak p \hat V} \to \hat V_{\frak p \hat V} \to \hat V_{\frak p \hat V}/\frak p\hat V_{\frak p \hat V}$ the right hand term is (derived) $p$-complete (recall $p \in \frak p$), whence the middle term is if and only if the first term is. Noting that $\frak p\hat V\isoto \frak p\hat V_{\frak p\hat V}$ since $\frak p\hat V$ is a prime ideal of the valuation ring $\hat V$, and $\frak p\hat V$ is $p$-adically complete by the previous paragraph (it is the kernel of $\hat V\to V/\frak p$).
\end{proof}

By gluing in contributions from \'etale cohomology we now remove the $p$-completeness condition:

\begin{proposition}\label{proposition_synt_coh_hv_excision}
For all $r\ge0$ and $j\in\bb Z$, the functors \[\bb Z_p(j)^\sub{syn}/p^r,\quad \tau^{\le j}(\bb Z_p(j)^\sub{syn}/p^r):\opp{CAlg}_\bb Z\To\D(\bb Z)\] satisfy henselian $v$-excision.
\end{proposition}
\begin{proof}
We begin with a remark on the truncated expression. For any henselian valuation ring $V$, the pullback square definition (\ref{eq:syn}) gives a cartesian square
\begin{equation}\begin{tikzcd}
\Z_p(j)^{\mathrm{syn}}(V)/p^r \ar{r} \ar{d} & R\Gamma_{\et}(V[\tfrac{1}{p}],\mu_{p^r}^{\otimes j})\ar{d}\\
\Z_p(j)^{\syn}(\hat V)/p^r \ar{r} & R\Gamma_{\et}(\hat V[\tfrac{1}{p}],\mu_{p^r}^{\otimes j}).
\end{tikzcd}\label{eqn:syn_pullback}
\end{equation}
This square remains cartesian after applying $\tau^{\le j}$: indeed, either $p$ is a unit in $V$ in which case the bottom row is zero, or else $p$ is a non-unit whence $V$ is $p$-henselian and both vertical arrows are equivalences by Lemma \ref{lemma_syn_p-hens}(2).

Now let $V$ be a henselian valuation ring and $\frak p\subset V$ a prime ideal; note that $V/\frak p$, $V_\frak p$, and $k(\frak p)$ are also henselian valuation rings by \cite[Lemma 3.3.5]{ElmantoHoyoisIwasaKelly2021}, so the previous paragraph applies to them. From the pullback square (\ref{eqn:syn_pullback}) and its truncated version, we must show (1) that $\bb Z_p(j)^\sub{syn}/p^r$ and $\tau^{\le j}(\bb Z_p(j)^\sub{syn}/p^r)$ carry
\[\xymatrix{
\hat V\ar[r]\ar[d] & \hat{V_{\frak p}}\ar[d]\\
\hat{V/\frak p}\ar[r] & \hat{k(\frak p)}
}\]
to cartesian squares, and (2) that $R\Gamma_\sub{\'et}(-,\mu_{p^r}^{\otimes j})$ and $\tau^{\le j} R\Gamma_\sub{\'et}(-,\mu_{p^r}^{\otimes j})$ carry both the squares
\[\xymatrix{
V[\tfrac1p]\ar[r]\ar[d] & V_{\frak p}[\tfrac1p]\ar[d]\\
V/\frak p[\tfrac1p]\ar[r] & k(\frak p)[\tfrac1p]
}\text{  and  }
\xymatrix{
\hat V[\tfrac1p]\ar[r]\ar[d] & \hat{V_{\frak p}}[\tfrac1p]\ar[d]\\
\hat{V/\frak p}[\tfrac1p]\ar[r] & \hat{k(\frak p)}[\tfrac1p]
}\]
to cartesian squares.
It is convenient to consider two cases to avoid degenerate statements.

Case $p\in V\setminus \frak p$. Then $\hat{V_\frak p}=\hat {k(\frak p)}=0$ and $\hat V\isoto\hat{V/\frak p}$ (for the latter isomorphism, observe that $\frak p$ is an ideal of $V_\frak p$, hence is $p$-divisible). So only the claims for the second of the three squares is non-trivial.
This one is the Milnor square associated with the henselian valuation ring $V[\tfrac1p]$ and prime ideal $\frak pV[\tfrac1p]$. But then rigidity for \'etale cohomology implies that $R\Gamma_\sub{\'et}(-,\mu_{p^r}^{\otimes j})$ (and any truncation thereof) induces equivalences on the vertical arrows in the second square.

Case $p\in\frak p$. In this case, in the second and third diagrams, the bottom terms are zero and the top horizontal arrows are isomorphisms; so the claims in (2) are trivial. Meanwhile, the square in (1) identifies, by Lemma \ref{lemma_p_compl_of_hv_square}, with the Milnor square associated with the $p$-complete valuation ring $\hat V$ and prime ideal $\frak p\hat V$; so the claims in (1) are covered by Proposition \ref{prop:nygaard-prism-exc}.
\end{proof}

\begin{proof}[Proof of Theorem~\ref{theorem:exc-cdarc}]
Using Theorem \ref{thm:ehik}, as already explained just after the statement of Theorem~\ref{theorem:exc-cdarc}, it is enough to prove that $\bb Z(j)^\sub{cdh}$ satisfies henselian $v$-excision. That is, for any henselian valuation ring $V$ and prime ideal $\frak p\subset V$, we must show that the square 
\[
\begin{tikzcd}
\bb Z(j)^\sub{cdh}(V) \ar{r} \ar{d} & \bb Z(j)^\sub{cdh}(V_{\mathfrak{p}}) \ar{d}\\
\bb Z(j)^\sub{cdh}(V/\mathfrak{p}) \ar{r} & \bb Z(j)^\sub{cdh}(\kappa(\mathfrak{p})).
\end{tikzcd}
\]
is cartesian.
It suffices to check this rationally and modulo any prime $p$. Rationally it follows from the decomposition of Theorem \ref{prop:basic_props_of_cdh_mot}(7) and the fact that $\KH$, and therefore $\KH_{\bb Q}$, satisfies Milnor excision \cite{Weibel1989a}.

Next Theorem \ref{theorem_syn_comp} implies that the previous square modulo $p$ identifies with
\[
\begin{tikzcd}
\tau^{\le j}\bb F_p(j)^\sub{syn}(V) \ar{r} \ar{d} & \tau^{\le j}\bb F_p(j)^\sub{syn}(V_{\mathfrak{p}}) \ar{d}\\
\tau^{\le j}\bb F_p(j)^\sub{syn}(V/\mathfrak{p}) \ar{r} & \tau^{\le j}\bb F_p(j)^\sub{syn}(\kappa(\mathfrak{p})).
\end{tikzcd}
\]
(we have implicitly removed $L_\sub{cdh}$, since the four rings are henselian valuation rings), which is indeed cartesian by Proposition \ref{proposition_synt_coh_hv_excision}.
\end{proof}

\subsection{Singular Beilinson-Lichtenbaum equivalence}\label{sec:singular-bl}
Recall that the integral Beilinson--Lichtenbaum equivalence states that on smooth schemes over a field or a mixed characteristic Dedekind domain, there is an equivalence of Nisnevich sheaves
\[
\bb Z(j)^{\bb A} \xrightarrow{\simeq} L_\sub{Nis}\tau^{\leq j+1}L_{\et}\Z(j)^{\bb A};
\]
see Corollary \ref{corollary:BL_over_B}. In this subsection we offer a singular analog of this result in the form of Theorem~\ref{thm:singular-bl}, involving the \'eh topology.

We begin by replacing the cdh topology in our construction of $\Z(j)^{\cdh}$ with the \'eh or h topology. The resulting \'eh-motivic cohomology provides a cdh-local analogue of \'etale motivic cohomology. We will not need the h-motivic cohomology in this paper but record in here for possible future reference.

\begin{definition}\label{def:h-eh-mot}
Let $\tau\in\{\text{\'eh},\text{h}\}$. For each $j\in\bb Z$, define $\bb Z(j)^\tau:=L_\tau\bb Z(j)^\sub{cdh}:\opp{Sch}^\sub{qcqs}\to\D(\bb Z)$ to be the $\tau$-sheafification of the cdh-motivic cohomology; equivalently, $\bb Z(j)^\tau$ is defined by $\tau$-sheafifying in Definition~\ref{def:cdh} instead of cdh sheafifying. We call $\bb Z(j)^\tau$ {\em weight-$j$ $\tau$-motivic cohomology}. Note that there are canonical comparison maps \[\bb Z(j)^\sub{cdh}\To \bb Z(j)^\sub{\'eh}\To \bb Z(j)^\sub{h},\] and that they all vanish if $j<0$.

Similarly to the cdh-local case in Remark~\ref{rem:rationalise}, we denote by \[\bb Q(j)^\tau:\opp{Sch}^\sub{qcqs}\To\D(\bb Q),\qquad X\mapsto \bb Z(j)^\tau(X)\otimes_\bb Z\bb Q\] the presheaf rationalisation of $\bb Z(j)^\tau$.
\end{definition}

The following technical lemma will be required to control sheafification:

\begin{lemma}[Singular Beilinson--Soul\'e with finite coefficients]\label{lemma_BeilinsonSoule}
For any qcqs scheme $X$, $j\ge 0$, and $m>0$, the complex $\bb Z(j)^\sub{cdh}(X)/m$ is coconnective.
\end{lemma}
\begin{proof}
By the comparison Theorem \ref{theorem_syn_comp} it is enough to prove the same for $L_\sub{cdh}\bb F_p(j)^\sub{BL}$, for all prime numbers $p$. By Proposition \ref{proposition_checking_on_points} this may be checked on henselian valuation rings $V$. But we already explained in the proof of Theorem \ref{thm:a1-comparison}(2) that $\bb F_p(j)^\sub{syn}(V)$ is coconnective, so the same certainly holds for $\bb F_p(j)^\sub{BL}(V)=\tau^{\le j}\bb F_p(j)^\sub{syn}(V)$.
\end{proof}

The following result summarises the main properties of the eh- and h-local motivic cohomologies, including in particular analogues of the cdh-local results of \S\ref{sec:lis-cdh}--\ref{ss:Milnor_excision}:

\begin{proposition}\label{prop_eh_motivic cohomology}
Let $\tau\in\{\text{\'eh},\text{h}\}$, and let $j \geq 0$.
\begin{enumerate}
\item (Rational structure) $\bb Q(j)^\tau$ is a $\tau$-sheaf and the canonical maps $\bb Q(j)^\sub{cdh}\to \bb Q(j)^\sub{\'eh}\to \bb Q(j)^\sub{h}$ of presheaves on $\opp{Sch}^\sub{qcqs}$ are equivalences.
\item (Syntomic comparison) For any prime $p$ and $r\ge1$, there are natural equivalences \[\bb Z(j)^\sub{\'eh}/p^r\quis L_\sub{cdh}\bb Z_p^\sub{syn}(j)/p^r,\qquad \bb Z(j)^\sub{h}/p^r\quis L_h\bb Z_p(j)^\sub{syn}/p^r.\]
\item (\'Etale comparison) For any prime $p$ and $r\ge1$, there is a natural equivalence \[\bb Z(j)^\sub{\'eh}/p^r\quis \bb Z(j)^\sub{h}/p^r \quis R\Gamma_\sub{\'et}(-,\mu_{p^r}^{\otimes j})\] on the category of qcqs $\bb Z[\tfrac1p]$-schemes
\item (Milnor excision) $\bb Z(j)^\tau$ is a finitary $\tau$-sheaf satisfying Milnor excision.
\item (Singular Hilbert 90) For any henselian valuation ring $V$, we have $H^{j+1}(\bb Z(j)^\sub{\'eh}(V))=0$.
\end{enumerate}
\end{proposition}
\begin{proof}
(1): Momentarily write $L_\bb Q=\colim_{m\in\bb Z\setminus\{0\}}$ for the rationalisation functor on presheaves; that is, $(L_\bb QF)(X)=F(X)\otimes_{\bb Z}\bb Q$; then $L_\tau L_\bb QL_\tau=L_\tau L_\bb Q$ for formal reasons. Applying this to $ F = \bb Z(j)^\sub{cdh}$, the presheaf $L_\tau L_\bb Q\bb Z(j)^\sub{cdh}$ is simply $\bb Q(j)^\sub{cdh}$ since we know that the latter is already a $\tau$-sheaf (Corollary~\ref{corol:qj2}).

Therefore, to prove that $L_\tau L_\bb QL_\tau \bb Z(j)^\sub{cdh}$ is $\bb Q(j)^\tau$, and so complete the proof of part (1), we will show in the rest of the proof that the presheaf $\bb Q(j)^\tau$ is already a $\tau$-sheaf. The problem here is that, a priori, it is not clear that rationalisation preserves $\tau$-sheaves. However, Lemma \ref{lemma_BeilinsonSoule} implies that multiplication by $m$, for any non-zero $m\in \bb Z$, is an equivalence on the presheaf $\tau^{< 0}\bb Z(j)^\sub{cdh}$; this property is preserved by any functor, in particular by $\tau$-sheafifying, so we have shown that the fibre of the map of $\tau$-sheaves $\bb Z(j)^\tau\to L_\tau\tau^{\ge 0}\bb Z(j)^\sub{cdh}$ is already rational. Applying $L_\bb Q$ we then deduce that the square of presheaves
\[\xymatrix{
\bb Z(j)^\tau\ar[r]\ar[d]& L_\tau\tau^{\ge 0}\bb Z(j)^\sub{cdh}\ar[d]\\
\bb Q(j)^\tau\ar[r]&L_\bb QL_\tau\tau^{\ge 0}\bb Z(j)^\sub{cdh}
}\]
is cartesian. The top row consists of $\tau$ sheaves by definition, as is the bottom right by Remark~\ref{lem:filtr-coconn} since it is a filtered colimit in presheaves of coconnective $\tau$ sheaves. Therefore the bottom left is also a $\tau$ sheaf, as required.

(2): \'{E}h-sheafifying the comparison Theorem \ref{theorem_syn_comp}, we obtain equivalences $\bb Z(j)^\sub{eh}/p^r\quis L_\sub{eh}\tau^{\le j}(\bb Z_p(j)^\sub{syn}/p^r)$. Now using Remark~\ref{rem:eh-sheaf} to write $L_\sub{\'eh}=L_\sub{cdh}L_\sub{\'et}$, and appealing to the fact that $L_\sub{\'et}\tau^{\le j}(\bb Z_p(j)^\sub{syn}/p^r)\quis \bb Z_p(j)^\sub{syn}/p^r$ (see the proof of Lemma \ref{lemma_BL_mod_p}(4)), we indeed obtain $\bb Z(j)^\sub{\'eh}/p^r\simeq L_\sub{cdh}\bb Z_p^\sub{syn}(j)/p^r$. For the analogous claim with h in place of \'eh, simply h-sheafify.

Part (3) follows from part (2), recalling Deligne's theorem that \'etale cohomology is already an h-sheaf.

(4): Since $\bb Z(j)^\tau$ is in particular a cdh sheaf, it is enough by Theorem \ref{thm:ehik} to check that it is finitary and satisfies henselian $v$-excision; it is moreover enough to check these properties rationally and modulo each prime number $p$. Rationally, they follow from part (1) and that we have already established the properties for cdh-motivic cohomology: Theorems \ref{prop:basic_props_of_cdh_mot}(4) and \ref{theorem:exc-cdarc}. Modulo any prime number $p$, part (2) reduces the claims to checking both finitariness and henselian $v$-excision for $L_\sub{cdh}\bb F_p^\sub{syn}(j)$ and $L_h\bb F_p^\sub{syn}(j)$. For $L_\sub{cdh}\bb F_p^\sub{syn}(j)$ use Theorem \ref{theorem_syntomic_properties}(2) and Proposition \ref{prop:finitary_conditions}, and it satisfies henselian $v$-excision by Proposition \ref{proposition_synt_coh_hv_excision}. For $L_h\bb F_p^\sub{syn}(j)$ the shortest proof is to identify it with \'etale cohomology of generic fibre (or the scheme itself when $j=0$), as in Theorem \ref{thm:a1-comparison}(2) and Remark \ref{rem_syn_weight_0}, and then appeal to h-descent and excision for \'etale cohomology \cite{BhattMathew2021}.

(5): The vanishing holds rationally by part (1) and the fact that $\bb Z(j)^\sub{cdh}(V)=\bb Z(j)^\sub{lse}(V)$ is supported in degrees $\le j$, by Lemma \ref{lemma_lke_from_affines}(3). So it remains to prove, for each prime number $p$, that $H^{j+1}(\bb Z(j)^\sub{\'eh}(V))$ is $p$-torsion-free, or in other words that $H^{j}(\bb Z(j)^\sub{\'eh}(V))\to H^{j}(\bb Z(j)^\sub{\'eh}(V)/p)$ is surjective. But by part (2) the target identifies (multiplicatively and compatibly with first Chern classes) with $H^j_\sub{syn}(V,\bb Z/p(j))$, which is generated by symbols by Theorem~\ref{thm_syntomic_Milnor}, and these even lift to $H^{j}(\bb Z(j)^\sub{cdh}(V))$.
\end{proof}

We finally establish the desired integral Beilinson--Lichtenbaum equivalence. By Lemma \ref{lemma_lke_from_affines}(3), we have that $\bb Z(j)^\sub{cdh}$ is cdh-locally supported in degrees $\le j+1$ (even $\le j$, but using $j+1$ leads to a stronger statement in the next theorem), so that the canonical map of presheaves $\bb Z(j)^\sub{cdh}\to \bb Z(j)^\sub{\'eh}$ induces a map of cdh sheaves \begin{equation}\bb Z(j)^\sub{cdh}\To L_\sub{cdh}\tau^{\le j+1}\bb Z(j)^\sub{\'eh}\label{eqn_BL}\end{equation} on $\opp{Sch}^\sub{qcqs}$.

\begin{theorem}[Singular Beilinson--Lichtenbaum equivalence]\label{thm:singular-bl}
For any $j\ge0$, the map of cdh sheaves (\ref{eqn_BL}) on qcqs schemes is an equivalence.
\end{theorem}
\begin{proof}
Since $\bb Q^\sub{cdh}(j)$ is cdh-locally supported in degrees $\le j+1$, the map (\ref{eqn_BL}) is an equivalence rationally by Proposition \ref{prop_eh_motivic cohomology}(1).

Hence it suffices to prove the equivalence modulo $p$ for every prime $p$. Since cdh and \'{e}h-motivic cohomology are both finitary by Theorem~\ref{prop:basic_props_of_cdh_mot}(4) and Proposition~\ref{prop_eh_motivic cohomology}(4) respectively, it suffices to check the claim on henselian valuation rings $V$. Using Proposition \ref{prop_eh_motivic cohomology}(5) to replace $(\tau^{\le j+1}\bb Z(j)^\sub{eh}(V))/p$ by $\tau^{\le j}(\bb Z(j)^\sub{\'eh}(V)/p)$, the desired claim is that the canonical map \[\bb F_p(j)^\sub{cdh}(V)\To \tau^{\le j}(\bb Z(j)^\sub{\'eh}(V)/p)\] should be an equivalence. But by the compatible comparisons of Theorem~\ref{theorem_syn_comp} and Proposition~\ref{prop_eh_motivic cohomology}(2), this map is precisely the identification $\bb F_p(j)^\sub{BL}(V)\quis \tau^{\le j}\bb F_p(j)^\sub{syn}(V)$, as required.
\end{proof}

\section{Conditional $\bb A^1$-invariance and projective bundle formula}\label{sec:pbf-a1}
The ideal goal of this section is to prove that the cdh-motivic cohomology $\Z(\star)^\cdh$ is $\A^1$-invariant and satisfies the projective bundle formula. Unfortunately, we manage to do this only under an additional assumption, which we call the \emph{key hypothesis}.
It holds, for example, in equal characteristic, and we conjecture that it holds in general.

Nonetheless, the results of this section are both fundamental to the first two authors' approach to non-$\bb A^1$-invariant motivic cohomology \cite{ElmantoMorrow2023} and serve as a crucial input to the main results of this paper which we discuss in \S\ref{sec:a1-comparison}. In particular, we will prove unconditional $\bb A^1$-invariance and projective bundle formula for cdh-motivic cohomology over fields and in ``highly ramified'' situations in Corollaries~\ref{cor:uncond_fields} and~\ref{corol:unconditional-pn} respectively.

\subsection{The key hypothesis}\label{ss:key}
Let $S$ be a scheme, $p$ a prime number, and $j\ge0$; we denote by $\Hyp(S, p, j)$ the following condition:
\begin{quote}
For every henselian valuation ring $V$ over $S$ (i.e., there exists a morphism $\Spec(V) \to S$) of mixed characteristic $(0,p)$,\footnote{In case of confusion, this means that $\opp{Frac}(V)$ has characteristic zero while the residue field of $V$ has characteristic $p$.} the cofibre of the mod-$p$ syntomic-to-\'etale comparison map \begin{equation}\gamma_\sub{syn}^\sub{\'et}\{j\}:\F_p(j)^\syn(V) \To R\Gamma_\et(V[\tfrac1p], \mu_p^{\otimes j})\simeq \F_p(j)^\syn(V[\tfrac1p]) \label{eqn_intro_syn_to_et}\end{equation}
(see Definition \ref{def:bms-decomplete}\footnote{This is the last time we use the notation $\gamma_\sub{syn}^\sub{\'et}\{j\}$: henceforth we identify syntomic and \'etale cohomology of $\bb Z[\tfrac1p]$-schemes, and so simply view the syntomic-to-\'etale comparison map as the canonical map $\bb Z_p(j)^\sub{syn}\to \bb Z_p(j)^\sub{syn}(-[\tfrac1p])$.}) is concentrated in cohomological degrees $\ge j$.
\end{quote}

In the case of a ring $A$ we will write $\Hyp(A,p,j)$ in place of $\Hyp(\Spec(A), p, j)$. 

\begin{remark} \label{rmk:Hyp-vacuous}
If $A$ contains a field then $\Hyp(A,p,j)$ automatically holds for all $p$ and $j$; indeed, it is vacuous since there are no mixed characteristic valuation rings under $A$. Similarly, if $p$ is invertible on $A$ then $\Hyp(A,p,j)$ holds for all $j$.
\end{remark}

\begin{remark}
Work of Bhatt--Mathew \cite[Proposition 5.2, Theorem 4.15]{BhattMathew2023}, recalled in Theorem \ref{theorem_syntomic_properties}(8), shows that for any regular local ring of residue characteristic $p$ the cofibre of the syntomic-to-\'etale comparison map is concentrated in cohomological degree $\ge j$. Therefore $\Hyp(S,p,j)$ would hold if any valuation ring under $S$ were known to be a filtered colimit of regular rings (which is expected to be true). In fact, Bhatt--Mathew introduce the notion of $F$-smoothness \cite[Definition 4.1]{BhattMathew2023}, and $\Hyp(S,p,j)$ would hold if it were known that any valuation ring under $S$ were $F$-smooth (which is also expected to be true).
\end{remark}

\begin{remark}\label{rem:key_reduces_to_rank1}
To establish the validity of hypothesis $\Hyp(S, p, j)$ in any given case, it suffices to treat those $V$ which are $p$-adically complete of rank one. Indeed, replacing $V$ as in the hypothesis by its $p$-adic completion does not change either side of (\ref{eqn_intro_syn_to_et}) by construction; see the pullback diagram~\eqref{eq:syn}. Therefore we may assume $V$ is $p$-adically complete. Then $\frak p:=\sqrt{pV}$ is a height one prime of $V$, so that $V_\frak p$ has rank one, and $V[\tfrac1p]=\opp{Frac} V=V_\frak p[\tfrac1p]$; recall also from Remark \ref{rmk:locn-compn-valn-ring} that $V_\frak p$ is $p$-adically complete. Applying the $v$-excision result of Proposition \ref{prop:nygaard-prism-exc} now yields a cartesian square
\[\xymatrix{
\opp{cofib}\left(\F_p(j)^\sub{syn}(V)\to R\Gamma_\et(V[1/p], \mu_p^{\otimes j}) \right)\ar[d]\ar[r] & \opp{cofib}\left(\F_p(j)^\sub{syn}(V_\frak p) \to R\Gamma_\et(V_\frak p[1/p], \mu_p^{\otimes j}) \right)\ar[d]\\
\F_p(j)^\sub{syn}(V/\frak p)\ar[r] & \F_p(j)^\sub{syn}(k(\frak p))
}\]
We wish to check that top left is concentrated in cohomological degrees $\ge j$ if and only if the same is true of the top right. So it is enough to show that the cofibre of the bottom horizontal arrow is supported in cohomological degrees $\ge j$. But both terms on the bottom row are supported in degrees $\ge j$, by Example~\ref{example_syn_in_char_p} since $k(\frak p)$ is a field and the valuation ring $V/\frak p$ is Cartier smooth; moreover, the map on $H^j$ is given by $\Omega^j_{V/\frak p,\sub{log}}\to \Omega^j_{k(\frak p),\sub{log}}$, which is injective since $\Omega^j_{V/\frak p}\to \Omega^j_{k(\frak p)}$ is injective by Gabber--Ramero \cite[Corollary 6.5.21]{GabberRamero2003}.
%
%
\end{remark}

The main theorems in this section are established for the $p$-adic motivic cohomology of $S$-schemes conditionally upon the veracity of $\Hyp(S,p,j)$ for $j\ge0$. We therefore summarise cases in which we know these conditions to hold. We begin with the trivial observation that $\Hyp(S,p,j)$ implies $\Hyp(S',p,j)$ whenever there is a morphism of schemes $S'\to S$. The mixed characteristic cases in which we know $\Hyp(S,p,j)$ to hold all depend on the following theorem of T.~Bouis \cite{Bouis2023}:

\begin{theorem}[Bouis] \label{thm:Bouis}
Let $\roi$ be a perfectoid valuation ring of mixed characteristic $(0,p)$. Then $\Hyp(\roi,p,j)$ holds for all $j\ge0$.
\end{theorem}

\begin{example}\label{example:hypothesis-holds} 
Bootstrapping from Bouis' theorem, we can check the key hypothesis in the following ``highly ramified'' examples. Let $A$ be any of the following:
\begin{enumerate}
\item a field; or
\item a zero-dimensional ring (e.g., an Artin local ring, or a possibly infinite product of fields).
\item $\bb Z^\sub{cyc}:=$ the integral closure of $\bb Z$ in the field $\bb Q(\zeta_m:m\ge1)$; or
\item an absolutely integrally closed ring.
\end{enumerate}
Then we claim that $\Hyp(A,p,j)$ is true for all prime numbers $p$ and all $j\ge0$. The easy proofs are below.

(1): See Remark \ref{rmk:Hyp-vacuous}. Case (2) reduces to (1), because any map from a zero-dimensional ring $A$ to a valuation ring factors through a residue field of $A$.

(3): Let $V$ be a $p$-adically complete valuation ring of mixed characteristic $(0,p)$ over $\bb Z^\sub{cyc}$. Then $V$ also receives a map from the $p$-adic completion of the ring of integers of $\bb Q(\zeta_{p^\infty})$, which is the prototypical example of a perfectoid valuation ring; so this case follows from Theorem \ref{thm:Bouis}.

(4): Let $A$ be an absolutely integrally closed ring and $V$ a $p$-adically complete, rank-one valuation ring of mixed characteristic $(0,p)$ over $B$. Replacing $A$ by its image in $V$ we may assume that $A$ is an integral domain such that $\opp{Frac}A$ has characteristic zero; therefore $A$ contains $\bb Z^\sub{cyc}$ and we may apply case (3).

\end{example}

In a different direction, if we restrict ourselves to low weights then we note that the key hypothesis holds unconditionally.

\begin{proposition}\label{prop:wt1-unconditional}
For any scheme $S$ and prime $p$, the conditions $\Hyp(S,p,0)$ and $\Hyp(S,p,1)$ are true.
\end{proposition}
\begin{proof}
Let $V$ be a valuation ring of mixed characteristic $(0,p)$.

For weight zero, we use that $\F_p(0)^\syn \wequi R\Gamma_\et(-, \F_p)$ by Remark \ref{rem_syn_weight_0}. So we need only prove that the map $H^0_\et(V, \F_p) \to H^0_\et(V[1/p], \F_p)$ is injective. This is clear because $H^0_\et(\ph,\F_p)$ is just the sheaf of locally constant $\F_p$-valued functions, $\Spec(V)$ is connected (being the spectrum of a local ring) and $\Spec(V[1/p])$ is nonempty.

For weight one, we use that $\F_p(1)^\syn \wequi \fib(R\Gamma_\et(-, \Gm) \xrightarrow{p} R\Gamma_\et(-, \Gm))$ by Proposition \ref{proposition_1st_Chern_class_syn}. Since $H^1_\et(-, \Gm) \wequi \opp{Pic}(-)$ \cite[Tag 03P7]{Stacks}, we see that for any local ring $A$, there are natural isomorphisms
\[
H^0_\sub{syn}(A,\F_p(1)) \cong \mu_p(A) \qquad H^1_\sub{syn}(A,\F_p(1)) \cong A^\times/p.
\]
It thus remains to prove that $\mu_p(V) \to \mu_p(V[\tfrac1p])$ is an isomorphism and that $V^\times/p \to V[\tfrac1p]^\times/p$ is injective. But $V$ is integrally closed in its field of fractions \cite[Tag 00IC]{Stacks}, so whenever $f\in V[\tfrac1p]^\times$ satisfies $f^p\in V^\times$ then in fact $f\in V^\times$; both claims easily follow from this.
\end{proof}

\subsection{Conditional $\A^1$-invariance of $\bb Z(j)^\sub{cdh}$}\label{sec:A1-inv}
In this section we prove the $\bb A^1$-invariance of cdh-motivic cohomology, conditionally on the key hypothesis. We begin by observing that our desired result is easy (and unconditional) rationally:

\begin{proposition}\label{prop:a1-rational} For any $j \in \Z$ and qcqs scheme $X$, the canonical map
\[
\Q(j)^{\cdh}(X) \rightarrow \Q(j)^{\cdh}(\A_X^1)
\]
is an equivalence.
\end{proposition}
\begin{proof} This follows immediately from $\bb A^1$-invariance of $\KH$ and the rational decomposition of Theorem~\ref{prop:basic_props_of_cdh_mot}(7).
\end{proof}

We now work towards establishing $\A^1$-invariance integrally, conditional on the key hypothesis; for the rest of this section we fix a prime number $p$. We begin with the following warm-up lemma.

\begin{lemma}\label{lem:i!Omegalog-A1}
For any field $k$ of characteristic $p$, the maps $H^n_\cdh(k, \Omega^j_{\log}) \rightarrow H^n_\cdh(k[T], \Omega^j_{\log})$ are isomorphisms for all $n,j\ge0$.
\end{lemma}
\begin{proof}
We offer two proofs of this result. The first proof is overkill, but illustrates one of our main techniques: we can extract information about motivic cohomology from $K$-theory. From the cdh-motivic Atiyah--Hirzebruch filtration of Theorem \ref{prop:basic_props_of_cdh_mot}(1)\&(3), and Corollary \ref{cor:singular_gl}, we know that for any qcqs $\bb F_p$-scheme $X$ of finite valuative dimension, $\KH(X)/p$ admits a natural bounded filtration with graded pieces $R\Gamma_\sub{cdh}(X,\Omega^j_\sub{log})[j]$ for $j\ge0$. If $X$ has valuative dimension $\le 1$ then it follows that the resulting natural edge maps $\KH_j(X;\bb Z/p\bb Z)\to H^0_\sub{cdh}(X,\Omega^j_\sub{log})$ are surjective, with kernel $H^1_\sub{cdh}(X,\Omega^{j+1}_\sub{log}$).  Applying this to both $k$ and $k[T]$ supplies us with a commutative diagram in which the rows are exact
\[\xymatrix{
0\ar[r]&H^1_\cdh(k, \Omega_{\log}^{j+1})\ar[d]\ar[r]&\KH_j(k;\Z/p) \ar[d]_{\cong} \ar[r] & H^0_\cdh(k, \Omega_{\log}^j) \ar[d]\ar[r]&0\\
0\ar[r]&H^1_\cdh(k[T], \Omega_{\log}^{j+1}) \ar[r] & \KH_j(k[T];\Z/p) \ar[r] & H^0_\cdh(k[T], \Omega_{\log}^j)\ar[r] & 0
}\]
The middle vertical isomorphism comes from the $\A^1$-invariance of $\KH$. Therefore the right vertical arrow is surjective; but it is also split injective, so it is an isomorphism, and therefore the left vertical arrow is also an isomorphism.

Here is a second, more direct proof, which we remark is similar to the final step of the proof of Lemma~\ref{lem:p1-milnor} below. We wish to show that $R\Gamma_\cdh(k, \Omega^j_{\log}) \rightarrow R\Gamma_\cdh(k[T], \Omega^j_{\log})$ is an equivalence. We note that both sides are unchanged if we replace cdh cohomology by Nisnevich cohomology: for $k$ this is because it is a point for the cdh topology, and for $k[T]$ it holds because its Nisnevich local rings are henselian discrete valuation rings, which are also points for the cdh topology. But the map $R\Gamma_\sub{Nis}(k, \Omega^j_{\log}) \rightarrow R\Gamma_\sub{Nis}(k[T], \Omega^j_{\log})$ is indeed an equivalence by Theorem \ref{thm_BL_coh} modulo $p$, in light of Example \ref{example_BL_cohomology}(2).
\end{proof}

The main way in which we will use the key hypothesis is to obtain a bound on the difference between $\bb F_p(j)^{\cdh}$ and mod-$p$ \'etale cohomology of the $p$-adic generic fibre. More precisely, in the rest of this section we will be interested in the $\D(\bb Z)$-valued presheaves on qcqs schemes
\begin{equation}
C(j):= \mathrm{cofib}(\bb F_p(j)^\cdh \To R\Gamma_{\et}(-[\tfrac{1}{p}],\mu_p^{\otimes j}))
\label{eqn:cdh-mot-to-et}\end{equation}
for $j\in\bb Z$, where the map appearing is the composition \[\bb F_p(j)^\cdh\stackrel{\sub{Thm.~\ref{theorem_syn_comp}}}\quis L\sub{cdh}\tau^{\le j}\bb F_p(j)^\sub{syn}\to L_\sub{cdh}\bb F_p(j)^\sub{syn}\to L_\sub{cdh}R\Gamma_\sub{\'et}(-[\tfrac1p],\mu_p^{\otimes j})=R\Gamma_\sub{\'et}(-[\tfrac1p],\mu_p^{\otimes j}).\]
(If $j\ge1$, then Theorem \ref{thm:a1-comparison} shows that this map identifies with $\bb F_p(j)^\cdh\to L_h\bb F_p(j)^\cdh$, but we do not need this). Here are some preliminary properties of $C(j)$ which will be required:

\begin{lemma} \label{lemm:use-key-hyp}
Let $S$ be a qcqs scheme of finite valuative dimension, fix $j\ge0$, and assume that $\Hyp(S, p, j)$ holds. Then we have the following:
\begin{enumerate}
\item For any qcqs $S$-scheme $X$, the complex $C(j)(X)$ is concentrated in cohomological degrees $\ge j-1$.
\item The lowest cohomology presheaf, namely $H^{j-1}(C(j)(-)):\Sch_S^\sub{op}\to\opp{Ab}$, is isomorphic to \[a_\sub{cdh}i_!\Omega^j_{\log}:=\text{the cdh sheafification of the presheaf of abelian groups }
X\mapsto \begin{cases} H_{\Zar}^0(X, \Omega^j_{\log}) & X \in \Sch^{\qcqs}_{S_{\F_p}} \\ 0 & \text{else.} \end{cases}
\]
\item The presheaf of abelian groups $H^{j-1}(C(j))$ is $\bb A^1$-invariant on qcqs $S$-schemes.
\end{enumerate}
\end{lemma}
\begin{proof}
(1): Since $C(j)$ is the cofibre of a map between two finitary cdh sheaves, it is itself a finitary cdh sheaf and therefore by Proposition \ref{proposition_checking_on_points}(1) is enough to show that $C(j)(V)$ is concentrated in cohomological degrees $\ge j-1$ whenever $V$ is a henselian valuation ring admitting a map $\Spec(V) \rightarrow S$. From \eqref{eqn:cdh-mot-to-et} we obtain (via the ``octahedral axiom'') a fibre sequence of presheaves
\[
\mathrm{cofib}(L_\cdh \tau^{\le j}\F_p(j)^{\syn} \to L_\cdh \F_p(j)^{\syn}) \To C(j) \To \mathrm{cofib}(L_\cdh \F_p(j)^{\syn} \rightarrow R\Gamma_\sub{\'et}(-[\tfrac1p],\mu_p^{\otimes j})),
\]
where the first term is clearly concentrated in cohomological degrees $\ge j+1$. Evaluating the third term on $V$ yields the cofiber of $\F_p(j)^{\syn}(V) \rightarrow R\Gamma_{\et}(V[\tfrac{1}{p}],\mu_p^{\otimes j})$, which we claim is concentrated in degree $\ge j-1$ (as required to complete the proof of (1)). If $V$ has mixed characteristic $(0,p)$ then it is even concentrated in degrees $\ge j$ by the key hypothesis $\Hyp(S, p,j)$; if $p\in V^\times$ then the map is an equivalence; finally, if $p=0$ in $V$ then the codomain vanishes and $\F_p(j)^{\syn}(V)$ is supported in degrees $\ge j$ by the final sentence in Example \ref{example_syn_in_char_p}.

(2): It will be useful to use the notation of the adjunction
\[
i_!:\PShv(\Sch^{\qcqs}_{S_{\F_p}},\D(\Z)) \leftrightarrows \PShv(\Sch^{\qcqs}_{S}, \D(\Z)):i^*;
\]
where $i^*$ is restriction and $i_!$ is given by
\[
i_!E(X) = \begin{cases} E(X) & X \in \Sch^{\qcqs}_{S_{\F_p}} \\ 0 & \text{else} \end{cases}.
\]
(The reader might find it helpful to consult Remark \ref{rem:values-of-functors}, even though it is rather $j_!$ which appears there.)

Now, on $\bb F_p$-schemes, the boundary map $C(j)[-1]\to \bb F_p(j)^\sub{cdh}$ is an equivalence (since the \'etale cohomology in \eqref{eqn:cdh-mot-to-et} vanishes); combined with Corollary \ref{cor:singular_gl} we deduce that $i^*C(j)[-1]\simeq R\Gamma_\sub{cdh}(-,W_r\Omega^j_\sub{log})[-j]$. Via the counit of the adjunction this induces $i_!R\Gamma_\sub{cdh}(-,W_r\Omega^j_\sub{log})[-j] \simeq i_!i^*C(j)[-1] \rightarrow C(j)[-1]$, which we then cdh sheafify (using $L_\sub{cdh}i_!L_\sub{cdh}\simeq L_\sub{cdh}i_!$) to obtain a comparison map
\[L_{\cdh}i_!H^0_\sub{Zar}(-,\Omega^j_\sub{log})[-j] \To C(j)[-1].\] 
Both sides of this comparison map are supported in degrees $\ge j$ (clearly for the left, and by part (1) for the right), and we claim that its cofiber is even supported in degrees $\ge j+1$ (which will complete the proof of (2)). As in part (1), it is enough to check this bound after evaluating on henselian valuation rings $V$ admitting a map $\Spec(V) \rightarrow S$. If $V$ has mixed characteristic $(0,p)$ then the left hand side vanishes and we saw in (1) that $C(j)(V)$ was supported in degrees $\ge j$. If $p\in V^\times$ then the left hand side again vanishes and we saw in (1) that $C(j)(V)$ was supported in degrees $\ge j+1$. Finally, if $p=0$ in $V$ then the comparison map for $V$ is an equivalence by construction.

(3): We claim that for any qcqs $S$-scheme $X$, the canonical map \begin{equation}H^{j-1}(C(j)(X)) \To \prod_{x\in X_{\bb F_p}} H^{j-1}(C(j)(k(x)))\label{eqn:inj}\end{equation} is an injection, where the product is over all points of $X$ of residue characteristic $p$. Since $H^{j-1}(C(j))(-)$ is a cdh sheaf of abelian groups by (2), and is thus cdh separated, the claim reduces to the case where $X$ is the spectrum of a henselian valuation ring $V$. We may further assume that $V$ has characteristic $p$ (as otherwise $H^{j-1}(C(j)(V))$ vanishes by (2)). In this case, $\Omega^n_{V,\sub{log}} \to \Omega^n_{\opp{Frac}(V),\sub{log})}$ is injective (this follows from Gabber--Ramero's result that $\Omega^n_{V} \to \Omega^n_{\opp{Frac}(V)}$ is injective \cite[Corollary 6.5.21]{GabberRamero2003}, as already used at the end of Remark \ref{rem:key_reduces_to_rank1}), proving the claim.

To prove part (3) we must show, for any qcqs $S$-scheme $X$, that the map $0^*:H^{j-1}(C(j)(\A^1_X)) \to H^{j-1}(C(j)(X))$ is injective. Using the injectivity of \ref{eqn:inj} for $\A_X^1$, this reduces to the case of $X=\Spec (k)$ where $k$ is a field of characteristic $p$. Using (2), this explicitly means showing that $0^*:H^0_\sub{cdh}(k[T],\Omega^j_\sub{log})\to H^0_\sub{cdh}(k,\Omega^j_\sub{log})$ is injective; but that follows from Lemma \ref{lem:i!Omegalog-A1}.
\end{proof}

\begin{remark}\label{rem:proof_of_BL_from_intro}
Lemma \ref{lemm:use-key-hyp}(1) is a conditional Beilinson--Lichtenbaum result, comparing the motivic cohomology of $X$ to the \'etale cohomology of $X[\tfrac1p]$. Combined with Corollary \ref{corol_cdh_is_A}, it establishes Theorem \ref{thm:BL_intro} of the introduction.
\end{remark}

The next theorem is the main $\bb A^1$-invariance result of the present paper.

\begin{theorem} \label{thm:A1inv}
Let $S$ be a scheme of finite valuative dimension, fix $N\ge0$, and assume that $\Hyp(S,p,j)$ holds for all $j \le N$. Then, for any qcqs $S$-scheme $X$ and integer $j\le N$, the canonical map \[ \Z_{(p)}(j)^\cdh(X) \to \Z_{(p)}(j)^\cdh(\A^1_X)\] is an equivalence.
\end{theorem}

\begin{proof} We have already proved that the map $\Z(j)^{\cdh}(X) \rightarrow \Z(j)^{\cdh}(\A^1_X)$ is a rational equivalence in Proposition~\ref{prop:a1-rational}. Hence it suffices to verify the claim after mod-$p$ reduction, i.e., we are left to prove that $\bb F_p(j)^{\cdh}(X) \rightarrow \bb F_p(j)^{\cdh}(\A^1_X) $ is an equivalence. For any $j$, we denote the cofibre of the map $\bb F_p(j)^{\cdh}(-) \rightarrow \bb F_p(j)^{\cdh}(\A^1_{(-)})$ in presheaves of complexes by $B(j)$; our goal is then to prove that $B(j) = 0$ for all $j \le N$.

We prove this in several steps, which we enumerate since the proof of Theorem \ref{thm:pbf-fields} will follow a similar sequence of logic.

{\bf Step 1.} We reduce to the assertion that for any $j \leq N$, and any henselian valuation ring $V$ of rank $\le 1$ admitting a map $\Spec(V)\to S$, we have $B(j)(V) = 0$. To perform this reduction, we need to verify that $B(j)$ is a finitary cdh sheaf which satisfies henselian $v$-excision. Certainly it is a finitary cdh sheaf, since $\Z(j)^{\cdh}$ is by Theorem \ref{prop:basic_props_of_cdh_mot}(4). Moreover, $\Z(j)^{\cdh}$ satisfies Milnor excision by Theorem~\ref{theorem:exc-cdarc}, and the same is therefore true of $\Z(j)^{\cdh}(\bb A^1_{-})$ since Milnor squares are stable under flat pullbacks \cite[Lemma 3.2.8]{ElmantoHoyoisIwasaKelly2021}. Therefore the mod-$p$ cofibre $B(j)$ satisfies Milnor excision, in particular henselian $v$-excision, as required. In the rest of the proof, let $V$ be a henselian valuation ring $V$ of rank $\le 1$ admitting a map $\Spec(V)\to S$.

{\bf Step 2.} We construct a diagram that we shall investigate in the next step. Evaluating \eqref{eqn:cdh-mot-to-et} on both $V$ and $V[T]$, and using $\A^1$-invariance of \'etale cohomology away from residual characteristics \cite[Corollaire XV.2.2]{SGA_IV_III}, we obtain the following diagram where each row and column is a cofibre sequence:

\begin{equation}\label{eq:cofibre-a1}
\begin{tikzcd}
\bb F_p(j)^\sub{cdh}(V) \ar{r} \ar{d} &   \bb F_p(j)^\sub{cdh}(V[T]) \ar{r} \ar{d} &  B(j)(V)  \ar{d} \\
R\Gamma_{\et}(V[\tfrac{1}{p}],\mu_p^{\otimes n}) \ar{r} \ar{d} & R\Gamma_{\et}(V[\tfrac1p][T],\mu_p^{\otimes j})  \ar{r} \ar{d} & 0 \ar{d} \\
C(j)(V) \ar{r} & C(j)(V[T]) \ar{r} & B(j)(V)[1].
\end{tikzcd}
\end{equation}

{\bf Step 3a.} We observe, for any $j\ge0$, that $B(j)(V)$ is concentrated in cohomological degrees $\le j+2$. Indeed, this follows from the top row of \eqref{eq:cofibre-a1}: $V[T]$ is of valuative dimension $\leq 2$, so the middle term is concentrated in cohomological degrees $\leq j+2$ by Theorem \ref{prop:basic_props_of_cdh_mot}, and similarly the left term is concentrated in cohomological degrees $\le j+1$ (even $\le j$ since $V$ is a point for the cdh topology).

{\bf Step 3b.} Next we claim for any $j \leq N$, that $B(j)(V)$ is also concentrated in cohomological degrees $\geq j+1$. This time we consider the third row of~\eqref{eq:cofibre-a1}. The complex $C(j)(V[T])$ is concentrated in degrees $\geq j-1$ by Lemma~\ref{lemm:use-key-hyp}(1), making use of the assumption $\Hyp(S, p, j)$, and so we obtain the same bound for $B(j)(V)[1]$ since the cofibre sequence is split. Furthermore, $H^{j-1}(B(j)(V)[1])$ is the cokernel of the canonical map $H^{j-1}(C(j)(V)) \to H^{j-1}(C(j))(V[T])$, which is zero by Lemma~\ref{lemm:use-key-hyp}(3). This completes the proof that $B(j)(V)[1]$ is concentrated in degrees $\geq j$, as claimed.

{\bf Step 4} We degenerate enough of a cdh-local Atiyah-Hirzebruch spectral sequence~\eqref{eq:cdh-ahss} to conclude that $B(j)(V) = 0$ for $j \leq N$. Indeed, equipping both $\KH(V)/p$ and $\KH(V[T])/p$ with the cdh-local Atiyah--Hirzebruch filtration of Theorem \ref{prop:basic_props_of_cdh_mot}(1)\&(3), and taking the cofibre, we see that $\opp{cofib}(\KH(V) \rightarrow \KH(V[T]))/p$ admits a bounded $\bb N$-indexed filtration with graded pieces $B(j)(V)[2j]$, for $j\ge0$. But $\KH(V)/p \rightarrow \KH(V[T])/p$ is an equivalence, so this translates into a convergent spectral sequence abutting to zero, namely $E_2^{i,j}=H^{i-j}(B(-j)(V)) \Rightarrow 0$. Writing $H^{i}(j):= H^i(B(j)(V))$ for simplicity, the bounds obtained in step 3 imply that this spectral sequence looks as follows:

\begin{equation*}
\begin{tikzcd}[row sep=small]
0& 0 & H^1(0) \ar{rrd} & H^2(0) & 0\\
0 & 0 & H^2(1) \ar{rrd}& H^3(1) & 0\\
0 & 0 & H^3(2)  \ar{rrd} & H^4(2) & 0\\
0 & 0  & H^4(3) \ar{rrd} & H^5(3) & 0 \\
\vdots & \vdots & \vdots & \vdots & \vdots\\
0 & 0  & H^{N+1}(N) \ar{rrd} & H^{N+2}(N) & 0 \\
\vdots & H^{N+1}(N+1) & H^{N+2}(N+1)  & H^{N+3}(N+1) & 0.\\
\end{tikzcd}
\end{equation*}
Evidently, whenever $j\le N$, then the groups $H^{j+1}(j)$ and $H^{j+2}(j)$ survive the spectral sequence and thus are zero. Therefore $B(j)(V) = 0$ for $j \leq N$, completing the proof.
\end{proof}

We explicitly point out some unconditional consequences of the previous theorem:

\begin{corollary}\label{cor:uncond_fields}
Let $A$ be a perfectoid valuation ring of finite rank, or a field, or a zero-dimensional ring, or $\bb Z^\sub{cyc}$, or an absolutely integrally closed ring of finite valuative dimension. Then, for any qcqs $A$-scheme $X$ and $j\in\bb Z$, the canonical map
\[ \Z(j)^\cdh(X) \To \Z(j)^\cdh(\A^1_X)\]
is an equivalence.
\end{corollary}
\begin{proof}
By Theorem \ref{thm:Bouis} and Example \ref{example:hypothesis-holds}, they key hypothesis holds for $A$ for all $j$ and all primes. So the result follows from Theorem \ref{thm:A1inv}.
\end{proof}

We do not know how to prove the following result without using algebraic $K$-theory:

\begin{corollary}\label{corol:low_weights_a1}
The presheaves $R\Gamma_\sub{cdh}(-,\bb Z)$ and $R\Gamma_\sub{cdh}(-,\bb G_m)$ are $\bb A^1$-invariant on all qcqs schemes.
\end{corollary}
\begin{proof}
Recall that $R\Gamma_\sub{cdh}(-,\bb Z)\simeq \bb Z(0)^\sub{cdh}$ and $R\Gamma_\sub{cdh}(-,\bb G_m)[-1]\simeq \bb Z(1)^\sub{cdh}$ by Theorem \ref{prop:basic_props_of_cdh_mot}(5)\&(6). The key hypothesis holds at weights $0$ and $1$ for all primes, by Proposition \ref{prop:wt1-unconditional}, so the $\bb A^1$-invariance claim follows from Theorem \ref{thm:A1inv}.
\end{proof}

\subsection{Conditional projective bundle formula}\label{sec:pbf-fields}
In this section we prove the projective bundle formula for cdh-motivic cohomology, conditionally on the key hypothesis. We first recall the set-up: in Theorem \ref{prop:basic_props_of_cdh_mot}(6) we constructed a first Chern class equivalence \[c_1^\sub{cdh}: R\Gamma_\sub{cdh}(-,\bb G_m)[-1]\quis\bb Z(1)^\sub{cdh}\] by cdh-locally left Kan extending the analogous first Chern class equivalence of Corollary~\ref{cor:low-wts-ddk-fields}. In particular, this induces a natural map which we abusively denote
\[ c_1^{\cdh}:\opp{Pic}(X) = H^1_\Zar(X, \Gm) \to H^1_{\cdh}(X, \Gm) \isoto H^2_{\cdh}(X, \bb Z(1))\] for any qcqs scheme $X$. This allows us to formulate the projective bundle formula in the same way as in Theorem~\ref{thm_pbf_Acdh}: for any qcqs scheme $X$ and any rank $d$ projective bundle $\pi:P\to X$, and $j\ge0$, we have the map
\begin{equation}\label{eqn_pbf_cdh}
\bigoplus_{i=0}^d\bb Z(j-i)^{\cdh}(X)[-2i]\To \bb Z(j)^{\cdh}(P),
\end{equation}
induced by powers of the first cdh Chern class $c_1^{\cdh}(\roi_P(1)) \in H^2_{\cdh}(P, \bb Z(1))$ of the canonical line bundle $\roi_{P}(1)$.


From what we have already done, it is easy to establish that the projective bundle formula holds rationally for cdh-motivic cohomology:

\begin{proposition}\label{prop:q-pbf}
Let $X$ be a qcqs scheme, $P\to X$ a rank $d$ projective bundle, and $j\ge0$. Then the map \eqref{eqn_pbf_cdh} is rationally an equivalence:
\[
\bigoplus_{i=0}^d\bb Q(j-i)^{\cdh}(X)[-2i]\quis \bb Q(j)^{\cdh}(P)
\]
\end{proposition}
\begin{proof}
The canonical map $\Z(\star)^{\cdh}\to \Z(\star)^{\bb A, \cdh}$ of presheaves of $\bb E_\infty$-algebras is rationally an equivalence by Corollary \ref{corol:qj2}, and compatible with first Chern classes (since $c_1^\sub{cdh}$ is defined by cdh-locally left Kan extending $c_1^{\bb A,\sub{cdh}}$). So the proposition reduces to the projective bundle formula for $\bb A^1$-cdh-motivic cohomology, established in Theorem \ref{thm_pbf_Acdh}.
\end{proof}

%
%
%
Unfortunately, we do not know at this moment that $\Z(\star)^{\cdh}[2\star]$ is represented by a motivic spectrum (even if we assume the key hypothesis to ensure $\bb A^1$-invariance); for this we need to at least know that it enjoys the $\bb P^1$-bundle formula at all weights. Therefore, Morel's Theorem~\ref{thm_pbf_for_general_E} does not apply to show that \eqref{eqn_pbf_cdh} is an equivalence integrally. Instead we proceed as in the case of $\bb A^1$-invariance, using the key hypothesis to extract the projective bundle formula from the analogous property of $\KH$.

\begin{theorem}\label{thm:pbf-fields}
Let $S$ be a scheme of finite valuative dimension, fix $N\ge0$, and assume that $\Hyp(S,p,j)$ holds for all $j \le N$. Then, for any rank $d$ projective bundle $\pi:P\to X$ of qcqs $S$-schemes and $j\le N$, the map from~\eqref{eqn_pbf_cdh}
\[
\bigoplus_{i=0}^d\bb Z_{(p)}(j-i)^{\cdh}(X)[-2i]\To \bb Z_{(p)}(j)^{\cdh}(P),
\]
is an equivalence.
\end{theorem}

\begin{proof} First, since for all $j \ge 0$ the presheaf $\Z_{(p)}(j)^{\cdh}$ is a Zariski sheaf and the maps from~\eqref{eqn_pbf_cdh} are compatible with restriction to opens of $X$, we may reduce to the case that $P$ is trivialisable and hence to $P = \bb P_X^d \to X$ for $d \geq 0$. The case $d = 0$ is vacuous. We now treat the case of $d = 1$, i.e., the $\bb P^1$-bundle formula.

Since we have already proved the claim rationally in Proposition~\ref{prop:q-pbf}, it suffices to prove the claim modulo $p$. Before we proceed, let us explain some notation. For any presheaf $F: \Sch_X^\sub{qcqs,op} \rightarrow \Spt$, denote by 
\[
\widetilde{F}(\P^1_X) := \mathrm{fib}(F(\P^1_X) \xrightarrow{\infty^*} F(X)) \qquad \widetilde{F}(\Gm_X) := \mathrm{fib}(F(\Gm_X) \xrightarrow{1^*} F(X)).
\]
Observing that 1) the map $c_1^\sub{cdh}(\scr O(1))\pi^*: \bb F_p(j-1)^{\cdh}(X)[-2] \rightarrow \bb F_p(j)^\sub{cdh}(\bb P^1_X)$ vanishes after postcomposition with $\bb F_p(j)^{\cdh}(\P^1_X) \xrightarrow{\infty^*} \bb F_p(j)^{\cdh}(X)$ and 2) the composite $\bb F_p(j)^\cdh(X)  \xrightarrow{\pi^*} \bb F_p(j)^\cdh(\P^1_X)  \xrightarrow{\infty^*} \bb F_p(j)^\cdh(X)$ is equivalent to the identity, the $\bb P^1$-bundle formula then amounts to showing that the map
\begin{equation}\label{eq:reduced}
\bb F_p(j-1)^{\cdh}(X)[-2] \xrightarrow{c^\sub{cdh}_1(\scr O(1))\pi^*} \widetilde{\bb F_p(j)^{\cdh}}(\P^1_X)
\end{equation}
is an equivalence for $j \leq N$.

Just as in Theorem~\ref{thm:A1inv} we employ an argument that involves the degeneration of the cdh-local Atiyah--Hirzebruch spectral sequence.
By the compatibility of the first Chern class for cdh-motivic cohomology with the orientation of $\KH$-theory, as presented in Lemma \ref{lemma:HZA-orient-KGL-compat} and Construction \ref{cons:low-weights}, we see that \eqref{eq:reduced} is the shifted graded pieces of the map 
\[
\Fil_{\cdh}^{\star-1}\KH(X) \xrightarrow{(1-[\scr O(-1)])\pi^*} \widetilde{\Fil_{\cdh}^{\star}\KH}(\P^1_X),
\]
which is an equivalence on underlying spectra by the $\bb P^1$-bundle formula for $\KH$. Let us write $D(j)(X)$ for the cofibre of~\eqref{eq:reduced}. We need to show that $D(j)(X) = 0$ for all $j \le N$, which we do following similar steps of argument as in the proof of Theorem~\ref{thm:A1inv}. 

{\bf Step 1.}
Exactly as in Step 1 of Theorem~\ref{thm:A1inv}, we can reduce to checking that for any $j \leq N$, and any henselian valuation ring $V$ of rank $\le 1$ admitting a map $\Spec(V)\to S$, we have $D(j)(V) = 0$.

{\bf Step 2.}
Completely analogously to Step 2 of Theorem~\ref{thm:A1inv}, we construct the following diagram, which we shall investigate in the next step:
\begin{equation}\label{eq:cofibre-p1}
\begin{tikzcd}
\bb F_p(j-1)^{\cdh}(V)[-2] \ar{r} \ar{d} &   \widetilde{\bb F_p(j)^{\cdh}}(\P^1_V) \ar{r} \ar{d} &  D(j)(V)  \ar{d} \\
R\Gamma_{\et}(V[\tfrac{1}{p}];\mu_p^{\otimes j-1})[-2] \ar{r} \ar{d} & \widetilde{R\Gamma_{\et}}(\P^1_{V[\tfrac{1}{p}]},\mu_p^{\otimes j})  \ar{r} \ar{d} & 0 \ar{d} \\
C(j-1)(V)[-2] \ar{r} & \widetilde{C(j)}(\P^1_V) \ar{r} & D(j)(V)[1].
\end{tikzcd}
\end{equation}
%
%
%
Each row and column is a cofibre sequence (for the middle row, use the projective bundle formula in étale cohomology, see Theorem \ref{theorem_syntomic_properties}(6)).

{\bf Step 3a.}
We observe, for any $j\ge0$, that $D(j)(V)$ is concentrated in cohomological degrees $\leq j+2$. This follows from the same reasoning as Step 3a of Theorem~\ref{thm:A1inv} since that the valuative dimension of $\P^1_V$ is $\leq 2$.

{\bf Step 3b.}
This step is more involved that the corresponding Step 3b of Theorem~\ref{thm:A1inv}. We claim for any $j \leq N$, that $D(j)(V)$ is concentrated in cohomological degrees $\geq j+1$.  To begin proving this claim, consider the third row of~\eqref{eq:cofibre-p1}.  By Lemma~\ref{lemm:use-key-hyp}(1), the first term is concentrated in cohomological degrees $\geq j$. We claim that the second term is, in fact, concentrated in degrees $\geq j$ as well (Lemma~\ref{lemm:use-key-hyp} tells us that it is concentrated in degrees $\geq j-1$). To see this, note that the functor $C(j)$ is $\A^1$-invariant since $\bb F_p(j)^{\cdh}$ is already known to be $\A^1$-invariant by Theorem~\ref{thm:A1inv} for $j \leq N$. Therefore, coupled with Zariski excision applied to the usual open cover of $\bb P^1_V$ using two copies of $\bb A^1_V$, we have that $\widetilde{C(j)}(\P^1_V) \simeq \widetilde{C(j)}(\bb G_{m,V})[-1]$. We thus conclude that $\widetilde{C(j)}(\P^1_V)$ is, in fact, concentrated in degrees $\geq j$. Therefore, $D(j)(V)[1]$ is concentrated in degrees $\geq j-1$ and thus $D(j)(V)$ is concentrated in degrees $\geq j$. 
Consequently $D(j)(V)$ is concentrated in degrees $\geq j+1$ if and only if the map $H^{j-2}(C(j-1)(V)) \to H^j(\widetilde{C(j)}(\P^1_V))$ is injective; we now verify this injectivity statement.  Denoting by $K$ the fraction field of $V$, we first claim that $H^{j-2}(C(j-1)(V)) \to H^{j-2}(C(j-1)(K))$ is injective. To see this, we note that the domain is zero if $V$ is not characteristic $p$ (as a special case of Lemma~\ref{lemm:use-key-hyp}(2)) and its injectivity was proved in the course of showing that~\eqref{eqn:inj} was injective when $V$ is of characteristic $p > 0$. Therefore, by the functoriality of the comparison map $H^{j-2}(C(j-1)(-)) \to H^j(\widetilde{C(j)}(\P^1_{-}))$ we may replace $V$ by its fraction field $K$. But now, since now $\P^1_K$ has valuative dimension $1$, the cohomology groups of $D(j)(K)$ are concentrated in two degrees (the bound of Step 3a improves by 1). Thus, similarly to step 4 of Theorem~\ref{thm:A1inv}, we can degenerate the Atiyah--Hirzebruch spectral sequence to prove that $D(j)(K) = 0$, which implies the desired injectivity.

{\bf Step 4.}
We finally argue exactly as in Step 4 of Theorem~\ref{thm:A1inv}, since the cohomology groups of $D(j)(V)$ are concentrated in at most two degrees.  This completes the proof of the theorem in the case $d=1$.

We now proceed by induction to treat the general case $d\ge1$, following \cite[Proof of Theorem 2.1.13]{Deglise2019}; we only sketch the argument, remarking that the result follows formally from the $\bb P^1$-bundle formula and the $\bb A^1$-invariance of $\bb Z_{(p)}(j)^\sub{cdh}$ for $j \leq N$. Consider the following diagram of fibre sequences (where the bottom sequence is split):
\[
\begin{tikzcd}
 \mathrm{fib}(\bb Z_{(p)}(j)(\bb P_X^d) \rightarrow \bb Z_{(p)}(j)(\bb P_X^{d-1})) \ar{r} &  \bb Z_{(p)}(j)(\bb P_X^d) \ar{r} & \bb Z_{(p)}(j)(\bb P_X^{d-1})\\
\bb Z_{(p)}(j-d)(X)[-2r] \ar{u} \ar{r} & \bigoplus^{d}_{i=0} \bb Z_{(p)}(j-i)(X)[-2i] \ar{r} \ar{u} & \bigoplus^{d-1}_{i=0} \bb Z_{(p)}(j-i)(X)[-2i] \ar[swap]{u}{\simeq}.
\end{tikzcd}
\]
Here the top right horizontal map is induced by the inclusion into the first $d-1$-coordinates $\bb P_X^{d-1} \rightarrow \bb P_X^d$, and the middle and rightmost vertical maps are the ones given by~\eqref{eqn_pbf_cdh}. The composite from the bottom left term to the top middle term is given by multiplication by $c_1(\roi_{\bb P^r}(1))^{d}$ and thus is null-homotopic when composed with the map into the top right term; this gives the factorization onto the fibre term. The inductive hypothesis gives the labeled equivalence. To finish the proof, it suffices to prove that the left-most vertical map is an equivalence. By Theorem~\ref{thm:A1inv}, the presheaves $\bb Z_{(p)}(j)$ are all $\bb A^1$-invariant for $j = 0, \cdots, N$ and they are, by definition, Nisnevich sheaves. Therefore the claimed equivalence follows from a calculation in unstable motivic homotopy theory (in other words, $\bb A^1$-invariant Nisnevich sheaves of pointed spaces) following the last paragraph of \cite[Proof of Theorem 2.1.13]{Deglise2019} and \cite[Lemma 2.1.14]{Deglise2019} which we can apply to these presheaves.
\end{proof}


Just as in the case of $\bb A^1$-invariance, we explicitly point out some unconditional consequences of the previous theorem:

\begin{corollary}\label{corol:unconditional-pn}
Let $A$ be a perfectoid valuation ring of finite rank, or a field, or a zero-dimensional ring, or $\bb Z^\sub{cyc}$, or an absolutely integrally closed ring of finite valuative dimension, or if $p$ is invertible in $A$.  Then, for any rank $d$ projective bundle $\pi:P\to X$ of qcqs $A$-schemes, the maps from~\eqref{eqn_pbf_cdh}
\[
\bigoplus_{i=0}^d\bb Z_{(p)}(j-i)^{\cdh}(X)[-2i]\To \bb Z_{(p)}(j)^{\cdh}(P),
\]
are equivalences for all $j \le N$.
\end{corollary}

Turning to low weights motivic cohomology we obtain calculations of cdh cohomology of the constant sheaf $\bb Z$ and of $\bb G_m$.

\begin{corollary} \label{corol_p_low_weights} Let $X$ be a qcqs $X$-scheme and $\pi: P \to X$ be a projective bundle. Then we have natural equivalences
\[
R\Gamma_{\cdh}(X, \bb Z) \To R\Gamma_{\cdh}(P, \bb Z)
\]
and
\[
R\Gamma_{\cdh}(X, \bb Z)[-1] \oplus R\Gamma_{\cdh}(X, \Gm) \To R\Gamma_{\cdh}(P, \Gm). 
\]
\end{corollary}

\begin{proof} After Theorem~\ref{thm:pbf-fields}, the result follows from the identification of weight zero and weight one motivic cohomology in Theorem \ref{prop:basic_props_of_cdh_mot}(5)\&(6) and the fact that the key hypothesis holds at weights $0$ and $1$ for all primes, by Proposition \ref{prop:wt1-unconditional}.
\end{proof}


\begin{remark}[Role in proof of the main theorem]\label{rem:pbf-val}
In Theorem~\ref{thm:pbf-a1} we will prove the $\P^1$-bundle formula unconditionally after a further $\A^1$-localisation, which we will then show implies the main results of this paper. We note that the proof of that theorem relies on two crucial applications of the unconditional $\P^1$-bundle formula of Corollary~\ref{corol:unconditional-pn}: firstly in Lemma~\ref{eq:rigidity-fib}(2) in the case of equicharacteristic schemes, and secondly in Lemma~\ref{lemm:HFpcdhinf-pbf} in the case of schemes where the relevant prime $p$ is invertible. In particular, we only use Theorem \ref{thm:pbf-fields} in cases where the key hypothesis hold vacuously. 
\end{remark}

\section{The main theorems}\label{sec:a1-comparison}
We now arrive at the main results of this paper. The following comparison theorem, whose proof takes up the bulk of this section, unlocks all our main results.

\begin{theorem}\label{thm:comparison}
For any qcqs scheme $X$ and $j\in\bb Z$, the canonical map
\[
L_{\A^1}\Z(j)^{\cdh}(X) \To \Z(j)^{\bb A, \cdh}(X)
\]
is an equivalence.
\end{theorem}

The proof of the theorem will be presented in \S\ref{ss:main_proof}; first, in \S\ref{sec:slice-conj}--\ref{ss:applications2}, we record some applications.

\subsection{Voevodsky's slice conjectures}\label{sec:slice-conj}
We first address some conjectures stated in \cite{Voevodsky2002a}. Voevodsky postulated the existence of a motivic spectrum, defined over any base scheme $X$, which one denotes by ``$H\bb Z_X$'', representing an $\A^1$-invariant theory of motivic cohomology and subject to certain expected properties. We propose the motivic spectrum representing $\A^1$-motivic cohomology as a candidate. In other words, for any qcqs scheme $X$ we set $H\bb Z_X^\bb A := s^0(\KGL_{X}) \in \SH(X)$. From this viewpoint, \cite[Conjectures 1 \& 3]{Voevodsky2002a} are tautological and our goal is to establish \cite[Conjectures 10 \&  17]{Voevodsky2002a}, relating $H\bb Z_X^\bb A$ to the motivic sphere spectrum and its base change properties, which we now do.

Firstly, we see that $\bb A^1$-cdh-motivic cohomology has no negative weights:

\begin{corollary}\label{prop:Fil1-vanishing} For any qcqs scheme $f: X \rightarrow \Spec(\Z)$, we have that $\Fil^1_\sub{slice}f^*s^0(\KGL_{\Z}) = 0$.  Equivalently, $\Z^{\A,\cdh}(j)(X) = 0$ for $j<0$ and all qcqs schemes $X$.
\end{corollary}
\begin{proof}
Given a motivic spectrum $E\in\SH(X)$, the universal property of $\Fil^1_\sub{slice}E$ shows that it vanishes if and only if $\map_{\sub{SH}(X)}(M_X(Y)\otimes\bb T_X,E)\simeq 0$ for all $Y\in\Sm_X$, or in other words if and only if $\map_{\sub{SH}(X)}(M_X(Y)\otimes\bb T_X^{\otimes j},E)\simeq 0$ for all $Y\in\Sm_X$ and $j<0$. Taking $E=f^*s^0(\KGL_\bb Z)$ and recalling the definition of $\Z(j)^{\A,\cdh}$, we see that $\Fil^1_\sub{slice}f^*s^0(\KGL_\bb Z)$ vanishes if and only if $\Z(j)^{\A,\cdh}(Y)\simeq 0$ for all $Y\in\Sm_X$ and $j\ge1$. Varying $X$, we have established the ``equivalently'' claim in the statement of the result.

We already noted after Definition \ref{def:cdh} that cdh-motivic cohomology vanishes in negative weights, so the desired vanishing of $\bb A^1$-cdh-motivic cohomology now follows from Theorem \ref{thm:comparison}.
\end{proof}

We can now show that the slice filtration on $\KGL$ is stable under base change, and hence compare $\Z(\star)^\A$ and $\Z(\star)^{\A,\cdh}$.
\begin{theorem}\label{thm:a1-a1cdh} Let $X$ be a qcqs scheme and $f: X \rightarrow \Spec(\Z)$ the structure map. Then the following hold:
\begin{enumerate}
\item the canonical maps
\[
f^*\Fil^j_\sub{slice}\KGL_{\mathbb{Z}} \To \Fil^j_\sub{slice}\KGL_X, \qquad j \in \mathbb{Z},
\]
are equivalences.
\item The canonical maps
\[
f^*s^j\KGL_{\mathbb{Z}} \To s^j\KGL_X, \qquad j \in \mathbb{Z},
\]
are equivalences. In particular, the map of $\bb E_\infty$-algebras in presheaves of complexes on $\mathbb{E}_{\infty}$-algebras on $\Sch_\bb Z^\sub{qcqs}$ 
\[
\Z(\star)^{\bb A, \cdh} \To \Z(\star)^{\bb A}
\]
is an equivalence. 
\end{enumerate}
\end{theorem}

\begin{proof}
Let us write
\[
\Fil^{<j}_\sub{slice}:= \mathrm{cofib}(\Fil^{j}_\sub{slice} \rightarrow \id) 
\]
for the ``quotients'' of the slice filtration. To prove (1), we may assume $j = 0$ by Bott periodicity as in \S\ref{subsub:slices-kgl}, i.e., the fact that $\Fil^j_\sub{slice} \KGL \wequi \bb T^{\otimes j} \otimes \Fil^0_\sub{slice} \KGL$. Applying $\Fil^0_\sub{slice}f^*$ to the cofibre sequence $\Fil^0_\sub{slice}\KGL_{\mathbb{Z}} \rightarrow \KGL_{\bb Z} \rightarrow \Fil^{<0}_\sub{slice}\KGL_{\mathbb{Z}}$ yields a cofibre sequence
\[
\Fil^0_\sub{slice}f^*\Fil^0_\sub{slice}\KGL_{\mathbb{Z}} \rightarrow \Fil^0_\sub{slice}\KGL_X \rightarrow \Fil^0_\sub{slice}f^*\Fil^{<0}_\sub{slice}\KGL_{\mathbb{Z}}.
\]
Since $f^*$ preserves effective objects, we have that $\Fil^0_\sub{slice}f^*\Fil^0_\sub{slice}\KGL_{\mathbb{Z}} \simeq f^*\Fil^0_\sub{slice}\KGL_{\mathbb{Z}}$, and therefore it suffices to prove that $\Fil^0_\sub{slice}f^*\Fil^{<0}_\sub{slice}\KGL_{\mathbb{Z}} \simeq 0$. 

To prove this, it is enough by exhaustiveness of the slice filtration (Remark \ref{rem:exhaustive}) to show that the maps
\[
\Fil_\sub{slice}^0f^*\Fil_\sub{slice}^{<0}\KGL_{\mathbb{Z}} \rightarrow \Fil_\sub{slice}^0f^*\Fil_\sub{slice}^{<-1}\KGL_{\mathbb{Z}}  \rightarrow \cdots \rightarrow \Fil_\sub{slice}^0f^*\Fil_\sub{slice}^{<-j}\KGL_{\mathbb{Z}} \rightarrow \cdots.
\]
are all equivalences.  For any $j\ge0$, this follows by applying $\Fil_\sub{slice}^0f^*$ to the cofibre sequence
\[
\mathbb{T}^{\otimes -j-1} \otimes s^0\KGL_{\mathbb{Z}} \rightarrow  \Fil_\sub{slice}^{<-j}\KGL_{\mathbb{Z}} \rightarrow \Fil_\sub{slice}^{<-j-1}\KGL_{\mathbb{Z}}.
\]
and observing that
\begin{eqnarray*}
\Fil_\sub{slice}^0f^*(\mathbb{T}^{\otimes -j-1} \otimes s^0\KGL_{\mathbb{Z}}) & \simeq & \Fil_\sub{slice}^0(\mathbb{T}^{\otimes -j-1} \otimes f^*s^0\KGL_{\mathbb{Z}}) \\
 & \simeq & \mathbb{T}^{\otimes -j-1}\otimes\Fil_\sub{slice}^{j+1}f^*s^0\KGL_{\mathbb{Z}}\\
 & \simeq & 0,
\end{eqnarray*}
where the second equivalence is~\eqref{eq:tate-slice} and the third equivalence is Corollary~\ref{prop:Fil1-vanishing}.

Part (2) is immediate since $s^j$ is the cofibre of the map $\Fil^{j+1}_\sub{slice} \rightarrow \Fil^j_\sub{slice}$. 
\end{proof}

We now present several corollaries of Theorem \ref{thm:a1-a1cdh}. First we explicitly record the conditional agreement of the cdh- and $\bb A^1$-motivic cohomologies:

\begin{corollary}\label{corol_cdh_is_A}
Let $S$ be a qcqs scheme of finite valuative dimension, $p$ a prime number, and $N\ge0$; assume that $\Hyp(S,p,j)$ holds for all $j\le N$. Then, for any qcqs $S$-scheme $X$ and integer $j\le N$, the canonical map \[\bb Z_{(p)}(j)^\sub{cdh}(X)\To \bb Z_{(p)}(j)^{\bb A}(X)\]
is an equivalence.
\end{corollary}
\begin{proof}
By Theorem \ref{thm:A1inv}, for any $j\le N$ the cdh-motivic cohomology $\bb Z_{(p)}(j)^\sub{cdh}$ is $\bb A^1$-invariant on qcqs $S$-schemes. So the desired equivalence follows by combining Theorems \ref{thm:comparison} and \ref{thm:a1-a1cdh}(2).
\end{proof}

The next consequence treats the convergence of the slice filtration:

\begin{corollary}\label{cor:bdd-a1}
For any qcqs scheme $X$, the canonical comparison maps of filtered spectra 
\[
L_{\bb A^1}\mathrm{Fil}_\sub{cdh}^{\star}\KH(X) \To \mathrm{Fil}_{\bb A, \cdh}^{\star}\KH(X)\To\mathrm{Fil}_{\bb A}^{\star}\KH(X)
\]
are equivalences. If $X$ furthermore has finite valuative dimension $\le d$, then $\Fil_\bb A^j\KH(X)$ is $j-d$-connective for all $j\in\bb Z$, and so the filtration is complete.
\end{corollary}
\begin{proof}
The second equivalence follows by applying $\omega^\infty$ to Theorem \ref{thm:a1-a1cdh}. The first equivalence holds by induction because the filtrations are $\bb N$-indexed, agree on $\Fil^0$ (they are both $\KH(X)$), and agree on graded pieces by Theorem \ref{thm:comparison}.

We now prove the connectivity bound of the filtration $L_{\bb A^1}\mathrm{Fil}_\sub{cdh}^{\star}\KH$; let $X$ be a qcqs scheme of valuative dimension $\le d$. Then, for any $m\ge0$, the term $X\otimes_\bb Z \Delta_\bb Z^m=X\otimes_\bb Z\bb Z[t_0,\dots,t_m]/(1-\sum_{i=0}^mt_i)$ appearing in the Suslin construction has valuative dimension $\le d+m$, and so $\mathrm{Fil}_\sub{cdh}^{j}\KH(X\otimes_\bb Z \Delta_\bb Z^m)$ is $j-d-m$-connective by Theorem \ref{prop:basic_props_of_cdh_mot}(3). It now follows formally\footnote{Let $K_{\bullet}: \Delta^{\op} \rightarrow \Spt$ be a simplicial spectrum and assume, for some fixed $j\ge0$, that $K_{m}$ is $j-m$-connective for all $m \ge 0$. Then the geometric realisation $|K_{\bullet}|$ is $j$-connective. Indeed, writing $|K_{\bullet}| \simeq \colim_{m \rightarrow \infty} (\colim_{\Delta^{\op}_{\leq m}} K|_{\Delta^{\op}_{\leq m}})$ by \cite[Remark 4.3.4]{LurieDAGVIII} it is enough, since filtered colimits preserve connectivity, to show that each term $\colim_{\Delta^{\op}_{\leq m}} K|_{\Delta^{\op}_{\leq m}}$ is $j$-connective. We do this by induction on $m$. In the base case of $m = 0$, it is the assumption that $K_0$ is $m$-connective. Otherwise, we have a cofibre sequence [loc.~cit.]
\[
\colim_{\Delta^{\op}_{\leq (m-1)}} K|_{\Delta^{\op}_{\leq (m-1)}} \rightarrow \colim_{\Delta^{\op}_{\leq m}} K|_{\Delta^{\op}_{\leq n}} \rightarrow K'_m[m].
\]
where the first term is $j$-connective by induction and the last term is a summand of $K_m[m]$ and hence $j$-connective by hypothesis; so the middle term is $j$-connective, as desired.
} that the geometric realisation \[L_{\bb A^1}\Fil_\sub{cdh}^j\KH(X)=\opp{colim}_{m\in\Delta^\op} \Fil_\sub{cdh}^j\KH(X\times\Delta^m)\] is $j-d$-connective.
\end{proof}

The next corollary concerns the effective cover $\kgl$ of $\KGL$ and its framed model $\cal V$ (see \S\ref{subsub:V} for recollections).

\begin{corollary}\label{cor:eff-base-change}Let $X$ be a qcqs scheme and $f: X \rightarrow \Spec(\Z)$ the structure map. Then:
\begin{enumerate}
\item The canonical map
\[
f^* \kgl_{\bb Z} \rightarrow \kgl_X
\]
is an equivalence.
\item The canonical map of \eqref{eq:fr-to-kgl}
\[
\scr V_X \rightarrow \kgl_X 
\]
is an equivalence. 
\item The canonical map $\scr V_X \rightarrow \kgl_X$ induces an equivalence
\[
\scr V_X/\beta \xrightarrow{\simeq} s^0(\kgl_X).
\]
\end{enumerate}
\end{corollary}

\begin{proof}
Point (1) is the special case $j=0$ of Theorem~\ref{thm:a1-a1cdh}(1). Point (2) has already been established over $\bb Z$ in Proposition \ref{prop:specz_slices}, and so follows for $X$ by pulling back and using point (1). Point (3) then follows by modding out by $\beta$ and recalling from  \S\ref{subsub:slices-kgl} that $s^0(\kgl_X)\quis\kgl_X/\beta$.
\end{proof}

Next we show that the $0$-slice of the motivic sphere is stable under pullback and identifies with the $0$-slice of $\KGL$:

\begin{corollary}\label{cor:conjectures} Let $X$ be a qcqs scheme and $f: X \rightarrow \Spec(\Z)$ the structure map. Then the unit map $1_{X} \rightarrow \KGL_X$ induces an commutative diagram of equivalences of $\bb E_\infty$-algebras:
\begin{equation}\label{eq:1-to-k}
\begin{tikzcd}
f^*s^0(1_\bb Z) \ar[swap]{d}{\simeq} \ar{r}{\simeq} & f^*s^0(\KGL_{\Z}) \ar{d}{\simeq}\\
s^0(1_X) \ar[swap]{r}{\simeq} & s^0(\KGL_X).
\end{tikzcd}
\end{equation}
\end{corollary}
\begin{proof}
Over $\Z$, we already have that $s^0(1_\bb Z) \simeq s^0(\KGL_{\Z})$ by Proposition~\ref{prop:specz_slices}, hence the top horizontal arrow of~\eqref{eq:1-to-k} is an equivalence. The right vertical map is an equivalence by Theorem~\ref{thm:a1-a1cdh}. It then suffices to prove that the bottom horizontal arrow is an equivalence. To do this, it suffices to prove that $F_X:= \mathrm{fib}(1_{X} \rightarrow \kgl_X)$ is $1$-effective. But $1$ and $\kgl$ are stable under base change (the latter by Corollary~\ref{cor:eff-base-change}), so the same is true of $F$ and therefore it suffices to check that $F_{\mathbb{Z}}$ is $1$-effective; that is exactly Proposition~\ref{prop:specz_slices}. 
\end{proof}

Finally we discuss applications to motives. Using the notation introduced at the beginning of the subsection, Corollary \ref{cor:conjectures} states that we have isomorphic absolute motivic ring spectra $X\mapsto s^0(1_X)$, $H\bb Z_X^\bb A$. We propose that the $\infty$-category of $H\bb Z_X^\bb A$-modules in $\SH(X)$, i.e., the presentably symmetric monoidal stable $\infty$-category \begin{equation}\mathrm{DM}(X):=\opp{Mod}_{H\bb Z_X^\bb A}(\SH(X)),\label{eqn:DM}\end{equation} is a good candidate for the derived $\infty$-category of $\bb A^1$-invariant motives over $X$. Indeed, this is known to agree with Voevodsky's usual $\infty$-category $\DM(k)$ whenever $X$ is the spectrum of a characteristic zero field $k$ \cite{rondigs-ostvaer}, while the fact that $H\bb Z^\bb A$ is an absolute motivic ring spectrum implies the following:

\begin{corollary}\label{corol:6functor}
The assignment on qcqs schemes $X\mapsto \opp{Mod}_{H\bb Z_X^\bb A}(\SH(X))$ upgrades to a six-functor formalism compatible with that of $\SH(-)$.
\end{corollary}
\begin{proof}
The traditional reference that this follows from $H\bb Z^\bb A$ being an absolute motivic spectrum is \cite[Proposition 11.4.7]{CisinskiDeglise2019}, but we provide here a short sketch of the result in the current modern language of six-functor formalisms. We thank Lucas Mann for discussions about this.

In principle the motivic stable homotopy category $\SH$, equipped with its various functors, assembles to form a $\opp{Pr}^L$-valued six-functor formalism on qcqs schemes in the sense of Heyer--Mann \cite{HeyerMann2024}, with the $!$-able morphisms being those of finite type. This should follow from the results of Ayoub \cite{Ayoub2007, Ayoub2007II} and the machinary of \cite[\S3.3]{HeyerMann2024}. Base changing this along the functor $-\otimes H\bb Z_{\Spec(\bb Z)}:\SH(\bb Z)\to \mathrm{DM}(\bb Z)$ defines a new six functor formalism \[X\mapsto \SH(X)\otimes_{\SH(\bb Z)}\mathrm{DM}(\bb Z).\] But, for any qcqs scheme $f:X\to\Spec(\bb Z)$, the fact that $f^*(H\bb Z_{\Spec(\bb Z)}^\bb A)\quis H\bb Z_X^\bb A$, i.e., Corollary \ref{cor:conjectures}, implies that the canonical functor \[\SH(X)\otimes_{\SH(\bb Z)}\mathrm{DM}(\bb Z)\To\mathrm{DM}(X)\] in $\opp{Pr}^L$ is an equivalence \cite[Theorem 4.8.4.6]{LurieHA}. This defines the desired six-functor formalism on $\mathrm{DM}$.
\end{proof}

\begin{remark}
Definition \eqref{eqn:DM} is more accessible than it might initially appear. Since $H\bb Z_X^\bb A=H(\bb Z(\star)^\bb A[2\star])$, by applying Remark \ref{rem:modules_over_HF}, replacing $\Spt$ by $\D(\bb Z)$, and shearing, we obtain an equivalence \[\mathrm{DM}(X)\simeq \opp{Mod}_{\bb Z(\star)^\bb A}(\Gr\opp{Sh}_{\Nis,\bb A^1}(\Sm_X,\D(\bb Z))_\sub{c,pbf}),\] i.e., 
it is the $\infty$-category of $\bb Z(\star)^\bb A$-modules $M^\star$ in $\Gr\opp{Sh}_{\Nis,\bb A^1}(\Sm_X,\D(\bb Z))$ satisfying the $\bb P^1$-bundle formula (that is, multiplication by $c_1^\bb A(\roi(1))\in H^2_\bb A(\bb P_X^1,\bb Z(1))$ induces $M^j[2]\quis \fib(M^{j+1}(\bb P^1_-)\xto{\infty^*}M^{j+1})$ for all $j\in\bb Z$). In particular, if the $\bb A^1$-motivic cohomology is accepted as a black box, then $\mathrm{DM}(X)$ can be defined without mentioning $\SH(X)$.
\end{remark}

\subsection{Low weight applications}\label{ss:applications2}
The first part of the next application shows in particular that weight-$0$ motivic cohomology identifies with the cdh cohomology of the constant sheaf $\bb Z$, as was conjectured by Voevodsky \cite[Conjecture 12]{Voevodsky2002}. It follows that the $\bb Z$-algebra structure on weight-$0$ $\bb A^1$-motivic cohomology, which we defined using unstable slices in Construction \ref{cons:low-weights}, is unique. We suspect that part (2) should hold for arbitrary regular Noetherian schemes, but cannot prove it.

\begin{corollary}\label{corol:low-wts-integral}
\begin{enumerate}
\item For any qcqs scheme $X$, the maps \[R\Gamma_\sub{cdh}(X, \bb Z) \xto{\sub{(\ref{eq:wt0-cdh})}} \bb Z(0)^{\sub{cdh},\bb A}(X)\qquad \mathrm{and}\qquad c_1^{\sub{cdh},\bb A}:R\Gamma_\sub{cdh}(X,\bb G_m)[-1]\xto{(\ref{eq:c1-det})} \bb Z(1)^{\sub{cdh},\bb A}(X)\] are equivalences.
\item Let $B$ be a mixed characteristic Dedekind domain or a field. Then, for any smooth $B$-scheme $X$, the canonical change-of-topology maps 
\[R\Gamma_\sub{Nis}(X, \bb Z) \To R\Gamma_\sub{cdh}(X, \bb Z)\qquad\mathrm{and}\qquad R\Gamma_\sub{Nis}(X,\bb G_m)\To R\Gamma_\sub{cdh}(X,\bb G_m)\]
are equivalences.
\end{enumerate}
\end{corollary}
\begin{proof}
(1): When $j\le 1$, the key hypothesis holds by Proposition~\ref{prop:wt1-unconditional}, so $\bb Z(j)^\sub{cdh}$ is $\bb A^1$-invariant by Theorem~\ref{thm:A1inv}, and therefore Theorem~\ref{thm:comparison} shows that the canonical map $\bb Z(j)^{\cdh} \to \bb Z(j)^{\sub{cdh},\bb A}$ is an equivalence. Now apply Theorem \ref{prop:basic_props_of_cdh_mot}(5)\&(6) to obtain the desired equivalences.

(2) Let $X$ be a smooth $B$-scheme. When $j=0$ we consider the composition \[R\Gamma_\sub{Nis}(X,\bb Z)\To R\Gamma_\sub{cdh}(X,\bb Z) \To \bb Z(0)^{\sub{cdh},\bb A}(X)\To \bb Z(0)^{\bb A}(X).\] The middle arrow is an equivalence by part (1), the final arrow is an equivalence by Theorem \ref{thm:a1-a1cdh} (though we remark that, in this case of smooth $B$-scheme, the equivalence $\bb Z(j)^{\sub{cdh},\bb A}(X)\quis \bb Z(j)^{\bb A}$ could already have been proved in \S\ref{ss:slice-dedekind}), and the composition is an equivalence by Corollary~\ref{cor:low-wts-ddk-fields}. Therefore the first arrow is also an equivalence. The case of $j=1$ is treated by the same reasoning.
\end{proof}

\subsection{Proof of the main theorem}\label{ss:main_proof}
The rest of this section is devoted to proving Theorem \ref{thm:comparison}, namely that the canonical comparison map $L_{\A^1} \Z(\star)^\cdh\to \Z(\star)^{\A,\cdh}$ is an equivalence. Since the proof is somewhat involved, we begin with an overview. We first show in \S\ref{subsub:reduce-to-P1} that it is enough to prove that $L_{\A^1} \Z(\star)^\cdh$ satisfies the $\P^1$-bundle formula. It is clearly enough to do this rationally and modulo $p$ for every prime $p$, and the rational case is straightforward (see \S\ref{subsub:P1-rational}). The majority of our work thus takes place with mod-$p$ coefficients.

Write $j: \Sch_{\Z[1/p]} \to \Sch$ and $i: \Sch_{\F_p} \to \Sch$ for the complementary immersions.
Given a cdh sheaf $E$ on $\Sch$, we can form the null composite \[ L_\cdh j_! j^* F\to F \to i_*i^* F. \]
(see \S\ref{subsub:localisation} for details on the functors $j_!, i_*$, etc.) and in some special cases this is is a fibre sequence, i.e., the induced map \[ L_\cdh j_!j^* F \to \fib(E \to i_*i^* F) \] is an equivalence; in this case we say that \emph{$F$ has a localisation sequence}.
Now, assume further that $F$ comes equipped with a first Chern class as formulated by the category introduced in~\eqref{eq:gr-big}.  In such a case, in order to prove that $F$ satisfies the $\P^1$-bundle formula, it is enough to prove it for $i_*i^* F$ and $L_\cdh j_*j^* F$. In our examples of interest, the former case will be easy because we will already know the $\P^1$-bundle formula in equal characteristic $p$, thanks to the work in \S\ref{sec:pbf-fields}. Establishing it for $L_\cdh j_!j^*E$ will be the main new work.

One way to establish localisation of a cdh sheaf $F$ is to exploit rigidity: if, for every henselian valuation ring $V$ with residue field $k$ of characteristic $p$ we have $F(V) \quis F(k)$, then $F$ has a localisation sequence (at least in situations where equivalences can be checked stalkwise, for example for bounded above sheaves or over schemes of finite valuative dimension).
One prominent cdh sheaf with this property is cofiber of the comparison map $\K \to \K^\sub{Sel}$ which we call Kato $K$-theory and denote by $\K^\sub{Kato}$ (see Definition~\ref{eq:gr-big}). Up to a shift and on $p$-henselian rings, it agrees with mod-$p$ infinitesimal $K$-theory, the fibre of the cyclotomic trace map $\K \to \TC$. We will  also study its motivic variant $\F_p(\star)^\sub{Kato}$, called Kato motivic cohomology, in \S\ref{subsub:infinitesimal-mot}. Our main result will be deduced from the $\bb P^1$-bundle formula for the $\bb A^1$-localisation of Kato motivic cohomology.

It will be apparent by construction that Kato $K$-theory has a $\P^1$-bundle formula.
Using the localisation sequence and a trick involving $v_1$-self maps, we will deduce from this that $L_\cdh j_!j^* \KH$ has a $\P^1$-bundle formula (Lemma \ref{lem:inf-k-fibre}(3)). Employing a degeneration-of-spectral-sequence argument, we then show that $L_\cdh j_!j^* \F_p(\star)^\cdh$ satisfies a $\P^1$-bundle formula and, by the localisation sequence, we deduce the same for $\F_p(\star)^\sub{Kato}$ and hence $L_{\A^1} \F_p(\star)^\sub{Kato}$. From this we conclude: see \S\ref{subsub:conclusion}.

We now begin the proof of Theorem \ref{thm:comparison}.

\subsubsection{Reduction to the $\P^1$-bundle formula} \label{subsub:reduce-to-P1}
Theorem~\ref{thm:comparison} reduces to the following $\P^1$-bundle formula for the $\A^1$-localisation of $\Z(j)^{\cdh}$:

\begin{theorem}\label{thm:pbf-a1}
For any qcqs scheme $X$, and $j\in \bb Z$, the map 
\begin{equation}\label{eq:pbf-a1}
L_{\A^1}\Z(j)(X)^{\cdh} \oplus L_{\A^1}\Z(j-1)(X)^{\cdh}[-2] \xrightarrow{\pi^* \oplus c_1^\sub{cdh}(\scr O(1))\pi^*} L_{\A^1}\Z(j)^{\cdh}(\P^1_{X})
\end{equation}
is an equivalence. 
\end{theorem}

We can use this to prove the main theorem:

\begin{proof}[Proof that Theorem~\ref{thm:pbf-a1} implies Theorem~\ref{thm:comparison}]
We apply Corollary \ref{corol_pullback_of_EM} with $X=\Spec(\Z)$ and $\bb Z(\star)^\bb A[2\star]\in \Gr\opp{Sh}_{\cdh, \bb A^1}(\Sch_\bb Z^\sub{qcqs},\Spt)_\sub{c,pbf}$: Theorem~\ref{thm:pbf-a1} implies that the $L_\sub{pbf}$ is redundant, while we saw in Remark \ref{rem:motivic_is_traditional} that $H(\bb Z(\star)^\bb A[2\star])=s^0(\KGL_\bb Z)$, so the corollary exactly yields the equivalence claimed in Theorem~\ref{thm:comparison}.
\end{proof}

The proof of the $\P^1$-bundle formula occupies the rest of this section.
Our strategy will be to first establish the $\P^1$-bundle formula for a plethora of other cohomology theories and then bootstrap our way to Theorem~\ref{thm:pbf-a1}.

\subsubsection{Rational $\bb P^1$-bundle formula} \label{subsub:P1-rational}
We first note that Theorem \ref{thm:pbf-a1} holds rationally. Note that $L_{\bb A^1} \bb Q(j)^{\cdh} \simeq L_{\bb A^1}\bb Z(j)^{\cdh} \otimes \bb Q$ since $L_{ \bb A^1}$ preserves filtered colimits.

\begin{lemma}\label{lem:rational-p1} Let $X$ be a qcqs scheme. Then for any $j \in \bb Z$, the natural maps
\begin{enumerate}
\item $\bb Q(j)^{\cdh}(X) \rightarrow L_{\bb A^1} \bb Q(j)^{\cdh}(X)$, and
\item $L_{\A^1}\bb Q(j)(X)^{\cdh} \oplus L_{\A^1}\bb Q(j-1)(X)^{\cdh}[-2] \xrightarrow{\pi^* \oplus c_1^\sub{cdh}(\scr O(1))\pi^*} L_{\A^1}\bb Q(j)^{\cdh}(\P^1_{X})$
\end{enumerate}
are equivalences.
\end{lemma}
\begin{proof}
(1): In other words, we are claiming that $\Q(j)^{\cdh}$ is an $\A^1$-invariant presheaf; this follows from the rational decomposition of Theorem \ref{prop:basic_props_of_cdh_mot}(7) since $\KH$ is $\bb A^1$-invariant.

(2): In light of (1) this follows from the $\bb P^1$-bundle formula for rational cdh-motivic cohomology, which was proved in Proposition~\ref{prop:q-pbf}.
\end{proof}

To complete the proof of Theorem \ref{thm:pbf-a1} it now suffices to prove it modulo every prime number. For the rest of this section, we fix a prime number $p$.

\subsubsection{Localisation sequences and cohomology theories} \label{subsub:localisation}
We consider the open and closed immersions into $\Spec(\Z)$ given by
\[
\Spec(\Z[\tfrac{1}{p}] )\stackrel{j}\To \Spec(\Z) \stackrel{i}\longleftarrow \Spec(\F_p);
\]
setting $\cal D:=\Spt$ or $\D(\bb Z)$, these immersions induce adjunctions on $\infty$-categories of presheaves 
\[
j_!:\PShv(\Sch^{\qcqs}_{\Z[\tfrac{1}{p}]}, \cal D) \rightleftarrows \PShv(\Sch^{\qcqs}_{\Z}, \cal D):j^*.
\]
\[
i^*:\PShv(\Sch^{\qcqs}_{\Z},\cal D) \rightleftarrows \PShv(\Sch^{\qcqs}_{\F_p}, \cal D):i_*
\]
For any presheaf $F$ on $\Sch_\bb Z^\sub{qcqs}$ the composition \[j_!j^*F\To F\To i_*i^*F\] is naturally null homotopic, by factoring it as $j_!j^*F\to i_*i^*j_!j^*F\simeq 0\to i_*i^*F$ (where we use that $i^*j_!\simeq 0$, as follows for example from the explicit formula of the next remark).

\begin{remark}[Explicit formulae]\label{rem:values-of-functors} 
In other words, $j^*$, $j_!$, and $i^*$ are the unique colimit preserving functors fitting into commutative diagrams
\[
\begin{tikzcd}
\Sch^{\qcqs}_{\Z[\tfrac{1}{p}]} \ar{r}{X  \mapsto X } \ar{d}  & \Sch^{\qcqs}_{ \Z} \ar{d} \\
\PShv(\Sch^{\qcqs}_{\Z[\tfrac{1}{p}]},\Spt) \ar{r}{j_!} & \PShv(\Sch^{\qcqs}_{ \Z},\Spt).
\end{tikzcd}
\qquad
\begin{tikzcd}
\Sch^{\qcqs}_{\Z} \ar{r}{X \mapsto X[\tfrac{1}{p}]} \ar{d}  & \Sch^{\qcqs}_{ \Z[\tfrac{1}{p}]} \ar{d} \\
\PShv(\Sch^{\qcqs}_{\Z},\cal D) \ar{r}{j^*} & \PShv(\Sch^{\qcqs}_{ \Z[\tfrac{1}{p}]},\cal D)
\end{tikzcd}
\]
\[
\begin{tikzcd}
\Sch^{\qcqs}_{\Z} \ar{r}{X \mapsto X_{\bb F_p}} \ar{d}  & \Sch^{\qcqs}_{\bb F_p} \ar{d} \\
\PShv(\Sch^{\qcqs}_{\Z},\cal D) \ar{r}{i^*} & \PShv(\Sch^{\qcqs}_{ \bb F_p},\cal D),
\end{tikzcd}
\]
and $i_*$ is the right adjoint to $i^*$.

Moreover, observe that if $Y \in \Sch^{\qcqs}_{\Z[1/p]}$ (respectively $ \in \Sch^{\qcqs}_{\F_p}$) and $X \in \Sch^{\qcqs}_{\Z}$ and we are given a morphism $X \to Y$, then automatically $X \in \Sch^{\qcqs}_{\Z[1/p]}$ (respectively $ \in \Sch^{\qcqs}_{\F_p}$).
Using the pointwise formula for left Kan extensions, one can therefore explicitly determine the values of these functors, which we record below for the readers' convenience:
\begin{enumerate}
\item $j_!F(X) = \begin{cases} F(X) & X \in \Sch^{\qcqs}_{\Z[1/p]}, \\ 0 & \text{else}; \end{cases}$
\item $j^*F(X) = F(X)$ for all $X\in\Sch_{\bb Z[1/p]}^\sub{qcqs}$;
\item $i^*F(X)=F(X)$ for all $X\in\Sch_{\bb F_p}^\sub{qcqs}$;
\item $i_*F(X) = F(X_{\F_p})$ for all $X\in\Sch^\sub{qcqs}$.
\end{enumerate}
In particular, using explicit formula 2--4, we see that $j^*$, $i^*$, and $i_*$ preserve sheaves for any given Grothendieck topology; in fact, $j^*$ and $i^*$ even commute with sheafification. In this remainder of this section we will often implicitly use these facts in the case of the cdh topology.
\end{remark}

We need a ``big'', $\D(\bb Z)$-linear variant of the $\infty$-category $\Gr\opp{Sh}_{\Nis, \bb A^1}(\Sm_X,\Spt)_\sub{c}=\opp{Mod}_{\cal Q_X}(\Gr\opp{Sh}_{\Nis, \bb A^1}(\Sm_X,\Spt))$ of Construction \ref{cons:lin-cohom-theories}, in which we will not impose sheafiness or $\bb A^1$-invariance. 

Let $B$ be a base ring and $\cal Q_{B,\sub{big}}\in \CAlg(\Gr \PShv(\Sch^\qcqs_B, \D(\bb Z)))$ the free $\bb E_\infty$-algebra on $\bb Z[(\bb P^1_B,\infty)]:=(\Sigma^\infty\bb P^1_B)\otimes_{\bb S}\bb Z$ (in other words, $\opp{coker}(\bb Z\xto{\infty} \bb Z[\Hom_{\Sch^\qcqs_B}(-,\bb P^1_B)])$) placed in degree $1$, and \begin{equation}\label{eq:gr-big}
\Gr\opp{PSh}(\Sch^\sub{qcqs}_B,\D(\bb Z))_\sub{c}:=\opp{Mod}_{\cal Q_{B,\sub{big}}}(\Gr\opp{PSh}(\Sch_B^\sub{qcqs},\D(\bb Z)))
\end{equation} the $\infty$-category of modules over $\cal Q_{B,\sub{big}}$. In other words, as in Construction \ref{cons:lin-cohom-theories} and the discussion before Remark \ref{remark:traditional}, objects $M$ of $\Gr\PShv(\Sch_B^\qcqs,\D(\bb Z))_{c}$ consist of graded presheaves of complexes $M^\star$ together with coherent ``multiplication by the first Chern class of $\scr O(1)$'' maps, and objects $F$ of $\CAlg(\Gr\PShv(\Sch_B^\qcqs,\D(\bb Z))_{c})$ are the same as pairs of $F^\star \in \CAlg(\Gr\PShv(\Sch_B^\qcqs,\D(\bb Z)))$ and a map $\bb Z[(\bb P^1_B,\infty)]\to F^1$.

Note that, since cdh sheafification is a lax symmetric monoidal endofunctor of presheaves, it induces a cdh sheafification endofunctor $L_\sub{cdh}$ of $\CAlg(\Gr\PShv(\Sch_B^\qcqs,\D(\bb Z))_{c})$, given termwise on the underlying graded presheaf.

\begin{example}
As in Remark  \ref{remark:traditional}, many of our cohomology theories are equipped with a theory of first Chern classes and so, after shearing, upgrade to objects of $\CAlg(\Gr\PShv(\Sch^\qcqs_B,\D(\bb Z))_c)$; the same is true for all relevant maps, since they will respect the first Chern classes.

For example, the $\bb E_\infty$-algebras in graded $\D(\bb Z)$-valued presheaves $\Z(\star)^\cdh[2\star]$, $\bb Z(\star)^{\bb A}[2\star]$, $\bb Z(\star)^{\bb A,\sub{cdh}}[2\star]$, $\bb Z_p(\star)^\sub{BL}[2\star]$, $L_\sub{cdh}\bb Z_p(\star)^\sub{BL}[2\star]$, $\bb Z_p(\star)^\sub{syn}[2\star]$, $L_\sub{cdh}\bb Z_p(\star)^\sub{syn}[2\star]$ and their mod-$p$ variants all upgrade to objects of $\CAlg(\Gr\PShv(\Sch^\qcqs,\D(\bb Z))_c)$, denoted respectively by $\Z^\cdh$, $\bb Z^{\bb A}$, $\bb Z^{\bb A,\sub{cdh}}$, $\bb Z_p^\sub{BL}$, $L_\sub{cdh}\bb Z_p^\sub{BL}$, $\bb Z_p^\sub{syn}$, $L_\sub{cdh}\bb Z_p^\sub{syn}$, and similarly mod $p$. These are related by various maps, for example \begin{equation}\F_p^\cdh\To L_\sub{cdh}\bb F_p^\sub{BL}\To L_\sub{cdh}\bb F_p^\sub{syn},\label{eqn:cdh_to_cdhsyn}\end{equation} where the first map is induced by \eqref{eq:cdh-syn} modulo $p$ and the second map is induced by the canonical map from Beilinson--Lichtenbaum to syntomic cohomology.
\end{example}

The functors $j_!$, $j^*$, $i^*$, and $i_*$ induce adjunctions on $\Gr\PShv(\Sch^\qcqs_{(\ph)},\D(\bb Z))_c$, which we abusively give the same names\footnote{We apologise that in this section $j$ will be used simultaneously for the inclusion of the generic fibre and for the weight of various cohomologies; it should not cause confusion because in the first case it will always occur with a decoration like $_!$ or $^*$, e.g., $j_!j^*\bb F_p(j)$.}
\begin{gather*}
j_!:\Gr\PShv(\Sch^{\qcqs}_{\Z[1/p]}, \D(\bb Z))_c \rightleftarrows \Gr\PShv(\Sch^{\qcqs}_{\Z}, \D(\bb Z))_c:j^*\\
i^*:\Gr\PShv(\Sch^{\qcqs}_{\Z}, \D(\bb Z))_c \rightleftarrows \Gr\PShv(\Sch^{\qcqs}_{\F_p}, \D(\bb Z))_c :i_*
\end{gather*}
and which are computed termwise, i.e., $(j_! M)^\star \wequi j_!(M^\star)$, etc. Indeed, for the functors $i^*$ and $j^*$ this holds because they are symmetric monoidal and preserve $\cal Q_\sub{big}$; for the functor $i_*$ we use that it is lax symmetric monoidal; and finally for $j_!$ we use the projection formula $\cal Q_{\bb Z[1/p],\sub{big}} \otimes j_!(F) \wequi j_!(j^* \cal Q_{\bb Z,\sub{big}} \otimes F)$.


\begin{remark}[Localisation sequences]\label{rem:local_seq}
As just before Remark \ref{rem:values-of-functors}, these functors fit into a null sequence $j_!j^*\to\opp{id}\to i_*i^*$ (with the null homotopy induced by $i^*j_!= 0$), which we will also study after termwise cdh sheafifying.
\end{remark}

Given $M \in \Gr\PShv(\Sch^\qcqs_B,\D(\bb Z))_c$ we obtain, for each $X \in \Sch^\qcqs_B$, the adjoint bonding map \[M^j(X)\To\fib(M^{j+1}(\bb P^1_X)\xto{\infty^*}M^{j+1}(X)),\] and we say that \emph{$E$ satisfies the $\P^1$-bundle formula} if these maps are equivalences for all $X$ and $j\in\bb Z$; as usual let $\Gr\PShv(\Sch^\qcqs_B,\D(\bb Z))_{c,\sub{pbf}}$ denote the full subcategory of such objects. In contrast to Construction \ref{cons:lin-cohom-theories}, we do not need to introduce any functor $L_\sub{pbf}$ forcing the projective bundle formula to hold. Using Remark~\ref{rem:values-of-functors}, we see that the functors $j_!$, $j_*$, $i^*$, $i_*$ all preserve objects satisfying the projective bundle formula.

\begin{example}\label{ex:pbfGr}
For example, the objects $\bb Z^\bb A$, $\bb Z^{\bb A,\cdh}$, and $\bb Z_p^\sub{syn}$ (as well as their mod-$p$ variants) satisfy the projective bundle formula by Theorems \ref{thm_pbf_Acdh} and \ref{theorem_syntomic_properties}(6).

Another example, which we will need in the proof of the main theorem, is $L_{\bb A^1}L_{\cdh}\F_p^{\syn}$, obtained by termwise $\bb A^1$-localising $L_\sub{cdh}\bb F_p^\sub{syn}$. Indeed, after unshearing, the $\bb E_\infty$-algebra in $\Gr\opp{PSh}(\Sch,\D(\bb Z))$ underlying $L_{\bb A^1}L_{\cdh}\F_p^{\syn}$ is \[L_{\bb A^1}L_\sub{cdh}\bb F_p(\star)^\sub{syn}\simeq R\Gamma_{\et}(-[\tfrac{1}{p}],\mu_p^{\otimes \star})\] by Theorem~\ref{thm:a1-comparison} (and the fact that $L_{\bb A^1}L_\sub{cdh}L_{\bb A^1}\simeq L_{\bb A^1}L_\sub{cdh})$, equipped with the usual first Chern class from Kummer theory (because of Lemma \ref{lemma_first_Chern_p_hens}). The desired projective bundle formula therefore reduces to the case of \'etale cohomology.
\end{example}

\subsubsection{Kato motivic cohomology} \label{subsub:infinitesimal-mot}

\begin{definition}\label{def:cdhinf} We define
\[
\F_p^\sub{Kato} := \opp{cofib}(  \F_p^{\cdh} \xto{\sub{\eqref{eqn:cdh_to_cdhsyn}}} L_{\cdh}\F_p^{\syn}) \in \Gr\PShv(\Sch^\qcqs,\D(\bb Z))_\sub{c,pbf},
\]
which we call \emph{mod-$p$ Kato motivic cohomology}.\footnote{It is a version of the complex appearing in the so-called Kato conjecture  \cite{Kato1986a}, measuring the difference between motivic and \'etale motivic cohomology.} In line with our notation for other cohomologies, let $\bb F_p(\star)^\sub{Kato}:=(\bb F_p^\sub{Kato})^\star[-2\star]\in\Gr\opp{Sh}(\Sch^\qcqs,\D(\bb Z))$ be the unshearing of the graded presheaf underlying $\bb F_p^\sub{Kato}$; that is, for any $j \in \Z$, the presheaf $\F_p(j)^{\sub{Kato}}$ fits into a fibre sequence
\[
\F_p(j)^{\cdh}\To L_{\cdh}\F_p(j)^{\syn}\To \F_p(j)^{\sub{Kato}}. 
\]
Since $\F_p(j)^{\sub{Kato}}$ is the cofibre of a map between two finitary cdh sheaves, it is itself a finitary cdh sheaf.
\end{definition}

\begin{remark}[Explicit description of Kato motivic cohomology]
Identifying the left hand term of the previous fibre sequence with $L_\sub{cdh}\tau^{\le j}\bb F_p(j)^\sub{syn}$ using Theorem \ref{theorem_syn_comp}, we obtain equivalences \begin{equation}\F_p(j)^{\sub{Kato}}\simeq L_\sub{cdh}\tau^{>j}\bb F_p(j)^\sub{syn},\qquad j\in\bb Z.\label{eqn:cdhinf}\end{equation} In particular, on $\bb F_p$-schemes we have $\bb F_p(j)^\sub{Kato}\simeq R\Gamma_\sub{cdh}(-,\tilde\nu(j))[-j-1]$ for $j\ge0$ (and $\simeq 0$ for $j<0$), by \eqref{eqn:cdhinf} and Remark \ref{rmk_syntomic_props}(2). Meanwhile, on $\bb Z[\tfrac1p]$-schemes, we have $\bb F_p(j)^\sub{Kato}\simeq L_\sub{cdh}\tau^{>j}R\Gamma_\sub{\'et}(-,\mu_p^{\otimes j})$ for all $j\in\bb Z$.
\end{remark}

Kato motivic cohomology has a cdh-local localisation sequence and satisfies the projective bundle formula. We remark that this latter result requires the unconditional projective bundle formula (but only for $P = \bb P^1_X$) for equicharacteristic schemes proved in Corollary~\ref{corol:unconditional-pn}. 

\begin{lemma}\label{lem:inf-fibre}
\begin{enumerate}
\item The null sequence
\begin{equation}\label{eq:rigidity-fib}
L_{\cdh}j_!j^*\F_p^{\sub{Kato}} \rightarrow \F_p^{\sub{Kato}} \rightarrow i_*i^*\F_p^{\sub{Kato}}
\end{equation}
(obtained by cdh sheafifying Remark \ref{rem:local_seq}) in $\Gr\PShv(\Sch^{\qcqs}_{\bb Z}, \D(\bb Z))_c$ is a fibre sequence.
\item $i_*i^*\F_p^{\sub{Kato}}$ satisfies the $\P^1$-bundle formula.
\end{enumerate}
\end{lemma}
\begin{proof} 
(1): Recalling that the functors $j^*$, $i^*$, and $i_*$ preserve cdh sheaves, when we cdh sheafify the null sequence of Remark \ref{rem:local_seq} applied to $\F_p^{\sub{Kato}}$ we do indeed obtain a null sequence \eqref{eq:rigidity-fib}, with the three terms being finitary cdh sheaves. (Here, and repeatedly below without explicit mention, we use that $L_\sub{cdh}$ preserves finitariness by Proposition~\ref{prop:finitary_conditions}.) To check it is a fibre sequence, it is therefore enough to show for every $j\in\bb Z$ and henselian valuation ring $V$ that the null sequence
\[j_!j^*\F_p(j)^{\sub{Kato}}(V) \To \F_p(j)^{\sub{Kato}}(V) \To i_*i^*\F_p(j)^{\sub{Kato}}(V)\]
is a fibre sequence. When $p\in V^\times$ the sequence reads $\F_p(j)^{\sub{Kato}}(V)\xto{=} \F_p(j)^{\sub{Kato}}(V)\to 0$, which is obviously a fibre sequence. In the remaining case that $p\notin V^\times$ the sequence reads $0\to \F_p(j)^{\sub{Kato}}(V)\to \F_p(j)^{\sub{Kato}}(V/p)$, which is again a fibre sequence since the latter map is an equivalence: indeed, when $j<0$ both terms vanish, and when $j\ge0$ we first replace $V$ by its $p$-completion as in Lemma \ref{lemma_syn_p-hens}(3) then appeal to Remark \ref{rmk_syntomic_props}.
\comment{

We first establish the fibre sequence~\eqref{eq:rigidity-fib}. Note that there is an induced map
\[
L_{\cdh}j_!j^*\F_p^{\sub{Kato}} \rightarrow \mathrm{fib}(\F_p^{\sub{Kato}} \rightarrow i_*i^*\F_p^{\sub{Kato}}),
\]
since $\F_p^{\sub{Kato}} \rightarrow  i_*i^*\F_p^{\sub{Kato}}$ is a morphism cdh sheaves.
Since every term in sight is finitary, we can check that the map is an equivalence after evaluation on henselian valuation rings. If $\tfrac{1}{p} \in V$ then since $L_{\cdh}j_!j^*\F_p^{\sub{Kato}}(V) = j_!j^*\F_p^{\sub{Kato}}(V)$, we get the claim in this case by the previous paragraph. Otherwise, $V$ is $p$-henselian. In this case~\eqref{eq:tilde-nu} and~\eqref{eq:fujiwara-gabber} imply that the map $\F_p^{\sub{Kato}}(V) \rightarrow \F_p^{\sub{Kato}}(V/p)$ is an equivalence. Hence, we get the desired claim as well in this case since $j_!E(V)$ is always zero if $p$ is not invertible in $V$.
}

(2) It suffices to establish that $i_*i^*\F_p^{\cdh}$ and $i_*i^*L_\sub{cdh}\F_p^{\syn}$ satisfy the $\bb P^1$-bundle formula, or in other words that $\F_p^{\cdh}$ and $L_{\cdh}\F_p^{\syn}$ satisfy the $\bb P^1$-bundle formula on $\bb F_p$-schemes. This was proved in Corollary~\ref{corol:unconditional-pn} and \cite[Theorem 5.14]{ElmantoMorrow2023} respectively (which also uses Corollary~\ref{corol:unconditional-pn} as input).
\end{proof}

\subsubsection{Kato $K$-theory} \label{subsub:Kinfsel}
In this section we will be concerned with certain variants of $K$-theory and the $\P^1$-bundle formulae they satisfy. We work in the following context, as in \cite[\S5.1]{ElmantoMorrow2023}. Let $\opp{PSh}(\Sch^\qcqs,\Spt)$ be the presentably symmetric monoidal stable $\infty$-category of presheaves of spectra on qcqs schemes, and $\K^\sub{cn}\in \CAlg(\opp{PSh}(\Sch^\qcqs,\Spt))$ the $\bb E_\infty$-algebra given by connective $K$-theory. Let $\opp{Mod}_{\K^\sub{cn}}(\opp{PSh}(\Sch^\qcqs,\Spt))$ be the $\infty$-category of modules over $\K^\sub{cn}$.

To formulate the $\bb P^1$-bundle formula for modules over $\K^\sub{cn}$, let $[\roi(-1)]:\bb S\to \K^\sub{cn}(\bb P^1_\bb Z)$ be the map classifying the line bundle $\roi(-1)$ on $\bb P_\bb Z^1$. For any $E \in \opp{Mod}_{\K^\sub{cn}}(\opp{PSh}(\Sch^\qcqs,\Spt))$ and any qcqs scheme $f:X\to\Spec(\bb Z)$, we have the map ``multiplication by $c_1(\roi(-1))$''
\[
\bb S \otimes E(\bb P^1_X) \stackrel{(f^* \circ [\roi(-1)]) \otimes \sub{id}}{\To} \K^\sub{cn}(\bb P^1_{X}) \otimes E(\bb P^1_X) \stackrel{\sub{act}}{\To} E(\bb P^1_X),
\]
where we also write $f: \bb P^1_X \rightarrow \bb P^1_{\bb Z}$ for the map induced by the structure map $f:X \rightarrow \bb Z$. This clearly defines a natural morphism of presheaves $c_1(\roi(-1)): E(\bb P^1_-) \rightarrow E(\bb P^1_-)$. We say that $E\in \opp{Mod}_{\K^\sub{cn}}(\opp{PSh}(\Sch^\qcqs,\Spt))$ {\em satisfies the $\bb P^1$-bundle formula} if, for any qcqs scheme $X$, the map of spectra \[E(X)\oplus E(X)\stackrel{\pi^* \oplus c_1(\roi(-1)) \circ \pi^*}{\To} E(\bb P_X^1)\] is an equivalence, where $\pi: \bb P^1_X \rightarrow X$ denotes the projection map. In line with our notation elsewhere, let $\opp{Mod}_{\K^\sub{cn}}(\opp{PSh}(\Sch^\qcqs,\Spt))_\sub{pbf}$ denote the full subcategory of $\K^\sub{cn}$-modules spanned by those satisfying the $\bb P^1$-bundle formula.

\begin{example}[Additive invariants]
We briefly review how $\K^\sub{cn}$-modules arising from $\Spt$-valued additive invariants of $\bb Z$-linear categories satisfy the $\bb P^1$-bundle formula. See also \cite[Construction 5.5]{ElmantoMorrow2023} for further details. Let $\opp{Fun}_\sub{add}(\opp{Cat}^\sub{perf}_\bb Z,\Spt)$ be the $\infty$-category of $\bb Z$-linear additive invariants which are \emph{not necessarily finitary}. While we refer the reader to \cite[5.4]{HoyoisScherotzkeSibilla2017} for more details, they should keep in mind that in [op.~cit.] additive invariants are required to be finitary. 

By construction $\opp{Fun}_\sub{add}(\opp{Cat}^\sub{perf}_\bb Z,\Spt) \subset \opp{Fun}(\opp{Cat}^\sub{perf}_\bb Z,\Spt)$, where the latter is a symmetric monoidal $\infty$-category via Day convolution, and $\K^{\cn}$ defines a $\bb E_{\infty}$-algebra object in $\opp{Fun}(\opp{Cat}^\sub{perf}_\bb Z,\Spt)$ by identifying lax monoidal functors with $\bb E_{\infty}$-algebra objects \cite{Glasman2015}; furthermore $\K^{\cn} \in \opp{Fun}_\sub{add}(\opp{Cat}^\sub{perf}_\bb Z,\Spt)$.

Now let $E \in  \opp{Mod}_{\K^\sub{cn}}(\opp{Fun}_\sub{add}(\opp{Cat}^\sub{perf}_\bb Z,\Spt)$).\footnote{In fact it is expected that the forgetful functor $ \opp{Mod}_{\K^\sub{cn}}(\opp{Fun}_\sub{add}(\opp{Cat}^\sub{perf}_\bb Z,\Spt)) \stackrel{\simeq}{\To}  \opp{Fun}_\sub{add}(\opp{Cat}^\sub{perf}_\bb Z,\Spt)$ is an equivalence. This is indeed true if one requires additive invariants to be finitary, by \cite[Theorem 5.24]{HoyoisScherotzkeSibilla2017}, though we do not need this result.} For the purposes of this paper, $E$ typically arises in the following way: we begin with some multiplicative maps in $\opp{Fun}_\sub{add}(\opp{Cat}^\sub{perf}_\bb Z,\Spt)$, for example $\K \rightarrow \TC$, $L_1\K$, $L_1\TC$, or $\KH$, and some multiplicative maps between the codomains under $\K$; then we take either fibres or cofibres in $ \opp{Mod}_{\K^\sub{cn}}(\opp{Fun}_\sub{add}(\opp{Cat}^\sub{perf}_\bb Z,\Spt))$ to get $E$. Note the forgetful functor $\opp{Mod}_{\K^\sub{cn}}(\opp{Fun}_\sub{add}(\opp{Cat}^\sub{perf}_\bb Z,\Spt)) \to \opp{Fun}(\opp{Cat}^\sub{perf}_\bb Z,\Spt)$ creates finite limits and colimits and thus the underlying presheaves of the various fibres and cofibres are calculated pointwise.

 Precomposing with $\opp{Perf}:\Sch^\sub{qcqs}\to \opp{Cat}^\sub{perf}_\bb Z$, this shows that for any such $E$, the presheaf \[E\circ\opp{Perf}:\Sch^\sub{qcqs}\To\Spt\]
(which, following standard abuse of notation, we will usually denote simply by $E$) is naturally a module over $\K^\sub{cn}$. Furthermore as explained in \cite[Lemma 5.6]{ElmantoMorrow2023}, this presheaf $E \circ \Perf$ satisfies the $\bb P^1$-bundle formula (thanks to the semiorthogonal decomposition of perfect complexes on $\bb P^1$).

In summary, we have explained that precomposing with $\opp{Perf}(-)$ upgrades to a symmetric monoidal functor \[\opp{Mod}_{\K^\sub{cn}}(\opp{Fun}_\sub{add}(\opp{Cat}^\sub{perf}_\bb Z,\Spt))\to \opp{Mod}_{\K^\sub{cn}}(\opp{PSh}(\Sch^\qcqs,\Spt))_\sub{pbf}.\] We will often speak below of a $K^\sub{cn}$-module ``extending to'' an additive (or, even better, a localising or truncating) invariant; by this we mean that it is in the essential imagine of this functor, and in practice we will have explicitly given an additive/localising/truncating invariant lifting it.

%
\end{example} 

A supply of $\K^\sub{cn}$-modules satisfying the $\bb P^1$-bundle formula which we will use will be produced by the following commutative diagram of presheaves on qcqs schemes:
\begin{equation}\xymatrix{
&L_1\K\ar[r] & L_1\TC\\
\K\ar[r]\ar[d]  & \K^\sub{Sel}\ar[r]\ar[u]\ar[d] & \TC\ar[u]\ar[d]\\
\KH\ar[r] & L_\sub{cdh}\K^\sub{Sel}\ar[r] & L_\sub{cdh}\TC
}\label{eqn:Ksel}\end{equation}
in which the three squares will be shown to be cartesian. Here $\K^{\Sel}$ denotes Clausen's Selmer $K$-theory \cite[Definition 6.1]{ClausenMathew2021}, defined as the pullback of $L_1\K\to L_1\TC\leftarrow \TC$ where $L_1$ denotes localisation at the complex topological $K$-theory spectrum; so the top right square is cartesian by definition. The bottom row is by definition the cdh-sheafification of the middle row. 

We now proceed to prove that both squares at the bottom are cartesian. First, we claim that the fibre of $L_1\K\to L_1\TC$ is a cdh sheaf: indeed, with finite coefficients both terms are truncating invariants  \cite[Corollary 4.23]{LandMathewMeierTamme} and so are cdh sheaves \cite[Appendix A]{LandTamme2019}, while rationally the $L_1$ has no effect and so we recover the rationalised fibre of the trace map $K\to \TC$, which is also truncating hence a cdh sheaf [op.~cit.]. This shows that the bottom right hand square has total fiber a cdh sheaf, which immediately implies that it is cartesian. The bottom rectangle is a pullback, again because the fibre of the trace map is a cdh sheaf, and so the bottom left hand square is a pullback.

Since the bottom left hand square is cartesian (hence also cocartesian) we can lift $L_\sub{cdh}\K^\sub{Sel}$ to a localizing invariant by taking the pushout of the maps $\KH \leftarrow \K \rightarrow \K^\sub{Sel}$ in the $\infty$-category $\opp{Mod}_{\K^\sub{cn}}(\Fun_\sub{add}(\opp{Cat}^\sub{perf}_{\bb Z}, \Spt))$. Similarly, pushing out the bottom rectangle lifts $L_{\cdh}\TC$ to a localising invariant, as already exploited in the proof of \cite[Lemma 5.7]{ElmantoMorrow2023}. In conclusion, the whole diagram \eqref{eqn:Ksel} extends to a commutative diagram of $\K^\sub{cn}$-modules in $\Fun_\sub{add}(\opp{Cat}^\sub{perf}_{\bb Z}, \Spt)$, in which the three squares are still cartesian.

The next object will play a crucial role in what follows:

\begin{definition}\label{def:Kato_K}
Define {\em Kato $K$-theory} to be the additive invariant \[ \K^{\sub{Kato}} = \opp{cofib}(\K \To \K^{\Sel}) \in\Fun(\opp{Cat}^\sub{perf}_{\bb Z},\Spt).\] Equivalently, by diagram \eqref{eqn:Ksel}, $\K^\sub{Kato}=\opp{cofib}(\KH \to L_\sub{cdh}\K^{\Sel})$, which shows in particular that (the restriction along $\Perf: \Sch^{\qcqs} \rightarrow \opp{Cat}^\sub{perf}_{\bb Z}$ of) Kato $K$-theory is a cdh sheaf. Furthermore, since localising invariants are stable under finite colimits, it is evidently a localising invariant and therefore it is a $\K^\sub{cn}$-module which satisfies the $\bb P^1$-bundle formula. 

We note that diagram \eqref{eqn:Ksel} defines a canonical comparison map $\K^\sub{Kato}[-1]\to \K^\sub{inf}:=\fib(\K\to\TC)$; this comparison is an equivalence modulo $p$ on any $p$-henselian ring \cite[Example~6.4]{ClausenMathew2021}.
\end{definition}

\begin{remark}
The relation between Kato $K$-theory and the Kato motivic cohomology of Definition \ref{def:cdhinf} is as follows; we will not use this relation and so do not provide proofs of the following claims. For any qcqs scheme $X$ there exists a natural filtration on $\K^\sub{Sel}(X)/p$ with graded pieces $\bb F_p(j)^\sub{syn}(X)[2j]$, for $j\in\bb Z$; after cdh sheafification this is compatible with the cdh-motivic filtration filtration on $\KH$ and so, taking cofibres, one obtains a natural filtration on $\K^\sub{Kato}(X)/p$ with graded pieces $\bb F_p(j)^\sub{Kato}(X)[2j]$. One may even define an integral form $\bb Z(\star)^\sub{Kato}$ of Kato motivic cohomology, appearing as the sheared graded pieces of a filtration on $\K^\sub{Kato}$.

It is due to this relation to Kato motivic cohomology that we define Kato $K$-theory as the cofiber of $\K\to \K^\sub{Sel}$, rather than the fibre. Indeed, had we taken the fibre as the definition then we would have been obliged to do the same for Kato motivic cohomology, which would have caused the unpleasant appearance of a shift $[-1]$ in the right hand side of \eqref{eqn:cdhinf}.
\end{remark}
%
%
%
%
%
%
%
%
%
%
%

The first parts of the following lemma are a $K$-theoretic version of Lemma~\ref{lem:inf-fibre}:

\begin{lemma}\label{lem:inf-k-fibre}
\begin{enumerate}
\item The null sequence of presheaves on qcqs schemes
\begin{equation}\label{eq:rigidity-fib-k}
L_{\cdh}j_!j^*\K^{\sub{Kato}}/p \rightarrow \K^{\sub{Kato}}/p \rightarrow i_*i^*\K^{\sub{Kato}}/p
\end{equation}
is a fibre sequence.
\item $i_*i^* \K^{\sub{Kato}}$ and $L_\cdh j_! j^* \K^{\sub{Kato}}/p$ satisfy the $\P^1$-bundle formula.
\item $L_\cdh j_! j^* \K/p\simeq L_\cdh j_! j^* \KH/p$ satisfies the $\P^1$-bundle formula (where the equivalence is due to the fact that $\K/p$ is $\bb A^1$-invariant on $\bb Z[\tfrac1p]$-schemes).
\end{enumerate}
\end{lemma}
\begin{proof}
(1) We establish the fibre sequence~\eqref{eq:rigidity-fib-k} by following the same logic as in Lemma~\ref{lem:inf-fibre}. Let $V$ be a henselian valuation ring. If $p\in V^\times$ then evaluating the null sequence \eqref{eq:rigidity-fib-k} on $V$ yields the fibre sequence $\K^\sub{Kato}(V)/p\xto{=} \K^\sub{Kato}(V)/p\to0$. If $p\notin V^\times$ then $V$ is henselian along $p$ and we must instead show that the map $\K^{\sub{Kato}}(V)/p \rightarrow \K^{\sub{Kato}}(V/p)/p$ is an equivalence. But $V$ and $V/p$ are both $p$-henselian so, as noted in Definition \ref{def:Kato_K}, this is the same as showing that the map $\K^{\sub{inf}}(V)/p \rightarrow \K^{\sub{inf}}(V/p)/p$ is an equivalence; this follows from \cite[Theorem A]{ClausenMathewMorrow2021} and the fact that valuation rings have no negative $K$-groups \cite[Theorem 3.4(3)]{KellyMorrow2021}.

(2): The presheaf $\K^{\sub{Kato}}$ satisfies the $\P^1$-bundle formula, as remarked in Definition~\ref{def:Kato_K}, for all qcqs schemes. In particular, it does so in characteristic $p > 0$, and therefore $i_*i^* \K^{\sub{Kato}}$ satisfies the $\bb P^1$-bundle formula. It then follows for $L_\sub{cdh}j_!j^*\K^\sub{Kato}/p$ from part (1).

(3): Recall that the Moore spectrum $\bb S/p \in \Spt$ has a $v_1$-self map $v_1^n: \bb S/p[r] \to \bb S/p$ for some $n,r > 0$ (see \cite{Adams1966} or \cite[Example 2.4.1]{Ravenel1992} for a textbook reference), such that, for any spectrum $E$, the canonical map \[ \colim(E/p \xrightarrow{v_1^n} E/p[-r] \xrightarrow{v_1^n} \dots  \xrightarrow{v_1^n} E/p[-dr] \dots)\To (L_{K(1)} E)/p\] is an equivalence; indeed, this holds by the veracity of the telescope conjecture at height $1$ \cite{Mahowald1982, Miller1981}. In particular, for any qcqs $\bb Z[\tfrac1p]$-scheme $X$, recalling that the canonical map 
 $\K^\Sel(X)/p \to L_{K(1)}\K(X)/p$ is an equivalence \cite[Example 6.3]{ClausenMathew2021}), we see that the self map $v_1^n:\K^\Sel(X)/p[r]\to \K^\Sel(X)/p$ is an equivalence; similarly for any iterate of $v_1^n$. We have shown that the canonical map \[\K^\sub{Kato}(X)/(p,v_1^{nr})\To \K(X)/(p,v_1^{nr})\] is an equivalence for any qcqs $\bb Z[\tfrac1p]$-scheme $X$ and $r\ge1$; in other words, $j_!j^* \K^\sub{Kato}/(p,v_1^{nr}) \quis j_!j^* \K/(p,v_1^{nr})$.
 
 Cdh-sheafifying now yields \[L_\sub{cdh}j_!j^* \K^\sub{Kato}/(p,v_1^{nr}) \quis L_\sub{cdh}j_!j^* \K/(p,v_1^{nr}),\] and so we deduce from part (2) that the right hand side satisfies the $\bb P^1$-bundle formula. To prove the $\bb P^1$-bundle formula for $L_\sub{cdh}j_!j^* \K/p$ itself we may restrict attention to schemes $X$ of finite type over $\bb Z$ (since the presheaf is finitary), in which case we conclude from the fact that $L_\sub{cdh}j_!j^* \K(X)/p$ is homologically bounded below (since $K$-theory is cdh-locally connective \cite[Theorem 3.4(3)]{KellyMorrow2021} and the cdh cohomological dimension of $X$ is bounded by its Krull dimension) and so 
\[
L_\cdh j_!j^* \K(X)/p \wequi \lim_r (L_\cdh j_!j^* \K(X))/(p,v_1^{nr}),
\]
and similarly for $\bb P_X^1$.
\end{proof}

Using Lemma \ref{lem:inf-k-fibre} we are now prepared to prove the following key lemma, establishing a special case of the projective bundle formula for $L_\sub{cdh}j_!j^*\bb F_p^\sub{cdh}$. Indeed observe that, for all $j\in\bb Z$, Corollary~\ref{cor:singular_bl} defines equivalences of presheaves on qcqs schemes \begin{equation}L_\sub{cdh}j_!j^*\bb F_p(j)^\sub{cdh}\simeq L_{\cdh}j_! \tau^{\leq j}R\Gamma_{\et}(-, \mu_p^{\otimes j})=:\Lambda(j)\label{eqn:Lambda}\end{equation} (here we have implicitly used that $L_\sub{cdh}j_!\quis L_\sub{cdh}j_!L_\sub{cdh}$). We temporarily denote this presheaf occasionally by the shorthand $\Lambda(j)$, as it will appear several times in the remainder of this section.

In principle, this lemma is a degeneration argument similar to those we have seen in \S\ref{sec:pbf-a1} based on a motivic filtration on $L_{\cdh}j_!j^*\KH/p$ and its $\bb P^1$-bundle formula, proved in Lemma~\ref{lem:inf-k-fibre}(3). However, unlike Theorem~\ref{thm:pbf-fields}, it does not require any compatibility between Chern classes.

\begin{lemma}\label{lem:p1-milnor}
Let $V$ be a mixed characteristic henselian valuation ring of rank $1$; assume that $p\notin V^\times$ and that $V$ contains a primitive $p^\sub{th}$ root of unity $\zeta_p$. Then
\[
\Lambda(j)(\P^1_V) = 0
\]
for all $j\in\bb Z$.
\end{lemma}
\begin{proof}
The assertion is vacuous when $j<0$ since then $\Lambda(j)$ vanishes. We define the following $\D(\Z)$-valued presheaves on $\Sch_{\Z[\zeta_p][1/p]}$, for $j\in\bb Z$:
\[
k_j:= H^n_\et(\ph, \Z/p) \qquad M(j):=\tau^{\leq j}R\Gamma_{\et}(-; \mu_p^{\otimes j}),
\]
so that $\Lambda(j)=L_\sub{cdh}j_!M(j)$. Since we are working in the presence of a primitive $p$-th root of unity, there are fibre sequences
\begin{equation}\label{eq:tr-cofib}
M(j-1) \xto{\zeta_p} M(j) \To {k}_j[-j],
\end{equation}
where the first map is the canonical map $\tau^{\le j-1}\to\tau^{\le j}$ followed by multiplication by $\zeta_p$.

Next consider the $\bb N$-indexed, decreasing $\cdh$-motivic filtration $\Fil_{\cdh}^{\star}\KH/p$ on  $\KH/p$, restricted to $\bb Z[\tfrac1p]$-schemes, from Theorem~\ref{prop:basic_props_of_cdh_mot}. In light of the equivalence \eqref{eqn:Lambda} we have that
\[
L_\sub{cdh}j_!\gr^j_{\cdh}\KH/p \simeq \Lambda(j)[2j]
\]
for $j\in\bb Z$ and therefore, to prove the result, it will suffice to establish the claim that $L_\cdh j_!\Fil_{\cdh}^j \KH (\P^1_V)/p = 0$ for all $j\ge0$ 

We first establish the base case $j=0$. By Lemma \ref{lem:inf-k-fibre}(3) and the hypothesis that $p\notin V^\times$, we have \[ L_\cdh j_!\Fil^0_{\cdh}\KH(\P^1_V)/p = L_\cdh j_! \KH/p(\P^1_V) \wequi L_\cdh j_! \KH(V)/p^{\oplus 2} = 0, \] as required.

Now fix some $j>0$ and assume that the claim has been proved for all integers $\le j$. In particular $\Lambda(j-1)(\P^1_V) = 0$, and so applying $L_{\cdh}j_!$ to the cofibre sequence~\eqref{eq:tr-cofib} yields \[ \Lambda(j)(\P^1_V) \quis L_\cdh j_! k_j(\P^1_V)[-j].\] Combining this with the cofibre sequence \[ L_\cdh j_! \Fil_{\cdh}^{j+1} \KH(\P^1_V)/p \To L_\cdh j_! \Fil_{\cdh}^{j} \KH(\P^1_V)/p \To \Lambda(j)(\P^1_V)[2j]\] and the inductive hypothesis (i.e., the vanishing of the middle term) we deduce that  \[ L_\cdh j_! \Fil_{\cdh}^{j+1} \KH(\P^1_V)/p[1] \wequi L_\cdh j_! {k}_j(\P^1_V)[j]. \]
Since $\Fil_{\cdh}^{j+1} \KH/p$ is cdh-locally concentrated in cohomological degrees $\le -j-1$ by Theorem~\ref{prop:basic_props_of_cdh_mot}(3), and $\P^1_V$ has valuative dimension $\le 2$, we see that the left hand side of the previous equivalence is concentrated in cohomological degrees $\le -j$.
On the other hand the right hand side is concentrated in cohomological degrees $\ge -j$. Therefore both sides are concentrated in a single homotopy group, given by $H^0_\cdh(\P^1_V, j_! {k}_j)$.

To conclude the induction, we must show that $H^0_\cdh(\P^1_V, j_! {k}_n)=0$. It is enough, since the presheaf $H^0_{\cdh}(-, j_!{k}_n)$ is cdh-separated, to prove that whenever $W$ is a henselian valuation ring with a map $\Spec(W) \to \P^1_V$, then the induced map of abelian groups \begin{equation} H^0_\cdh(\P^1_V, j_!{k}_j) \to H^0_\cdh(W, j_!{k}_j)=j_!k_j(W) \label{eqn:P12W}\end{equation} is zero. If $p\notin W$ this is true simply because the target vanishes, so henceforth assume that $p\in W^\times$, whence the map $\Spec(W) \to \P^1_V$ factors through $\P^1_{V[1/p]}$.
Therefore the map \eqref{eqn:P12W} is the top horizontal composite in the following commutative square
\begin{equation*}
\begin{CD}
H^0_\cdh(\P^1_V, j_! {k}_j) @>>> H^0_\cdh(\P^1_{V[1/p]}, {k}_j) @>>> H^0_\cdh(W, {k}_j)=k_j(W) \\
@VVV                            @V{i_0^*}VV \\
H^0_\cdh(V, j_! {k}_j) @>>> H^0_\cdh(V[1/p], {k}_j).
\end{CD}
\end{equation*}
Since the bottom left corner is zero (as $p\notin V^\times$), it is finally enough to show that the map $i_0^*$ is injective. But $F:=V[\tfrac1p]$ is a field, and $\bb P_F^1$ is a smooth curve over a field (whence its Nisnevich local rings are henselian valuation rings), so in both the domain and codomain of $i_0^*$ we may replace $H^0_\sub{cdh}$ by $H^0_\sub{Nis}$. The injectivity now follows from $\bb A^1$-invariance of $H^0_\sub{Nis}(-,k_j)$ on smooth $F$-schemes: see Theorem \ref{thm:PST}(1) and recall that $k_j$ is a homotopy invariant presheaf with transfers, as explained in the proof of Theorem \ref{thm_BL_coh} over fields.
\end{proof}

\subsubsection{Adding roots of unity}
In order to apply Lemma \ref{lem:p1-milnor} we will need to be able to add a primitive root of unity; the following lemma will allow us to do that.

\begin{lemma}\label{lem:galois-descent}
Let $j \in\bb Z$. Then the following hold.
\begin{enumerate}
\item $\Lambda(j)$ satisfies Milnor excision;
\item For any qcqs scheme $X$, the canonical map \[
\Lambda(j)(X) \rightarrow \Lambda(j)(X \otimes_{\Z} \Z[\zeta_p])^{hC_{p-1}},
\]
is an equivalence, where the action of $C_{p-1}$ on $\Lambda(j)(X \otimes_{\Z} \Z[\zeta_p])$ is induced by its action on $\Z[\zeta_p]$.
\item Statements (1) and (2) also hold if we replace $\Lambda(j)=L_\sub{cdh}j_!\tau^{\le j}R\Gamma_\sub{\'et}(-,\mu_p^{\otimes j})$ by $L_{\cdh}j_!j^*\bb F_p(j)^\sub{syn}=L_{\cdh}j_! R\Gamma_{\et}(-, \mu_p^{\otimes j})$ (i.e. we drop the truncation).
\end{enumerate}
\end{lemma}

\begin{proof}
Fixing any $i\ge 0$ we will actually prove claims (1) and (2) for the presheaf $\Lambda^i(j):=L_\sub{cdh}j_!H^i_\sub{\'et}(-,\mu_p^{\otimes j})$. We then obtain the claims for $L_\sub{cdh}j_!\tau^{\le i}R\Gamma_\sub{\'et}(-,\mu_p^{\otimes j})$ by induction, in particular for $\Lambda(j)$ by setting $i=j$, and for $L_{\cdh}j_! R\Gamma_{\et}(-, \mu_p^{\otimes j})$ by taking $\opp{colim}_i$.

(1) We use the criterion of Theorem \ref{thm:ehik}: that is, since $\Lambda^i(j)$ is a finitary cdh sheaf, we need only check henselian $v$-excision. By Remark \ref{rem:rigid-exc}, for this it suffices to check that if $V$ is a henselian valuation $V$ ring with residue field $k$, then $\Lambda^i(j)(V) \to \Lambda^i(j)(k)$ is an equivalence. If $p\not\in V^\times$ then both sides vanish, so assume $p\in V^\times$; then the equivalence follows from Gabber's rigidity for étale cohomology \cite{Gabber1994a}.

(2): Since $\Lambda^i(j)$ is a finitary $\D(\bb Z)$-valued cdh sheaf which satisfies excision, the same holds for $\Lambda^i(j)(-\otimes_\bb Z\bb Z[\zeta_p])^{hC_{p-1}}$; the only subtlety is the finitariness, since taking $C_{p-1}$-fixed points is not a finite limit, but the higher cohomology groups of any $p$-torsion $C_{p-1}$-module vanish and so $H^m(\Lambda^i(j)(-\otimes_\bb Z\bb Z[\zeta_p])^{hC_{p-1}})=H^m(\Lambda^i(j)(-\otimes_\bb Z\bb Z[\zeta_p]))^{C_{p-1}}$ for all $m\in\bb Z$, from which the finitariness is clear.

To prove the desired equivalence of statement (2) (in other words, for $\Lambda^i(j)$) it is therefore enough to show, for every henselian valuation ring $V$ of rank $\le 1$, that $\Lambda^i(j)(V) \quis \Lambda^i(j)(V\otimes_{\Z} \Z[\zeta_p] )^{hC_{p-1}}$. There are three cases to treat. If $p=0$ in $V$ then both sides vanish. If $p\in V^\times$ then $V\otimes_{\bb Z}\bb Z[\zeta_p]=V\otimes_{\bb Q}\bb Q(\zeta_p)$, which is a finite \'etale $V$-algebra and so is a finite product of henselian valuation rings by Nagata \cite[Lemma A.5]{ElmantoMorrow2023}; that justifies the first of the following equivalences (i.e., removing $L_\sub{cdh}$): \[\Lambda^i(j)(V\otimes_{\bb Z}\bb Z[\zeta_p])^{hC_{p-1}}\stackrel\sim\To H^i_\sub{\'et}(V\otimes_{\bb Z}\bb Z[\zeta_p],\mu_p^{\otimes j})^{hC_{p-1}}\stackrel\sim\To H^i_\sub{\'et}(V,\mu_p^{\otimes j})=\Lambda^i(j)(V)\] and second equivalence is Galois descent again together with vanishing of the higher cohomology of $C_{p-1}$.

It remains to treat the case that $V$ has mixed characteristic $(0,p)$; it is enough to show that $\Lambda^i(j)(A)\simeq0$, where $A:=V\otimes_\bb Z\bb Z[\zeta_p]$. We will use some elementary commutative algebra. First note that $\opp{Frac}(V)\otimes_\bb Q\bb Q(\zeta_p)$ is a finite product of finite field extensions $F$ of $\opp{Frac}(V)$, since $\bb Q\subset \bb Q(\zeta_p)$ is a finite separable extension. Next, $A$ is a subring of $\opp{Frac}(V)\otimes_\bb Q\bb Q(\zeta_p)$, finite over $V$, and so $\tilde A=\tilde V$ where the $\tilde{\phantom{-}}$ denote integral closures within $\opp{Frac}(V)\otimes_\bb Q\bb Q(\zeta_p)$. But $\tilde V$ is given by the product of the integral closures of $V$ in each $F$, and each of these is a henselian valuation ring of mixed characteristic $(0,p)$. So we have shown that $\tilde A$ is a finite product of henselian valuation rings of mixed characteristic $(0,p)$, whence $\Lambda^i(j)(\tilde A)\simeq 0$. We now use Milnor excision to replace $\tilde A$ by $A$. More precisely, write $\tilde A=\opp{colim}A_\al$ as the filtered union of its finite $A$-subalgebras $A_\al$. For any given $\al$, there is some $m\gg0$ such that $p^mA_\al\subset A$ (as $A_\al$ is finite over $A$ and $A[\tfrac1p]=A_\al[\tfrac1p]=\opp{Frac}(V)\otimes_\bb Q\bb Q(\zeta_p)$), and so we have a Milnor square
\[\xymatrix{
A\ar[d]\ar[r] & A_\al\ar[d] \\
A/p^mA_\al\ar[r] & A_\al/p^{m}A_\al
}\]
But $\Lambda^i(j)$ vanishes on both terms on the bottom row (since it is a cdh sheaf and vanishes on $\bb F_p$-schemes), and satisfies Milnor excision by part (1), so we deduce that the canonical map $\Lambda^i(j)(A)\to\Lambda^i(j)(A_\al)$ is an equivalence. By now taking the colimit over $\al$ (recall we saw in the proof of part (1) that $\Lambda^i(j)$ is finitary) we obtain $\Lambda^i(j)(A)\quis\Lambda^i(j)(\tilde A)\simeq 0$, as required.
\end{proof}

\subsubsection{Conclusion of proof of Theorem~\ref{thm:pbf-a1}} \label{subsub:conclusion}
We have now carried out the majority of the work required to prove Theorem~\ref{thm:pbf-a1}. The following lemma represents our second application of  the unconditional projective bundle formula of Corollary~\ref{corol:unconditional-pn}.

\begin{lemma} \label{lemm:HFpcdhinf-pbf}
The following objects of $\Gr\opp{PSh}(\Sch^\sub{qcqs},\D(\bb Z))_\sub{c}$ satisfy the $\P^1$-bundle formula:
\[L_\cdh j_! j^*\F^{\cdh}_p,\qquad L_\cdh j_! j^*\F_p^\syn,\qquad L_\cdh j_! j^*\F_p^{\sub{Kato}}.\]
\end{lemma}
\begin{proof}
We first treat the case of $L_\cdh j_! j^*\F^{\cdh}_p$, where we must show that
\begin{equation}\label{eq:j!pbf}
 L_\cdh j_! j^*\bb F_p(j)^{\cdh}(X) \oplus L_\cdh j_! j^*\bb F_p(j-1)^{\cdh}(X)[-2] \xrightarrow{\pi^*\oplus c_1(\roi(1))\pi^*} L_\cdh j_! j^*\bb F_p(j)^{\cdh}(\P^1_X)
 \end{equation}
is an equivalence for every qcqs scheme $X$; recall from \eqref{eqn:Lambda} that $L_\cdh j_! j^*\bb F_p(j)^{\cdh}\simeq \Lambda(j)$ for all $n$. So we may assume by Lemma \ref{lem:galois-descent}(2) that $X$ is a $\bb Z[\zeta_p]$-scheme; then using Lemma \ref{lem:galois-descent}(1), the finitariness of each $L_\cdh j_! j^*\bb F_p(j)^{\cdh}$, and Proposition \ref{proposition_checking_on_points} we reduce to the case that $X=\Spec(V)$ is the spectrum of a henselian valuation ring $V$ of rank $\le 1$ lying under $\bb Z[\zeta_p]$. We now treat three cases. If $p\in V^\times $ then we may drop the three occurrences of $j_!j^*$ and see that \eqref{eq:j!pbf} is an equivalence by Corollary~\ref{corol:unconditional-pn}, noting that the key hypothesis holds vacuously by Remark~\ref{rmk:Hyp-vacuous}. If $p=0$ in $V$ then both sides of \eqref{eq:j!pbf} vanish (just because $L_\sub{cdh}j_!$ of any presheaf vanishes on $\bb F_p$-schemes). The third case is when $p\notin V^\times$ and $V$ is of mixed characteristic, containing a primitive $p^\sub{th}$ root of unity; then the left hand side of \eqref{eq:j!pbf} vanishes by removing $L_\sub{cdh}$ and using the definition of $j_!$, while the right hand side vanishes by the key Lemma \ref{lem:p1-milnor}.

Next we treat $L_\cdh j_! j^*\F_p^\syn$. Arguing exactly as at the beginning of the previous paragraph, we reduce to establishing that
\begin{equation}\label{eq:j!pbf2}
 L_\cdh j_! j^*\bb F_p(j)^{\sub{syn}}(V) \oplus L_\cdh j_! j^*\bb F_p(j-1)^{\sub{syn}}(V)[-2] \xrightarrow{\pi^*\oplus c_1(\roi(1))\pi^*} L_\cdh j_! j^*\bb F_p(j)^{\sub{syn}}(\P^1_V)
 \end{equation}
 is an equivalence whenever $V$ is a henselian valuation ring of rank $\le 1$ lying under $\bb Z[\zeta_p]$. If $p\in V^\times$ this is precisely the $\bb P^1$-bundle formula for \'etale cohomology. If $p=0$ in $V$ then both sides vanish. Finally, if $p\notin V^\times$ and $V$ is of mixed characteristic, containing a primitive $p^\sub{th}$ root of unity, then again we must show that the right hand side vanishes; but we are working over $\bb Z[\zeta_p]$, so that \[j^*\bb F_p(j)^\syn=R\Gamma_\sub{\'et}(-,\mu_p^{\otimes j})\quis \opp{colim}(\tau^{\le 0}R\Gamma_\sub{\'et}(-,\mu_p^{\otimes 0})\xto{\zeta_p}\tau^{\le 1}R\Gamma_\sub{\'et}(-,\mu_p^{\otimes 1})\xto{\zeta_p}\cdots),\] whence applying $L_\sub{cdh}j_!$ yields an equivalence $L_\sub{cdh}j_!j^*\bb F_p(j)^\sub{syn}\simeq \opp{colim}_m\Lambda(m)$ and so the desired vanishing again follows from the key Lemma \ref{lem:p1-milnor}.

Finally, by definition we have a cofibre sequence in $\Gr\PShv(\Sch^\qcqs, \D(\bb Z))_c$
\[
L_\cdh j_! j^*\F^{\cdh}_p \To L_\cdh j_!j^* \F_p^\syn\To L_\cdh j_! j^*\F_p^{\sub{Kato}},
\]
and so the $\bb P^1$-bundle formulae for $L_\cdh j_! j^*\F^{\cdh}_p$ and $L_\cdh j_! j^*\F_p^\syn$ imply it for $L_\cdh j_! j^*\F_p^{\sub{Kato}}$.
\end{proof}

We now prove the main theorem; as already noted in \S\ref{subsub:reduce-to-P1}, it remains only to prove it modulo $p$.

\begin{proof}[Proof of Theorem~\ref{thm:pbf-a1} modulo $p$.]

By Definition~\ref{def:cdhinf} and the fact that $L_{\A^1}$ is exact, we have a fibre sequence
\begin{equation}\label{eq:inf-fibre}
 L_{\A^1}\F_p^{\cdh} \To L_{\A^1}L_{\cdh}\F_p^{\syn} \To L_{\A^1}\F_p^{\sub{Kato}}
\end{equation}
in $\Gr\PShv(\Sch^\qcqs, \D(\bb Z))_c$, and so to prove that the first term satisfies the $\P^1$-bundle formula it is enough to prove that the other two terms do. The middle term was already treated (easily) in Example \ref{ex:pbfGr}, while the third term was treated (even before $\bb A^1$-localisation) in Lemma \ref{lemm:HFpcdhinf-pbf}.
\end{proof}

\section{Variant for Hermitian $K$-theory (away from $2$)} \label{sec:GW}

In this section, we prove variants of our main theorems for Hermitian $K$-theory\footnote{Following common practice (see, for example, \cite{calmes2024motivic}), for a qcqs scheme $X$ we have the \emph{Grothendieck--Witt spectrum} $\rm GW(X)$ but the functor $\rm GW: \Sch^\sub{qcqs,op} \rightarrow \Spt$ itself is called \emph{Hermitian $K$-theory}.} in place of $K$-theory, for schemes on which $2$ is invertible\footnote{Our main references will be  the series of papers \cite{Hermitian1,Hermitian2,Hermitian3}. While the main innovation of these papers is a good theory of Hermitian $K$-theory which does not require $2$ to be invertible, we use them even in this setting because they are conveniently written in the language of the present paper. We also note the reference \cite{Schlichting2017} which is a rather comprehensive account of a classical version of theory where $2$ is invertible; in fact most of the theorems we need are already covered there but in an older language. An extension of these results where $2$ is not assumed to be invertible will be treated in a sequel.}.  This theory is a quadratic refinement of $\K$-theory: while $\K$-theory is ultimately built from vector bundles, Hermitian $K$-theory is built from vector bundles equipped with a nondegenerate symmetric bilinear form. As a consequence, we produce an $\bb A^1$-invariant motivic filtration on homotopy Hermitian $K$-theory (i.e. $\bb A^1$-localised Hermitian $K$-theory), producing a quadratic refinement of the motivic filtration on homotopy $\K$-theory in this setting. The attendant motivic cohomology theory is often called \emph{Milnor-Witt motivic cohomology}.

Most of our arguments are very similar to those already performed for $\K$-theory and motivic cohomology and hence we will omit certain details where appropriate. In particular, we use the language and technology of motivic stable homotopy theory more judiciously in this section.

 Let us briefly explain what we are after. As explained in~\S\ref{sec:slice-conj} the results of the paper shows that the motivic spectrum $s^0(\KGL)$ is a good candidate for an $\bb A^1$-invariant motivic cohomology theory. Therefore, for the rest of this section and for the rest of the paper, we set:
\[
H\bb Z_X := s^0(\KGL_X) (\simeq s^0(1_{X})),
\]
Working with qcqs schemes over $\Spec(\Z[\tfrac{1}{2}])$, we have a motivic spectrum $\KO$ representing (homotopy invariant) Hermitian $K$-theory (sometimes this is also written as $\mathrm{KQ}$ in the motivic homotopy literature). The role of $H\Z_X$ here is played instead by the spectrum $H \tilde{\Z}_X$ representing \emph{Milnor-Witt motivic cohomology}; we refer the reader to \cite{bachmann2025} for a textbook account of this theory. We warn the reader that $H \tilde{\Z}_X$ is \emph{not} simply the zero-th slice of $\KO$. The slices of the latter are known over fields but are rather complicated \cite{RoendigsOestvaer2016}. Instead, $H \tilde{\Z}_X$ appears as a kind of generalised slice, first considered by the first author over fields in \cite{bachmann-very-effective}; we refer the reader to Remark~\ref{rem:zero-gen} for a precise statement.  The main result of this section is Theorem~\ref{thm:ko}, which is the exact $\KO$ analog of Theorem~\ref{thm:a1-a1cdh} which describes the slices of $\KGL$ over an arbitrary base scheme $X$. The motivic filtration on $\bb A^1$-invariant Hermitian $K$-theory is a consequence recorded in Corollary~\ref{thm:ko-main}.

\subsection{Generalised slices}\label{ss:gen-slices}
We briefly discuss these generalised slices. In \S\ref{ss:slice} we have not only encountered the $\infty$-category $\SH(X)^\eff$ but also the $\infty$-category of very effective spectra $\SH(X)^\veff$. There are other variants of the slice filtration which use $\SH(X)^\veff$ instead of $\SH(X)^\eff$ as the building block. Here we briefly discuss the resulting \emph{generalised slice filtrations}. 

Since $\SH(X)^\veff \otimes \bb T_X^{\otimes b}[a-2b] \subset \SH(X)$ is stable under colimits, it admits a right adjoint. The composition of this right adjoint with inclusion is denoted by
\[
\tilde{\Fil}_\sub{slice}^{a,b}: \SH(X) \rightarrow \SH(X).
\]
We call $\tilde{\Fil}_\sub{slice}^{a,b}E$ a \emph{generalised slice cover} of $E$. Using this procedure, out of a sub-poset $I$ of $\bb Z \times \bb Z$, we can produce an $I$-indexed filtration of $E \in \SH(X)$, called a generalised slice filtration.

The relevant one for us is defined in the following way; it comes under the name of the  \emph{generalised slice filtration} and was first studied in \cite{bachmann-very-effective,spitzweck2012motivic}.  We set:
\[ 
\tilde{\Fil}_\sub{slice}^n = \tilde{\Fil}_\sub{slice}^{2n,n}
\]
By the same reasoning as in Remark~\ref{rem:mult}, it assembles into a lax monoidal functor
\[
\tilde{\Fil}_\sub{slice}^\star: \SH(X) \rightarrow \Fil\SH(X).
\]
Just as in the slice filtration, we will write
\[
\tilde{s}_\sub{slice}^{\star}E
\]
for the graded pieces of this filtration evaluated on $E \in \SH(X)$. We will apply this to the motivic spectrum $\KO_X$ in \S\ref{sec:motfilt-ko}, and therefore endow $\KO_X$ with a multiplicative filtration
 \[
 \tilde{\Fil}_\sub{slice}^\star\KO_X \rightarrow \KO_X.
 \]

In practice, it is also convenient to consider the following more refined filtration where we introduce ``half slices.'' It is, however, \emph{not} monoidal:
\[ \tilde{\Fil}_\sub{slice}^n = \tilde{\Fil}_\sub{slice}^{2n,n} \quad\text{and}\quad \tilde{\Fil}_\sub{slice}^{n+1/2} = \tilde{\Fil}_\sub{slice}^{2n+1,n}. \]
We will not employ the slice $s^{\star}$ notation for the layers of this filtration but instead write:
 \[
   \tilde{\Fil}_\sub{slice}^{m}(E)/\tilde{\Fil}_\sub{slice}^{n}(E) := \mathrm{cofib}(\tilde{\Fil}_\sub{slice}^{n}(E) \rightarrow \tilde{\Fil}_\sub{slice}^{m}(E)) \qquad n \geq m. 
 \]

\begin{remark}\label{rem:cofiber} Since $\SH(X)^\veff \otimes \bb T_X^{\otimes b}[a-2b] \subset \SH(X)$ is not a stable subcategory we note that the functor $\tilde{\Fil}_\sub{slice}^{a,b}$ is not exact. Nonetheless, we will repeatedly use the following observation: if $A \rightarrow B \rightarrow C$ is a cofiber sequence in $\SH(X)$ then we have an induced equivalence $\tilde{\Fil}_\sub{slice}^{a,b}A \stackrel{\simeq}{\To} \tilde{\Fil}_\sub{slice}^{a,b}B$ as soon as the mapping \emph{space} $\Map(Y, C) \simeq *$ for all $Y \in \SH(X)^\veff \otimes \bb T_X^{\otimes b}[a-2b]$. This latter condition is equivalent to asking that the presheaf of pointed spaces $\Omega^{\infty}\omega^{\infty}(C[2b-a]): \Sm_X^{\op} \rightarrow \Spc_*$ is contractible. 
\end{remark}

\begin{remark}\label{rem:detection} If $X$ is the spectrum of Dedekind domain, then Proposition~\ref{prop:detecting-veff}(2) gives a criterion for checking when  an effective spectrum is very effective. Unfortunately, the proof is very sensitive to the fact that $X$ is of Krull dimension one and we do not know an analog for higher dimensional schemes. Nonetheless, note that Proposition~\ref{prop:detect-effective} holds more generally in that a spectrum $E \in \SH(X)^\veff \otimes \bb T_X^{\otimes b}[a-2b]$ if and only if for each point for every point $x\in X$, the pullback $i_x^*E\in\SH(k(x))^\veff \otimes \bb T_{k(x)}^{\otimes b}[a-2b]$ (see, again, \cite[Proposition B.3]{BachmannHoyois2021}). 
\end{remark}

\subsection{$L$-theory and the key result}
The key to bootstrapping from results about motivic cohomology and $K$-theory to results about Hermitian $K$-theory is to study the cdh-sheafification of the abelian presheaf given by $L$-theory\footnote{Roughly, the difference between $K$-theory and Hermitian $K$-theory is ``the $L$-theory part.'' We refer to \cite[Main Theorem]{Hermitian2} for a precise statement.} in degree zero. Nonetheless, we need $L$-theory as a presheaf of spectra, which we now review. 

Just like $K$-theory, there are two $L$-theory spectra, the ``connective'' version $\L^\sub{cn}$ and a ``non-connective version'' $\L$ (this notation aligns with the notation of this paper for $K$-theory, but the former is usually denoted by $\L$ in the literature and the latter by $\bb L$).
However, unlike $K$-theory, neither of them takes values in connective spectra.
The most general definition of the former is \cite[Definition 4.4.4]{Hermitian2}.
The $\L$-groups obtained in this way agree with Balmer's triangular Witt groups \cite[discussion on p. 39]{Hermitian1}.
In particular, $\L_0(A) = \pi_0\L^\sub{cn}(A)$ is the usual Witt group \cite[Theorem 4.3]{balmer-triangular-2} in the classical sense of \cite[Chapter 1]{MilnorHusemoller}. If $X$ is a scheme it coincides with the construction of Knebusch \cite{Knebusch1977}. 
For $\L$, we refer the reader to \cite[\S2.3]{calmes2024motivic}; though for schemes where $2$ is invertible, this should also coincide with the spectrum constructed by \cite{Schlichting2017}. In any case, this latter spectrum only plays an auxiliary role.
By construction there is a natural transformation of $\bb E_\infty$-rings $\L^\sub{cn} \to \L$.

The next proposition summarises the key properties of $L$-theory that we will use.

\begin{proposition} \label{prop:L-basics}
The functors $\L^\sub{cn}, \L: \Sch_{\Z[1/2]}^\sub{qcqs,op} \to \Spt$ satisfy the following properties:
\begin{enumerate}
\item $\L^\sub{cn}$ is $4$-periodic, i.e. there exists a unit in $\L_4(\Z[1/2])$.
\item For $A$ a local ring, $\pi_i\L^\sub{cn}(A) = 0$ for $i \not\equiv 0 \pmod{4}$.
\item The map $\L^\sub{cn}(X) \to \L_{\A^1}\L^\sub{cn}(X)$ is an equivalence if $X$ is regular or the spectrum of a valuation ring.
\item The map $\L^\sub{cn}(X) \to \L(X)$ is an equivalence if $X$ is regular or the spectrum of a valuation ring.
\item There exists a class $c \in \pi_{-1}\L^\sub{cn}(\P^1_{\Z[1/2]})$ inducing for all $X \in \Sch^\sub{qcqs}_{\Z[1/2]}$ a $\bb P^1$-bundle formula: \[  \L(X) \oplus  \L(X)[-1] \xrightarrow{p^* \oplus c \cdot p^*} \L(\P^1_X), \] where $p: \P^1_X \rightarrow X$ is the projection map.
\item $\L^\sub{cn}$ and $\L$ are both finitary presheaves of spectra.
\end{enumerate}
\end{proposition}
\begin{proof}

(1) See e.g. \cite[Proposition 1.14]{balmer-triangular-1}.

(2) See e.g. \cite[Theorem 5.6]{balmer-triangular-2}.

(3) The result for regular schemes is established by Balmer in \cite[Theorem 3.1]{balmer-homotopy-invariance}.
By (1), the result need only be proved for $\pi_i\L$, $i \in \{0,1,2,3\}$.
For $i \in \{0,2\}$, Balmer cites \cite[Theorem 3.1]{MR1805408}, which does not require any regularity hypotheses.
For $i \in \{1,3\}$, Balmer's proof only requires that $\K_0(A) \to \K_0(A[t_1, \dots, t_n])$ is surjective (see the discussion before \cite[Lemma 3.3]{balmer-homotopy-invariance}).
This holds if $A$ is a valuation ring since $\K^\cdh \wequi \KH$.

(4) Via Karoubi cofinality \cite[(2.3.1)]{calmes2024motivic}, this is immediate from the analogous claim about $K$-theory.

(5) See \cite[Theorem 6.1.6]{calmes2024motivic} for the most general statement (and \cite[Proposition 1.0.3]{calmes2024motivic} for the case of the $\bb P^1$-bundle formula), but the results of \cite[Theorem A \& B]{Rohrbach2022} are sufficient since we are working with schemes where $2$ is invertible. 


(6) For $\L^\sub{cn}$, see \cite[Corollary 4.4.6]{Hermitian2}.
The case of $\L$ follows from this and the construction.\NB{would be nice to have a definition so this claim even makes sense...}
\end{proof}

Now, according to \cite[Definition 8.1.1]{calmes2024motivic}, we have an absolute motivic spectrum (the base change stability follows from \cite[Theorem 8.1.6 and Proposition 8.2.2]{calmes2024motivic}) denoted by $\KW_X \in \CAlg(\SH(X))$; note that this coincides with the motivic spectrum was first constructed by Hornbostel in \cite{hornbostel2005a1} at least when $X$ is a regular scheme where $2$ is invertible. It has the property that
\[
(\omega^{\infty,\gr}\KW_X)^j \simeq L_{\A^1}\L[i]. 
\] 
Nonetheless, we will only consider $\KW_X$ for schemes $X$ with $1/2 \in \roi_X$. The following assertion about the cdh-sheafification of $\L^\sub{cn}$ and $\L$ follows from the same argument as its counterpart for $K$-theory and Proposition~\ref{prop:L-basics}.
%

\begin{corollary} \label{cor:A1-cdh-L} We have natural equivalences of presheaves on $\Sch^\sub{qcqs}_{\Z[1/2]}$:
\[
L_\cdh \L^\sub{cn} \quis L_\cdh \L \quis L_{\A^1} \L.
\]
\end{corollary}
\begin{proof}
The fact that $\KW$, as constructed above, is an absolute motivic spectrum, implies that the right most term is a cdh sheaf (use Proposition \ref{prop:auto-cdh}).
Thus the canonical maps $\L^\sub{cn} \to \L \to L_{\A^1} \L$ induce $L_\cdh \L^\sub{cn} \to L_\cdh \L \to L_{\A^1} \L$.
All three terms being finitary we may check the maps are equivalence on stalks, which holds by Proposition \ref{prop:L-basics}(3, 4).
\end{proof}

Composing the functor $\L$ with $\pi_0$ we have 
 \[ \pi_0\L^\sub{cn} := \L_0: \Sch_{\bb Z[\tfrac{1}{2}]}^{\qcqs,\op} \rightarrow \CAlg(\mathrm{Ab}) \subset \CAlg(\Spt). \]
The following are basic facts about $\L_0$ that we will need. 
 \begin{proposition} \label{prop:L0-basics}
\begin{enumerate}
\item The functor $\L_0$ is finitary.
\item The functor $\L_0$ is rigid (see Definition \ref{def:rigid}).
\end{enumerate}
\end{proposition}
\begin{proof}
Since $\L_0$ is $\pi_0$ of a finitary functor (by Proposition~\ref{prop:L-basics}(6)), claim (1) holds. Using this, (2) reduces to the Noetherian case (use \cite[Tag 0AGV]{Stacks}), for which this is exactly \cite[Lemma 5.2]{jacobson2018cohomological}.
\end{proof}

The following is an $L$-theoretic analog of what we have proved for low-weights motivic cohomology (Corollaries~\ref{corol:low_weights_a1} and~\ref{corol_p_low_weights}). Under this analogy the $\L_0$ presheaf should be thought of as the constant presheaf at $\bb Z$. This turns out to be the key result in calculating the generalised slices of $\KO$.
 
\begin{theorem} \label{thm:key-L}
The cdh sheaf \[ R\Gamma_{\cdh}(-,\L_0): \Sch_{\bb Z[\tfrac{1}{2}]}^{\qcqs,\op} \rightarrow \CAlg(\Spt) \] is $\A^1$-invariant and satisfies the $\P^1$-bundle formula (in the sense of Remark~\ref{rmk:P1-bundle-L0}).
\end{theorem}

\begin{remark}[Formulation of $\bb P^1$-bundle formula] \label{rmk:P1-bundle-L0}
The last assertion of Theorem \ref{thm:key-L} requires some further explanation.
The precise claim is that there exists a certain canonical class $c \in H_{\cdh}^1(\P^1_{\Z[1/2]}, \L_0)$ such that for any $X \in \Sch_{\bb Z[\tfrac{1}{2}]}^{\qcqs}$ the canonical map \[ R\Gamma_{\cdh}(X, \L_0) \oplus R\Gamma_{\cdh}(X, \L_0)[-1] \xrightarrow{p^* \oplus c \cdot p^*} R\Gamma_{\cdh}(\bb P^1_X, L_0) \] is an equivalence. 

We now construct $c$ explicitly. Granting that we know the functor $R\Gamma_{\cdh}(-, \L_0)$ is $\bb A^1$-invariant (as in the first part of Theorem~\ref{thm:key-L}), we have that $H^1_{\cdh}(\P^1_{\Z[1/2]}, \L_0) = \ker(H^0_{\cdh}(\Gm_{\Z[1/2]}, \L_0) \xrightarrow{1^*}H^0_{\cdh}(\Spec(\Z[1/2]), \L_0))$. Therefore, it suffices to construct a class in $H^0_{\cdh}(\Gm_{\Z[1/2]}, \L_0)$ which vanishes when restricted along $1 \in \Gm_{\Z[1/2]}$. We have a comparison map $W(\Z[\tfrac{1}{2}, t^{\pm 1}]) = L_0(\Gm_{\Z[\tfrac{1}{2}]}) \rightarrow H^0_{\cdh}(\Gm_{\Z[1/2]}, \L_0)$ and the element to use is the image of the class $\lra{t}-1 \in W(\Z[\tfrac{1}{2}, t^{\pm 1}])$ under the comparison map (see \cite[Definition 2.6]{bachmann-etaZ} for more details on a similar construction with the Nisnevich topology instead of the cdh topology).

%
\end{remark}

\begin{remark}\label{rem:no-orientations} Unlike $K$-theory or motivic cohomology, the cohomology theories studied in this section are \emph{not oriented}. For example, they do not satisfy a projective bundle formula; we refer the reader to \cite[Theorem 6.1.6]{calmes2024motivic} for how the $\L$-theory of a projective bundle looks like. 
\end{remark}

To prove Theorem~\ref{thm:key-L}, we first prove the following lemma which employs the same degeneration techniques as in \S\ref{sec:pbf-a1}.

\begin{lemma} \label{lemm:compute-KW}
Let $X \in \Sch_{\Z[1/2]}$ have valuative dimension $\le 3$.
Then for $n \in \Z$ and $0 \le i \le 3$ we have \[ \pi_{4n-i} L_{\cdh}\L^\sub{cn}(X) \wequi H^i_\cdh(X, \L_0). \]
\end{lemma}
\begin{proof}
By Corollary \ref{cor:A1-cdh-L} and the construction of $\KW_X$ we have that, $\omega^{\infty}\KW_X(X) \wequi L_\cdh \L^\sub{cn}(X)$. In particular, $L_\cdh \L^\sub{cn}$ satisfies a $\bb P^1$-bundle formula compatible with the one coming Proposition~\ref{prop:L-basics}(5), i.e., before cdh sheafification.
Using Proposition \ref{prop:L-basics}(1,2) we see that $a_\cdh \L_i \wequi a_\cdh \L_0$ for $i \equiv 0 \pmod 4$, and $=0$ else.
The strongly convergent cdh descent spectral sequence converging to the homotopy groups of $L_{\cdh}L^\sub{cn}$ then collapses and the result follows (after noting that the construction of $c$ from Remark~\ref{rmk:P1-bundle-L0} is compatible with the cdh sheafification of the map induced from Proposition~\ref{prop:L-basics}(5) with the same name).
\end{proof}

\begin{proof}[Proof of Theorem \ref{thm:key-L}.]
Since $\L_0$ is rigid it satisfies henselian $v$-excision by Remark~\ref{rem:rigid-exc}. Since $\L_0$ is in addition finitary it follows that $L_\cdh \L_0$ satisfies Milnor excision by Theorem~\ref{thm:ehik} and hence, in order to verify $\A^1$-invariance or the $\P^1$-bundle formula, it suffices to prove this when the base $X$ has valuative dimension $\le 1$.
In this case both $\A^1_X$ and $\P^1_X$ have valuative dimension $\le 2$ \cite[Proposition 2.3.2(7,1)]{ElmantoHoyoisIwasaKelly2021}. Therefore, Lemma~\ref{lemm:compute-KW} applies to describe the cdh cohomology of $\L_0$ in terms of the homotopy groups of $L_{\cdh}\L^\sub{cn}$. From this, $\bb A^1$-invariance follows immediately by Corollary~\ref{cor:A1-cdh-L} and the $\bb P^1$-bundle formula was already noted in the proof of Lemma~\ref{lemm:compute-KW}.

\end{proof}

\newcommand{\I}{\mathrm{I}}

\subsection{Base change results for Milnor--Witt motivic cohomology and very effective cover of $\KO$} Theorem~\ref{thm:key-L} lets us deduce a result about Milnor--Witt motivic cohomology and the very effective cover of $\KO$, the motivic spectrum representing Hermitian $K$-theory, over an arbitrary base (the latter by confirming a criterion of the first author). This will feed into the construction of the motivic filtration on $\A^1$-invariant Hermitian $K$ theory in the next section. 

Let $X$ be a qcqs scheme. Recall that $\SH(X)^\veff \subset \SH(X)^\eff$ determines the non-negative part of a $t$-structure on $\SH(X)^\eff$ by formal reasons \cite[Proposition 1.4.4.11]{LurieHA}. We denote by \[ \tau_{\le 0}^\eff: \SH(X)^\eff \to \SH(X)^\eff \] the associated truncation functor.
The \emph{Milnor-Witt motivic cohomology spectrum} is defined to be \[ H\tilde{\Z}_X = \tau_{\le 0}^\eff(1_X) \simeq \pi_0^{\eff}(1_X) \in \CAlg(\SH(S)). \]

\begin{remark}\label{rem:zero-slice-analog} We note the formal similarity with the zero-th slice of the sphere $s^0(1_X) := \mathrm{cofib}(\mathrm{Fil}_\sub{slice}^{1}1_X \rightarrow 1_X)$. We had established in Theorem~\ref{cor:conjectures} that $s^0(1_X) \simeq H\Z_S$ and that it is stable under base change. The goal of this subsection is to prove the same result of $H\tilde{\Z}_X$.
\end{remark}

Following Definition~\ref{definition_A^1_mot-_coh} we can extract the following cohomology theory. 

\begin{definition}\label{definition_A^1_mot-_coh_mw}
For a qcqs scheme $X$, its {\em $\bb A^1$-Milnor--Witt motivic cohomology} is defined by \begin{equation}\tilde \Z(j)^{\bb A}(X):=\map_{\SH(X)}(1_X, \bb T^{\otimes j}_X \otimes H\tilde{\Z}_X))[-2j]\in\opp{Sp}\label{eqn_ZjAMW}\end{equation} for $j\in\bb Z$. Write $H^i_{\bb A}(X,\tilde \Z(j)):=H^i(\tilde \Z(j)^\bb A(X))$, for $i,j\in\bb Z$, for the corresponding {\em $\bb A^1$-Milnor--Witt cohomology groups}. 
\end{definition}

\begin{remark} The cohomology theory $H^i_{\bb A}(X,\tilde \Z(j))$ provides a ``quadratic refinement'' of $\bb A^1$-motivic cohomology. It is an environment in which one can study cycles ``equipped with an orientation'' and plays a central role in enumerative geometry over fields other than the complex numbers \cite{KassWicklegren2019}.
 \end{remark}

We defer a more extensive development of this cohomology theory to the future. For now, our goal is to prove that motivic spectrum $H\tilde{\Z}_X$ is stable under base change. To do this, we need an auxiliary motivic spectrum which plays the role of Milnor $K$-theory for $L$-theory.

\begin{definition}\label{def:kw} The motivic spectrum of \emph{Witt $K$-theory} $\ul{\K}^W \in \SH(\Z[1/2])$ is a version of Morel's spectrum \cite{A1-alg-top} (constructed over fields) whose construction we sketch. There is a natural transformation $\L_0 \to \Z/2$ of presheaves of (ordinary) rings on qcqs schemes called the \emph{rank map}.
We denote its kernel by $\I \subset \L_0$, which is a presheaf of ideals. Its powers are denoted $\I^n \subset \L_0$. 

In \cite[Definition 2.6, Corollary 2.8]{bachmann-etaZ}, the first author constructed $\ul{\K}^W \in \CAlg(\SH(\Z[1/2]))$ in the following way (in the language of this paper): we have the multiplicative graded presheaf\footnote{To obtain a multiplicative structure compatible with the grading, one needs to shear. This can be freely done since we are only working with $\D(\Z)$-valued presheaves as we have done many times already.}
\[
R\Gamma_{\Nis}(-,I^{\star})[\star] \in \Gr\opp{Sh}_{\Nis, \bb A^1}(\Sm_{\Z[1/2]},\Spt).
\]
Note that, by construction, $R\Gamma_{\Nis}(-,I^{j}) = R\Gamma_{\Nis}(-,\L_0)$ for $j \leq 0$. Indeed, that each term is a $\bb A^1$-invariant was verified in \cite[Proof of Corollary 2.8]{bachmann-etaZ}. We also have a ``first Chern class'' map $c:\Sigma^{\infty}\P_{\Z[1/2]}^1 \rightarrow R\Gamma_{\Nis}(-,I^1)[1]$ constructed in \cite[Definition 2.6]{bachmann-etaZ}; here we used that $R\Gamma_{\Nis}(-,I^1)[1]$ is already an $\bb A^1$-invariant, Nisnevich sheaf to translate from a class in reduced cohomology of $\bb G_{m}$ to the required map. That $(R\Gamma_{\Nis}(-,I^{\star})[\star], c) \in \Gr\opp{Sh}_{\Nis, \bb A^1}(\Sm_{\Z[1/2]},\Spt)_\sub{c,pbf}$ was verified in \cite[Corollary 2.8]{bachmann-etaZ}. Using the motivic Eilenberg-MacLane functor we set
\[
\ul{\K}^W := H(R\Gamma_{\Nis}(-,I^{\star})[\star]) \in \CAlg(\SH(\Z[1/2]). 
\]
Using \cite[Definition 2.6, Lemma 2.7]{bachmann-etaZ} (which uses, ultimately, Voevodsky's resolution of the Milnor conjectures \cite{voevodsky-z2}) we obtain a map in $\CAlg(\SH(\Z[1/2])$:
\begin{equation}  s:\ul{\K}^W \rightarrow H \Z_{\Z[1/2]}/(2,\tau)\label{eq:s-map}.\end{equation} Here, $\tau: \bb T \otimes H\bb Z/2[-1] \rightarrow H\bb Z/2$ is the map in mod-$2$ motivic cohomology which classifies the class of $-1 \in  \mu_2(\Z[\tfrac{1}{2}])$. A slightly different description of this spectrum is used in [loc.~cit.] to show that $H \Z_{\Z[1/2]}/(2,\tau)$ is even an $\bb E_{\infty}$-ring.
In particular, by the description of \cite[Lemma 2.7]{bachmann-etaZ}, $\omega^{\infty,\gr}_{\otimes}$ of~\eqref{eq:s-map} is given by the map of multiplicative graded presheaves
\begin{equation}\label{eq:s-explicit}
R\Gamma_{\Nis}(-,I^{\star})[\star] \xrightarrow{s^{\star}} R\Gamma_{\Nis}(-, \scr H_{\et}^{\star}(\bb F_2))[\star]. 
\end{equation}
\end{definition}

The following is our main theorem about the motivic spectrum $H\tilde \Z$:

\begin{theorem}\label{thm:hz-tilde-main} Let $X$ be a qcqs scheme where $2 \in \roi_X^{\times}$ and let $f: X \rightarrow \Spec(\Z[1/2])$ be the structure morphism. Then
\begin{enumerate}
\item the motivic spectrum $f^*H\tilde{\Z}_{\Z[1/2]}$ is very effective.
 \item There is a canonical pullback square in $\CAlg(\SH(X))$:
\begin{equation}\label{eq:pullback}
\begin{CD}
H\tilde{\Z}_X @>>> H\Z_X \\
@VVV           @VVV  \\
f^*\ul{\K}^W @>f^*s>> H\Z/(2,\tau)_X.
\end{CD}
\end{equation}
\item the canonical map
\[
f^*H\tilde{\Z}_{\Z[1/2]} \rightarrow H\tilde{\Z}_{X}
\]
is an equivalence.
\end{enumerate}
\end{theorem}

The techniques of the present paper are first used to establish the following description of $f^*\ul{\K}^W \in \SH(X)$:

\begin{proposition} \label{prop:determine-KW}
Let $X \in \Sch_{\bb Z[\tfrac{1}{2}]}^{\qcqs}$ with structure map $f: X \to \Spec(\Z[1/2]).$ 
Then we have a equivalence of presheaves on $\Sm_X$:
\[
R\Gamma_{\cdh}(-, \I^j)[j] \to \omega^{\infty}(\bb T^{\otimes j} \otimes f^*\ul{\K}^W),
\]
which is multiplicative as $j \in \Z$ varies.
\end{proposition}

\begin{proof} The comparison map is produced in a completely analogous fashion to the map $\Z(j)^{\cdh} \rightarrow \Z(j)^{\bb A,\cdh}$. In particular, it is multiplicative. Just as in the beginning of the proof of Theorem~\ref{thm:comparison}, it suffices to prove that $R\Gamma_{\cdh}(-, \I^j)$ is  finitary, cdh-locally left Kan extended from $\Sm_{\Z[1/2]}$, is $\bb A^1$-invariant and satisfies the $\bb P^1$-bundle formula. 

First, note that the map of functors $\L_0 \rightarrow \Z/2$ is a morphism between finitary, rigid presheaves; this is obvious for the target and Proposition~\ref{prop:L0-basics} for the domain. Therefore, the same is true for the kernel and its powers. We conclude from Lemma~\ref{lemma_rigid_implies_lke} that $\I^j$ is left Kan extended from smooth $\Z[1/2]$-algebras for all $j \in \Z$. This implies that $R\Gamma_{\cdh}(-, \I^j)$ is finitary (by Proposition~\ref{prop:finitary_conditions}(2)) and cdh-locally left Kan extended from $\Sm_{\Z[1/2]}$ as desired.

Next we prove $\A^1$-invariance. We have have cofiber sequences \[ R\Gamma_\Nis(-, \I^{j+1})|_{\Sm_{\Z[1/2]}} \to R\Gamma_\Nis(-,\I^{j})|_{\Sm_{\Z[1/2]}} \to R\Gamma_{\Nis}(-, \scr H_{\et}^{j}(\bb F_2)). \] Here, the first map is induced by the inclusion of presheaves $\I^{j+1} \rightarrow \I^j$ and the second map is obtained from the comparison map of~\eqref{eq:s-explicit}; see also \cite[Theorem 2.1]{bachmann-etaZ} for a description of this map.

Left Kan extending and cdh sheafifying we get cofiber sequences
\[
R\Gamma_\cdh(-, \I^{j+1})[j]|_{\Sm_{\Z[1/2]}} \to R\Gamma_\cdh(-,\I^{j})[j]|_{\Sm_{\Z[1/2]}} \to R\Gamma_\cdh(-, \scr H_{\et}^{j}(\bb F_2))[j]. 
\]
Via Corollary~\ref{cor:singular_bl} the last term identifies (up to a shift) with the cofiber $\bb F_2(j-1)^{\cdh} \xrightarrow{\cdot \tau} \bb F_2(j)^{\cdh}$ (in other words, it is described as the last term in the cdh-sheafification of the he cofiber sequence~\eqref{eq:tr-cofib} when $p =2$). Therefore, by Theorem~\ref{thm:A1inv}, it is $\A^1$-invariant on schemes on which $2$ is invertible (since $\Hyp(\Z[1/2],2, j)$ holds for all $j$ vacuously by Remark~\ref{rmk:Hyp-vacuous}). Thus the result follows, by induction, from the case $\I^0 = \L_0$, which is exactly Theorem \ref{thm:key-L}.

It remains to establish the $\P^1$-bundle formula, which is done in the same way by induction, Theorem \ref{thm:key-L}, and Corollary~\ref{corol:unconditional-pn}. This is left to the reader. 
\end{proof}

%

\begin{proof}[Proof of Theorem~\ref{thm:hz-tilde-main}] By construction $H\tilde\Z_{\Z[1/2]}$ is very effective and therefore its pullback to any $X$ is very effective; this proves (1). 

Next, let us remark that the square in (2) is cartesian in $\CAlg(\SH(\Z[1/2]))$ by \cite[Definition 4.1, Corollary 4.9]{bachmann-etaZ}. For any $X$, let us write $P_X$ to be the pullback of the cospan $f^*\ul{\K}^W \rightarrow H\Z_X/(2, \tau) \leftarrow H \Z/2_X$ in $\SH(X)$. We therefore have an induced unit map $\eta_X:1_X \rightarrow P_X$; we shall show that this induces an equivalence $\pi_0^{\eff}(1_X)  = H\tilde\Z \rightarrow P_X$. For this, it will suffice to show that $P_X \in \SH(S)^\veff$, $\fib(u_X) \in \SH(X)^\veff[1]$ and, for every $X \in \Sm_S$, $\Map_{\SH(X)}(M_X(Y)[1], P_S) = \ast$. The first two properties are stable under base change, so we can check them when $X=\Spec(\Z[1/2])$, where they are already known. For the third property,  we have a long exact sequence
\[
\cdots \rightarrow  \pi_i(\Map_{\SH(X)}(M_X(Y)[1], P_X)) \rightarrow H_{\bb A}^{1-i}(X, \Z(0)) \oplus H_{\cdh}^{1-i}(X, L_0) \rightarrow H^{1-i}_{\cdh}(-,\bb F_2) \rightarrow \cdots,
\]
after applying the identification of Proposition \ref{prop:determine-KW}. Using Corollary~\ref{corol:low-wts-integral} to identify the weight zero $\bb A^1$-motivic cohomology with cdh cohomology of $\bb Z$, the result then follows from the fact that there is no negative cdh cohomology of an abelian sheaf.

Part (3) then follows immediately from (2) since we have proved that $H\bb Z$ is stable under pullbacks in Theorem~\ref{thm:a1-a1cdh}(2).
\end{proof}

\subsubsection{The very effective cover of $\KO_X$} We now explain $\KO$, the motivic spectrum representing hermitian $K$-theory. For every qcqs scheme $X$,  there is the motivic spectrum \[ \KO_X \in \CAlg(\SH(X)), \] representing homotopy invariant Hermitian $K$ theory; in the literature it is also sometimes called (homotopy) Hermitian $K$-theory.
This was defined in full generality, as a periodic $\bb E_{\infty}$-algebra in \cite[\S 8.1]{calmes2024motivic}, but we will only be interested in the case where $X \in \Sch^\sub{qcqs}_{\Z[1/2]}$ for which already see \cite[\S 7]{HoyoisJelisiejewNardinYakerson2022} and \cite[\S 5]{Carmody2021} for a highly structured statement. This is an absolute motivic spectrum by \cite[Proposition 1.0.7]{calmes2024motivic} or \cite[\S 6]{Carmody2021}. By construction, we have a natural equivalence 
\[
\rm GWH(X) \simeq \omega^{\infty}KO_X(X)
\] for any qcqs scheme $X$, where $\rm GWH$ is the $\A^1$-localisation of the non-connective Hermitian $K$-theory $\rm GW$ of \cite[\S 2.4]{calmes2024motivic} (where this theory is denoted by $\bb G\bb W$).

The following establishes a version of Theorem~\ref{thm:a1-a1cdh}(1) for Hermitian $K$ theory for $j =0$. We will establish a more general result for other effective covers in the next section.

\begin{theorem}\label{thm:base-change-ko-veff} Let $X \in \Sch_{\bb Z[\tfrac{1}{2}]}^{\qcqs}$ with structure map $f: X \to \Spec(\Z[1/2]).$ Then, the canonical map
\[
f^*\tilde{\Fil}_\sub{slice}^0\KO_{\Z[1/2]} \rightarrow \tilde{\Fil}_\sub{slice}^0\KO_X 
\]
is an equivalence. 
\end{theorem}

\begin{proof} According to \cite[Theorem 2.6]{Bachmann2022}, we need only show that $\Fil^1_\sub{slice}f^*H\bb Z_{\bb Z[1/2]} = 0$ and that 
\[
\Map_{\SH(X)}(M(Y)[2], f^*\HW_{\Z[1/2]}) = 0.
\] Here, the spectrum $\HW_{\Z[1/2]}$ can be identified with the $\eta$-periodization of $f^*\ul{\K}^W$ by \cite[Lemma 3.9]{bachmann-etaZ} where $\eta$ is the motivic Hopf map. The first condition was already verified in Corollary~\ref{prop:Fil1-vanishing}, while the second condition follows from the description of $f^*\ul{\K}^W$ in Proposition~\ref{prop:determine-KW}, noting that the $\eta$-periodization (which is given by filtered colimits) does not increase connectivity.

\end{proof}

\subsection{The motivic filtration on $\KO$}\label{sec:motfilt-ko}

We finally come to our main results about Hermitian $K$-theory, a description of the generalised slice filtration of $\KO_X$ for any qcqs scheme $X$ over $\Z[1/2]$.

 \begin{remark}[$4$-periodicity] For any scheme $X$, we have the \emph{Hermitian Bott element} defined as in \cite[Lemma 6.4]{HoyoisJelisiejewNardinYakerson2022} which is of the form $\bb T^{\otimes 4}_X \rightarrow \KO_X$. The induced map $\bb T^{\otimes 4}_X \otimes \KO_X \to \KO_X$ is an equivalence by the proof of [op. cit.]; this is analogous to Bott periodicity for $\KGL_X$. Just as in~\eqref{eq:tate-slice} (see \cite[Lemma 8]{bachmann-very-effective} for a proof), we have
 \[
 \tilde{\Fil}_\sub{slice}^{n+1}(\bb T_X \otimes E) \simeq \bb T_X \otimes \tilde{\Fil}_\sub{slice}^{n}E,
 \]
 where $n$ is possibly a half integer.
 Therefore, the generalised slice filtration on Hermitian $K$-theory, and its refinement, also enjoys periodicity, analogous to the discussion for $\KGL$ in~\S\ref{subsub:slices-kgl}. In particular, to describe all of its slices we need only describe the layers in the following filtration
 \[
  \tilde{\Fil}_\sub{slice}^{0}\KO_X \leftarrow   \tilde{\Fil}_\sub{slice}^{1/2}\KO_X \leftarrow  \tilde{\Fil}_\sub{slice}^{1}\KO_X \cdots  \leftarrow \tilde{\Fil}_\sub{slice}^{3+1/2}\KO_X \leftarrow  \tilde{\Fil}_\sub{slice}^{4}\KO_X.
 \]
 \end{remark}

\begin{theorem}\label{thm:ko}
Let $X \in \Sch^{\qcqs}_{\Z[1/2]}$ and $n \in \Z$.
There are canonical equivalences
\begin{align*}
  \tilde{\Fil}_\sub{slice}^{4n}(\KO_X)/\tilde{\Fil}_\sub{slice}^{4n+1/2}(\KO_X) &\wequi \bb T^{\otimes 4n}\otimes H\tilde{\Z}_X \\
  \tilde{\Fil}_\sub{slice}^{4n+1/2}(\KO_X)/\tilde{\Fil}_\sub{slice}^{4n+1}(\KO_X) &\wequi \bb T^{\otimes 4n} \otimes H\Z/2_X[1] \\
  \tilde{\Fil}_\sub{slice}^{4n+1}(\KO_X)/\tilde{\Fil}_\sub{slice}^{4n+2}(\KO_X) &\wequi \bb T^{\otimes 4n+1}\otimes H\Z/2_X \\
  \tilde{\Fil}_\sub{slice}^{4n+2}(\KO_X)/\tilde{\Fil}_\sub{slice}^{4n+3}(\KO_X) &\wequi \bb T^{\otimes 4n+2}\otimes H\Z_X. \\
   \tilde{\Fil}_\sub{slice}^{4n+3}(\KO_X)/\tilde{\Fil}_\sub{slice}^{4n+4}(\KO_X) &\wequi 0. \\
\end{align*}
\end{theorem}

\begin{remark}[Zero-th generalised slice]\label{rem:zero-gen} In particular we get an equivalence in $\SH(X)$;
\[
\mathrm{cofib}(\tilde{\Fil}_\sub{slice}^{1/2}(\KO_X) \rightarrow \tilde{\Fil}_\sub{slice}^{0}(\KO_X)) \simeq H\tilde{\Z}_X.
\]
This should be compared with the result for $\KGL$ in Corollary~\ref{cor:conjectures}.
\end{remark}

\begin{proof}
We begin with the case $X=\Spec(\Z[1/2])$. The first author, in \cite[above Lemma 2.5]{Bachmann2022} has constructed a $5$-step filtration on $\tilde{\Fil}_\sub{slice}^0 \KO$
\[
F_5 \rightarrow F_4 \rightarrow F_3 \rightarrow F_2 \rightarrow F_1 \rightarrow \tilde{\Fil}_\sub{slice}^0 \KO,
\]
out of which one can extract the following fiber sequences in $\SH(\Z[1/2])$.
\begin{gather*}
  F_1 \to \tilde{\Fil}_\sub{slice}^0 \KO \to H\tilde{\Z} \\
  F_2 \to F_1 \to H\Z/2[1] \\ 
  F_3 \to F_2 \to (\bb T \otimes H\Z)/(2,\tau) \\ 
  F_4 \to F_3 \to  H\Z/2[2] \\ 
  F_5 \to F_4 \to \bb T^{\otimes 2} \otimes H\Z.
\end{gather*}

In brief, the construction is as follows: begin with the framed presheaf on $\Sm_{\Z[1/2]}$ given by the stack $\rm Bil$ of symmetric bilinear forms. The map $\rm Vect \to \rm Bil$ assigning a vector bundle to its associated hyperbolic symmetric bilinear map induces a map on homotopy orbits and group completions $\rm Vect^\sub{gp}_{hC_2} \rightarrow \rm Bil^\sub{gp}$ which is an isomorphism on homotopy sheaves in degrees $1$ and $2$ \cite[Lemma 2.4]{Bachmann2022}. A calculation of the Nisnevich homotopy sheaves of $\rm Vect^\sub{gp}_{hC_2}$ and the $0$-th Nisnevich homotopy sheaf of  $\rm Bil^\sub{gp}$ gives rise to a $4$-step filtration on the framed presheaf $\rm Bil^\sub{gp}$. Via the theory of framed presheaves and using  \cite[Corollary 22]{Hoyois2021}, \cite[Theorem 7.3]{HoyoisJelisiejewNardinYakerson2022}, \cite[Theorem 2.6]{Bachmann2022} we obtain the filtration $F_4 \rightarrow \cdots \rightarrow F_1 \rightarrow \tilde{\Fil}_\sub{slice}^0 \KO$ with the graded pieces as described above. On the other hand, using the stack $\rm Alt$ of alternating forms, we have a map of framed presheaves $\rm Alt \otimes \bb T^{\otimes 2} \rightarrow F_4$ via  \cite[Lemma 2.5]{Bachmann2022} with cofibre $\bb T^{\otimes 2} \otimes H \Z$. We set $F_5$ to be $\rm Alt \otimes \bb T^{\otimes 2}$.

We now proceed in steps. 

\paragraph{Step 1.}
We claim that $F_1 \wequi \tilde{\Fil}_\sub{slice}^{1/2} \KO$. Applying $\tilde{\Fil}_\sub{slice}^{1/2}$ to the fibre sequence $F_1 \to \tilde{\Fil}_\sub{slice}^0 \KO \to H\tilde{\Z}$,  our goal is to prove that $\tilde{\Fil}_\sub{slice}^{1/2}F_1 \to \tilde{\Fil}_\sub{slice}^{1/2}\tilde{\Fil}_\sub{slice}^0 \KO \simeq \tilde{\Fil}_\sub{slice}^{1/2}\KO$ is an equivalence and that $F_1$ is its own $\tilde{\Fil}_\sub{slice}^{1/2}$-cover. For the former, by Remark~\ref{rem:cofiber}, we need to prove that $\Map_{\SH(X)}(M_X(Y)[1], H\tilde{\Z}) \simeq \ast$; this holds by definition of $H\tilde{\Z}$ as $\pi_0^\sub{eff}(1_X)$, hence a discrete object in the $t$-structure on $\SH(X)^{\eff}$. For the latter, by Remark~\ref{rem:detection}, we may verify after pulling back to fields, where it is known by the ``Moreoever'' part of \cite[Theorem 16]{bachmann-very-effective}.

\paragraph{Step 2.}
By the same arguments, we see that $F_2 \wequi \tilde{\Fil}_\sub{slice}^{1} \KO, F_4 \wequi \tilde{\Fil}_\sub{slice}^2 \KO$ and $F_5 \wequi \tilde{\Fil}_\sub{slice}^4 \KO$.
.

\paragraph{Step 3.}
Next, we want to prove that $\tilde{\Fil}_\sub{slice}^1 \KO/\tilde{\Fil}_\sub{slice}^2 \KO \wequi \bb T \otimes H\Z/2$.
Our filtration instead yields a cofiber sequence of the form \[ F_3/F_4 \wequi H\Z/2[2] \to F_2/F_4 \to F_2/F_3 \wequi (\bb T \otimes H\Z)/(2,\tau), \] and we will eventually need to show that this identifies with the cofiber sequence \[ H\Z/2[2] \xrightarrow{\tau} \bb T \otimes H\Z/2 \to (\bb T \otimes H\Z)/(2,\tau). \]
For now, we need only observe that this identification works over fields.
This is true essentially by construction, see \cite[Lemma 2.4]{bachmann-etaZ} and its proof.

\paragraph{Step 4.}
We now prove that $F_2/F_4 \wequi H\Z/2$.
The intuition here is that $\tilde{\Fil}_\sub{slice}^{1/2} \KO /\tilde{\Fil}_\sub{slice}^1 \KO$ and $\tilde{\Fil}_\sub{slice}^2\KO/\tilde{\Fil}_\sub{slice}^1\KO$ should be linked by multiplication by $\eta$, as is the case in the classical orthogonal $K$-theory spectrum, and for $\KO$ over fields.
We thus contemplate the diagram
\begin{equation*}
\begin{tikzcd}
              & \tilde{\Fil}_\sub{slice}^2 \KO \ar[d, "a"] \\
\tilde{\Fil}_\sub{slice}^1\KO \otimes \Gm \ar[d] \ar[r, "\eta"] \ar[ur, dashed] & \tilde{\Fil}_\sub{slice}^1 \KO \ar[d, "b"] \\
\tilde{\Fil}_\sub{slice}^{1/2}\KO \otimes \Gm \ar[r, "\eta"], \ar[ur, dashed] & \tilde{\Fil}_\sub{slice}^{1/2} \KO.
\end{tikzcd}
\end{equation*}
The solid square commutes.
By what we have said, we know $\cof(b) \wequi H\Z/2[1]$ and so $\omega^\infty \bb T^{\otimes -1}[1] \otimes \cof(b) = 0$.
Thus the lower dashed arrow exists uniquely.
A similar argument applies to the upper dashed arrow.
We hence obtain a commutative square involving the dashed arrows.
Taking vertical cofibers in this new square we obtain a canonical map \[ \bb T \otimes H\Z/2 \wequi (\tilde{\Fil}_\sub{slice}^{1/2}\KO/\tilde{\Fil}_\sub{slice}^1\KO) \otimes \Gm \to \tilde{\Fil}_\sub{slice}^1\KO/\tilde{\Fil}_\sub{slice}^2\KO. \]
It remains to prove that this is an equivalence, which we may check on fields, where the result is known \cite{bachmann-very-effective} \NB{actually prove this?}.

\paragraph{Conclusion of the proof.}
We have now proved the result in case $X=\Spec(\Z[1/2])$.
It remains to deal with general $X$ with structure map $f: X \rightarrow \Spec(\Z[1/2])$.
Theorem~\ref{thm:base-change-ko-veff} implies that the very effective slice cover is stable under base change. Hence it remains to prove that the map $f^*\tilde{\Fil}_\sub{slice}^{\alpha}(\KO_{\Z[1/2]}) \to \tilde{\Fil}_\sub{slice}^{\alpha}(\KO_{X})$ is an equivalence, for $\alpha \in \{4n,4n+1/2,4n+1,4n+2 \mid n \in \Z \}$.
This will conclude the proof.
By periodicity we may assume $n=0$.
We illustrate the case $\alpha=1/2$, the others are similar.
We have the fiber sequence \[ f^*(\tilde{\Fil}_\sub{slice}^{1/2}(\KO_{\Z[1/2]})) \to f^*(\tilde{\Fil}_\sub{slice}^{0}(\KO_{\Z[1/2]})) \to f^*H\tilde{\Z}_{\Z[1/2]}. \]
Using Theorems~\ref{thm:hz-tilde-main}(3) and~\ref{thm:base-change-ko-veff}, the above identifies with a fibre sequence
\[ f^*(\tilde{\Fil}_\sub{slice}^{1/2}(\KO_{\Z[1/2]})) \to \tilde{\Fil}_\sub{slice}^{0}\KO_X \to H\tilde{\Z}_{X}. \]
Since $\tilde{\Fil}_\sub{slice}^{1/2} \wequi \tilde{\Fil}_\sub{slice}^{1/2}\tilde{\Fil}_\sub{slice}^0$, to prove that $f^*\tilde{\Fil}_\sub{slice}^{1/2}(\KO_{\Z[1/2]}) \to \tilde{\Fil}_\sub{slice}^{1/2}\KO_{X}$ is an equivalence it will suffice, by Remark~\ref{rem:cofiber}, to show that $f^*(\tilde{\Fil}_\sub{slice}^{1/2}(\KO_{\Z[1/2]})) \in \SH(S)^\veff[1]$ and $\Map(M(Y)[1],H\tilde{\Z}_X)\simeq \ast$ for all $Y \in \Sm_X$. Both statements are clear.
\end{proof}

\begin{remark}\label{rem:zz2z20z000z} From Theorem~\ref{thm:ko} one can formally deduce the other half slices:
\begin{align*}
  \tilde{\Fil}_\sub{slice}^{4n}(\KO_X)/\tilde{\Fil}_\sub{slice}^{4n+1/2}(\KO_X) &\wequi \bb T^{\otimes 4n}\otimes H\tilde{\Z}_X \\
  \tilde{\Fil}_\sub{slice}^{4n+1/2}(\KO_X)/\tilde{\Fil}_\sub{slice}^{4n+1}(\KO_X) &\wequi \bb T^{\otimes 4n} \otimes H\Z/2_X[1] \\
  \tilde{\Fil}_\sub{slice}^{4n+1}(\KO_X)/\tilde{\Fil}_\sub{slice}^{4n+3/2}(\KO_X) & \wequi \bb T^{\otimes 4n+1}\otimes H\Z/2_X\\
  \tilde{\Fil}_\sub{slice}^{4n+3/2}(\KO_X)/\tilde{\Fil}_\sub{slice}^{4n+2}(\KO_X) & =  0 \\
  \tilde{\Fil}_\sub{slice}^{4n+2}(\KO_X)/\tilde{\Fil}_\sub{slice}^{4n+5/2}(\KO_X) &\wequi \bb T^{\otimes 4n+2}\otimes H\Z_X\\
    \tilde{\Fil}_\sub{slice}^{4n+5/2}(\KO_X)/\tilde{\Fil}_\sub{slice}^{4n+3}(\KO_X)  = 0 \\
   \tilde{\Fil}_\sub{slice}^{4n+3}(\KO_X)/\tilde{\Fil}_\sub{slice}^{4n+7/2}(\KO_X) &=  0 \\
      \tilde{\Fil}_\sub{slice}^{4n+7/2}(\KO_X)/\tilde{\Fil}_\sub{slice}^{8n}(\KO_X) & = 0 \\
\end{align*}
In this way, we offer an algebro-geometric analog of  Bott's calculation of the homotopy groups of real $K$-theory in topology.
\end{remark}

As noted earlier, it is not clear that the slices from Remark~\ref{rem:zz2z20z000z} assembles into a multiplicative graded object. In lieu of this, we offer the following description of a multiplicative filtration on $\rm GWH$.

\begin{theorem} \label{thm:ko-main}
There exist a functorial, multiplicative, $\bb Z$-indexed, exhaustive filtration 
\[
\mathrm{Fil}_{\bb A}^{\star}\rm GWH(X) \to GWH(X)
\] 
on $\rm GWH(X)$ whose graded pieces are naturally given as follows for $j \in  \bb Z$:
\begin{equation}\label{eq:gr}
\rm gr^j_{\bb A}GWH(X) \simeq \begin{cases}
\omega^{\infty}\widetilde{s}^j(\KO_X) & j \equiv 0 \mod 4\\
\Z/2(j-1)^{\bb A}(X)[2j-2] & j \equiv 1 \mod 4\\
\Z(j-1)^{\bb A}(X)[2j-2] & j \equiv 2 \mod 4\\
0 & j \equiv 3 \mod 4. 
\end{cases}
\end{equation}
Furthermore:
\begin{enumerate}
\item We have a functorial cofibre sequence for $j \equiv 0$ mod $4$
\[
\bb Z/2(j)^{\bb A}(X)[2j+1] \to \rm gr^j_{\bb A}GWH(X) \to \tilde{\Z}(j)^{\bb A}(X)[2j].
\]
\item The filtration is complete for any noetherian scheme of finite dimension.  
\end{enumerate}
\end{theorem}

\begin{proof} We set
\[
\Fil^j_{\bb A}\rm GWH(X) := \omega^{\infty}\tilde{\Fil}_\sub{slice}^j(\KO_X)
\]

Then $\Fil^j_{\bb A}\rm GWH(X)$ is clearly a multiplicative, $\bb Z$-indexed filtration which is complete and converges to $\omega^{\infty}\KO_X(X) = \rm GWH(X)$. The description of the graded pieces and point (1) follow immediately from Theorem~\ref{thm:ko}. Claim (2) about completeness of filtration follows from Proposition~\ref{prop:detecting-veff}.

\end{proof}

\begin{remark} Using the results of \cite{calmes2024motivic} and the same arguments as in Corollary~\ref{cor:bdd-a1} we can improve completeness of the filtration in Theorem~\ref{thm:ko-main} to include qcqs schemes of finite valuative dimension. We leave details to the interested reader. 
\end{remark}

\bibliographystyle{acm}
\bibliography{MMBigBibliography}

\end{document}